\documentclass[12pt]{amsart}
 
 \makeatletter
\def\@tocline#1#2#3#4#5#6#7{\relax
  \ifnum #1>\c@tocdepth 
  \else
    \par \addpenalty\@secpenalty\addvspace{#2}%
    \begingroup \hyphenpenalty\@M
    \@ifempty{#4}{%
      \@tempdima\csname r@tocindent\number#1\endcsname\relax
    }{%
      \@tempdima#4\relax
    }%
    \parindent\z@ \leftskip#3\relax \advance\leftskip\@tempdima\relax
    \rightskip\@pnumwidth plus4em \parfillskip-\@pnumwidth
    #5\leavevmode\hskip-\@tempdima
      \ifcase #1
       \or\or \hskip 1em \or \hskip 2em \else \hskip 3em \fi%
      #6\nobreak\relax
    \dotfill\hbox to\@pnumwidth{\@tocpagenum{#7}}\par
    \nobreak
    \endgroup
  \fi}
\makeatother

\usepackage{charter}

\usepackage{setspace}
\usepackage{macros}
\standardmargins\standardrenews\standardlabeling\standardmakes
\colorcommentstrue

\newcommand{\FS}{\mathrm{FS}}

\draftfalse

\RequirePackage{color}
\RequirePackage[usenames,dvipsnames]{xcolor}

\usepackage{pgf,tikz,pgfplots}
\pgfplotsset{compat=1.15}
\usepackage{mathrsfs}
\usetikzlibrary{arrows}

\newcommand\redball[1]{\fill[color = red, opacity=0.76] #1; \draw #1;}

\usepackage{pgfplots}
\pgfplotsset{compat=1.8}
\usepackage{mathtools}

\usepackage{esint}

\usepackage[hidelinks]{hyperref}

\newcommand{\Sing}{\text{Sing}}
\newcommand{\VSing}{\text{VSing}}

\newcommand{\tbeta}{\gamma}

\newcommand{\tspace}{\eta}
\newcommand{\tstar}{\delta_*}
\newcommand{\tbigspace}{\delta_{**}}

\newcommand{\cspace}{C_\eta}
\newcommand{\kspace}{k_\eta}

\newcommand{\dset}{\OC 0d_\Z}

\newcommand{\turn}{k}

\newcommand{\dimsing}{\delta_\dims}
\newcommand{\dirichlet}{{\tfrac\qdim\pdim}}

\renewcommand{\SS}{\mathcal S}

\renewcommand{\bfA}{A}
\renewcommand{\bfB}{X}
\renewcommand{\bfC}{Y}
\renewcommand{\bfD}{Z}

\newcommand{\bigOC}[2]{\big(#1,#2\big]}

\renewcommand{\index}{k}
\newcommand{\mink}{h}
\newcommand{\Mink}{\hh}
\newcommand{\pert}{b}
\newcommand{\Pert}{\bb}
\newcommand{\exceptionalset}{\{\tfrac1\pdim,-\tfrac1\qdim\}}

\newcommand{\timeerror}{s}

\newcommand{\dynexp}{\tau}

\newcommand{\dir}{L}
\newcommand{\diff}{M}
\renewcommand{\vert}{\LL_-}
\newcommand{\niol}{N}

\begin{document}
\title[A variational principle in the parametric geometry of numbers]{A variational principle in the parametric geometry of numbers}

\authortushar\authorlior\authordavid\authormariusz

\subjclass[2020]{11K55, 11J13 (primary), 28A80, 28A78, 37A15, 37A17, 37C85, 37D40, 91A05, 91A44 (secondary)}
\keywords{Diophantine approximation, Hausdorff dimension, packing dimension, geometry of numbers, lattices, simultaneous approximation, successive minima, Schmidt games, topological games, homogeneous dynamics, divergent trajectories, diagonal flows}
\dedicatory{
Dedicated to S. G. Dani, G. M. Margulis, and W. M. Schmidt
}

\begin{Abstract}
We extend the parametric geometry of numbers (initiated by Schmidt and Summerer, and deepened by Roy) to Diophantine approximation for systems of $m$ linear forms in $n$ variables, and establish a new connection to the metric theory via a variational principle that computes fractal dimensions of a variety of sets of number-theoretic interest. The proof of our variational principle relies on two novel ingredients: a variant of Schmidt’s game capable of computing the Hausdorff and packing dimensions of any set, and the notion of \textsf{templates}, which generalize Roy’s \textsf{rigid systems}. We use our variational principle to compute the Hausdorff and packing dimensions of the set of singular systems of linear forms and show they are equal, resolving a conjecture of Kadyrov, Kleinbock, Lindenstrauss and Margulis, as well as a question of Bugeaud, Cheung and Chevallier. As a corollary of Dani's correspondence principle, the divergent trajectories of a one-parameter diagonal action on the space of unimodular lattices with exactly two Lyapunov exponents with opposite signs has equal Hausdorff and packing dimensions. 
Other applications include quantitative strengthenings of theorems due to Cheung and Moshchevitin, which originally resolved conjectures due to Starkov and Schmidt respectively; as well as dimension formulas with respect to the uniform exponent of irrationality for simultaneous and dual approximation in two dimensions, completing partial results due to Baker, Bugeaud, Cheung, Chevallier, Dodson, Laurent and Rynne.
\end{Abstract}
\maketitle

\tableofcontents

\draftnewpage
\part{Introduction}
\label{partintroduction}

\section{Readers' Guide}
\label{sectionreadersguide}

The following brief guide will aid non-linear navigation across the paper. To prevent misunderstanding, the reader should first acquaint themselves with {\bf Conventions 1 through 7}, which may be found at the start of Section \6\ref{sectionconventions}. 
The conventions are followed by a glossary of notation (in the order of their appearance), which may be skipped on a first reading.
After the conventions one must read {\bf Section \6\ref{sectionmain} (Main Results)} and Section {\bf \6\ref{sectionvariational} (The Variational Principle)}, which contain statements of all the main theorems as well as fundamental definitions that are germane to the sequel. Section \6\ref{sectionfuture} contains a sample of future research directions.
The several theorems of Section \6\ref{sectionmain} are all consequences of a single \emph{variational principle} in the parametric geometry of numbers, which provides a unifying perspective to both old and new  results in the metric theory of Diophantine approximation. Theorem \ref{theoremvariationaluniform} in {Section \6\ref{sectionvariational}} is the version of this variational principle we prove in the sequel.

At this stage, there are a few potential routes ahead. Readers keen to get directly to the various applications in Section \6\ref{sectionmain} could take the variational principle (Theorem \ref{theoremvariationaluniform}) for granted and move directly to {\bf Part \ref{partmainproofs} (Proofs of main theorems using the variational principle)}. This allows one to better familiarize themselves with how to apply the variational principle before entering the myriad details that its intricate proof entails.

An alternate route would be to skip the proofs of the applications in Part \ref{partmainproofs}, and instead move straight to the heart of the paper, viz. our proof of the variational principle (Theorem \ref{theoremvariationaluniform}). This proof involves reading {\bf Part \ref{partgames} (Dimension games)} and {\bf Part \ref{partproofvariational} (Proof of the variational principle)} in order. 
We note that the proof of the upper bound in Section \6\ref{sectionupper} is significantly shorter than that of the lower bound in Section \6\ref{sectionlower}. 

Readers particularly interested in our variant of Schmidt's game (that computes the Hausdorff and packing dimensions of any Borel set in a doubling metric space) may read {\bf Section \6\ref{sectiondimensionprelim} (Preliminaries on measures and dimensions)} and {\bf Section \6\ref{sectiongame} (A characterization of Hausdorff and packing dimensions using games)} (both in Part \ref{partgames}) independently of all other sections in the paper.  

\section{Conventions and Glossary of Notation}
\label{sectionconventions}

We begin with our most important conventions, which should not be skipped and may be especially useful for a non-linear reader.

\begin{convention}
We denote the nonnegative integers as $\N \df \{0,1,2,\ldots\}$.
\end{convention}
\begin{convention}
Where applicable, the nonzero integers $m$, $n$, and $d \df m + n$ are treated as constant.
\end{convention}
\begin{convention}
\label{measureconvention}
All measures and sets are assumed to be Borel, and measures are assumed to be locally finite. Sometimes we restate these hypotheses for emphasis.
\end{convention}
\begin{convention}
We use uppercase letters $X,Y,\dots$ for matrices and bold letters $\xx,\yy,\dots$ for vectors. 
\end{convention}
\begin{convention}
\label{spanconvention}
Given a vector space $V$ and some index set $I$ we use the notation 
\[
\lb x_i \in V  :  i \in I \rb
\] 
to mean the subspace generated by $\{ x_i \in V  :  i \in I \}$, or the smallest subspace containing $\{ x_i \in V  :  i \in I \}$.
\end{convention}
\begin{convention}
In what follows, $A \lesssim B$ or $A \lesssim_\times B$ means that there exists a constant $C$ (the \emph{implied constant}) such that $A \leq CB$. $A\asymp B$ or $A \asymp_\times B$ means $A \lesssim B \lesssim A$. Similarly, $A\lesssim_\plus B$ means that $A \leq B + C$ for some constant $C$.
When we write $A \lesssim_{\beta} B$ or $A \lesssim_{\plus,\beta} B$ this signifies that the implied constant depends on $\beta$.
We use $A\asymp_+ B$ to mean $A\lesssim_\plus B$ and $B\lesssim_\plus A$. For instance, this allows us to write $A \asymp_+ B = C \asymp_\plus D$ without having to write $O(1)$ everywhere, which would obscure some of the information and also be more cluttered.
\end{convention}
\begin{convention}
Recall that $\Theta(x)$ denotes any number such that $x/C \leq \Theta(x) \leq C x$ for some uniform constant $C$. Similarly, $\Omega(x)$ and $O(x)$ denote numbers such that $x/C \leq \Omega(x)$ and $|O(x)| \leq C x$ for some uniform positive constant $C$, respectively.
\end{convention}

\hrulefill

\begin{gloss*}
\emph{For the reader's convenience we summarize a partial list of notations and terminology in the order that they appear in the sequel.}
\begin{itemize}
\item{\hyperref[]{$\N$}} \dotfill The non-negative integers
\item{\hyperref[matrixspace]{$\MM$}} \dotfill The space of $m\times n$ matrices with real entries
\item{\hyperref[sing]{$\Sing(\pdim,\qdim)$}} \dotfill The set of singular $\pdim\times \qdim$ matrices
\item{\hyperref[theoremsing]{$\dimsing$}} \dotfill $\dimsing \df \dimprod\big(1 - \tfrac1{\dimsum}\big)$
\item{\hyperref[BA]{$\mathrm{BA}(\pdim,\qdim)$}} \dotfill The set of badly approximable $\pdim\times \qdim$ matrices
\item{\hyperref[VWA]{$\mathrm{VWA}(\pdim,\qdim)$}} \dotfill The set of very well approximable $\pdim\times \qdim$ matrices
\item{\hyperref[sectiondimensionprelim]{$\HD(S)$}} \dotfill The Hausdorff dimension of a set $S$
\item{\hyperref[sectiondimensionprelim]{$\PD(S)$}} \dotfill  The packing dimension of a set $S$
\item{\hyperref[subsectionDani]{$\mathrm I_k$}} \dotfill The $k$-dimensional identity matrix
\item{\hyperref[subsectionDani]{$d$}} \dotfill $d \df \pdim + \qdim$
\item{\hyperref[subsectionDani]{$\lambda_j(\Lambda)$ $(1\leq j \leq d)$}} \dotfill The $j$th minimum of a lattice $\Lambda \subset \R^d$
\item{\hyperref[subsectionDani]{$g_t$}} \dotfill For $t\in\R$, $g_t \df \left[\begin{array}{ll}
e^{t/\pdim} \mathrm I_\pdim &\\
& e^{-t/\qdim} \mathrm I_\qdim
\end{array}\right] \in \SL_d(\R)$
\item{\hyperref[subsectionDani]{$u_\bfA$}} \dotfill For an $\pdim\times \qdim$ matrix $\bfA$, $u_\bfA \df \left[\begin{array}{ll}
\mathrm I_\pdim & \bfA\\
& \mathrm I_\qdim
\end{array}\right] \in \SL_d(\R)$
\item{\hyperref[UEI]{$\what\omega(\bfA)$}} \dotfill The uniform exponent of irrationality of an $\pdim\times \qdim$ matrix $\bfA$
\item{\hyperref[vsing]{$\VSing(\pdim,\qdim)$}} \dotfill The set of very singular $\pdim\times \qdim$ matrices, i.e. $\{\bfA:\what\omega(\bfA) > \qdim/\pdim\}$
\item{\hyperref[theoremdani2]{$\what\dynexp(\bfA)$}} \dotfill $\what\dynexp(\bfA) \df \liminf_{t\to\infty} \frac{-1}t\log\lambda_1(g_t u_\bfA \Z^d)$
\item{\hyperref[dani]{$\tau$}} \dotfill $\tau = \frac1\qdim \frac{\omega - \dirichlet}{\omega + 1}$
\item{\hyperref[singomega]{$\Sing_\dims(\omega)$}} \dotfill $\Sing_\dims(\omega) \df \{\bfA : \what\omega(\bfA) = \omega\} = \{\bfA : \what\tau(\bfA) = \tau\}$
\item{\hyperref[remarktriviallysingular]{\emph{trivially singular}}} \dotfill See Section $\6$ \ref{remarktriviallysingular}
\item{\hyperref[remarktriviallysingular]{$\Sing_\dims^*(\omega)$}} \dotfill $\Sing_\dims^*(\omega) \df \{\bfA\in \Sing_\dims(\omega) : \bfA\text{ is not trivially singular}\}$
\item{\hyperref[cusp]{$\PP(\bfA)$}} \dotfill $\PP(\bfA) \df \lim_{\epsilon\to 0} \liminf_{T\to\infty} \frac1T \lambda\big(\big\{t\in [0,T] : \lambda_1(g_t u_\bfA \Z^d) \leq \epsilon\big\}\big)$
\item{\hyperref[cusp]{\emph{singular on average}}} \dotfill $\bfA$ is \emph{singular on average} if $\PP(\bfA) = 1$
\item{\hyperref[defksingulargeneral]{\emph{$k$-singular}}} \dotfill See Def. \ref{defksingulargeneral}
\item{$\{x\}$} \dotfill $\{x\}$ denotes the fractional part of $x \in \R$
\item{\hyperref[fkdef]{$f_\dims(k)$}} \dotfill $f_\dims(k) \df \dimprod - \frac{k(\dimsum-k)\dimprod}{(\dimsum)^2} - \left\{\frac{k\pdim}{\dimsum}\right\} \left\{\frac{k\qdim}{\dimsum}\right\}$
\item{\hyperref[hitdef]{$\Mink,\Mink_\bfA, \mink_i(t)$}} \dotfill $\Mink = \Mink_\bfA = (\mink_1,\ldots,\mink_d) : \Rplus \to \R^d$, \hyperref[hitdef]{$\mink_i(t) \df \log\lambda_i(g_t u_\bfA \Z^d)$}
\item{\hyperref[Vjt]{$V_{j,t}$}} \dotfill $V_{j,t} \df \mathrm{span}_{\R}\{\rr\in \Z^d : \|g_t u_\bfA \rr\| \leq \lambda_j(g_t u_\bfA \Z^d)\}$
\item{\hyperref[FjI]{$F_{j,I}(t)$}} \dotfill $F_{j,I}(t) \df \log\|g_t u_\bfA (V_{j,t}\cap \Z^d)\|$
\item{\hyperref[slopeset]{\emph{$d_+$, $d_-$}}} \dotfill $d_+ \df \pdim$ and $d_- \df \qdim$
\item{\hyperref[slopeset]{$[a,b]_\Z$}} \dotfill $[a,b]_\Z \df [a,b]\cap\Z$
\item{\hyperref[slopeset]{$Z(j)$}} \dotfill $Z(j) \df \left\{\tfrac{\dir_+}{\pdim} - \tfrac{\dir_-}{\qdim}: \dir_\pm \in [0,d_\pm]_\Z,\;\;\dir_+ + \dir_- = j\right\}$

\item{\hyperref[definitiontemplate]{\emph{template}}} \dotfill See Def. \ref{definitiontemplate}
\item{\hyperref[definitiontemplate]{\emph{balanced template}}} \dotfill See Def. \ref{definitiontemplate}
\item{\hyperref[definitiontemplate]{\emph{partial template}}} \dotfill See Def. \ref{definitiontemplate}
\item{\hyperref[definitiontemplate]{$\TT_{\dims}$}} \dotfill The space of $\pdim\times\qdim$ templates
\item{\hyperref[definitiontemplate]{$F_j$}} \dotfill $F_j \df \sum_{0 < i\leq j} f_i$ for a map $\ff:\Rplus\to\R^d$
\item{\hyperref[definitiontemplate]{\emph{convexity condition}}} \dotfill $F_j$ is convex when $f_j < f_{j+1}$
\item{\hyperref[definitiontemplate]{\emph{quantized slope condition}}} \dotfill Slopes of the pieces of $F_j$ are in $Z(j)$ when $f_j < f_{j+1}$
\item{\hyperref[definitiontemplate]{$f_0, f_{d+1}$}} \dotfill $f_0 \df -\infty$ and $f_{d+1} \df +\infty$
\item{\hyperref[Df]{$\DD(\ff)$}} \dotfill $\DD(\ff) \df \{\bfA : \Mink_\bfA \asymp_\plus \ff\}$
\item{\hyperref[DF]{$\DD(\FF)$}} \dotfill $\DD(\FF) \df \bigcup_{\ff\in\FF} \DD(\ff)$
\item{\hyperref[NSC]{$\NN(\SS,C)$}} \dotfill $\NN(\SS,C) \df \{\gg:\Rplus\to\R^d : \|\gg - \ff\| \leq C \text{ for some } \ff\in \SS\}$
\item{\hyperref[lowerbound]{$\DD(\ff,C)$}} \dotfill $\DD(\ff,C) \df \DD(\NN(\{\ff\},C))$
\item{\hyperref[MSdef]{$\DD(\SS)$}} \dotfill $\DD(\SS) = \{\bfA \in \MM : \Mink_\bfA \in \SS\}$
\item{\hyperref[Lqdef]{$\dir_\pm = \dir_\pm(\ff,I,q)$} \dotfill Chosen so that $\dir_+ + \dir_- = q$ and $F_q' = \frac{\dir_+}{m} - \frac{\dir_-}{n}$ on $I$}
\item{\hyperref[Mpqdef]{$\diff_\pm = \diff_\pm(p,q) = \diff_\pm(\ff,I,p,q)$} \dotfill $\diff_\pm(p,q) = \dir_\pm(q) - \dir_\pm(p)$}
\item{\hyperref[definitionintervalequality]{\emph{interval of equality}}} \dotfill See Def. \ref{definitiondimtemplate}
\item{\hyperref[Splusdef1]{$S_\pm = S_\pm(\ff,I)$}} \dotfill See \eqref{Splusdef1} and \eqref{Sminusdef1} in Def. \ref{definitiondimtemplate}
\item{\hyperref[dimfI]{$\delta(\ff,I)$}} \dotfill See Def. \ref{definitiondimtemplate}
\item{\hyperref[DeltafT]{$\Delta(\ff,T)$}} \dotfill See Def. \ref{definitiondimtemplate}
\item{\hyperref[DeltaT1T2]{$\Delta(\ff,[T_1,T_2])$}} \dotfill See Def. \ref{definitiondimtemplate}
\item{\hyperref[deltaT+T-]{$\delta(T_+,T_-) $}} \dotfill See Def. \ref{definitiondimtemplate}

%
%
%

\item{\hyperref[definitiondimtemplate]{$\underline\delta(\ff), \overline\delta(\ff)$}} \dotfill Lower and upper average contraction rates of a template $\ff$

\item{\hyperref[UDE]{$\what\dynexp(\ff)$}} \dotfill The uniform dynamical exponent of $\ff$: $\what\dynexp(\ff) \df \liminf_{t\to\infty} \frac{-1}t f_1(t)$
\item{\hyperref[super]{$\w\Sing_\dims^*(\omega)$}} \dotfill $\w\Sing_\dims^*(\omega) \df \{\bfA:\what\omega(\bfA) \geq \omega,\;\bfA\text{ not trivially singular}\}$

\item{\hyperref[definitionintervallinearity]{\emph{interval of linearity}}} \dotfill See Lemma \ref{lemmaphiprimebound}
\item{\hyperref[definitionstandardtemplate]{\emph{standard template for a pair of points}}} \dotfill See Def. \ref{definitionstandardtemplate}
\item{\hyperref[definitionstandardtemplate]{$\mbf s[(t_k,-\epsilon_k),(t_{k+1},-\epsilon_{k+1})]$}} \dotfill See Def. \ref{definitionstandardtemplate}

\item{\hyperref[definitionstandardtemplate2parameter]{\emph{standard template for a sequence of points}}} \dotfill See Def. \ref{definitionstandardtemplate2parameter}

\item{\hyperref[definitionstandardtemplate2parameter]{\emph{$\ff[\tau,\lambda]$, the standard template for parameters $\tau\geq 0$ and $\lambda > 1$}}} \dotfill See Def. \ref{definitionstandardtemplate2parameter}

\item{\hyperref[expequivariance]{\emph{exponential $\lambda$-equivariance}}} \dotfill See Def. \ref{definitionstandardtemplate2parameter}

\item{\hyperref[HM]{$\scrH^s(A)$}} \dotfill The $s$-dimensional Hausdorff measure of a set $A\subset\R^d$
\item{\hyperref[PM]{$\scrP^s(A)$}} \dotfill The $s$-dimensional packing measure of a set $A\subset\R^d$
\item{\hyperref[HDPD]{$\HD(A)$}} \dotfill The Hausdorff dimension of a set $A\subset\R^d$
\item{\hyperref[HDPD]{$\PD(A)$}} \dotfill The packing dimension of a set $A\subset\R^d$
\item{$B(\xx,\rho)$} \dotfill The closed ball centered at $\xx \in \R^d$ with radius $\rho>0$
\item{\hyperref[nbhd]{$\NN(A,\epsilon)$}}  \dotfill The $\epsilon$-neighborhood of a set $A\subset\R^d$  
\item{\hyperref[localdim]{$\underline\dim_\xx(\mu)$}} \dotfill $\underline\dim_\xx(\mu) \df \liminf_{\rho\to 0} \log\mu(B(\xx,\rho)) / \log\rho$
\item{\hyperref[localdim]{$\overline\dim_\xx(\mu)$}} \dotfill $\overline\dim_\xx(\mu) \df \limsup_{\rho\to 0} \log\mu(B(\xx,\rho)) \ \log\rho$

\item{\hyperref[definitiongames1]{\emph{$\delta$-dimensional Hausdorff and packing $\beta$-games}}} \dotfill See Def. \ref{definitiongames1}

\item{\hyperref[rhosep]{\emph{$\rho$-separated set}}} \dotfill See Footnote \ref{rhosep}

\item{\hyperref[Hausdorff]{$\underline\delta(\AA)$}} \dotfill $\underline\delta(\AA) \df \liminf_{k\to\infty} \frac{1}{k} \sum_{i = 0}^k - \log\#(A_i)/ \log(\beta)$
\item{\hyperref[packing]{$\overline\delta(\AA)$}} \dotfill $\overline\delta(\AA) \df \limsup_{k\to\infty} \frac{1}{k} \sum_{i = 0}^k - \log\#(A_i)/ \log(\beta)$

\item{\hyperref[definitionmodifiedgame]{\emph{modified $\delta$-dimensional Hausdorff and packing $\beta$-games}}} \dotfill See Section \ref{sectiongamesgeometry}
\item{\hyperref[outcome2]{\emph{$\xx_\infty$}}} \dotfill $\xx_\infty \df \xx_0 + \sum_{k = 1}^\infty \beta^k \rho_{-1} \xx_k$, see \eqref{outcome2}

\item{\hyperref[Lambdakdef]{\emph{$\alpha$}}} \dotfill $\alpha \df \frac{\dimprod}{\dimsum}$, in \eqref{Lambdakdef}
\item{\hyperref[Lambdakdef]{\emph{$Y_k$}}} \dotfill $Y_k \df X_0 + \sum_{i=1}^k \beta^i \rho_{-1} X_i$, in \eqref{Lambdakdef}
\item{\hyperref[Lambdakdef]{\emph{$\Lambda_{k+1}$}}} \dotfill $\Lambda_{k+1} \df g_{-\alpha\log(\beta^{k+1} \rho_{-1})} u_{Y_k} \Z^{\dimsum}$, in \eqref{Lambdakdef}
\item{\hyperref[gamma]{\emph{$\tbeta$ and $g$}}} \dotfill $\tbeta \df -\alpha\log(\beta) > 0$ and $g \df g_\tbeta$, in Notation \ref{gamma}

\item{\hyperref[hLambda]{$\Mink(\Lambda)$} \dotfill $\hh(\Lambda) \df (\log\lambda_1(\Lambda),\ldots,\log\lambda_d(\Lambda))$}

\item{\hyperref[mink2]{\emph{Minkowski's second theorem}}} \dotfill Theorem \ref{mink2} 
 \item{\hyperref[definitionLambdarational]{\emph{$\Lambda$-rational}}} \dotfill A subspace $V\subset \R^d$ is $\Lambda$-rational if $V\cap\Lambda$ is a lattice in $V$, Def. \ref{definitionLambdarational}
\item{\hyperref[Lambdarational]{$\VV_q(\Lambda)$}} \dotfill Set of all $q$-dimensional $\Lambda$-rational subspaces of $\R^d$
\item{\hyperref[Lambdarational]{$\|V\|$}} \dotfill Covolume of $V\cap\Lambda$ in $V$, see Notation \ref{Lambdarational}
\item{\hyperref[contracting]{$\vert$}} \dotfill $\vert \df \{\0\}\times \R^\qdim$ is the subspace of $\R^d$ contracted by the $(g_t)$ flow
\item{\hyperref[contracting]{$\CC(V,\epsilon)$}} \dotfill Conical $\epsilon$-neighborhood of a subspace $V \subset \R^d$
\item{\hyperref[definitionintegral]{\emph{$\tspace$-integral}}} \dotfill see Def. \ref{definitionintegral}
\item{\hyperref[definitionsimple]{\emph{splits, mergers, transfers}}} \dotfill See Def. \ref{definitionsimple}
\item{\hyperref[definitionsimple]{\emph{simple}}} \dotfill See Def. \ref{definitionsimple}

\item{\hyperref[definitiontstar]{\emph{$\tstar$}}} \dotfill $\tstar \df d\cdot(d^2)!\tspace \in \tspace\N$ in proof of Lemma \ref{lemmafindtemplate}
\item{\hyperref[definitiontbigspace]{\emph{$\tbigspace$}}} \dotfill $\tbigspace \df \dimprod d^{4d}\tstar \in \tspace\N$ in proof of Lemma \ref{lemmafindtemplate}

\item{\hyperref[definitionconvexhull]{\emph{convex hull function}}} \dotfill See Def. \ref{definitionconvexhull}

\item{\hyperref[definitionCmatch]{\emph{$C$-match}}} \dotfill See Def. \ref{definitionCmatch} 

\item{\hyperref[good1]{\emph{good on turn $k$}}} \dotfill See \eqref{good1} and \eqref{good2}
\item{\hyperref[grass]{$\GG = \GG(d,n) = \GG_n(\R^d)$}} \dotfill The Grassmannian variety of $n$-dim. subspaces of $\R^d$

\item{\hyperref[lemmahperturbation]{\emph{$\Pert$-perturbation} of $\ff$ at $t_0$}} \dotfill See Lemma \ref{lemmahperturbation}
\item{\hyperref[qinterval]{\emph{$q$-interval}}} \dotfill See page \pageref{qinterval}
\item{\hyperref[intervalmixing]{\emph{interval of mixing}}} \dotfill See page \pageref{intervalmixing}


\end{itemize}
\end{gloss*}

\section{Statements of Main results}
\label{sectionmain}

The notion of singularity (in the sense of Diophantine approximation) was introduced by Khintchine, first in 1937 in the setting of simultaneous approximation \cite{Khinchin6}, and later in 1948 in the more general setting of matrix approximation \cite{Khinchin4}. Since then this notion has been studied within Diophantine approximation and allied fields, see Moshchevitin's excellent yet far from comprehensive 2010 survey \cite{Moshchevitin3}. 

Let $\MM$\label{matrixspace} denote the set of all $\pdim\times \qdim$ matrices with real entries. A matrix $\bfA \in \MM$ is called \emph{singular}\label{sing} if for all $\epsilon > 0$, there exists $Q_\epsilon$ such that for all $Q \geq Q_\epsilon$, there exist integer vectors $\pp\in \Z^\pdim$ and $\qq\in \Z^\qdim$ such that
\begin{align*}
\|\bfA \qq + \pp\| \leq \epsilon Q^{-\qdim/\pdim}\;\;\;\; \text{ and } \;\;\;\;
0 < \|\qq\| \leq Q.
\end{align*}
Here and from now on $\|\cdot\|$ is used to denote two fixed norms\Footnote{Note that many definitions, such as the one above, and all our main theorems, are insensitive to the choice of these norms. In some cases, e.g. in the course of a proof, we specify a particular norm for computational convenience.}, one on $\R^m$ and the other on $\R^n$. We denote the set of singular $\pdim\times\qdim$ matrices by $\Sing(\pdim,\qdim)$. For $1\times 1$ matrices (i.e. numbers), being singular is equivalent to being rational, and in general any matrix $\bfA$ which satisfies an equation of the form $\bfA \qq = \pp$, with $\pp,\qq$ integral and $\qq$ nonzero, is singular. However, Khintchine proved that there exist singular $2\times 1$ matrices whose entries are linearly independent over $\Q$ \cite[Satz II]{Khinchin3}\Footnote{Although Khintchine's seminal 1926 paper \cite{Khinchin3} includes a proof of the existence of $2\times 1$ and $1\times 2$ matrices possessing a certain property which clearly implies that they are singular, it does not include a definition of singularity nor discuss any property equivalent to singularity.}, and his argument generalizes to the setting of $\pdim\times\qdim$ matrices for all $(\pdim,\qdim) \neq (1,1)$. The name \emph{singular} derives from the fact that $\Sing(\pdim,\qdim)$ is a Lebesgue nullset for all $\pdim,\qdim$, see e.g. \cite[p.431]{Khinchin6} or \cite[Chapter 5, \67]{Cassels}. Note that singularity is a strengthening of the property of \emph{Dirichlet improvability} introduced by Davenport and Schmidt \cite{DavenportSchmidt4}.

In contrast to the measure zero result mentioned above, the computation of the Hausdorff dimension of $\Sing(\pdim,\qdim)$ has been a challenge that so far only met with partial progress. The first breakthrough was made in 2011 by Cheung \cite{Cheung}, who proved that the Hausdorff dimension of $\Sing(2,1)$ is $4/3$; this was extended in 2016 by Cheung and Chevallier \cite{CheungChevallier}, who proved that the Hausdorff dimension of $\Sing(\pdim,1)$ is $\pdim^2/(\pdim+1)$ for all $\pdim\geq 2$; while most recently Kadyrov, Kleinbock, Lindenstrauss, and Margulis (KKLM) \cite{KKLM} proved that the Hausdorff dimension of $\Sing(\pdim,\qdim)$ is at most $\dimsing \df \dimprod\big(1 - \tfrac1{\dimsum}\big)$, and went on to conjecture that their upper bound is sharp for all $(\pdim,\qdim)\neq (1,1)$ (see also \cite[Problem 1]{BCC}).

Cheung and Chevallier's result for singular vectors was an equality and they needed to develop separate tools to deal with upper and lower bounds. They developed the notion of {\it best approximation vectors} and a multidimensional extension of Legendre's theorem on convergents of real continued fraction expansions, as well as the notion of {\it self-similar coverings} that construct Cantor sets with ``inhomogeneous'' tree structures. On the other hand, though KKLM were only able to prove an upper bound rather than an equality, their methods leveraged the technology of integral inequalities developed by Eskin, Margulis and Mozes \cite{EMM} and extend Cheung and Chevallier's upper bound to the matrix framework.

Without relying on the aforementioned results and techniques, we prove (as announced in \cite{DFSU_singular_announcement}) that KKLM's conjecture is correct, and further that the packing dimension of $\Sing(\pdim,\qdim)$ is the same as its Hausdorff dimension, thus answering a question of Bugeaud, Cheung, and Chevallier \cite[Problem 7]{BCC}. To summarize:
\begin{theorem}
\label{theoremsing}
For all $(\pdim,\qdim)\neq (1,1)$, we have
\[
\HD(\Sing(\pdim,\qdim)) = \PD(\Sing(\pdim,\qdim)) = \dimsing \df \dimprod\big(1 - \tfrac1{\dimsum}\big),
\]
where $\HD(S)$ and $\PD(S)$ denote the Hausdorff and packing dimensions of a set $S$, respectively.
\end{theorem}
Note that we provide a new proof of the lower bound as well as a proof of the upper bound.

\subsection{Dani correspondence}
\label{subsectionDani}
The set of singular matrices is linked to homogeneous dynamics via the \emph{Dani correspondence principle} \cite{Dani4,KleinbockMargulis2}. For each $t\in\R$ and for each matrix $\bfA \in \MM$, let
\begin{align*}
g_t &\df \left[\begin{array}{ll}
e^{t/\pdim} \mathrm I_\pdim &\\
& e^{-t/\qdim} \mathrm I_\qdim
\end{array}\right],&
u_\bfA &\df \left[\begin{array}{ll}
\mathrm I_\pdim & \bfA\\
& \mathrm I_\qdim
\end{array}\right],
\end{align*}
where $\mathrm I_k$ denotes the $k$-dimensional identity matrix. Finally, let $d \df \pdim + \qdim$, and for each $j = 1,\ldots,d$, let $\lambda_j(\Lambda)$ denote the $j$th minimum of a lattice $\Lambda \subset \R^d$, i.e. the infimum of $\lambda$ such that the set $\{\rr\in\Lambda : \|\rr\| \leq \lambda\}$ contains $j$ linearly independent vectors. Then the Dani correspondence principle is a dictionary between the Diophantine properties of a matrix $\bfA$ on the one hand, and the dynamical properties of the orbit $(g_t u_\bfA \Z^d)_{t\geq 0}$ on the other. 

Recall that an $\pdim\times \qdim$ matrix $\bfA$ is called \emph{badly approximable}\label{BA} if there exists
$c > 0$ such that for all integer vectors $\pp\in \Z^\pdim$ and $\qq\in \Z^\qdim \setminus \{\0\}$ we have $\|\bfA \qq + \pp\| \geq c \|\qq\| ^{-\frac{\qdim}{\pdim}}$; and is called \emph{very well approximable}\label{VWA} if there exist $\varepsilon > 0$ and infinitely many integer vectors $\pp\in \Z^\pdim$ and $\qq\in \Z^\qdim \setminus \{\0\}$ such that $\|\bfA \qq + \pp\| \leq \|\qq\|^{-(\frac{\qdim}{\pdim} + \varepsilon)}$. Such classes have been intensively studied within the field of metric Diophantine approximation \cite{BernikDodson, Bugeaud, DodsonKristensen}.

\begin{center}
\begin{tabular}{|c|c|}
\hline
\spc{{\bf Diophantine properties of $\bfA$}}&
\spc{{\bf Dynamical properties of $(g_t u_\bfA x_0)_{t\geq 0}$}}\\
\hline
$\bfA$ is \emph{badly approximable}&
$(g_t u_\bfA x_0)_{t\geq 0}$ is bounded\\
\hline
$\bfA$ is \emph{singular}&
$(g_t u_\bfA x_0)_{t\geq 0}$ is divergent\\
\hline
$\bfA$ is \emph{very well approximable}&
$\limsup_{t\to\infty} \frac1t\dist(x_0,g_t u_\bfA x_0) > 0$\\
\hline
\end{tabular}
\end{center}

We denote the sets of badly approximable, singular, and very well approximable matrices by $\mathrm{BA}(\pdim,\qdim)$, $\Sing(\pdim,\qdim)$, and $\mathrm{VWA}(\pdim,\qdim)$, respectively. Using the Dani correspondence principle, the fact that they are all Lebesgue null sets can now be seen to follow from the ergodicity of the ($g_t$)-action (see \cite[Corollary 2.2 in Chapter III]{BekkaMayer}). Indeed, in each case it suffices to show that any trajectory that equidistributes is not in the respective set. An equidistributed trajectory is not bounded because the orbit must be dense, proving that $\mathrm{BA}(\pdim,\qdim)$ is Lebesgue null. An equidistributed trajectory is not divergent because that would imply escape of mass, proving that $\Sing(\pdim,\qdim)$ is Lebesgue null. Finally, an equidistributed trajectory does not escape to infinity at a linear rate because this would imply that it spends a proportionally long time near infinity infinitely often, which would imply escape of mass (along a subsequence); thereby proving that $\mathrm{VWA}(\pdim,\qdim)$ is Lebesgue null.

It follows from the Dani correspondence principle that Theorem \ref{theoremsing} implies that the set of divergent trajectories of the one-parameter diagonal ($g_t$)-action (on the space of unimodular lattices that has exactly two Lyapunov exponents with opposite signs) has equal Hausdorff and packing dimensions. In the sequel, we focus on Diophantine statements and leave it to the interested reader to translate our results in the language of homogeneous dynamics.

Let us precisely state the result mentioned in the middle row of the table above as it is particularly germane to our theme.
\begin{theorem}[{\cite[Theorem 2.14]{Dani4}}]
\label{theoremdani}
A matrix $\bfA \in \MM$ is singular if and only if the trajectory $(g_t u_\bfA \Z^d)_{t\geq 0}$ is divergent in the space of unimodular lattices in $\R^d$, or equivalently (via Mahler's compactness criterion \cite[Theorem 11.33]{EinsiedlerWard}) if
\[
\lim_{t\to\infty} \lambda_1(g_t u_\bfA \Z^d) = 0.
\]
\end{theorem}

It is natural to ask about the set of matrices such that the above limit occurs at a prescribed rate, such as the set of matrices such that $-\log\lambda_1(g_t u_\bfA \Z^d)$ grows linearly with respect to $t$. This question is closely linked with the concept of uniform exponents of irrationality. The \emph{uniform exponent of irrationality}\label{UEI} of an $\pdim\times \qdim$ matrix $\bfA$, denoted $\what\omega(\bfA)$, is the supremum of $\omega$ such that for all $Q$ sufficiently large, there exist integer vectors $\pp\in \Z^\pdim$ and $\qq\in\Z^\qdim$ such that
\begin{align*}
\|\bfA \qq + \pp\| &\leq Q^{-\omega} \text{ and }
0 < \|\qq\| \leq Q.
\end{align*}
By Dirichlet's theorem (\cite{Dirichlet} or \cite[Theorem 1E in \6II]{Schmidt3}), every $\pdim\times\qdim$ matrix $\bfA$ satisfies $\what\omega(\bfA) \geq \dirichlet$. Moreover, it is immediate from the definitions that any matrix $\bfA$ satisfying $\what\omega(\bfA) > \dirichlet$ is singular. We call a matrix \emph{very singular}\label{vsing} if it satisfies the inequality $\what\omega(\bfA) > \dirichlet$, in analogy with the set of \emph{very well approximable} matrices, which satisfy a similar inequality for the regular (non-uniform) exponent of irrationality. We denote the set of very singular $\pdim\times\qdim$ matrices by $\VSing(\pdim,\qdim)$. The relationship between uniform exponents of irrationality and very singular matrices on the one hand, and homogeneous dynamics on the other, is given as follows:

\begin{theorem}
\label{theoremdani2}
A matrix $\bfA$ is very singular if and only if $\what\tau(\bfA) > 0$, where
\[
\what\dynexp(\bfA) \df \liminf_{t\to\infty} \frac{-1}t\log\lambda_1(g_t u_\bfA \Z^d).
\]
Moreover, the quantities $\tau = \what\tau(\bfA)$ and $\omega = \what\omega(\bfA)$ are related by the formula
\begin{equation}
\label{dani}
\tau = \frac1\qdim \frac{\omega - \dirichlet}{\omega + 1}\cdot
\end{equation}
\end{theorem}

This theorem is a straightforward example of the Dani correspondence principle and is probably well-known, but we have not been able to find a reference.

\begin{proof}
The first assertion follows from \eqref{dani}, so it suffices to prove \eqref{dani}. Let $\omega = \what\omega(\bfA)$, and let $\tau$ be given by \eqref{dani}; then we need to prove that $\what\tau(\bfA) = \tau$. We prove the $\geq$ direction; the $\leq$ direction is similar. Fix $\epsilon > 0$ and $t \geq 0$, and let $Q = e^{(1/n - \tau) t}$. By the definition of $\omega$, if $t$ (and thus $Q$) is sufficiently large then there exist $\pp,\qq$ such that $\|A\qq + \pp\| \leq Q^{-\omega+\epsilon}$ and $0 < \|\qq\| \leq Q$. Now let
\[
\rr = g_t u_A (\pp,\qq) = (e^{t/m}(A\qq + \pp), e^{-t/n}\qq).
\]
Then
\begin{align*}
\lambda_1(g_t u_A \Z^d) \leq \|\rr\| &\asymp \max(e^{t/m}\|A\qq + \pp\|, e^{-t/n}\|\qq\|)\\
&\leq \max(e^{t/m} Q^{-\omega+\epsilon},e^{-t/n} Q)\\
&= \max(e^{t/m} e^{(1/n-\tau)(-\omega+\epsilon)},e^{-\tau t})\\
&= \exp\left(-t\min\left(\tau, \left(\tfrac1n-\tau\right)(\omega-\epsilon) - \tfrac1m\right)\right).
\end{align*}
Since $t$ was arbitrary, it follows that
\[
\what\dynexp(\bfA) = \liminf_{t\to\infty} \frac{-1}t\log\lambda_1(g_t u_\bfA \Z^d) \geq \min\left(\tau, \left(\tfrac1n-\tau\right)(\omega-\epsilon) - \tfrac1m\right).
\]
Taking the limit as $\epsilon\to 0$ we get
\[
\what\dynexp(\bfA) \geq  \min\left(\tau, \left(\tfrac1n-\tau\right)\omega - \tfrac1m\right) = \min(\tau,\tau) = \tau.
\qedhere\]
\end{proof}

\subsection{Dimensions of very singular matrices}
Perhaps unsurprisingly, the set of very singular matrices has the same dimension properties as the set of singular matrices.

\begin{theorem}
\label{theoremvsing}
For all $(\pdim,\qdim)\neq (1,1)$, we have
\[
\HD(\VSing(m,n)) = \PD(\VSing(m,n)) = \dimsing.
\]
\end{theorem}

One can also ask for more precise results regarding the function $\what\omega$. Specifically, for each $\omega > \dirichlet$ we can consider the levelset\Footnote{For results considering the superlevelset, see Theorem \ref{theoremvariational5}.}
\begin{equation}
\label{singomega}
\Sing_\dims(\omega) \df \{\bfA : \what\omega(\bfA) = \omega\} = \{\bfA : \what\tau(\bfA) = \tau\} \df \Sing_\dims(\tau),
\end{equation}
where $\tau$ is given by \eqref{dani}.\Footnote{This is somewhat of an abuse of notation since $\Sing_\dims$ is being used in two separate senses in the equation \eqref{singomega}. We avoid confusion by using $\Sing_\dims$ in the second sense only when the parameter is named $\tau$.} Elements of the set above are called \emph{$\omega$-singular} or \emph{$\tau$-singular}.

It would be desirable to obtain precise formulas for the Hausdorff and packing dimensions of $\Sing_\dims(\omega)$ in terms of $\omega$, $\pdim$, and $\qdim$, see e.g. \cite[Problem 2]{BCC}. However, this appears to be extremely challenging at the present juncture. We have made significant progress towards this question: solving it completely in the cases $(\pdim,\qdim) = (1,2)$ and $(\pdim,\qdim) = (2,1)$, and for packing dimension in the case where $n \geq 2$. See Theorems \ref{theorempacking} and \ref{theoremspecialcase} for details.

In general, we have obtained asymptotic formulas of two types: estimates valid when $\omega$ is small and estimates valid when $\omega$ is large. Note that while the minimum value of $\what\omega$ is always $\dirichlet$ (corresponding to $\what\tau = 0$), the maximum value depends on whether or not $\qdim$ is at least $2$. If $\qdim \geq 2$, then the maximum value of $\what\omega$ is $\infty$ (corresponding to $\what\tau = \frac1\qdim$), while if $\qdim = 1$, then the maximum value of $\what\omega$ (excluding rational points) is $1$ (corresponding to $\what\tau = \frac{\pdim-1}{2\pdim}$).\Footnote{The reason for this is that if $\qdim = 1$, then for trivial reasons the value of $\what\omega$ at a point $\xx\in\R^\pdim$ is at most the minimum value of $\what\omega$ over the coordinates $x_1,\ldots,x_m$, and if $\xx$ is irrational, then for some $i = 1,\ldots,m$, $x_i$ is irrational and therefore (since we are in one dimension) satisfies $\what\omega(x_i) = 1$.} Consequently, we have two different asymptotic estimates of the dimensions of $\Sing_\dims(\omega)$ when $\omega$ is large corresponding to these two cases. 
In all of the formulas below, $\tau$ is related to $\omega$ by the formula \eqref{dani}.

\begin{theorem}
\label{theoremhsmall}
Suppose that $(\pdim,\qdim) \neq (1,1)$. Then for all $\omega > \dirichlet$ sufficiently close to $\dirichlet$, we have\vspace{-0.23in}
\begin{align*}
\phantom{\HD(\Sing_\dims(\omega))} &\hspace{2 in}&
\phantom{\PD(\Sing_\dims(\omega))} &\hspace{1.2 in}\\
\HD(\Sing_\dims(\omega)) &= \dimsing - \Theta\left(\sqrt{\omega - \dirichlet}\right)&
\PD(\Sing_\dims(\omega)) &= \dimsing - \Theta\Big(\omega - \dirichlet\Big)\\
&= \dimsing - \Theta\left(\sqrt{\tau}\,\right)&
&= \dimsing - \Theta\left(\tau\right)
\end{align*}
unless $(\pdim,\qdim) = (2,2)$, in which case\vspace{-0.23in}
\begin{align*}
\phantom{\HD(\Sing_\dims(\omega))} &\hspace{2 in}&
\phantom{\PD(\Sing_\dims(\omega))} &\hspace{1.2 in}\\
\HD(\Sing_\dims(\omega)) &= \dimsing - \Theta\Big(\omega - \dirichlet\Big)&
\PD(\Sing_\dims(\omega)) &= \dimsing - \Theta\Big(\omega - \dirichlet\Big)\\
&= \dimsing - \Theta\left(\tau\right)&
&= \dimsing - \Theta\left(\tau\right).
\end{align*}
In the sequel, we refer to the dimension formulas in the case $(\pdim,\qdim) \notin \{ (1,1), (2,2) \}$ as ``the first case of Theorem \ref{theoremhsmall}'', and to the dimension formulas in the case $(\pdim,\qdim) = (2,2)$ as ``the second case of Theorem \ref{theoremhsmall}''.
\end{theorem}
\vspace{.05 in}
\begin{theorem}
\label{theoremN2}
Suppose that $\qdim \geq 2$. Then for all $\omega<\infty$ sufficiently large, we have\vspace{-0.23in}
\begin{align*}
\phantom{\HD(\Sing_\dims(\omega))} &\hspace{2 in}&
\phantom{\PD(\Sing_\dims(\omega))} &\hspace{1.2 in}\\
\HD(\Sing_\dims(\omega)) &= \dimprod - 2\pdim + \Theta\left(\tfrac1\omega\right)&
\PD(\Sing_\dims(\omega)) &= \dimprod - \pdim.\\
&= \dimprod - 2\pdim + \Theta\left(\tfrac1\qdim - \tau\right)
\end{align*}
\end{theorem}
\vspace{0 in}
\begin{theorem}
\label{theoremN1}
Suppose that $\qdim = 1$ and $\pdim\geq 2$. Then for all $\omega<1$ sufficiently close to $1$, we have\vspace{-0.23in}
\begin{align*}
\phantom{\HD(\Sing_\dims(\omega))} &\hspace{2 in}&
\phantom{\PD(\Sing_\dims(\omega))} &\hspace{1.2 in}\\
\HD(\Sing_\dims(\omega)) &= \Theta\left(1 - \omega\right)&
\PD(\Sing_\dims(\omega)) &= 1.\\
&= \Theta\Big(\tfrac{\pdim-1}{2\pdim} - \tau\Big)
\end{align*}
\end{theorem}

Beyond the results above, we have a precise formula for the packing dimension when $n\geq 2$, which remains a lower bound when $n=1$.
\begin{theorem}
\label{theorempacking}
Define the function
\[
\overline\delta_\dims(\tau) \df \max\left(mn - m, \; \dimsing - \frac{mn}{m+n}(d+m)\tau, \; mn - \frac{mn}{m+n}\frac{1+m\tau}{1-\frac{mn}{m-1}\tau}\right) \cdot
\]
Then we have
\begin{equation}
\label{packinginequality}
\PD(\Sing_{m,n}(\tau)) \geq \overline\delta_\dims(\tau),
\end{equation}
with the understanding that the last piece of $\overline\delta_\dims(\tau)$ is ignored if $m=1$. If $n \geq 2$, then equality holds in \eqref{packinginequality}.
\end{theorem}

\begin{remark*}
The cases of the maximum correspond to $\tau\in [\tau_2,\tfrac 1n]$, $\tau\in [\tau_1,\tau_2]$, and $\tau\in [0,\tau_1]$, respectively, where $\tau_1 = \frac{m^2 - d}{mn(d+m)}$ and $\tau_2 = \frac{m}{n(m+d)}$. Note that $\tau_1 > 0$ if and only if $m^2 > d$. When $\tau_1 \leq 0$, then the second case of the maximum holds for all $\tau\in [0,\tau_2]$.
\end{remark*}

When $n=1$, the inequality \eqref{packinginequality} is strict for some values of $\tau$, as shown by the following theorem:

\begin{theorem}
\label{theorempacking2}
We have
\begin{align*}
\PD(\Sing_{m,1}(\tau)) &\geq 1 &\text{ for all } 0 &< \tau \leq \tfrac{m-1}{2m}, \text{ and}\\
\PD(\Sing_{m,1}(\tau)) &\geq m-1 &\text{ for all } 0 &< \tau \leq \tfrac{1}{m^2}.
\end{align*}
\end{theorem}

\begin{remark*}
To see that Theorem \ref{theorempacking2} implies that the inequality \eqref{packinginequality} in Theorem \ref{theorempacking} is strict for some values of $\tau$, note that $\overline\delta_{m,1}(\frac{m-1}{2m}) = \frac12 < 1$. For $m\geq 3$, we have $$\overline\delta_{m,1}\left(\frac1{m^2}\right) = m - 1 - \frac1{m^2-m-1} < m-1.$$ When $m=2$, we instead have $$\overline\delta_{m,1}\left(\frac1{m^2}\right) = \overline\delta_{m,1}\left(\frac{m-1}{2m}\right) = \frac12 < 1 =  m - 1.$$
\end{remark*}

\subsubsection{Trivially singular matrices}
\label{remarktriviallysingular}
Call a matrix $\bfA$ \emph{trivially singular} if there exists $j = 1,\ldots,d-1$ such that
\[
\log\lambda_{j+1}(g_t u_\bfA \Z^d) - \log\lambda_j(g_t u_\bfA \Z^d) \to \infty \text{ as } t\to\infty.
\]
Then all of the formulas above in Theorems \ref{theoremhsmall}-\ref{theorempacking2} remain true if $\Sing_\dims(\omega)$ is replaced by the set
\[
\Sing_\dims^*(\omega) \df \{\bfA\in \Sing_\dims(\omega) : \bfA\text{ is not trivially singular}\}.
\]
Similarly, the formulas in Theorems \ref{theoremsing} and \ref{theoremvsing} above and in Theorems \ref{theoremspecialcase}-\ref{theoremkmessenger} below remain true if we restrict to the respective sets of matrices that are not trivially singular.
The reason for this is since while proving lower bounds none of the templates (see Definition \ref{definitiontemplate}) we construct are trivially singular.

Moreover, for $\qdim \geq 2$ we have\vspace{-0.23in}
\begin{align*}
\phantom{\HD(\Sing_\dims(\infty))} &\hspace{1 in}&
\phantom{\PD(\Sing_\dims(\infty))} &\hspace{1 in}\\
\HD(\Sing_\dims^*(\infty)) &= \dimprod - 2\pdim&
\PD(\Sing_\dims^*(\infty)) &= \dimprod - \pdim
\end{align*}
and for $\qdim = 1$, $\pdim \geq 2$ we have\vspace{-0.23in}
\begin{align*}
\phantom{\HD(\Sing_\dims(\infty))} &\hspace{1 in}&
\phantom{\PD(\Sing_\dims(\infty))} &\hspace{1 in}\\
\HD(\Sing_\dims^*(1)) &= 0&
\PD(\Sing_\dims^*(1)) &= 1.
\end{align*}
Note that the class of trivially singular matrices is smaller than the class of matrices with degenerate trajectories in the sense of \cite[Definition 2.8]{Dani4}, but larger than the class considered in \cite[p.2]{BCC} consisting of matrices $\bfA$ such that the group $\bfA\Z^\qdim + \Z^\pdim$ does not have full rank. A $d\times 1$ or $1\times d$ matrix is trivially singular if and only if it is contained in a rational hyperplane of $\R^d$.

\subsection{$1\times 2$ and $2\times 1$ matrices}
Beyond our asymptotic formulas stated in the previous section, we obtain precise formulas for the Hausdorff and packing dimensions of $\Sing_\dims(\omega)$ for the cases $(\pdim,\qdim) = (1,2)$ and $(\pdim,\qdim) = (2,1)$. Our dimension formulas complete a cornucopia of bounds due to Baker, Bugeaud--Laurent, Laurent, Dodson, Yavid, Rynne, and Bugeaud--Cheung--Chevallier (1977--2016). We refer to \cite{BCC} for a detailed history of the prior results.

\begin{theorem}
\label{theoremspecialcase}
For all $\omega \in (2, \infty)$ (corresponding to $\tau \in (0,1/2)$) we have
\begin{align*}
\HD(\Sing_{1,2}(\omega)) &= \begin{cases}
\frac43 - \frac43\sqrt{\tau - 6\tau^3 + 4\tau^4} - 2\tau + \frac83 \tau^2
& \text{ if }\tau \leq \tau_0 \df \frac{3\sqrt2-2}{14}\\
\frac{1-2\tau}{1+\tau}
& \text{ if }\tau \geq \tau_0
\end{cases}\\
\PD(\Sing_{1,2}(\omega)) &= \begin{cases}
\tfrac{4-8\tau}{3} & \text{ if } \tau \leq \tau_1 \df \frac18\\
1 & \text{ if } \tau \geq \tau_1
\end{cases}
\end{align*}
(cf. Figure \ref{figurespecialcasegraph}).
\end{theorem}

\begin{remark*} There had been a lot of partial progress towards the Hausdorff dimension part of Theorem \ref{theoremspecialcase}. In particular, the $\geq$ direction follows from \cite[Corollary 2 and Theorem 3]{BCC}. For $\tau \geq \tau_0$ the upper bound follows from \cite[Corollary 2]{BCC} and for $\tau < \tau_0$, a non-optimal upper bound is given in \cite[Theorem 1]{BCC}. 
\end{remark*}

\begin{center}
\begin{figure}
\begin{tikzpicture}[scale=1.3]
 \begin{axis}[
   xmin=0,
   ymin=0,
   ]
\addplot[domain=0.5858:1.3333, ultra thick, smooth,blue] ((6*(3*x^3-12*x^2+14*x-4)^(1/2) + 9*x-16)/(2*(12*x-25)),x);
\addplot[domain=0.16:0.5, ultra thick,smooth,blue] {2*(1/2-x)/(1+x)};
\addplot[domain=0.0:0.125, ultra thick,smooth] {(4-8*x)/3};
\addplot[domain=0.125:0.5, ultra thick,smooth] {1};
\addplot[mark=*] coordinates {(0.159,0.586)};
\addplot[mark=*] coordinates {(0.125,1.0)};
\addplot[mark=*] coordinates {(0,4/3)};
\addplot[mark=*] coordinates {(1/2,0)};
\addplot[mark=*] coordinates {(1/2,1)};
\end{axis}
\node at (2.38,0.51) {\textcolor{blue}{$f_2(\tau) =  \HD(\Sing_{1,2}(\omega))$}};
\node at (4.77,3.46) {$f_1(\tau) = \PD(\Sing_{1,2}(\omega))$};
\node at (3.46,2.5) {\textcolor{blue}{$(\frac{3\sqrt 2 - 2}{14}, 2-\sqrt 2)$}};
\node at (2,4.3) {$(\frac18,1)$};
\node at (0.67,5.3) {$(0,\frac43)$};
\end{tikzpicture}
\caption{Graphs of the dimension functions 
\[
f_1(\tau) \df \PD(\Sing_{1,2}(\omega)) ~\text{and}~ f_2(\tau) \df \HD(\Sing_{1,2}(\omega)).
\]
The packing dimension function $f_1$ is linear on the intervals $[0,1/8]$ and $[1/8,1/2]$, while the Hausdorff dimension function $f_2$ is real-analytic on the intervals $[0,\tau_0]$ and $[\tau_0,1/2]$, where $\tau_0 = (3\sqrt2-2)/14 \sim 0.1602$.}
\label{figurespecialcasegraph}
\end{figure}
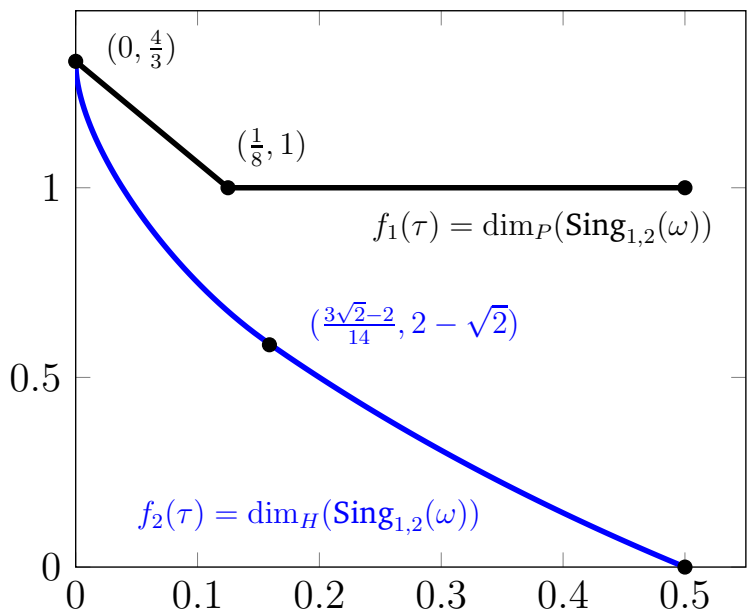
\end{center}

\begin{remark*}
Identify the space of $1 \times 2$ matrices with that of $2 \times 1$ matrices using the transpose isomorphism. Then
by Jarn\'ik's identity \cite{Jarnik4} (see also \cite[Theorem A]{German2}), for all $\omega\in \CO2\infty$ we have
\[
\Sing_{1,2}(\omega) = \Sing_{2,1}(\omega')
\]
where $\omega' = 1 - \frac1\omega$~, and
\[
\Sing_{1,2}(\infty) = \Sing_{2,1}(1) \cup \Sing_{2,1}(\infty).
\]
Thus by applying an appropriate substitution to the above formulas and using the fact that $\Sing_{2,1}(\infty)$ is countable (it is the set of rational points), it is possible to get explicit formulas for $\HD(\Sing_{2,1}(\omega'))$ and $\PD(\Sing_{2,1}(\omega'))$, either in terms of $\omega'$ or in terms of 
\[
\tau' = \frac{\omega'-\frac12}{\omega'+1} = \frac{\tau}{1+2\tau} \cdot
\]
However, the resulting formulas are not very elegant so we omit them.
\end{remark*}

\begin{remark*}
The transition point $\tau_0 = (3\sqrt 2 - 2)/14$ in the above formula for Hausdorff dimension corresponds to 
\[
\omega_0 = 2+\sqrt2 ~,~ \omega'_0 = \sqrt2/2 ~,~ \tau'_0 = (4-3\sqrt2)/2 ~,~\text{and}~ \HD(\Sing_{1,2}(\omega_0)) = 2-\sqrt 2.
\]
The transition point $\tau_1 = 1/8$ for packing dimension corresponds to 
\[
\omega_1 = 3 ~,~ \omega'_1 = 2/3 ~,~ \tau'_1 = 1/10 ~,~ \text{and}~ \PD(\Sing_{1,2}(\omega_1)) = 1.
\]
\end{remark*}

\begin{remark*}
Theorem \ref{theoremspecialcase} implies that $\HD(\Sing_{1,2}(\omega)) < \PD(\Sing_{1,2}(\omega))$ for all $\omega \in (2,\infty)$. This answers the first part of \cite[Problem 7]{BCC} in the affirmative.
\end{remark*}

\subsection{Singularity on average}
A different way of quantifying the notion of singularity is the notion of \emph{singularity on average} introduced in \cite{KKLM}. Given a matrix $\bfA$, we define the \emph{proportion of time spent near infinity} to be the number
\label{cusp}
\[
\PP(\bfA) \df \lim_{\epsilon\to 0} \liminf_{T\to\infty} \frac1T \lambda\big(\big\{t\in [0,T] : \lambda_1(g_t u_\bfA \Z^d) \leq \epsilon\big\}\big) \in [0,1],
\]
where $\lambda$ denotes Lebesgue measure. The matrix $\bfA$ is said to be \emph{singular on average} if $\PP(\bfA) = 1$. Clearly, every singular matrix is singular on average.

\begin{theorem}
\label{theoremsingonaverage}
For all $p\in [0,1]$, we have
\[
\HD(\{\bfA : \PP(\bfA) = p\}) = 
\PD(\{\bfA : \PP(\bfA) = p\}) = 
p \dimsing + (1-p)\dimprod.
\]
In particular, the dimension of the set of matrices singular on average is $\dimsing$.
\end{theorem}

Note that the Hausdorff dimension part of this theorem proves the conjecture stated in \cite[Remark 2.1]{KKLM}, where the upper bound was proven. However, we give an independent proof of the upper bound. Also note that when $p = 1$, the lower bound for Hausdorff dimension follows from Theorem \ref{theoremsing}.

\subsection{Starkov's conjecture}
In \cite[p.213]{Starkov2}, Starkov asked whether there exists a singular vector (i.e. $\pdim\times 1$ singular matrix) which is not very well approximable. Here, we recall that a matrix $\bfA$ is called \emph{very well approximable} if for some $\omega > \dirichlet$, there exist infinitely many pairs $(\pp,\qq)\in\Z^\pdim\times\Z^\qdim$ such that
\begin{equation}
\|\bfA \qq + \pp\| \leq \|\qq\|^{-\omega},
\end{equation}
or equivalently in terms of the Dani correspondence principle, a matrix $\bfA$ is very well approximable if $\limsup_{t\to\infty} -\frac1t\log\lambda_1(g_t u_\bfA \Z^d) > 0$. This question was answered affirmatively by Cheung \cite[Theorem 1.4]{Cheung} in the case $\pdim = 2$. In fact, Cheung showed that if $\psi$ is any function such that $q^{1/2}\psi(q) \to 0$ as $q\to\infty$, then there exists a $2\times 1$ singular vector which is not $\psi$-approximable. Here, a matrix $\bfA$ is called \emph{$\psi$-approximable}\label{psiapproximable} if there exist infinitely many pairs $(\pp,\qq)\in\Z^\pdim\times\Z^\qdim$ such that $\qq\neq \0$ and
\[
\|\bfA \qq + \pp\| \leq \psi(\|\qq\|).
\]
The following theorem improves on Cheung's result both by generalizing it to the case of arbitrary $\pdim,\qdim$ (i.e. to the matrix approximation framework), and also by computing the dimension of the set of matrices with the given property:
\begin{theorem}
\label{theoremstarkov}
If $\psi$ is any function such that $q^{\qdim/\pdim} \psi(q) \to 0$ as $q\to\infty$, then the set of $\pdim\times\qdim$ singular matrices that are not $\psi$-approximable has Hausdorff dimension $\dimsing$. Equivalently, if $\phi$ is any function such that $\phi(t) \to \infty$ as $t \to \infty$, then the set of $\pdim\times\qdim$ singular matrices $\bfA$ such that $-\log\lambda_1(g_t u_\bfA \Z^d) \leq \phi(t)$ for all $t$ sufficiently large has Hausdorff dimension $\dimsing$. The same is true for the packing dimension.
\end{theorem}
Note that this theorem is optimal in the sense that if $\psi(q) \geq c q^{-\qdim/\pdim}$ for some constant $c$, then it is easy to check that every singular $\pdim\times\qdim$ matrix is $\psi$-approximable.

\subsection{Schmidt's conjecture}
In \cite[p.273]{Schmidt7}, Schmidt conjectured that for all $2\leq k \leq \pdim$, there exists an $\pdim\times 1$ matrix $\bfA$ such that
\begin{align}
\label{ksingular}
\lambda_{k-1}(g_t u_\bfA \Z^d) \to 0 \text{ and }
\lambda_{k+1}(g_t u_\bfA \Z^d) \to \infty \text{ as } t\to\infty.
\end{align}
This conjecture was proven by Moshchevitin \cite{Moshchevitin4}, who constructed an $\pdim\times 1$ matrix $\bfA$ satisfying \eqref{ksingular} and not contained in any rational hyperplane\Footnote{As observed by Moshchevitin \cite[Corollary 2]{Moshchevitin4}, proving Schmidt's conjecture by constructing an $m\times 1$ matrix $\bfA$ satisfying \eqref{ksingular} which is contained in a rational hyperplane is actually trivial: let $\bfA = (\xx,\0)$ where $\xx\in \R^{k-1}$ or $\xx\in\R^{k-2}$ is a badly approximable vector. We assume that if Schmidt had noticed this example, he would have included in his conjecture the requirement that $\bfA$ should not be contained in a rational hyperplane.} (see also \cite{Keita,Schleischitz3}). 
To extend this discussion to the matrix framework, we make the following definition.
\begin{definition}
\label{defksingulargeneral}
An $m\times n$ matrix $\bfA$ is \emph{$k$-singular} for $2\leq k \leq m+n-1$ if 
\begin{align}
\label{ksingulargeneral}
\lambda_{k-1}(g_t u_\bfA \Z^d) \to 0 \text{ and }
\lambda_{k+1}(g_t u_\bfA \Z^d) \to \infty \text{ as } t\to\infty.
\end{align}
(Note that any matrix satisfying \eqref{ksingulargeneral} is singular by Theorem \ref{theoremdani}.)
\end{definition}

We improve Moshchevitin's result by computing a lower bound on the Hausdorff dimension of the set of matrices witnessing Schmidt's conjecture in the matrix framework:

\begin{theorem}
\label{theoremkmessenger}
For all $(\pdim,\qdim)\neq (1,1)$ and for all $2\leq k \leq \dimsum - 1$, the Hausdorff dimension of the set of matrices $\bfA$ that satisfy \eqref{ksingulargeneral} is at least
\[
\max(f_\dims(k),f_\dims(k-1))
\]
where
\begin{equation}
\label{fkdef}
f_\dims(k) \df \dimprod - \frac{k(\dimsum-k)\dimprod}{(\dimsum)^2} - \left\{\frac{k\pdim}{\dimsum}\right\} \left\{\frac{k\qdim}{\dimsum}\right\} \cdot
\end{equation}
Here $\{x\}$ denotes the fractional part of a real number $x$. The same formula is valid for the set of matrices $\bfA$ that satisfy \eqref{ksingulargeneral} and are not trivially singular.
\end{theorem}

\begin{remark*}
The function $f_\dims$ satisfies $f_\dims(\dimsum-k)=f_\dims(k)$ and $f_\dims(1) = f_\dims(\dimsum-1) = \dimsing$. Moreover, for all $1\leq k \leq \dimsum-1$ we have $f_\dims(k) \leq \dimsing$. It follows that when $k = 2$ or $\dimsum - 1$, the Hausdorff and packing dimensions of the set of matrices $\bfA$ that satisfy \eqref{ksingulargeneral} are both equal to $\dimsing$.
\end{remark*}

\begin{remark*}
When $\pdim = 1$ or $\qdim = 1$, the fractional parts appearing in \eqref{fkdef} can be computed explicitly, leading to the formula
\[
f_\dims(k) = \dimprod - \frac{k(\dimsum-k)}{\dimsum}\cdot
\]
However, this formula is not valid when $\pdim,\qdim \geq 2$.
\end{remark*}

We conjecture that the lower bound in Theorem \ref{theoremkmessenger} is optimal for both the Hausdorff and packing dimensions (see Conjecture \ref{conjecturemessenger} below).

\subsection{A conjecture of BGMRV}

After the initial version of this paper was published on arXiv, Beresnevich, Guan, Marnat, Ramirez, and Velani (BGMRV) \cite{BGMRV} studied sets of the form
\[
\FS(\pdim,\qdim) \df \DI(\pdim,\qdim) \butnot(\BA(\pdim,\qdim) \cup \Sing(\pdim,\qdim))
\]
where $\DI(\pdim,\qdim)$ and $\BA(\pdim,\qdim)$ are the set of Dirichlet improvable and badly approximable matrices, respectively. They prove that $\FS(\pdim,1)$ has the cardinality of the continuum for all $\pdim$, and then conjecture that $\FS(\pdim,\qdim)$ has full dimension $\pdim\qdim$ for all $\pdim,\qdim$. They note that an obvious barrier to applying the main result of the current paper to prove this is that the constant relating a successive minima function to a template is dependent on the template rather than uniform. In the current version of the paper, however, the constant is uniform and we can therefore prove the following theorem:

\begin{theorem}
\label{theoremconjectureBGMRV}
For all $m,n$, $\HD(\FS(m,n)) \geq \dimsing$.
\end{theorem}

However, we cannot prove that $\FS(m,n)$ has full dimension, primarily because the set $\DI(m,n)$ is too sensitive to small perturbations, and our method requires fairly large (though uniformly bounded) perturbations.

\subsection{New proofs of old results}
In addition to our new results, our techniques now provide a uniform framework to prove classical results in metric Diophantine approximation. The following result was proven in the one-dimensional setting by Jarn\'ik (1928) and in the matrix setting by Schmidt (1969).

\begin{theorem}[Jarn\'ik--Schmidt, \cite{Jarnik1,Schmidt2}]
\label{theoremJS}
The Hausdorff dimension of the set of badly approximable matrices is $mn$.
\end{theorem}


Recall that for each $\omega > \dirichlet$, we say that a matrix $A$ is $\omega$-approximable if
\[
\limsup_{|\qq|\to\infty} \sup_{\pp\in\Z^m} \frac{-\log\|A\qq-\pp\|}{\log\|\qq\|} \geq \omega.
\]
It follows from the Dani correspondence principle that $A$ is $\omega$-approximable if and only if
\[
\limsup_{t\to\infty} \frac{-h_{A,1}(t)}{t}  \geq \tau
\]
where $\tau$ is as in \eqref{dani}.

The following theorem was proven in the one-dimensional case independently by Jarn\'ik (1929) and Besicovitch (1934), and in the matrix case by Bovey and Dodson (1986).

\begin{theorem}[Jarn\'ik--Besicovitch--Bovey--Dodson, \cite{Jarnik3, Besicovitch, BoveyDodson}]
\label{theoremBD}
The Hausdorff dimension of the set of $\omega$-approximable matrices is $mn(1 - \tau)$. In particular, the Hausdorff dimension of the very well approximable matrices is $mn$.
\end{theorem}

We provide proofs of these theorems in Sections \ref{sectionJS} and \ref{sectionBD} respectively.

%

\section{The variational principle}
\label{sectionvariational}

\subsection{Successive minima functions and templates}
All the theorems in the previous section (with the exception of Theorems \ref{theoremdani} and \ref{theoremdani2}) are consequences of a single \emph{variational principle} in the parametric geometry of numbers. This variational principle is a quantitative analogue of theorems due to Schmidt and Summerer \cite[\62]{SchmidtSummerer3} and Roy \cite[Theorem 1.3]{Roy3}. However, we will state their results in language somewhat different from the language used in their papers, due to the fact that the fundamental object we consider is the one-parameter family of unimodular lattices $(g_t u_\bfA \Z^d)_{t\geq 0}$ used by the Dani correspondence principle, rather than a one-parameter family of (non-unimodular) convex bodies as is done in \cite{SchmidtSummerer3,Roy3}. We leave it to the reader (see Appendix \ref{appendix}) to verify that the theorems we attribute below to \cite{SchmidtSummerer3} and \cite{Roy3} are indeed faithful translations of their results to our setting.
\comdavid{From David's email:\\
I am looking at Schmidt--Summerer's paper and recording the notation they use:\\
-- they are doing simultaneous approximation, so $n=1$\\
-- their $n$ is our $d$, their $y$ is our $r$, their $\xi$ is our $A$\\
-- however, they have $r = (q,p)$ instead of $r = (p,q)$\\
-- what they call $\Lambda(\xi)$ is what we would call $u_A \Z^d$\\
-- what they call $\KK(Q)$ is what we would call $g_{-t} B$, where
$Q = e^t$ and $B = [-1,1]^d$\\
-- their $T$ is our $g_{-1}$\\
Now, the fundamental reason that our theorems are the same as Schmidt--Summerer's is that $\lambda_i(g_t u_A \Z^d, B) = \lambda_i(u_A \Z^d, g_{-t} B)$. In their notation the right-hand side is $\lambda_i(\Lambda(\xi),\KK(Q))$. Thus $L_i(q)$ in their notation is the same as $h_i(t)$ in our notation, where $q = e^t$.\\
Roy's paper uses different notation, but I think Schmidt--Summerer's notation is closer to ours so it will be easier to just compare against that.}\\
\comtushar{From Damien Roy's email:\\
The first one is that the approach of Schmidt and Summerer in $[18]$ is quite close to yours.  If you look at their paper on page 52, you will see (equation $(1.3')$) that although they vary the convex body as a function of the parameter q, their family convex bodies depends only on $n$, and they have volume $1$ independently of $q$. The lattice is fixed but arbitrary. In your notation, the situation that they consider is $m=d-1$ and $n=1$.\\
I wish I could use the same setting as Schmidt and Summerer but, because of my training in transcendental number theory, it was much more natural for me to work with the integer lattice and with the dual setting involving one linear form.  I also changed the parametrization to get only slopes $0$ and $1$.  That helped me quite a lot in the combinatorial aspect of the work.  So, the confusion is partly my fault.\\
So Schmidt and Summerer have $m=d-1$ and $n=1$, while I have $m=1$ and $n=d-1$.  In a previous paper (Acta Arithmetica 140 (2009)), Schmidt and Summerer considered both situations with $d$ parameters (in fact their $d$ is $n$).\\
In Section 2 of their paper $[18]$, on pages 57-60, Schmidt and Summerer show that the map $L$ (which you denote $h$) has bounded difference with what they call an $(n,\gamma)$-system.  The quantity $\gamma$ accounts for the fact that Minkowski's theorem and Mahler's theory of compound bodies both involve bounded factors. I quote their result in my paper $[14]$ as Theorem 2.9 with the definition of an $(n,\gamma)$-system stated as Definition 2.8.  I use the same notation as them although their specific result is essentially dual to what I present.\\
In Section 3 of their paper, Schmidt and Summerer consider the limit case of an $(n,0)$-system and, at the beginning of Section 4, on page 62, they conjecture that the study of these systems should suffice to determine the spectra of the family of exponents of approximation that they are considering.  In the rest of their paper they develop a theory of cover of an $(n,\gamma)$-system and apply it to prove relations between several exponents of approximation.\\ 
I think that, your $(n-1) \times 1$ templates are exactly the $(n,0)$-systems of Schmidt and Summerer with the variable $q$ replaced by $t/(n-1)$. However, what I call an $(n,0)$-system (based on Definition 2.8 of $[14]$) is a map $(P_1(q),...,P_n(q))$ such that $(P_1(nt/(n-1))-t/(n-1),...,P_n(nt/(n-1))-t/(n-1))$ is a $1 \times (n-1)$ template.  I hope that I am giving you the correct transformation.\\
The rigid systems $P$ that I define in the introduction of $[14]$ are those for which I was able to construct a vector $u$ such that $P-L_u$ is bounded.  Once I had done this, it remained to show that every $(n,\gamma)$-system can be approximated by a rigid system (a special case of $(n,0)$-systems) up to bounded difference.  
}

The fundamental question of our version of the parametric geometry of numbers will be as follows: given a matrix $\bfA$, what does the function $\Mink = \Mink_\bfA = (\mink_1,\ldots,\mink_d) : \Rplus \to \R^d$ defined by the formula
\begin{equation}
\label{hitdef}
\mink_i(t) \df \log\lambda_i(g_t u_\bfA \Z^d)
\end{equation}
look like? The function $\Mink_\bfA$ will be called the \emph{successive minima function} of the matrix $\bfA$. The Dani correspondence principle shows that many interesting Diophantine questions about the matrix $\bfA$ are equivalent to questions about its successive minima function. Thus the dictionary in \6\ref{subsectionDani} may be translated as follows.

\begin{center}
\begin{tabular}{|c|c|}
\hline
\spc{{\bf Diophantine properties of $\bfA$}}&
\spc{{\bf Asymptotic properties of $h_{A,1}$}}\\
\hline
$\bfA$ is \emph{badly approximable}&
$\displaystyle\limsup_{t\to\infty} -h_{\bfA,1}(t) < \infty$\\
\hline
$\bfA$ is \emph{singular}&
$\displaystyle\liminf_{t\to\infty} -h_{\bfA,1}(t) = \infty$\\
\hline
$\bfA$ is \emph{very well approximable}&
$\displaystyle \limsup_{t\to\infty} \frac{-h_{\bfA,1}(t) }{t} > 0$\\
\hline
\end{tabular}
\end{center}

\medskip
The main restriction on the successive minima function comes from an application of Minkowski's second theorem on successive minima (see Theorem \ref{mink2} below) to certain subgroups of the lattice $g_t u_\bfA \Z^d$. Specifically, fix $j = 1,\ldots,d-1$ and let $I$ be an interval such that $\mink_j(t) < \mink_{j+1}(t)$ for all $t\in I$. 
For each $t\in I$, let\footnote{Here, $V_{j,t}$ is the smallest subspace containing $\{\rr\in \Z^d : \|g_t u_\bfA \rr\| \leq \lambda_j(g_t u_\bfA \Z^d)\}$. See Convention \ref{spanconvention}.} 
\label{Vjt}
\[
V_{j,t} \df \lb \rr\in \Z^d : \|g_t u_\bfA \rr\| \leq \lambda_j(g_t u_\bfA \Z^d) \rb \subset \R^d.
\]
Then the map $t\mapsto V_{j,t}$ is continuous, and therefore constant, on $I$. By Minkowski's second theorem (Theorem \ref{mink2}), we have
\label{FjI}
\[
\sum_{i\leq j} \mink_i(t) \asymp_\plus F_{j,I}(t) \df \log\|g_t u_\bfA (V_{j,t}\cap \Z^d)\|,
\]
where $\|\Gamma\|$ denotes the covolume of a discrete group $\Gamma \subset \R^d$ (relative to its linear span). Now an argument based on the exterior product formula for covolume and the definition of $g_t$ (see Lemma \ref{lemmaapproximatetemplate}) shows that $F_{j,I} \asymp_\plus G_{j,I}$ for some convex, piecewise linear function $G_{j,I}$ whose slopes are in the set
\begin{equation}
\label{slopeset}
Z(j) \df \left\{\tfrac{\dir_+}{\pdim} - \tfrac{\dir_-}{\qdim}: \dir_\pm \in [0,d_\pm]_\Z,\;\;\dir_+ + \dir_- = j\right\},
\end{equation}
where for convenience we write 
\begin{align*}
d_+ &\df \pdim, &
d_- &\df \qdim, &
[a,b]_\Z &\df [a,b]\cap\Z. 
\end{align*} 
This suggests that $\Mink$ can be approximated by a piecewise linear function $\ff$ such that whenever $f_j < f_{j+1}$ on an interval $I$, the function $F_j \df \sum_{i\leq j} f_i$ is convex and piecewise linear on $I$ with slopes in $Z(j)$. Moreover, it is obvious that $\mink_1\leq \cdots \leq \mink_d$, and the formula for $g_t$ implies that for all $i$, we have $-\frac1\qdim \leq \mink_i' \leq \frac1\pdim$ wherever $\mink_i$ is differentiable. We therefore make the following definition:

\begin{definition}
\label{definitiontemplate}
An $\pdim\times\qdim$ \emph{template} is a piecewise linear\Footnote{In this paper, a \emph{piecewise linear} function is assumed to be continuous, and to be linear on a locally finite collection of intervals whose union is its domain.} map $\ff:\Rplus\to\R^d$ with the following properties:
\begin{itemize}
\item[(I)] $f_1 \leq \cdots \leq f_d$.
\item[(II)] $-\frac1\qdim \leq f_i' \leq \frac1\pdim$ for all $i$.
\item[(III)] For all $j = 0,\ldots,d$ and for every interval $I$ such that $f_j < f_{j+1}$ on $I$, the function \[F_j \df \sum_{0 < i\leq j} f_i\] is convex and piecewise linear on $I$ with slopes in $Z(j)$. Here we use the convention that $f_0 = -\infty$ and $f_{d+1} = +\infty$. We will call the assertion that $F_j$ is convex the \emph{convexity condition}, and the assertion that its slopes are in $Z(j)$ the \emph{quantized slope condition}.
\end{itemize}
When $\pdim = 1$, templates are a slight generalization of reparameterized versions of the \emph{rigid systems} of \cite{Roy3}. We denote the space of $\pdim\times\qdim$ templates by $\TT_{\dims}$.

A template $\ff$ will be called \emph{balanced} if $F_d = f_1 + \ldots + f_d = 0$. Note that every template is equal to a constant plus a balanced template, since by condition (III), $F_d$ is piecewise linear with slopes in $Z(d) = \{0\}$, and thus constant. So for most purposes the distinction between balanced and unbalanced templates is irrelevant, but in some places it will make a difference.
A \emph{partial template} is a piecewise linear map $\ff$ satisfying (I)-(III) whose domain is a closed, possibly infinite, subinterval of $\Rplus$.
An example of a (partial) template is shown in Figure \ref{figuretemplate1}.
\end{definition}

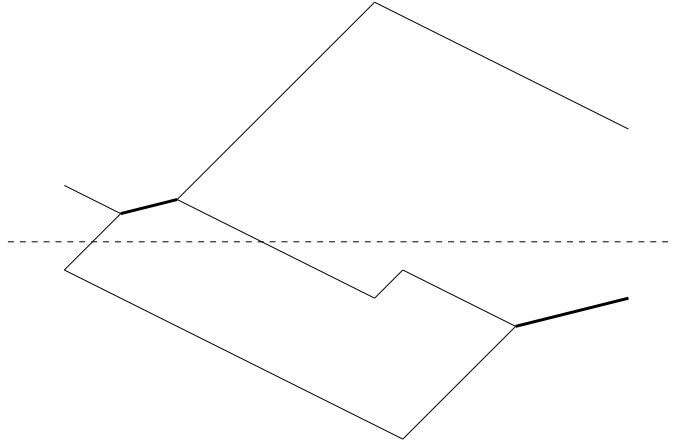
\begin{figure}
\scalebox{0.75}{
\begin{tikzpicture}
\clip(-1,-4) rectangle (11,5);
\draw[dashed] (-1,0) -- (11,0);
\draw (0,1)--(1,0.5);
\draw (0,-0.5)--(1,0.5);
\draw (0,-0.5)--(6,-3.5);
\draw[line width=1.5] (1,0.5) -- (2,0.75);
\draw (2,0.75) -- (5.5,4.25);
\draw (5.5,4.25) -- (10,2);
\draw (2,0.75) -- (5.5,-1);
\draw (5.5,-1) -- (6,-0.5);
\draw (6,-0.5) -- (8,-1.5);
\draw (6,-3.5) -- (8,-1.5);
\draw[line width=1.5] (8,-1.5) -- (10,-1);
\end{tikzpicture}
}
\caption{The joint graph of a $1\times 2$ partial template $\ff = (f_1,f_2,f_3)$, where the \emph{joint graph} of a template is the union of the graphs of its component functions.}
\label{figuretemplate1}
\end{figure}

The fundamental relation between templates and successive minima functions is given as follows:
\begin{theorem}
\label{theoremSSR}
~
\begin{itemize}
\item[(i)] For every $\pdim\times \qdim$ matrix $\bfA$, there exists an $\pdim\times \qdim$ template $\ff$ such that $\Mink_\bfA \asymp_\plus \ff$.
\item[(ii)] For every $\pdim\times \qdim$ template $\ff$, there exists an $\pdim\times \qdim$ matrix $\bfA$ such that $\Mink_\bfA \asymp_\plus \ff$.
\end{itemize}
\end{theorem}

In the case $\pdim = 1$, Theorem \ref{theoremSSR} follows from \cite[Theorem 1.3]{Roy3} (cf. \cite[Corollary 4.7]{Roy2} for part (ii)). 

Theorem \ref{theoremSSR}(ii) asserts that for every template $\ff$, the set
\begin{equation*}\label{Df}
\DD(\ff) \df \{\bfA : \Mink_\bfA \asymp_\plus \ff\}
\end{equation*}
is nonempty.\Footnote{To clarify the notation, $\DD(\ff)$ is the set of all $\bfA$ such that there exists a constant $C > 0$ such that $\|\Mink_\bfA(t) - \ff(t)\| \leq C$ for all $t\geq 0$.} It is natural to ask how big this set is in terms of Hausdorff and packing dimension. Moreover, given a collection of templates $\FF$, we can ask the same question about the set
\begin{equation*}\label{DF}
\DD(\FF) \df \bigcup_{\ff\in\FF} \DD(\ff).
\end{equation*}
It turns out to be easier to answer the second question than the first, assuming that the collection of templates $\FF$ is closed under finite perturbations. Here, $\FF$ is said to be \emph{closed under finite perturbations} if whenever $\gg \asymp_\plus \ff\in\FF$, we have $\gg\in \FF$.

\begin{theorem}[Variational principle, version 1]
\label{theoremvariational1}
Let $\FF$ be a (Borel) collection of templates closed under finite perturbations. Then
\begin{align}
\label{variational1}
\HD(\DD(\FF)) &= \sup_{\ff\in\FF} \underline\delta(\ff),&
\PD(\DD(\FF)) &= \sup_{\ff\in\FF} \overline\delta(\ff),
\end{align}
where the functions $\underline\delta,\overline\delta:\TT_\dims\to [0,\dimprod]$ are as in Definition \ref{definitiondimtemplate} below.
\end{theorem}

\begin{corollary}
\label{corollaryvariational}
With $\FF$ as above, we have
\begin{align}
\label{variational}
\HD(\DD(\FF)) &= \sup_{\ff\in \FF} \HD(\DD(\ff)),&
\PD(\DD(\FF)) &= \sup_{\ff\in \FF} \PD(\DD(\ff)).
\end{align}
\end{corollary}

However, note that Theorem \ref{theoremvariational1} does not imply that $\HD(\DD(\ff)) = \underline\delta(\ff)$ for an individual template $\ff$, since the family $\{\ff\}$ is not closed under finite perturbations. And indeed, since the function $\underline\delta$ is sensitive to finite perturbations, the formula $\HD(\DD(\ff)) = \underline\delta(\ff)$ cannot hold for all $\ff\in\TT_\dims$.

\begin{definition}
\label{definitiondimtemplate}
We define the lower and upper average contraction rate of a template $\ff$ as follows. Let $I$ be an open interval on which $\ff$ is linear. For each $q = 1,\ldots,d$ such that $f_q < f_{q + 1}$ on $I$, let $\dir_\pm = \dir_\pm(\ff,I,q) \in [0,d_\pm]_\Z$ be chosen to satisfy $\dir_+ + \dir_- = q$ and
\begin{equation}
\label{Lqdef}
F_q' = \sum_{i=1}^q f_i' = \frac{\dir_+}{\pdim} - \frac{\dir_-}{\qdim} \text{ on } I,
\end{equation}
as guaranteed by (III) of Definition \ref{definitiontemplate}. An \label{definitionintervalequality}\emph{interval of equality} for $\ff$ on $I$ is an interval $\OC pq_\Z$, where $0 \leq p < q \leq d$ satisfy
\begin{equation}
\label{pqdef}
f_p < f_{p+1} = \cdots = f_q < f_{q+1} \text{ on } I.
\end{equation}
As before, we use the convention that $f_0 = -\infty$ and $f_{d+1} = +\infty$. Note that the collection of intervals of equality forms a partition of $[1,d]_\Z$. If $\OC pq_\Z$ is an interval of equality for $\ff$ on $I$, then we let $\diff_\pm(p,q) = \diff_\pm(\ff,I,p,q)$, where
\begin{equation}
\label{Mpqdef}
\diff_\pm(\ff,I,p,q) = \dir_\pm(\ff,I,q) - \dir_\pm(\ff,I,p),
\end{equation}
or equivalently, $\diff_\pm(p,q)$ are the unique integers such that
\[
\diff_+ + \diff_- = q - p \text{ and } \sum_{i = p + 1}^q f_i' = \frac{\diff_+}{m} - \frac{\diff_-}{n} \text{ on } I.
\]
Note that we have $\diff_\pm \geq 0$ by (II) of Definition \ref{definitiontemplate}.\Footnote{Indeed, we have \[\frac{m + n}{mn} \diff_+ - \frac{q - p}{n} = \sum_{i = p + 1}^q f_i' \geq - \frac{q - p}{n}\] on $I$, and thus $\diff_+ \geq 0$, and similarly $\diff_- \geq 0$.} Next, let
\begin{align} 
\label{Splusdef1}
S_+ = S_+(\ff,I) &= \bigcup_{\OC pq_\Z} \bigOC p{p+\diff_+(p,q)}_\Z\\ 
\label{Sminusdef1}
S_- = S_-(\ff,I) &= \bigcup_{\OC pq_\Z} \bigOC{p+\diff_+(p,q)}q_\Z
\end{align}
where the unions are taken over all intervals of equality for $\ff$ on $I$. Note that $S_+$ and $S_-$ are disjoint and satisfy $S_+\cup S_- = [1,d]_\Z$, and that $\#(S_+) = \pdim$ and $\#(S_-) = \qdim$. Next, let
\begin{equation}
\label{dimfI}
\delta(\ff,I) = \#\{(i_+,i_-)\in S_+\times S_- : i_+ < i_-\} \in [0,\dimprod]_\Z,
\end{equation}
and note that
\begin{equation}
\label{codimfI}
\dimprod - \delta(\ff,I) = \#\{(i_+,i_-)\in S_+\times S_- : i_+ > i_-\}.
\end{equation}

The \emph{lower and upper average contraction rates} of $\ff$ are the numbers
\begin{align}
\label{deltaFH}
\underline\delta(\ff) &\df \liminf_{T\to\infty} \Delta(\ff,T),&
\overline\delta(\ff) &\df \limsup_{T\to\infty} \Delta(\ff,T),
\end{align}
where
\[
\label{DeltafT}
\Delta(\ff,T) \df \frac1T \int_0^T \delta(\ff,t) \;\dee t.
\]
Here we abuse notation by writing $\delta(\ff,t) = \delta(\ff,I)$ for all $t\in I$ (this will cause $\delta(\ff,t)$ to be well-defined for all $t$ outside of a discrete set of corner points for $\ff$). We will also have occasion later to use the notations
\[
\label{DeltaT1T2}
\Delta(\ff,[T_1,T_2]) = \frac{1}{T_2 - T_1} \int_{T_1}^{T_2} \delta(\ff,t) \;\dee t
\]
and
\begin{equation}
\label{deltaT+T-}
\delta(T_+,T_-) = \#\{(i_+,i_-)\in T_+\times T_- : i_+ < i_-\} \in [0,\dimprod]_\Z.
\end{equation}
Note that according to \eqref{deltaT+T-}, $\delta(\ff,I) = \delta(S_+,S_-)$.
\end{definition}

\begin{figure}
\scalebox{0.75}{
\begin{tikzpicture}
\clip(-1,-5) rectangle (11,5);
\draw[dashed] (-1,0) -- (11,0);
\draw (0,-4.5) -- (2,-4.5);
\draw (5.5,-4.5) -- (6,-4.5);
\draw[line width=1.5] (6,-4.5) -- (10,-4.5);
\node at (0.5,-4) {1};
\node at (1.5,-4) {1};
\node at (3.75,-4) {0};
\node at (5.75,-4) {1};
\node at (7,-4) {2};
\node at (9,-4) {2};
\draw (0,1)--(1,0.5);
\draw (0,-0.5)--(1,0.5);
\draw (0,-0.5)--(6,-3.5);
\draw[line width=1.5] (1,0.5) -- (2,0.75);
\draw (2,0.75) -- (5.5,4.25);
\draw (5.5,4.25) -- (10,2);
\draw (2,0.75) -- (5.5,-1);
\draw (5.5,-1) -- (6,-0.5);
\draw (6,-0.5) -- (8,-1.5);
\draw (6,-3.5) -- (8,-1.5);
\draw[line width=1.5] (8,-1.5) -- (10,-1);
\node at (0.6,1.2) {\scalebox{1.5}{$\downarrow$}};
\node at (0.6,-0.35) {\scalebox{1.5}{$\updownarrow$}};
\node at (1.5,1) {\scalebox{1.5}{$\downarrow$}};
\node at (1.5,0.2) {\scalebox{1.5}{$\uparrow$}};
\node at (1.5,-0.85) {\scalebox{1.5}{$\downarrow$}};
\node at (4,2.2) {\scalebox{1.5}{$\uparrow$}};
\node at (4,-0.6) {\scalebox{1.5}{$\downarrow$}};
\node at (4,-2.1) {\scalebox{1.5}{$\downarrow$}};
\draw[dashed] (1,-4.5) -- (1,5);
\draw[dashed] (2,-4.5) -- (2,5);
\draw[dashed] (5.5,-4.5) -- (5.5,5);
\draw[dashed] (6,-4.5) -- (6,5);
\draw[dashed] (8,-4.5) -- (8,5);
\end{tikzpicture}
}
\caption{The joint graph in Figure \ref{figuretemplate1}, with an illustration of the sets $S_\pm(\ff,I)$ and the contraction rates $\delta(\ff,I)$ for each interval of linearity $I$. The ``one-dimensional physics'' interpretation of templates can be seen in this picture as follows: first one particle is going up while two are going down; then the top two collide into each other and their new velocity is determined by conservation of momentum; then they split apart again. 
Given this interpretation of the motion occurring in $I$ as being the result of ``collisions'' between $m$ particles going up and $n$ particles going down, $\delta(\ff,I)$ counts the number of particle pairs that are ``moving towards'' each other (including particles ``colliding'' with each other). 
}
\label{figuredimtemplate}
\end{figure}
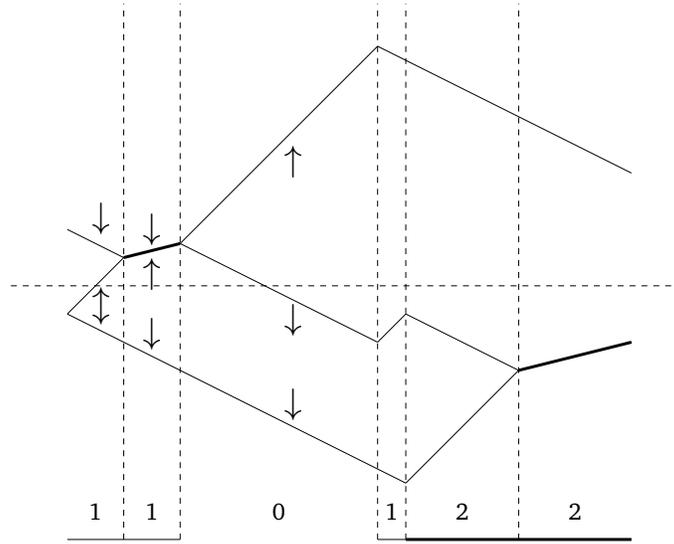

Definition \ref{definitiondimtemplate} can be understood intuitively in terms of a simple version of one-dimensional physics with sticky collisions and conservation of momentum; cf. Figure \ref{figuredimtemplate}. Suppose that we observe particles $P_1,\ldots,P_d$ travelling along trajectories $f_1,\ldots,f_d$ during a time interval $I$ along which $\ff$ is linear, and we want to infer the velocities of these particles before they collided, based on the following background information: before the collision $\pdim$ of the particles were travelling upwards at a speed of $\frac1\pdim$, and $\qdim$ of the particles were travelling downwards at a speed of $\frac1\qdim$. When particles collide (that is, when the velocities of the particles of lower index are more upwards than the velocities of the particles of higher index at the same location), they join forces to move as a unit, and their new velocity is determined by conservation of momentum. However, we can still think of the group as being composed of a certain number of ``upwards'' particles and a certain number of ``downwards'' particles.

The equations \eqref{Splusdef1} and \eqref{Sminusdef1} can be understood as suggesting a particular solution to this problem of inference: assume that within each group, all of the upwards-travelling particles started out below all of the downwards-travelling particles. This is not the only possible solution but it is the nicest one for certain purposes. Specifically, we can imagine a force of ``gravity'' attempting to bring all of the particles together, which acts between any two particles by imposing a fixed energy cost if the two particles are travelling away from each other.\Footnote{This is of course unlike real gravity, which imposes an energy cost that varies with respect to distance.} The total energy cost is then the codimension $\dimprod - \delta(\ff,I)$ defined by \eqref{codimfI}. The equations \eqref{Splusdef1} and \eqref{Sminusdef1} can then be thought of as giving the solution that minimizes this cost.

The idea of codimension as an energy cost is also useful for computing the suprema \eqref{variational1} in certain circumstances, since it suggests principles like the conservation of energy. However, one needs to be careful since the stickiness of collisions means that some naive formulations of conservation of energy are violated.\\

In most cases of interest, the collection $\FF$ in Theorem \ref{theoremvariational1} is defined by some Diophantine condition. In this case, generally rather than $\DD(\FF)$ the set we are really interested in is the set of all matrices whose corresponding successive minima functions satisfy the same Diophantine condition. Now Theorem \ref{theoremSSR}(i) implies that these two sets are the same and thus Theorem \ref{theoremvariational1} is equivalent modulo Theorem \ref{theoremSSR}(i) to the following:

\begin{theorem}[Variational principle, version 2]
\label{theoremvariational2}
Let $\SS$ be a (Borel) collection of functions from $\Rplus$ to $\R^d$ which is closed under finite perturbations, and let
\begin{equation}
\label{MSdef}
\DD(\SS) \df \{\bfA \in \MM : \Mink_\bfA \in \SS\}.
\end{equation}
Then
\begin{align}
\label{variational2}
\HD(\DD(\SS)) &= \sup_{\ff\in\SS\cap \TT_\dims} \underline\delta(\ff),&
\PD(\DD(\SS)) &= \sup_{\ff\in\SS\cap \TT_\dims} \overline\delta(\ff)
\end{align}
with the understanding that $\HD(\emptyset) = \PD(\emptyset) = \sup(\emptyset) = -\infty$ (or $0$ if desired).
\end{theorem}
\begin{proof}[Proof of equivalence]
Theorem \ref{theoremvariational2} implies Theorem \ref{theoremvariational1} since we can take $\SS = \{\gg : \gg \asymp_\plus \ff \in \FF\}$. Conversely, Theorem \ref{theoremvariational1} implies Theorem \ref{theoremvariational2} modulo Theorem \ref{theoremSSR}(i) since we can take $\FF = \SS\cap\TT_\dims$.
\end{proof}

In fact, we will prove a uniform version of Theorem \ref{theoremvariational2}. For each $C > 0$ and collection of functions $\SS$ let
\begin{align} 
\label{NSC}
\NN(\SS,C) &= \{\gg:\Rplus\to\R^d : \|\gg - \ff\| \leq C \text{ for some } \ff\in \SS\},\\ 
\label{DS}
\DD(\SS) &= \{\bfA \in \MM : \Mink_\bfA \in \SS\}
\end{align}
(i.e. $\DD(\SS)$ is as in \eqref{MSdef}).

\begin{theorem}
\label{theoremvariationaluniform}
For all $\epsilon > 0$, there exists $C > 0$ such that for every template $\ff$,
\begin{align*}
\HD(\DD(\NN(\ff,C))) &\geq \underline\delta(\ff) - \epsilon\\
\PD(\DD(\NN(\ff,C))) &\geq \overline\delta(\ff) - \epsilon,
\end{align*}
and for every Borel collection $\SS$ of functions from $\Rplus$ to $\R^d$,
\begin{align*}
\HD(\DD(\SS)) &\leq \sup_{\ff\in \NN(\SS,C)\cap \TT_{\pdim,\qdim}} \underline\delta(\ff) + \epsilon\\
\PD(\DD(\SS)) &\leq \sup_{\ff\in \NN(\SS,C)\cap \TT_{\pdim,\qdim}} \overline\delta(\ff) + \epsilon.
\end{align*}
\end{theorem}

Theorem \ref{theoremvariational2} can be thought of as a quantitative strengthening of Theorem \ref{theoremSSR}, as shown by the following equivalent formulation:

\begin{theorem}[Variational principle, version 3]
\label{theoremvariational3}
~
\begin{itemize}
\item[(i)] Let $S$ be a (Borel) set of $\pdim\times\qdim$ matrices of Hausdorff (resp. packing) dimension $>\delta$. Then there exist a matrix $\bfA\in S$ and a template $\ff \asymp_\plus \Mink_\bfA$ whose lower (resp. upper) average contraction rate is $> \delta$.
\item[(ii)] Let $\ff$ be a template whose lower (resp. upper) average contraction rate is $> \delta$. Then there exists a (Borel) set $S$ of $\pdim\times\qdim$ matrices of Hausdorff (resp. packing) dimension $>\delta$, such that $\Mink_\bfA \asymp_\plus \ff$ for all $\bfA\in S$.
\end{itemize}
\end{theorem}
\begin{proof}[Proof of equivalence]
Part (i) is equivalent to the $\leq$ direction of \eqref{variational2}, and part (ii) to the $\geq$ direction. For the first equivalence, for the forwards direction take $S = \{\bfA : \Mink_\bfA \in \SS\}$, and for the backwards direction take $\SS = \{\gg : \gg\asymp_\plus \Mink_\bfA, \; \bfA \in S\}$. For the second equivalence, for the backwards direction take $S = \DD(\ff)$ and $\SS = \{\gg : \gg\asymp_\plus \ff\}$.
\end{proof}

It is worth stating the special case of Theorem \ref{theoremvariational2} that occurs when the collection $\SS$ is defined by the Diophantine conditions defining $\Sing_\dims(\omega)$ and $\Sing_\dims^*(\omega)$ for some $\omega\geq \dirichlet$. Thus, we define the \emph{uniform dynamical exponent} of a map $\ff:\Rplus\to\R^d$ to be the number
\begin{equation*}\label{UDE}
\what\dynexp(\ff) \df \liminf_{t\to\infty} \frac{-1}t f_1(t).
\end{equation*}
Moreover, $\ff$ is said to be \emph{trivially singular} if $f_{j+1}(t) - f_j(t) \to \infty$ as $t\to\infty$ for some $j = 1,\ldots,d-1$. Letting $\SS = \{\ff : \what\dynexp(\ff) = \dynexp\}$ or $\SS = \{\ff : \what\dynexp(\ff) = \dynexp, \; \ff \text{ not trivially singular}\}$ in Theorem \ref{theoremvariational2} yields the following result:

\begin{theorem}[Special case of variational principle]
\label{theoremvariational4}
For all $\omega \geq \dirichlet$, we have
\begin{align*}
\HD(\Sing_\dims(\omega)) &= \sup\{\underline\delta(\ff): \ff\in\TT_\dims,\;\;\what\dynexp(\ff) = \dynexp\}\\
\PD(\Sing_\dims(\omega)) &= \sup\{\overline\delta(\ff): \ff\in\TT_\dims,\;\;\what\dynexp(\ff) = \dynexp\}\\
\HD(\Sing_\dims^*(\omega)) &= \sup\{\underline\delta(\ff): \ff\in\TT_\dims,\;\;\what\dynexp(\ff) = \dynexp, \; \ff \text{ not trivially singular}\}\\
\PD(\Sing_\dims^*(\omega)) &= \sup\{\overline\delta(\ff): \ff\in\TT_\dims,\;\;\what\dynexp(\ff) = \dynexp, \; \ff \text{ not trivially singular}\}\\
\end{align*}
where $\dynexp$ is as in \eqref{dani}. 
\end{theorem}

Theorem \ref{theoremvariational2} can also be used to compute the dimensions of the set
\label{super}
\[
\w\Sing_\dims^*(\omega) \df \{\bfA:\what\omega(\bfA) \geq \omega,\;\bfA\text{ not trivially singular}\} = \bigcup_{\omega'\geq \omega} \Sing_\dims^*(\omega').
\]
\begin{theorem}[Special case of variational principle]
\label{theoremvariational5}
For all $\omega \geq \dirichlet$, we have
\begin{align*}
\HD(\w\Sing_\dims^*(\omega)) &= \sup_{\omega' \geq \omega} \HD(\Sing_\dims^*(\omega'))\\
\PD(\w\Sing_\dims^*(\omega)) &= \sup_{\omega' \geq \omega} \PD(\Sing_\dims^*(\omega')).
\end{align*}
\end{theorem}
(Theorem \ref{theoremvariational5} is also true with the stars removed, but in that case it is not as interesting because $\HD(\Sing_\dims(\infty))$ is ``too large'', whereas $\HD(\Sing_\dims^*(\infty))$ is the ``correct'' size according to \6\ref{remarktriviallysingular}.)

It is natural to expect that the map $\omega\mapsto \HD(\Sing_\dims^*(\omega))$ is monotonically decreasing, in which case Theorem \ref{theoremvariational5} would imply that
\[
\HD(\w\Sing_\dims^*(\omega)) = \HD(\Sing_\dims^*(\omega)).
\]

\section{Directions to further research}
\label{sectionfuture}
We conclude our introduction with a small sample of problems and research directions, which we hope will illustrate the wide scope awaiting future exploration.

\subsection{Exact Hausdorff and packing dimensions}
Determine whether an appropriate gauge function exists with respect to which the Hausdorff measure of the singular matrices have positive and finite measure. The same question for packing measures is also open. It would be natural to expect that the $\dimsing$-dimensional Hausdorff measure of $\Sing(\pdim,\qdim)$ is zero, and that the $\dimsing$-dimensional packing measure of $\Sing(\pdim,\qdim)$ is infinite. In general, determining exact dimensions for any of the sets we have studied in this paper would be an interesting challenge.



\subsection{Quantitative Schmidt's conjecture}
We conjecture that the inequality in Theorem \ref{theoremkmessenger} is actually an equality:
\begin{conjecture}
\label{conjecturemessenger}
For $2\leq k \leq m+n-1$, the Hausdorff and packing dimensions of the set of $k$-singular $m\times n$ matrices (see Definition \ref{defksingulargeneral}) are both equal to
\[
\max(f_{m,n}(k),f_{m,n}(k-1)), \text{ where}
\]

\[
f_\dims(k) \df \dimprod - \frac{k\dimprod}{\dimsum}\left(1 - \frac k\dimsum\right) - \left\{\frac{k\pdim}{\dimsum}\right\} \left\{\frac{k\qdim}{\dimsum}\right\} \cdot
\]
Here, $\{x\}$ denotes the fractional part of a real number $x$.
\end{conjecture}

\begin{remark}
When $k = 2$ or $\dimsum - 1$, the Hausdorff and packing dimensions of the set of $k$-singular matrices are both equal to $\dimsing$.
\end{remark}

\subsection{Regularity of dimension functionals}
\begin{problem}
Determine when/whether the functions
\begin{align*}
\omega &\mapsto \HD(\Sing_\dims(\omega)),&
\omega &\mapsto \PD(\Sing_\dims(\omega))
\end{align*}
are decreasing and continuous. 
\end{problem}

Although it is natural to suspect that these functions are in fact decreasing and continuous for all $(\dims) \neq (1,1)$, Theorem \ref{theorempacking2} seems to suggest otherwise: it suggests that the function $\tau\mapsto\PD(\Sing_{m,1}(\tau))$ may have a discontinuity at $\tau = 1/m^2$ for all $m\geq 3$. Indeed, the proof of Theorem \ref{theorempacking2} gives us no reason to suspect that the inequality is strict in Theorem \ref{theorempacking} for $\tau$ slightly greater than $1/m^2$. If in fact equality holds for such $\tau$, then there \emph{is} a discontinuity! If this were the case, it would show that the conjecture we made in the announcement of this paper \cite[Conjecture 2.10]{DFSU_singular_announcement} was too optimistic.

\subsection{Intersecting standard and uniform exponent level sets}
Let $\omega(\bfA)$ and $\what\omega(\bfA)$ denote the standard and uniform exponents of irrationality of a matrix $\bfA$, respectively:
\[
\what\omega(\bfA) \df \liminf_{Q\to\infty}\;\; \sup_{0 < \|\qq\| \leq Q} \sup_\pp \frac{-\log\|A\qq + \pp\|}{\log Q}
\]
\[
\omega(\bfA) \df \limsup_{Q\to\infty}\;\; \sup_{0 < \|\qq\| \leq Q} \sup_\pp \frac{-\log\|A\qq + \pp\|}{\log Q}
\]
The Hausdorff dimensions of the levelsets of $\omega$ are well-known, and we have provided many results on the Hausdorff dimensions of the levelsets of $\what\omega$. However, it is natural to ask about the dimension of the intersection of two such sets:
\begin{question}
What is the behavior of the function
\[
(\omega,\what\omega) \mapsto \HD(\{\bfA : \omega(\bfA) = \omega,\; \what\omega(\bfA) = \what\omega\})?
\]
\end{question}

\subsection{Precise dimension formulas for uniform exponent level sets}
As mentioned previously, it is very challenging to obtain precise formulas for the Hausdorff and packing dimensions of $\Sing_\dims(\omega) = \{\bfA : \what\omega(\bfA) = \omega\}$ in terms of $\omega$, $\pdim$, and $\qdim$. Though we have completely solved (see Theorems \ref{theorempacking} and \ref{theoremspecialcase} for details) this problem in the cases $(\pdim,\qdim) = (1,2)$ and $(\pdim,\qdim) = (2,1)$, and for packing dimension in the case where $n \geq 2$, it is plausible that finding a closed form expression in all scenarios is hopeless.
To express the limit of our current understanding, note that we do not have conjectural formulas for Hausdorff dimension even for the cases when $(\pdim,\qdim) \in \{ (1,3), (3,1), (2,2) \}$ at present.

\subsection{Metric theory for $\epsilon$-Dirichlet improvable matrices}
Given $0 < \epsilon < 1$, an $m\times n$ matrix $A$ is called \emph{$\epsilon$-Dirichlet improvable} (see \cite{DavenportSchmidt4}) if for all sufficiently large $Q$, there exists $(\pp,\qq)\in\Z^{m+n}$ such that
\[
\|A \qq - \pp\| \leq \epsilon Q^{-n/m}
\;\;\; \text{and} \;\;\; 0 < \|\qq\| < Q .
\]
An $m\times n$ matrix $A$ is \emph{Dirichlet improvable} if it is $\epsilon$-Dirichlet improvable for some $0 < \epsilon < 1$. Singular matrices are $\epsilon$-Dirichlet improvable for \emph{all} $0 < \epsilon < 1$.

\begin{question}
How do the Hausdorff and packing dimensions of the set of $\epsilon$-Dirichlet improvable $m\times n$ matrices vary as functions of $\epsilon$? It would already be interesting just to give estimates on these dimensions, if not precisely determine them.
\end{question}

\subsection{Weighted singular matrices and general diagonal flows}
In the parametric geometry of numbers and the Dani correspondence principle we are generally concerned with the $(g_t)$ flow as defined in \6\ref{subsectionDani}. What happens if the $(g_t)$ flow is replaced by some other diagonal flow $(h_t)$, for example
\[
h_t = \diag(e^{a_1 t},\ldots,e^{a_m t}, e^{-b_1 t},\ldots, e^{-b_n t}) \in \SL_{m+n}(\R)
\]
where $a_1,\ldots,a_m,b_1,\ldots,b_n$ are positive real numbers? For example, is it possible to compute the Hausdorff and packing dimensions of the set of $m\times n$ matrices $A$ such that the trajectory $(h_t u_A \Z^{m+n})_{t\geq 0}$ is divergent as a function of $a_1,\ldots,a_m,b_1,\ldots,b_n$? When $m=2$ and $n=1$, this question in case of the Hausdorff dimension has been addressed by Liao, Shi, Solan, and Tamam \cite{LSST}. Without obtaining dimension formulas, Guan and Shi proved that the Hausdorff dimension of the set of divergent-on-average trajectories for a one-parameter subgroup action on a finite-volume homogeneous space is not full, \cite{GuanShi}. 
The recent work of Solan \cite{Solan} made great progress in the attempt to extend our results to the setting of general diagonal flows, proving a variational principle for the Hausdorff dimension with respect to a certain modified metric. However, significant work remains to be done in this vein; in particular, the challenge to obtain exact formulas (instead of bounds) for the variational principle for the standard metric remains.

\subsection{Inhomogeneous Diophantine approximation}
Our results fall within the domain of homogeneous Diophantine approximation. It would be of interest to investigate analogues of our results in the frameworks of inhomogeneous approximation, see \cite{BugeaudLaurent, Cassels, Laurent}.
In this setting, given an $\pdim\times \qdim$ matrix $\bfA$ and $\xx \in \R^m$, the pair $(\bfA,\xx)$ is called \emph{singular} if for all $\epsilon > 0$, there exists $Q_\epsilon$ such that for all $Q \geq Q_\epsilon$, there exist integer vectors $\pp\in \Z^\pdim$ and $\qq\in \Z^\qdim$ such that
\begin{align*}
\|\bfA \qq + \pp + \xx\| \leq \epsilon Q^{-\qdim/\pdim}\;\;\;\; \text{ and } \;\;\;\;
0 < \|\qq\| \leq Q.
\end{align*}
It is also natural to study the inhomogeneous approximation frameworks where we fix one coordinate of the pair $(\bfA,\xx)$ and let the other vary. Extending our variational principle (Theorem \ref{theoremvariational2}) and its corollaries to such inhomogeneous frameworks would be a natural next step. When $m=n=1$, this question in case of the Hausdorff dimension has been recently investigated by Kim and Liao \cite{KimLiao}.

\subsection{Parametric geometry of numbers in arbitrary characteristic}
It would be of interest to develop the technology introduced in this work to study questions of Diophantine approximation in the function field setting, see Roy and Waldschmidt \cite{RoyWaldschmidt}.

\section{Acknowledgements}
\label{sectionacknowledgements} 
This research began on 28$^{th}$ November 2016 when the authors met at the American Institute of Mathematics in San Jose, California, via their SQuaRE program. We thank the institute and their staff for their hospitality and excellent working conditions. In particular, we thank Estelle Basor for her singular encouragement and support.
The first-named author was supported in part by a 2017-2018 Faculty Research Grant from the University of Wisconsin-La Crosse. The second-named author was supported in part by the Simons Foundation grant \#245708. The third-named author was supported in part by the EPSRC Programme Grant EP/J018260/1, and is currently supported by a Royal Society University Research Fellowship URF\tbs R1\tbs180649. The fourth-named author was supported in part by the NSF grant DMS-1361677. We thank Pieter Allaart, Val\'erie Berth\'e, Nicolas Chevallier, Seonhee Lim, Antoine Marnat, Damien Roy, Johannes Schleischitz, Andreas Wieser, and Hao Xing for helpful comments and clarifying questions. In particular, we thank Damien Roy for his meticulous reading and criticism, as well as for pointing out several translations between the notation in his papers and those of Schmidt--Summerer and ours leading to the inclusion of Appendix \ref{appendix}. 
We thank Lingmin Liao for their punctilious reading and several discussions that greatly helped improve the exposition. 
We thank Barak Weiss for leading a semester-long study of the variational principle during the \href{http://www.math.tau.ac.il/~barakw/seminar/}{Fall 2021 {\it Seminar on homogeneous dynamics and applications}} at Tel-Aviv University, which helped elicit excellent questions that in turn helped us improve the exposition at several points. 
Finally, we thank the anonymous referees for their extremely scrupulous reports, which helped us improve several points throughout the paper, and pushed us to clarify many facts that were previously ``tacitly assumed and never spelled out''. The quest of refereeing a long and at times necessarily arduous paper is largely a thankless endeavor for which we are greatly appreciative.
We dedicate this paper to S. G. Dani, G. M. Margulis, and W. M. Schmidt -- for their pioneering perspectives that persist in persuading us to persevere in building bridges between Diophantine and dynamical worlds.

\draftnewpage
\part{Proof of main theorems using the variational principle}
\label{partmainproofs}

\section{Leitfaden to Part \ref{partmainproofs}}
\label{sectionleitfaden}
In this part we prove all the theorems of Section \ref{sectionmain} (with the exception of Theorems \ref{theoremdani} and \ref{theoremdani2}) as well as Theorem \ref{theoremSSR} from Section \ref{sectionvariational}, making full use of the variational principle whose involved proof we have deferred to Part \ref{partproofvariational}. 

For reference, the following theorems are proven in the following subsections:
\begin{itemize}
\item \fbox{Theorems \ref{theoremsing} and \ref{theoremvsing} are proved in \6\ref{subsectionmainupperbound} and \6\ref{subsectionmainlowerbound}.}\\ 
To prove Theorems \ref{theoremsing} and \ref{theoremvsing} it suffices\footnote{This follows from the monotonicity of the Hausdorff and packing dimensions, and the fact that the latter is bounded below by the former (see Section \ref{sectiondimensionprelim}).} to show that
\begin{align} \label{mainupperbound}
\PD(\Sing(\dims)) &\leq \dimsing,\\ \label{mainlowerbound}
\HD(\VSing(\dims)) &\geq \dimsing.
\end{align}
We prove these inequalities first (in \6\ref{subsectionmainupperbound} and \6\ref{subsectionmainlowerbound} respectively), since their proofs provide the best basic illustration of our techniques. 
\item \fbox{Theorem \ref{theorempacking} is proven in \6\ref{subsectionpsi} and \6\ref{subsectionPDgeq}.}\\
The packing dimension upper bound (valid for $n \geq 2$) is proven in \6\ref{subsectionpsi}. The packing dimension lower bound is proven in \6\ref{subsectionPDgeq}.
\item \fbox{Theorem \ref{theorempacking2} is proven in \6\ref{subsectionN1PDgeq}.}
\item \fbox{Theorem \ref{theoremhsmall} is proven in \6\ref{subsectionmainupperbound}, \6\ref{subsectionmainlowerbound}, \6\ref{subsectionhsmallHDleq}, \6\ref{subsectionhsmallHDgeq}, and \6\ref{subsectionPDgeq}.}\\
In \6\ref{subsectionmainupperbound}, after proving \eqref{mainupperbound}, we obtain the upper bound for packing dimension in Theorem \ref{theoremhsmall}. 
The packing dimension lower bound in Theorem \ref{theorempacking} (proven in \6\ref{subsectionPDgeq}) implies the packing dimension lower bound in Theorem \ref{theoremhsmall}. This completes the proof for the packing dimension asymptotic formula.
Regarding the Hausdorff dimension, there are two asymptotic formulas that have to be proved.
For the first case of Theorem \ref{theoremhsmall}: the lower bound for Hausdorff dimension is obtained in \6\ref{subsectionmainlowerbound}, after proving \eqref{mainlowerbound}; and the upper bound for Hausdorff dimension is proven in \6\ref{subsectionhsmallHDleq}. 
For the second case of Theorem \ref{theoremhsmall}: the lower bound for Hausdorff dimension is proven in \6\ref{subsectionhsmallHDgeq}; and the upper bound for Hausdorff dimension follows from that for packing dimension (proven in \6\ref{subsectionmainupperbound}).
\item \fbox{Theorem \ref{theoremN2} is proven in \6\ref{subsectionpsi} \6\ref{subsectionPDgeq}, \6\ref{subsectionN2HDgeq}, and \6\ref{subsectionN2HDleq}.}\\ 
Theorem \ref{theorempacking} (proven in \6\ref{subsectionpsi} and \6\ref{subsectionPDgeq}) implies the packing dimension formula in Theorem \ref{theoremN2}. The upper and lower bounds for the Hausdorff dimension formula in Theorem \ref{theoremN2} are proven in \6\ref{subsectionN2HDgeq} and \6\ref{subsectionN2HDleq}, respectively.
\item \fbox{Theorem \ref{theoremN1} is proven in \6\ref{subsectionN1PDgeq},\6\ref{subsectionN1HDgeq}, \6\ref{subsectionN1HDleq}, and \6\ref{subsectionN1PDleq}}.\\ 
The packing dimension upper bound in Theorem \ref{theoremN1} is proven in \6\ref{subsectionN1PDleq}. The packing dimension lower bound is implied by Theorem \ref{theorempacking2} (proven in \6\ref{subsectionN1PDgeq}).
The lower and upper bounds for the Hausdorff dimension formula in Theorem \ref{theoremN1} are proven in \6\ref{subsectionN1HDgeq} and \6\ref{subsectionN1HDleq}, respectively.
\item \fbox{Theorem \ref{theoremspecialcase} is proven in \6\ref{subsectionspecialcase}.}
\item \fbox{Theorem \ref{theoremsingonaverage} is proven in \6\ref{subsectionsingonaverage}.}
\item \fbox{Theorem \ref{theoremstarkov} is proven in \6\ref{subsectionstarkov}.}
\item \fbox{Theorem \ref{theoremkmessenger} is proven in \6\ref{subsectionkmessenger}.}
\item \fbox{Theorem \ref{theoremSSR} is proven in \6\ref{subsectionSSR}.}
\end{itemize}

\section{Proof of \eqref{mainupperbound} + Theorem \ref{theoremhsmall}, upper bound for packing dimension}
\label{subsectionmainupperbound}
In some sense, the variational principle means that it is harder to prove upper bounds on dimension than lower bounds: for a lower bound one only needs to exhibit a template or sequence of templates with the appropriate dimension properties, while for an upper bound one needs to prove something about all possible templates. This is in contrast to the usual situation in which it is easier to prove upper bounds. Our technique for proving upper bounds is based on continuing the analogy with physics (cf. Figure \ref{figuredimtemplate} and the three paragraphs following Definition \ref{definitiondimtemplate}) by defining a function that measures the ``potential energy'' of any configuration of particles: the potential energy is larger the farther apart the particles are. We then prove an inequality relating the change in potential energy and the contraction rate. Integrating this inequality gives a relation between the potential energy at a given point in time, which is always positive, and the average contraction rate up to that time. This then yields a bound on the average contraction rate up to any point in time.

Let $\ff:\Rplus\to\R^d$ be a balanced\Footnote{Since any template can be written as a translation of a balanced template, we can without loss of generality consider only balanced templates in what follows.} template (cf. Definition \ref{definitiontemplate}). We define the ``potential energy of $\ff$ at time $t$'' to be the number
\begin{equation}
\label{phidef}
\phi(t) = \phi_\ff(t) = \max\left(\frac{\pdim^2 \qdim}{\dimsum}|f_1(t)|,\frac{\pdim \qdim^2}{\dimsum}|f_d(t)|\right).
\end{equation}
Note that $\phi(t) \geq 0$ for all $t \geq 0$. The motivation for the definition of $\phi$ is the following lemma:
\begin{lemma}
\label{lemmaphiprimebound}
Let $I$ be an interval of linearity\Footnote{\label{definitionintervallinearity}I.e. an interval on which $\ff$ is linear. If $I$ is an interval of linearity for $\ff$, we will denote the constant value of $\ff'$ on $I$ by $\ff'(I)$.} for $\ff$ such that $\phi'(t)$ is well-defined for all $t\in I$, and such that $\ff(t) \neq \0$ for all $t\in I$. Then
\begin{equation}
\label{phiprimebound}
\phi'(t) \leq \dimsing - \delta(\ff,I)
\end{equation}
for $t\in I$. Equality holds in precisely the following cases:
\begin{itemize}
\item[1.] $S_+(\ff,I) = \{1,\ldots,m\}$;
\item[2.] $S_+(\ff,I) = \{1,\ldots,m-1,m+1\}$, and $f_1 = \ldots = f_m$ and $f_{m+1} = \ldots = f_{m+n}$ on $I$ (and in particular since $\ff$ is balanced we have $\pdim |f_1| = \qdim |f_d|$ on $I$);
\item[3a.] $S_+(\ff,I) = \{2,\ldots,m+1\}$, and $\pdim |f_1| \geq \qdim |f_d|$ on $I$;
\item[3b.] $S_+(\ff,I) = \{1,\ldots,m-1,m+n\}$, and $\qdim |f_d| \geq \pdim |f_1|$ on $I$.
\end{itemize}
If equality does not hold, then the difference between the two sides of \eqref{phiprimebound} is at least $1/\max(m,n)$.
\end{lemma}
Note that when $m=1$, cases 2 and 3a are equivalent, and when $n=1$, cases 2 and 3b are equivalent.
\begin{proof}
Note that the cases 3a and 3b are symmetric with respect to the operation of replacing the $m\times n$ template $\ff$ by the $n\times m$ template $-\ff$, while the other two cases are individually symmetric with respect to this operation. Thus, we may without loss of generality suppose that
\begin{equation}
\label{WLOGphi}
\phi = \frac{\pdim^2 \qdim}{\dimsum}|f_1|
\;\;\;\text{ i.e. }\;\;\;
\pdim |f_1| \geq \qdim |f_d|
\end{equation}
on $I$. Let $j \geq 1$ be the largest number such that
\[
f_j = f_1 \text{ on } I.
\]
Note that since $\ff$ is balanced and $\ff(t) \neq \0$ for all $t\in I$, \eqref{WLOGphi} implies that $j\leq m$. Since $I$ is an interval of linearity for $\ff$, it follows that $f_j < f_{j+1}$ on $I$. Accordingly, let $\dir_\pm = \dir_\pm(\ff,I,j)$ and $S_\pm = S_\pm(\ff,I)$. Then by \eqref{WLOGphi} and \eqref{Lqdef} we have
\[
\phi'(t) = \frac{\pdim^2 \qdim}{\dimsum} \frac{-F_j'(t)}{j} = \frac{\pdim^2 \qdim}{(\dimsum)j}\left[\frac{\dir_-}{\qdim} - \frac{\dir_+}{\pdim}\right]
\]
and on the other hand, by \eqref{codimfI} we have
\begin{equation}
\label{S-0jS+jd}
\dimprod - \delta(\ff,I) \geq \#\big(S_-\cap \OC 0j\big) \cdot\#\big(S_+\cap \OC jd\big) = \dir_- (\pdim - \dir_+)
\end{equation}
and thus
\[
\dimsing - \delta(\ff,I) \geq \dir_- (\pdim - \dir_+) - \frac{\dimprod}{\dimsum}\cdot
\]
So to demonstrate \eqref{phiprimebound} it suffices to show that
\[
\frac{\pdim^2 \qdim}{(\dimsum)j}\left[\frac{\dir_-}{\qdim} - \frac{\dir_+}{\pdim}\right] \leq \dir_- (\pdim - \dir_+) - \frac{\dimprod}{\dimsum}\cdot
\]
Indeed, since $\dir_+ + \dir_- = j$, we have
\begin{align*}
\frac{\pdim^2 \qdim}{(\dimsum)j}\left[\frac{\dir_-}{\qdim} - \frac{\dir_+}{\pdim}\right]
&= \frac{\pdim^2 \qdim}{(\dimsum)j}\left[\dir_-\left(\frac1{\pdim}+\frac1{\qdim}\right) - \frac{j}{\pdim}\right]
= \frac{\dir_- \pdim}{j} - \frac{\dimprod}{\dimsum}
\end{align*}
so we need to show that
\begin{equation}
\label{NTS1}
\frac{\dir_- \pdim}{j} \leq \dir_- (\pdim - \dir_+).
\end{equation}
If $\dir_- = 0$, then this inequality is trivial (and equality holds). So suppose that $\dir_- > 0$. Since $j = \dir_- + \dir_+ \leq \pdim$, we have $\dir_+ < j \leq \pdim$, so $(j-1) (m-L_+ - 1) \geq 0$, and thus
\begin{equation}
\label{mjmL1}
m \leq j + (m - \dir_+) - 1 \leq j(m-\dir_+),
\end{equation}
and rearranging yields \eqref{NTS1}. This completes the proof of \eqref{phiprimebound}.

Now suppose that equality holds in \eqref{phiprimebound}. The equality in \eqref{S-0jS+jd} implies that
\[
S_+ = \{1,\ldots,\dir_+\} \cup \{j+1,\ldots,j+m-\dir_+\}.
\]
On the other hand, the equality in \eqref{NTS1} implies that either $\dir_- = 0$, or equality holds in \eqref{mjmL1}. In the latter case we have $\dir_- = j - \dir_+ = 1$, and either $j = 1$ or $m-\dir_+ = 1$, from the left and right hand sides of \eqref{mjmL1}, respectively. So there are three cases:
\begin{itemize}
\item[1.] If $\dir_- = 0$, then $S_+ = \{1,\ldots,m\}$.
\item[2.] If $\dir_- = 1$ and $m-\dir_+ = 1$, then $S_+ = \{1,\ldots,m-1,m+1\}$. In this case $j=m$, i.e. $f_1 = \ldots = f_m$ on $I$. Combining with \eqref{WLOGphi} and using the fact that $\ff$ is balanced shows that $f_{m+1} = \ldots = f_{m+n}$ on $I$.
\item[3a.] If $\dir_- = 1$ and $j = 1$, then $S_+ = \{2,\ldots,m+1\}$.
\end{itemize}
Note that the case 3b does not appear in this list due to the fact that we made the assumption \eqref{WLOGphi} without loss of generality, using the fact that 3a and 3b are symmetric. The converse direction can be proved similarly.

Finally, suppose that equality does not hold in \eqref{phiprimebound}. Note that the difference between the two sides of \eqref{phiprimebound} is the sum of the difference between the two sides of \eqref{S-0jS+jd} and those of \eqref{NTS1}, i.e. $mn-\delta(\ff,I) - \dir_- m/j$. Since this is a positive rational number with denominator $j$, it must be $\geq 1/j \geq 1/\max(m,n)$.
\end{proof}

Now suppose that the template $\ff$ is singular, i.e. satisfies $f_1(t) \to -\infty$ as $t\to\infty$. Then $\ff(t) \neq \0$ for all sufficiently large $t$. So by Lemma \ref{lemmaphiprimebound}, \eqref{phiprimebound} holds for almost all sufficiently large $t$, and thus for all sufficiently large $T$ we have
\[
0 \lesssim_\plus \phi(T) - \phi(0) \lesssim_\plus \int_0^T [\dimsing - \delta(\ff,t)] \;\dee t = T [\dimsing - \Delta(\ff,T)].
\]
It follows that
\[
\overline\delta(\ff) = \limsup_{T\to\infty} \Delta(\ff,T) \leq \limsup_{T\to\infty} \left[\dimsing + O(1/T) - \frac{\phi(T)}{T}\right] = \dimsing - \liminf_{T\to\infty} \frac{\phi(T)}{T} \leq \dimsing,
\]
and applying Theorem \ref{theoremvariational2} to the set
\[
\SS = \{\ff: \Rplus \to \R^d ~|~ f_1(t) \to -\infty \text{ as } t\to\infty\}
\]
yields \eqref{mainupperbound}. Note that if $\ff$ is $\tau$-singular, i.e. $|f_1(t)| \geq \tau t$ for all sufficiently large $t$, then 
\[
\phi(T) \geq \frac{m^2 n}{m+n}\tau T
\]
for all sufficiently large $T$, and thus
\[
\overline\delta(\ff) \leq \dimsing - \frac{m^2 n}{m+n}\tau.
\]
Applying Theorem \ref{theoremvariational4} yields the upper bound of the packing dimension assertion of Theorem \ref{theoremhsmall}.

\section{Proof of \eqref{mainlowerbound} + Theorem \ref{theoremhsmall}, first formula, lower bound for Hausdorff dimension}
\label{subsectionmainlowerbound}

Lemma \ref{lemmaphiprimebound} provides motivation for how to construct a template yielding the lower bound \eqref{mainlowerbound}. Namely, the template $\ff$ should be constructed in a way such that most of the time, one of the four cases for the possible value of $S_+(\ff,I)$ listed in Lemma \ref{lemmaphiprimebound} holds. For example, there may be two consecutive intervals of linearity $I_1$ and $I_2$ such that $S_+(\ff,I_1) = \{2,\ldots,m+1\}$ and $S_+(\ff,I_2) = \{1,\ldots,m\}$; cf. Figure \ref{figurebasic}.

\begin{figure}[h!]
\scalebox{0.5}{
\begin{tikzpicture}
\clip (-2,-8) rectangle (14,2);
\draw[dashed] (-2,0) -- (14,0);
\draw (0,0) -- (8,-5) -- (12,0);
\draw[line width=3] (0,0) -- (8,1) -- (12,0);
\draw[line width=1.5] (0,-6) -- (8,-6);
\draw[line width=3] (8,-6) -- (12,-6);
\node at (0,-5) {\scalebox{2}{$\delta(\ff,I_i)$}};
\node at (4,-5) {\scalebox{2}{$mn-m$}};
\node at (10,-5) {\scalebox{2}{$mn$}};
\node at (0,-7) {\scalebox{2}{$|I_i|/|I|$}};
\node at (4,-7) {\scalebox{2}{$\tfrac{n}{m+n}$}};
\node at (10,-7) {\scalebox{2}{$\tfrac{m}{m+n}$}};
\node at (0,-1) {\scalebox{2}{$t_0$}};
\node at (8,-1) {\scalebox{2}{$t_1$}};
\node at (12,-1) {\scalebox{2}{$t_2$}};
\end{tikzpicture}
}
\caption{The joint graph of a partial template $\ff$ such that $S_+(\ff,I_1) = \{2,\ldots,m+1\}$ and $S_+(\ff,I_2) = \{1,\ldots,m\}$, where $I_1 = (t_0,t_1)$ and $I_2 = (t_1,t_2)$. In this picture we have $\ff(t_0) = \ff(t_2) = \0$, and thus $|I_1| = \tfrac{n}{m+n}|I|$ and $|I_2| = \tfrac{m}{m+n}|I|$, where $I = (t_0,t_2)$. Consequently,
\[
\frac1{|I|} \int_I \delta(\ff,t) \;\dee t = \frac{n}{m+n}(mn-m) + \frac{m}{m+n}(mn) = \dimsing
\]
i.e. the average contraction rate of $\ff$ over $I$ is $\dimsing$. Note that this partial template is exactly the standard template defined by the points $(t_0,0)$ and $(t_2,0)$ (cf. Definition \ref{definitionstandardtemplate}).}
\label{figurebasic}
\end{figure}
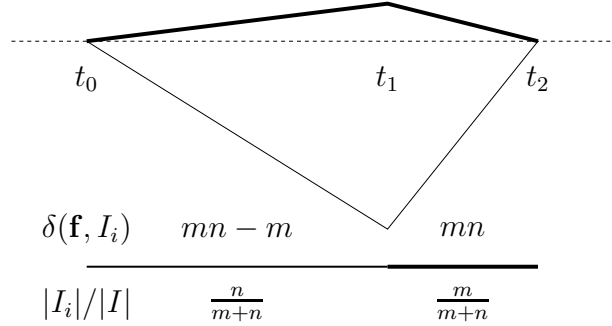

In contrast to the picture in Figure \ref{figurebasic}, if we want the template $\ff$ to be singular then we need $\ff(t) \neq \0$ for all $t$, so we will need to ``cut off'' a small part of the picture. By ``gluing'' infinitely many of these pictures together we will get a singular template of large Hausdorff dimension; cf. Figure \ref{figuremainlowerbound}.

\begin{figure}[h!]
\begin{tikzpicture}
\clip(-1,-3) rectangle (8,1);
\draw[dashed] (-1,0) -- (8,0);
\draw[line width=2] (-0.3,0.26) -- (0.5,0.05) -- (2,0.35) -- (3,0.06) -- (4.8,0.42) -- (6,0.07) -- (7.9,0.47);
\draw (-0.3,-1.3) -- (0.5,-0.25) -- (2,-1.75) -- (3,-0.3) -- (4.8,-2.1) -- (6,-0.35) -- (7.9,-2.35);
\fill[color=gray] (0.3,-0.8) rectangle (0.7,0.4);
\fill[color=gray] (2.75,-0.9) rectangle (3.25,0.45);
\fill[color=gray] (5.7,-1) rectangle (6.3,0.5);
\end{tikzpicture}
\caption{The joint graph of a template $\ff$ designed to be a singular template of large Hausdorff dimension. The gray regions represent intervals where the precise value of the template is irrelevant; what matters is that the template stays away from $\0$ on these regions.}
\label{figuremainlowerbound}
\end{figure}
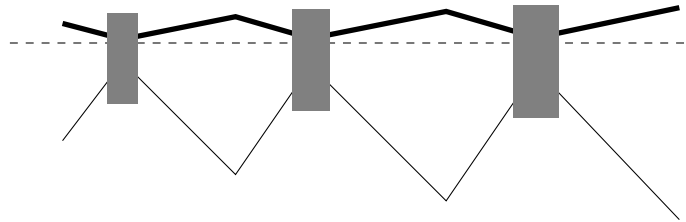

To make the idea conveyed in Figure \ref{figuremainlowerbound} rigorous, we introduce the notion of the \emph{standard template} defined by two points $(t_k,-\epsilon_k)$ and $(t_{k+1},-\epsilon_{k+1})$. The idea is that $\ff:[t_k,t_{k+1}]\to\R^d$ should satisfy $f_1(t_i) = f_{2}(t_i) = -\epsilon_i$ for $i=k,k+1$, and $f_1$ should be as small as possible given this restriction. Finally, the template should be chosen so that $f_d$ is as small as possible, given the previous restrictions. Formally we make the following definition:

\begin{definition}
\label{definitionstandardtemplate}
Fix $0\leq t_k < t_{k+1}$ and $\epsilon_k,\epsilon_{k+1}\geq 0$ and let $\Delta t = \Delta t_k = t_{k+1} - t_k$ and $\Delta\epsilon = \Delta\epsilon_k = \epsilon_{k+1} - \epsilon_k$. Assume that the following formulas hold:
\begin{equation}
\label{standardtemplate}
-\tfrac1m \Delta t \leq \Delta\epsilon \leq \tfrac1n \Delta t,
\end{equation}
\begin{equation}
\label{standardtemplate2}
\Delta\epsilon \geq -\tfrac{n-1}{2n}\Delta t \text{ if } m = 1 \text{ and }
\Delta\epsilon \leq \tfrac{m-1}{2m}\Delta t \text{ if } n = 1,
\end{equation}
\begin{equation}
\label{standardtemplate3}
\text{either }
(n-1)\left(\tfrac1n\Delta t - \Delta\epsilon\right) \geq d\epsilon_k \text{ or }
(m-1)\left(\tfrac1m\Delta t + \Delta\epsilon\right) \geq d\epsilon_{k+1}.
\end{equation}
Then the \emph{standard template} defined by the two points  $(t_k,-\epsilon_k)$ and $(t_{k+1},-\epsilon_{k+1})$ is the partial template $\ff:[t_k,t_{k+1}]\to\R^d$ defined as follows:
\begin{itemize}
\item Let $g_1,g_2:[t_k,t_{k+1}]\to\R$ be piecewise linear functions such that $g_i(t_j) = -\epsilon_j$, and $g_i$ has two intervals of linearity: one on which $g_i' = \frac1m$ and another on which $g_i' = -\frac1n$. For $i=1$ the latter interval comes first while for $i=2$ the former interval comes first; cf. Figure \ref{figurediamond}. The existence of such functions $g_1$ and $g_2$ is guaranteed by \eqref{standardtemplate}. Finally, let $g_3 = \ldots = g_d$ be chosen so that $g_1 + \ldots + g_d = 0$.
\item For each $t\in [t_k,t_{k+1}]$ let $\ff(t) = \gg(t)$ if $g_2(t) \leq g_3(t)$; otherwise let $f_1(t) = g_1(t)$ and let $f_2(t) = \ldots = f_d(t)$ be chosen so that $f_1 + \ldots + f_d = 0$.
\end{itemize}
We will sometimes denote the standard template defined by $(t_k,-\epsilon_k)$ and $(t_{k+1},-\epsilon_{k+1})$ by $\mbf s[(t_k,-\epsilon_k),(t_{k+1},-\epsilon_{k+1})]$.
\end{definition}

\begin{figure}[h!]
\begin{tikzpicture}
\clip(-1,-3) rectangle (8,2);
\draw[dashed] (-1,0) -- (8,0);
\draw (0,-0.5) -- (5,-3) -- (7,-1);
\draw (0,-0.5) --  (0.8,0.3);
\draw[dotted] (0.8,0.3) -- (2,1.5) -- (3.5,0.75);
\draw (3.5,0.75) -- (7,-1);
\draw[line width = 1.5] (0,0.5) --  (0.8,0.3);
\draw[line width = 1.5,dotted] (0.8,0.3) -- (2,0) -- (3.5,0.75);
\draw[line width = 1.5] (3.5,0.75) -- (5,1.5) -- (7,1);
\draw[line width = 2.5] (0.8,0.3) -- (3.5,0.75);
\fill (0,-0.5) circle (0.05);
\node at (0,-1) {$(t_1,-\epsilon_1)$};
\fill (7,-1) circle (0.05);
\node at (7,-0.3) {$(t_2,-\epsilon_2)$};
\end{tikzpicture}
\caption{The joint graphs of $\ff$ and $\gg$ on the interval $[t_1,t_2]$, where $\ff$ is the standard template $\mbf s[(t_1,-\epsilon_1),(t_{2},-\epsilon_{2})]$ as in Definition \ref{definitionstandardtemplate}. The map $\gg$ is shown dotted while $\ff$ is shown solid.}
\label{figurediamond}
\end{figure}
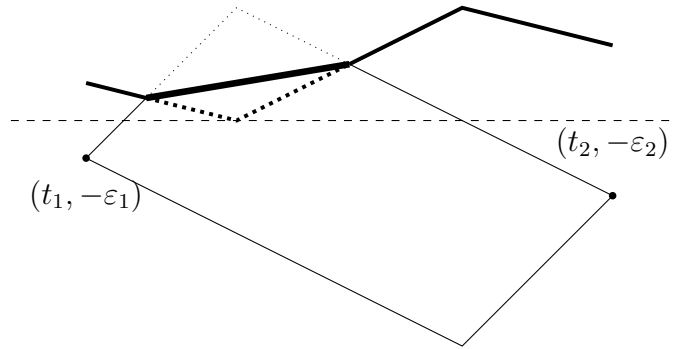

\begin{lemma}
A standard template is indeed a balanced partial template.
\end{lemma}
\begin{proof}
We show where the formulas \eqref{standardtemplate2} and \eqref{standardtemplate3} are needed, leaving the rest of the proof as an exercise to the reader. The condition \eqref{standardtemplate3} is equivalent to the assertion that $g_2(t) \geq g_3(t)$ where $t$ is the location of the maximum of $g_2$. This implies that $f_2(t) = f_3(t)$, guaranteeing that the convexity condition (cf. Definition \ref{definitiontemplate}) is satisfied at $t$. The condition \eqref{standardtemplate2} is equivalent to the assertion that there is no interval on which $f_1' = f_2' = \pm 1$. If such an interval exists, then $\ff$ cannot be a template because if it were, we would have $\{1,2\} \subset S_\pm$ but $\#(S_\pm) = d_\pm = 1$, a contradiction. Conversely, if there is no such interval then the sets $S_\pm$ can be computed in a consistent way on any interval of linearity for $\ff$.
\end{proof}

\begin{example}
The inequalities \eqref{standardtemplate}-\eqref{standardtemplate3} always hold when $\epsilon_k = \epsilon_{k+1} = 0$. In this case, the standard template $\ff$ defined by the points $(t_k,0)$ and $(t_{k+1},0)$ has only two intervals of linearity, and the average value of $\delta(\ff,\cdot)$ on $[t_k,t_{k+1}]$ is equal to $\dimsing$; see Figure \ref{figurebasic}.
\end{example}

\begin{definition}
\label{definitionstandardtemplate2parameter}
Let $(t_k)_0^\infty$ be an increasing sequence of nonnegative real numbers, and let $\epsilon_k\geq 0$ for each $k$. The \emph{standard template} defined by the sequence of points $(t_k,-\epsilon_k)$ is the partial template produced by gluing together the standard templates defined by the pairs of points $(t_k,-\epsilon_k)$ and $(t_{k+1},-\epsilon_{k+1})$ for each $k$. The \emph{standard template} defined by two parameters $\tau\geq 0$ and $\lambda > 1$, denoted $\ff[\tau,\lambda]$, is the one defined by the sequence of points $(t_k,\epsilon_k)_{k\in\Z}$, where $t_k = \lambda^k$ and $\epsilon_k = \tau t_k$ for all $k$. Note that in this case, \eqref{standardtemplate}-\eqref{standardtemplate3} become
\begin{equation}
\label{standardtemplateB}
\tau \leq \tfrac{1}{n},
\end{equation}
\begin{equation}
\label{standardtemplate2B}
\tau \leq \tfrac{m-1}{2m} \text{ if } n = 1,
\end{equation}
\begin{equation}
\label{standardtemplate3B}
\text{either }
(n-1)\left(\tfrac1n - \tau\right) \geq \tfrac1{\lambda - 1} d\tau \text{ or }
(m-1)\left(\tfrac1m + \tau\right) \geq \tfrac\lambda{\lambda - 1}d\tau .
\end{equation}
We refer to $\ff[\tau,\lambda]$ as being \label{expequivariance}\emph{exponentially $\lambda$-equivariant}, viz. that $\ff(\lambda t) = \lambda \ff(t)$ for all $t \geq 0$.
\end{definition}

Now fix $\tau > 0$ small and let $\lambda = 1+\sqrt\tau$ (or more generally $\lambda = 1 + \Theta(\sqrt\tau)$), and note that \eqref{standardtemplateB}-\eqref{standardtemplate3B} hold. Let $\ff = \ff[\tau,\lambda]$ and $t_k,\epsilon_k$ be as above. Now since the map $(\epsilon_1,\epsilon_2) \mapsto \Delta(\mbf s[(0,-\epsilon_1),(1,-\epsilon_2)],1)$ is Lipschitz continuous, it follows that
\begin{align*}
\Delta(\ff, [t_k,t_{k+1}])
&= \Delta(\mbf s[(t_k,-\epsilon_k),(t_{k+1},-\epsilon_{k+1})],[t_k,t_{k+1}])\\
&= \Delta\left(\mbf s\left[\left(0,-\tfrac{\epsilon_k}{\Delta t_k}\right),\left(1,-\tfrac{\epsilon_{k+1}}{\Delta t_k}\right)\right],1\right)\\
&= \Delta(\mbf s[(0,0),(1,0)],1) - O\left(\tfrac{\max(\epsilon_k,\epsilon_{k+1})}{\Delta t_k}\right)\\
&= \dimsing - O\left(\tfrac{\tau}{\lambda-1}\right) = \dimsing - O(\sqrt\tau)
\end{align*}
and thus for sufficiently large $k$
\[
\Delta(\ff,t_k) = \dimsing - O(\sqrt\tau).
\]
Given $T$ large, let $k$ be chosen so that $t_k \leq T < t_{k+1}$. Then
\[
\Delta(\ff,T) - \Delta(\ff,t_k) = O\left(\tfrac{T-t_k}{t_k}\right) = O(\lambda - 1) = O(\sqrt\tau)
\]
and thus
\[
\Delta(\ff,T) = \dimsing - O(\sqrt\tau).
\]
Taking the limit as $T\to\infty$ shows that
\[
\HD(\Sing_\dims(\tau)) \geq \underline\delta(\ff) =  \dimsing - O(\sqrt\tau),
\]
and taking the limit as $\tau\to 0$ completes the proof of \eqref{mainlowerbound}, as well as of the lower bound for Hausdorff dimension in the first case of Theorem \ref{theoremhsmall}.

\begin{remark}
The $O(\sqrt\tau)$ term in the above proof comes from combining two sources of error: one of size $O(\lambda-1)$ and another of size $O(\tfrac{\tau}{\lambda-1})$. We chose $\lambda = 1 + \Theta(\sqrt\tau)$ so as to minimize the sum of these two error terms.
\end{remark}

\begin{remark}
Via a more careful argument one could exactly compute $\underline\delta(\ff[\tau,\lambda])$ in terms of $\tau$ and $\lambda$ for the template $\ff$ described above. Using calculus one could then optimize over the variable $\lambda$ to get a lower bound which is the best possible using this technique.
\end{remark}

\section{Proof of Theorem \ref{theoremhsmall}, upper bound for Hausdorff dimension}
\label{subsectionhsmallHDleq}

Let $\ff$ be a $\tau$-singular template such that $\underline\delta(\ff) > \dimsing - z$, where $z > 0$ is small. We aim to show that $\dynexp = O(z^2)$ if $(m,n)\neq (2,2)$. Indeed, let $\phi$ be as in \eqref{phidef}, and let
\begin{equation}
\label{exceptional}
E \df \{t\geq 0 : \phi'(t) < \dimsing - \delta(\ff,t)\}.
\end{equation}
By Lemma \ref{lemmaphiprimebound}, we have
\[
\phi'(t) \leq \dimsing - \delta(\ff,t) - \max(m,n)^{-1} \big[t\in E\big]
\]
for all $t$ sufficiently large. Here $[t\in E]$ denotes $1$ if $t\in E$ and $0$ otherwise.
Integrating over $[0,T]$ gives
\[
\phi(T) - \phi(0) \leq T \big(\dimsing - \Delta(\ff,T)\big) - \max(m,n)^{-1} \lambda\big(E\cap [0,T]\big),
\]
where $\lambda$ denotes Lebesgue measure. On the other hand, since $\underline\delta(\ff) > \dimsing - z$, we have $\Delta(\ff,T) \geq \dimsing - z$ for all sufficiently large $T$, and thus rearranging the previous equation and using the fact that $\phi(T) \geq 0$ gives
\begin{equation}
\label{exceptionalsmall}
\lambda\big(E\cap [0,T]\big) = O(zT)
\end{equation}
and
\begin{equation}
\label{phibound}
\phi(T) = O(zT)
\end{equation}
assuming $T$ is sufficiently large. The trick now is that we also know $\phi(T) = \Omega(\tau T)$ since $\ff$ is $\tau$-singular (which means that $|f_1(t)| \geq \tau t$ for all sufficiently large $t$). So the question is what kind of templates satisfy both an upper bound and a lower bound for $\phi$, but for which the exceptional set $E$ is not large. The answer is given by the following lemma, in which the problem has been rescaled so that the upper bound for $\phi$ is just $1$:

\begin{lemma}
\label{lemmanonstatic}
Suppose that $(m,n)\neq (2,2)$, and fix $x > 0$. Let $\ff : I \df [t_-,t_+]\to\R^d$ be a partial template such that $x\leq \phi(t) \leq 1$ for all $t\in I$. Then if $|I|$ is sufficiently large depending on $m,n$, then
\[
\lambda(E) \gtrsim x,
\]
where the exceptional set $E$ is as in \eqref{exceptional}.
\end{lemma}
\begin{proof}
Let $y \df \lambda(E)$; we need to show that either $\phi(t) \lesssim y$ for some $t\in I$, or else $|I| = O(1)$.

Throughout this proof, we will call an interval $J$ a Type 1 interval if case 1 of Lemma \ref{lemmaphiprimebound} holds along it; we define Type 2/3a/3b intervals similarly. The basic idea is to reduce to the case of a Type 2 interval to the left of a Type 3 interval to the left of a Type 1 interval, modulo a small perturbation. Since $\ff$ cannot be static on any interval of fixed Type, the bound on $\phi$ implies a bound on the length of each interval and thus on the length of the whole interval $I$. The proof now splits into two cases.

{\bf Case 1}: Suppose first that there is some Type 1 interval which is to the left of a Type 2/3a/3b interval. Without loss of generality, we may assume that there are no Type (1/2/3a/3b) intervals between them. It follows that if the two intervals are $I_1 = (t_1,t_2)$ and $I_2 = (t_3,t_4)$, respectively, then we have $0 \leq t_3 - t_2 \leq y$.

If $I_2$ is Type 2, then $f_1(t_3) = \ldots = f_m(t_3)$ and $f_{m+1}(t_3) = \ldots = f_{m+n}(t_3)$. On the other hand, by the convexity condition we have $f_m(s) = f_{m+1}(s)$ for some $s\in [t_2,t_3]$. It follows that $|f_{m+1}(t_3) - f_m(t_3)| \lesssim y$ and thus $\phi(t_3) \asymp |\ff(t_3)| \lesssim y$.

If $I_2$ is Type 3a, then $m|f_1(t_3)| \geq n|f_d(t_3)|$. On the other hand, by the convexity condition, for each $j = 1,\ldots,m$ there exists $s_j\in [t_2,t_3]$ such that $f_j(s_j) = f_{j+1}(s_j)$. It follows that $|f_{j+1}(t_3) - f_j(t_3)| \lesssim y$, so
\begin{align*}
(m+1)|f_1(t_3)| &= -\sum_{i=1}^{m+1} f_i(t_3) + O(y) = \sum_{i=m+2}^{m+n} f_i(t_3) + O(y)\\
& \leq n|f_d(t_3)| + O(y) \leq m|f_1(t_3)| + O(y)
\end{align*}
and thus $\phi(t_3) \asymp |\ff(t_3)| \lesssim y$. A similar argument applies if $I_2$ is Type 3b.

{\bf Case 2}: On the other hand suppose that no Type 1 interval is to the left of any Type 2/3a/3b interval. Now let
\[
\psi(t) \df \big|m|f_1(t)| - n|f_d(t)|\big|
\]
and let $J$ be a Type 2/3a/3b interval. If $J$ is Type 2, then $\psi=0$ on $J$ and thus $\psi' = 0$. Suppose that $J$ is Type 3a. Then on $J$ we have $m|f_1| \geq n|f_d|$, $f_1' = -\frac1n$, and $f_d' \leq \frac1{n(m+n-1)}$, from which it follows that
\[
\psi' \geq c_{m,n} \df \frac mn - \frac{1}{m+n-1} \cdot
\]
Note that $c_{m,n} > 0$ unless $m=1$, in which case $c_{m,n} = 0$. Similar logic shows that if $J$ is Type 3b, then $\psi' \geq c_{n,m}$ on $J$.

Now let $A_i$ denote the union of the Type $i$ intervals in $I$. Note that $A_1\cup A_2\cup A_3 \cup E = I$ except for finitely many points. We can assume that $t_0 \df \sup(A_2) \leq t_1 \df \inf(A_1)$ and $\sup(A_3) \leq t_1$, as otherwise we are in Case 1 and we are done by the preceding argument. Since $t_0$ is the endpoint of a Type 2 interval, we have $\psi(t_0) = 0$. On the other hand, we have
\[
\psi(t_0) \geq \int_{t_-}^{t_0} \psi'(t) \;\dee t \geq c_{m,n}\lambda\big([t_-,t_0]\cap A_{3\mathrm a}\big) + c_{n,m}\lambda\big([t_-,t_0]\cap A_{3\mathrm b}\big) - O(y)
\]
so we have $\lambda\big([t_-,t_0]\cap A_{3\mathrm a}\big) = O(y)$ if $m\geq 2$ and $\lambda\big([t_-,t_0]\cap A_{3\mathrm b}\big) = O(y)$ if $n\geq 2$, respectively. On the other hand, if $m=1$ then $A_{3\mathrm a} = A_2$ and if $n=1$ then $A_{3\mathrm b} = A_2$. Consequently
\[
\lambda\big([t_-,t_0]\cap A_3\butnot A_2\big) = O(y)
\]
and thus since $\phi' = \dimsing - (mn-1)$ on $A_2$, we have
\begin{align*}
0&\asymp_\plus \phi(t_0) - \phi(t_-) = \int_{t_-}^{t_0} \phi'(t) \;\dee t\\
&= [\dimsing - (mn-1)]\lambda\big([t_-,t_0]\cap A_2\big) - O\left(\lambda\big([t_-,t_0]\cap E \cup A_3 \butnot A_2\right)\\
&= \left(1 - \frac{mn}{m+n}\right) (t_0 - t_-) - O(y)
\end{align*}
Since $(m,n)\neq (2,2)$ by assumption, we have $(m-1)(n-1)\neq 1$ and thus
\[
1 - \frac{mn}{m+n} \neq 0
\]
and thus $t_0 - t_- = O(1)$. Similarly, since $\phi' = \dimsing - (mn-m) = \frac{m^2}{m+n}$ on $A_{3\mathrm a}$ and $\phi' = \dimsing - (mn-n) = \frac{n^2}{m+n}$ on $A_{3\mathrm b}$, and since $(t_0,t_1) \subset A_3\cup E$, we have
\[
0 \gtrsim_\plus \frac{\min(m,n)^2}{m+n} (t_1 - t_0) - O(y).
\]
Since $\phi' = \dimsing - mn = -\frac{mn}{m+n}$ on $A_1$, and since $(t_1,t_+) \subset A_1\cup E$, we have
\[
0 \asymp_\plus -\frac{mn}{m+n} (t_+-t_1) + O(y).
\]
Thus $t_1 - t_0 = O(1)$ and $t_+ - t_1 = O(1)$, so combining gives $|I| = t_+ - t_- = O(1)$.
\end{proof}

Let $C > 0$ be the constant such that Lemma \ref{lemmanonstatic} is true whenever $|I| \geq C$. Notice that any partial template whose domain has length $\geq C$ can be split up into partial templates whose domains have length $=C$ which cover the majority of the original domain. It follows that in the context of Lemma \ref{lemmanonstatic}, in general we have
\[
\lambda(E) \gtrsim x |I| \text{ as long as } |I| \geq C,
\]
where $I$ is the domain of a partial template $\ff$ satisfying $x \leq \phi \leq 1$. Applying a scaling argument yields:

\begin{lemma}
\label{lemmanonstatic2}
Suppose that $(m,n)\neq (2,2)$, and fix $0 < x_0 \leq x_1$ and $I\subset \R$ such that $|I| \geq C x_1$. Let $\ff:I\to\R^d$ be a partial template such that $x_0 \leq \phi(t) \leq x_1$ for all $t\in I$. Then
\[
\lambda(E) \gtrsim \frac{x_0 |I|}{x_1}\cdot
\]
\end{lemma}

Now fix $T$ large, let $I = [T/2,T]$, and let $x_0 = \inf_I \phi$, $x_1 = \sup_I \phi$. Since $\ff$ is $\tau$-singular we have $x_0 \gtrsim \tau T >0$, while by \eqref{phibound} we have $x_1 = O(zT)$. In particular, if $z$ is sufficiently small then $T \geq Cx_1$. Consequently, by Lemma \ref{lemmanonstatic2} and \eqref{exceptionalsmall},
\[
\tau T^2 = O(x_0 T) = O\big(x_1 \lambda(E\cap I)\big) = O(zT)^2,
\]
which implies $\tau = O(z^2)$.

It follows that if $\ff$ is a $\tau$-singular template, then $\underline\delta(\ff) \leq \dimsing - \Theta(\sqrt\tau)$, since otherwise we can take $z = 2(\dimsing - \underline\delta(\ff))$ and apply the above argument. Taking the supremum over $\ff$ and applying Theorem \ref{theoremvariational4} shows that
\[
\HD\left(\Sing_\dims(\tau)\right) \leq \dimsing - \Theta(\sqrt\tau) \text{ if } (m,n) \neq (2,2).
\]
When $(m,n) = (2,2)$, the upper bound for Hausdorff dimension follows from the upper bound for packing dimension which we proved in \6\ref{subsectionmainupperbound}.

\section{Proof of Theorem \ref{theoremhsmall}, second formula, lower bound for Hausdorff dimension}
\label{subsectionhsmallHDgeq}

In this proof, we will employ a variant of the notion of a standard template defined by two parameters, as in Definition \ref{definitionstandardtemplate2parameter}, by introducing a third parameter.

Let $m=n=2$, and fix $0 < \tau < \frac1n = \frac12$. Fix $\lambda > 1$ and let $t_k = \lambda^k$ and $\epsilon_k = \tau t_k$. However, rather than letting $\ff = \ff[\tau,\lambda]$ (as in Definition \ref{definitionstandardtemplate2parameter}), we will introduce a new parameter $\gamma \in [1 + 6\tau + 2\lambda\tau,\lambda]$. We define $\ff$ as follows:
\begin{itemize}
\item On $[1,\gamma]$, we have $\ff = \mbf s[(1,-\tau),(\gamma,-\lambda\tau)]$. Note that \eqref{standardtemplate3} is satisfied due to the lower bound on $\gamma$.
\item Extend $\ff$ to $[\gamma,\lambda]$ via the requirement that $\ff$ is constant on $[\gamma,\lambda]$: $f_1 = f_2 = -\lambda\tau$ and $f_3 = f_4 = \lambda\tau$ on $[\gamma,\lambda]$.
\item Extend $\ff$ to $[0,\infty)$ via exponential equivariance\Footnote{Note that we form infinitely many periods when extending $\ff$ backwards from $1$ to $0$, and so $\ff$ now has infinitely many intervals of linearity in $[0,1]$. However, this does not cause any problems in what follows.}, i.e. so that $\ff(\lambda t) = \lambda \ff(t)$ for all $t \geq 0$.
\end{itemize}
For simplicity of calculation, we set $\gamma = 1 - 2\tau + 10\lambda\tau$ (this is possible as long as $1 - 2\tau + 10\lambda\tau \leq \lambda$), since this means that $\ff$ has only three intervals of linearity on $[1,\gamma]$ (otherwise $\ff$ has four intervals of linearity on $[1,\gamma]$):
\[
\ff'(t) = \begin{cases}
(-\tfrac12,\tfrac12,0,0) & 1 < t < 1 + 4\tau\\
(-\tfrac12,\tfrac16,\tfrac16,\tfrac16) & 1 + 4\tau < t < 1 - 2\tau + 6\lambda\tau\\
(\tfrac12,-\tfrac12,0,0) & 1 - 2\tau + 6\lambda\tau < t < 1 - 2\tau + 10\lambda\tau = \gamma
\end{cases}
\]
(cf. Figure \ref{figure2x2}). It follows that
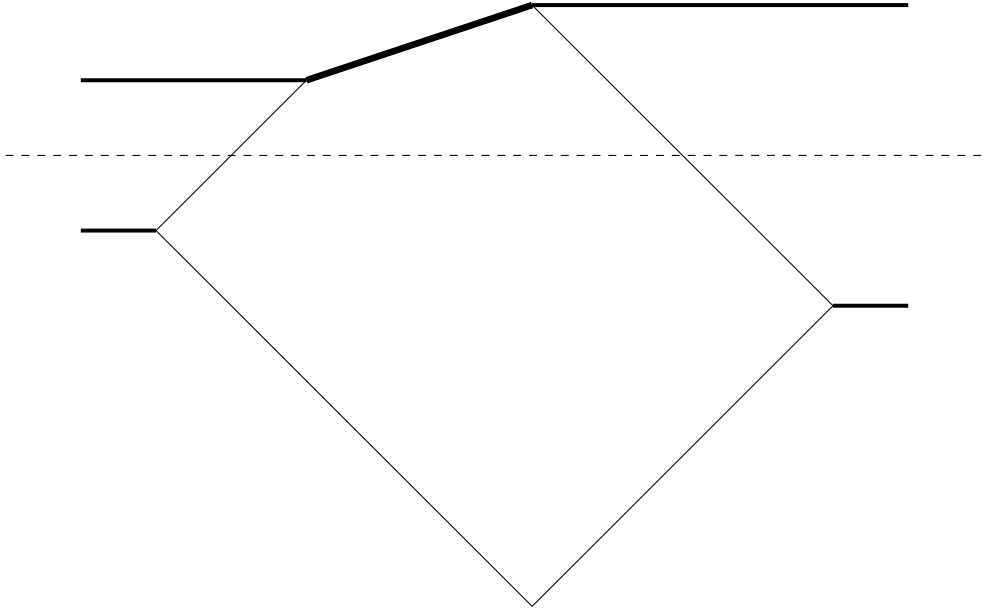
\begin{figure}
\begin{tikzpicture}
\clip(-1,-7) rectangle (12,3);
\draw[dashed] (-1,0) -- (12,0);
\draw[line width=1.5] (0,1)--(3,1);
\draw[line width=2.5] (3,1)--(6,2);
\draw[line width=1.5] (6,2) -- (11,2);
\draw[line width=1.5] (0,-1)--(1,-1);
\draw (1,-1) -- (3,1);
\draw (1,-1) -- (6,-6);
\draw (6,-6) -- (10,-2);
\draw (6,2) -- (10,-2);
\draw[line width=1.5] (10,-2) -- (11,-2);
\end{tikzpicture}
\caption{A period of an exponentially equivariant $2\times 2$ template, with $\gamma = 1 - 2\tau + 10\lambda\tau$. Here a template is called exponentially equivariant if it is equal to a scaled copy of itself; the ``period'' is an interval which is long enough to recover the template from this self-similarity property.}
\label{figure2x2}
\end{figure}
\[
\delta(\ff,t) = \begin{cases}
2 & 1 < t < 1 - 2\tau + 6\lambda\tau\\
3 & 1 - 2\tau + 6\lambda\tau < t < \lambda
\end{cases}
\]
and thus if we let $r = 1-2\tau+6\lambda\tau$, then the minima of the exponentially $\lambda$-periodic\Footnote{Meaning that $\Delta(\ff,\lambda T) = \Delta(\ff,T)$ for all $T > 0$.} function $\Delta(\ff,\cdot)$ occur at $\lambda^k r$ for $k\in\Z$. It follows that
\begin{align*}
\underline\delta(\ff) &= \Delta(\ff,r)
= \Delta(\ff,[\lambda^{-1} r,r])\\
&= \frac{3(1 - \lambda^{-1} r) + 2 (r - 1)}{r - \lambda^{-1} r}
= 3 - \frac{r - 1}{r - \lambda^{-1} r}\\
&= 3 - \frac{6\lambda \tau - 2\tau}{(1 - \lambda^{-1})(1 - 2\tau + 6\lambda\tau)}\\
&= 3 - \Theta(\tau),
\end{align*}
where the implied constant of $\Theta$ can depend on $\lambda$. This completes the proof. Note that as in \6\ref{subsectionmainlowerbound}, one can optimize over the parameter $\lambda$ to get the best possible bound using templates of this form, but we omit the required calculations. 

\section{Proof of Theorem \ref{theorempacking}, lower bound}
\label{subsectionPDgeq}

We consider a two-parameter standard template $\ff[\tau,\lambda]$ (as in Definition \ref{definitionstandardtemplate2parameter}). Fix $0 < \tau < 1/n$, such that $\tau < \frac{m-1}{2m}$ if $n=1$. Now if $\lambda$ is sufficiently large, then \eqref{standardtemplateB}-\eqref{standardtemplate3B} are satisfied (the left half of \eqref{standardtemplate3B} if $n\geq 2$, and the right half if $n=1$), and thus there is a standard template $\ff = \ff_\lambda = \ff[\tau,\lambda]$ defined by the sequence of points $(t_k,-\epsilon_k)_0^\infty$, where $t_k = \lambda^k$ and $\epsilon_k = \tau t_k$.

\begin{claim}
\label{claimuppercontractionlimit}
Let $\gg \df \mbf s[(0,0),(1,-\tau)]$ (as in Definition \ref{definitionstandardtemplate}). As $\lambda\to\infty$, the upper average contraction rate of $\ff_\lambda$ tends to
\begin{equation}
\label{deltataupackingdef}
\begin{split} 
\sup_{0 < T \leq 1} \Delta(\gg,T) &= \overline\delta_\dims(\tau)\\ 
&\df \max\left(mn - m, \; \dimsing - \frac{mn}{m+n}(d+m)\tau, \; mn - \frac{mn}{m+n}\frac{1+m\tau}{1-\frac{mn}{m-1}\tau}\right).
\end{split}
\end{equation}
\end{claim}

\begin{proof}
Indeed, first let $\gamma > 0$ be small enough so that $\gg' = (-\tfrac1n,\tfrac1{n(d-1)},\ldots,\tfrac1{n(d-1)})$ on $(0,2\gamma)$; the definition of $\gg$ guarantees that such $\gamma$ exists. Since $\ff_\lambda$ is exponentially $\lambda$-equivariant  and since $[\gamma,\lambda\gamma]$ is a period of $\ff_\lambda$ we have
\[
\overline\delta(\ff_\lambda) = \sup_{T\in [\gamma,\lambda\gamma]} \Delta(\ff_\lambda,T).
\]
Next, we extend $\gg$ to $\Rplus$ by stipulating that $g_1' = -\tfrac1n$ on $\CO 1\infty$ and then defining $g_2,\ldots,g_d$ on $\CO 1\infty$ in the same way as for standard templates (as in Definition \ref{definitionstandardtemplate}). Now since $\ff_\lambda' \to \gg'$ almost everywhere as $\lambda \to \infty$, it follows that $\Delta(\ff_\lambda,\cdot)\to \Delta(\gg,\cdot)$ uniformly on $[\gamma,\lambda\gamma]$, i.e. for every $\epsilon > 0 $ there exists $\lambda_0$ such that for all $\lambda \geq \lambda_0$ we have $|\Delta(\ff_\lambda,\cdot) - \Delta(\gg,\cdot)| < \epsilon$ on $[\gamma,\lambda\gamma]$. Thus, we have that 
\[
\lim_{\lambda\to\infty} \overline\delta(\ff_\lambda) = \sup_{T\in \CO\gamma\infty} \Delta(\gg,T).
\]
But since $\delta(\gg,t) = mn-m$ for all $t\in [0,\gamma] \cup \CO1\infty$, it follows that $\Delta(\gg,T) \leq \max(mn-m,\Delta(\gg,1)) = \max(\Delta(\gg,\gamma),\Delta(\gg,1))$ for all $T\in [0,\gamma] \cup \CO1\infty$, and thus
\[
\sup_{T\in \CO\gamma\infty} \Delta(\gg,T) = \sup_{0 < T \leq 1} \Delta(\gg,T).
\]
To complete the proof, we need to show that \eqref{deltataupackingdef} holds, i.e. that $\sup_{0 < T \leq 1} \Delta(\gg,T) = \overline\delta_\dims(\tau)$. Indeed, from the definition of $\gg$, it follows that there exist intervals $I_i = (t_i,t_{i+1})$, $i=0,1,2$, with $t_0 = 0$, $t_3 = 1$, as follows:

\begin{table}[h!]
\begin{tabular}{|l|l|l|l|}
\hline
& $(g_1',g_2')$ & $S_+(\gg,\cdot)$ & $mn-\delta(\gg,\cdot)$\\
\hline
$I_0$ & $(-\tfrac1n,\tfrac1{n(d-1)})$ & $\{2,\ldots,m+1\}$ & $m$\\
$I_1$ (case 1) & $(-\tfrac1n,-\tfrac1n)$ & $\{3,\ldots,m+2\}$ & $2m$\\
$I_1$ (case 2) & $(\tfrac1m,-\tfrac1{m(d-1)})$ & $\{1,\ldots,m\}$ & $0$\\
$I_2$ & $(\tfrac1m,-\tfrac1n)$ & $\{1,3,\ldots,m+1\}$ & $m-1$\\
\hline
\end{tabular}
\hspace{0.2 in}
\caption{Two cases for the intervals of linearity of $\gg$. See Figure \ref{2cases}.}
\label{tablestandardzerotemplate}
\end{table}

\begin{figure}[h!]
\begin{tikzpicture}
\clip(-1,-5) rectangle (7,2);
\draw[dashed] (-1,0) -- (7,0);
\draw (0,0) -- (4,-4);
\draw[line width=2.5] (0,0) -- (3,1);
\draw (3,1) -- (6,-2);
\draw (4,-4) -- (6,-2);
\draw[line width=2] (3,1) -- (6,1);
\end{tikzpicture}
\begin{tikzpicture}
\clip(-1,-5) rectangle (7,2);
\draw[dashed] (-1,0) -- (7,0);
\draw (0,0) -- (3,-3);
\draw[line width=2.5] (0,0) -- (3,1);
\draw (3,-3) -- (5.5,-0.5);
\draw[line width=2.5] (3,1) -- (4.5,0.5);
\draw (4.5,0.5) -- (5.5,-0.5);
\draw[line width=2] (4.5,0.5) -- (5.5,0.5);
\end{tikzpicture}
\caption{The joint graph of $\gg$ in Case 1 and Case 2, respectively. Note that the slope of the last top segment may be either negative or positive according to whether $m<n$ or $m>n$, respectively (in the picture we have $m=n$ which corresponds to a horizontal slope).}
\label{2cases}
\end{figure}
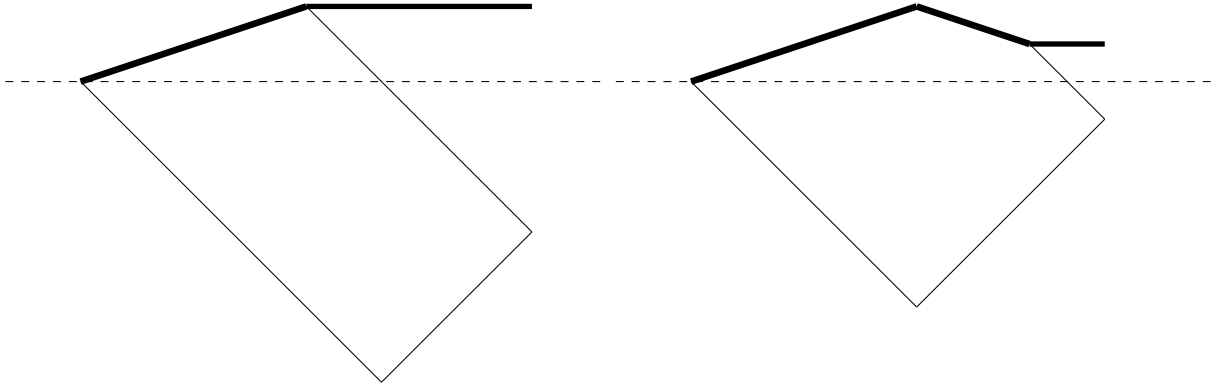

\noindent Here case 1 holds when $\tau \geq \frac{m-1}{n(d+m-1)}$, while case 2 holds when $\tau \leq \frac{m-1}{n(d+m-1)}$. (When equality holds, $I_1$ is empty and so the cases are compatible.) Now let $0 < T\leq 1$ be maximal such that $\Delta(\gg,\cdot)$ attains its maximum at $T$. Then $\delta(\gg,t) \geq \Delta(\gg,T)$ for $t$ slightly less than $T$, while $\delta(\gg,t) < \Delta(\gg,T)$ for $t$ slightly greater than $T$. Thus $T = t_i$ for some $i=1,2,3$. But if case 1 holds, then $\Delta(\gg,t_2) < mn-m = \Delta(\gg,t_1)$, so if $T=t_2$ then case 2 holds. Now it can be checked by direct calculation\Footnote{The calculation of $\Delta(\gg,t_3)$ is somewhat tedious and it is easier to use the equality case of Lemma \ref{lemmapsiprimebound} below instead of performing a direct computation, since $\psi_\gg(1) = \frac{mn}{m+n}(d+m)\tau$. Some other formulas useful for the calculations: when case 2 of Table \ref{tablestandardzerotemplate} holds we have $t_1 = \tfrac n{m+n}(1+m\tau)$ and $t_2 = 1 - \frac{mn}{m-1}\tau$.} that
\begin{align*}
\Delta(\gg,t_1) &= mn - m,\\
\Delta(\gg,t_2) &= mn - \frac{mn}{m+n}\frac{1+m\tau}{1-\frac{mn}{m-1}\tau} \text{ if case 2 holds},\\
\Delta(\gg,t_3) &= \dimsing - \frac{mn}{m+n} (d+m)\tau,
\end{align*}
which implies \eqref{deltataupackingdef}, since if $\tau \geq \frac{m-1}{n(d+m-1)}$ then 
\[
mn - \frac{mn}{m+n} \cdot \frac{1+m\tau}{1-\frac{mn}{m-1}\tau} \leq mn-m,
\] 
and thus when case 1 holds, the last term on the right-hand side of \eqref{deltataupackingdef} does not contribute to the maximum\Footnote{Note that when $m=1$, case 1 holds for all $\tau\geq 0$ and thus again the last term on the right-hand side of \eqref{deltataupackingdef} can be ignored.}. This concludes the proof of the claim.
\end{proof}

Applying the variational principle (Theorem \ref{theoremvariational2}) to $\ff_\lambda$ gives us that 
\[
\PD(\Sing_\dims(\tau)) \geq \lim_{\lambda\to\infty} \overline\delta(\ff_\lambda) = \sup_{0 < T \leq 1} \Delta(\gg,T) = \overline\delta_\dims(\tau).
\]

This completes the proof of the lower bound in Theorem \ref{theorempacking}.

\section{Proof of Theorem \ref{theorempacking}, upper bound when $n \geq 2$}
\label{subsectionpsi}

To prove the upper bound when $n \geq 2$ in Theorem \ref{theorempacking}, i.e. equality holds in \eqref{packinginequality}, we need a different definition of ``potential energy'' (cf. Section \ref{subsectionmainupperbound}). Let $\ff:\Rplus\to\R^d$ be a balanced $m\times n$ template. For each $t\geq 0$ let
\[
\psi(t) = \psi_\ff(t) \df \max\left(\frac{\pdim \qdim}{\dimsum}\big|(m+1)f_1(t) + (d-1)f_2(t)\big|,\frac{\pdim \qdim^2}{\dimsum} |f_d(t)|\right).
\]
Note that since $\ff$ is balanced,
\[
(m+1)f_1(t) + (d-1)f_2(t) \leq (m+1)f_1(t) + f_2(t) + \ldots + f_d(t) = m f_1(t) \leq 0
\]
and thus $\psi(t) \geq \phi(t) \geq 0$ for all $t \geq 0$. The analogous result to Lemma \ref{lemmaphiprimebound} is stated as follows:

\begin{lemma}
\label{lemmapsiprimebound}
Suppose that $n\geq 2$. Let $I$ be an interval of linearity for $\ff$ such that $\psi'(t)$ is well-defined for all $t\in I$, and such that $\ff(t) \neq \0$ for all $t\in I$. Then
\begin{equation}
\label{psiprimebound}
\psi'(t) \leq \dimsing - \delta(\ff,t)
\end{equation}
for $t\in I$. Equality holds in the following (non-exhaustive) cases:
\begin{itemize}
\item[1.] when $f_1 < f_2 = f_d$ on $I$,
\item[2.] when $f_1 < f_2< f_3 = f_d$, and $f_2' = -1/n$ on $I$.
\end{itemize}
\end{lemma}
Note that there is no symmetry here, unlike in the proof of Lemma \ref{lemmaphiprimebound}, since $\psi$ is not symmetric with respect to $\ff\mapsto -\ff$.
\begin{proof}
The proof is similar to that of Lemma \ref{lemmaphiprimebound}. We can suppose that
\begin{equation}
\label{psiassumption}
\big|(m+1)f_1(t) + (d-1)f_2(t)\big| \geq n |f_d(t)|
\end{equation}
for $t\in I$, since otherwise $\psi = \phi$ on $I$ and Lemma \ref{lemmaphiprimebound} implies the conclusion. Let $j \geq 2$ be the largest number such that
\[
f_2 = f_j \text{ on } I.
\]
Since $I$ is an interval of linearity for $\ff$, we have $f_j < f_{j+1}$ on $I$. Let $L_\pm = L_\pm(\ff,I,j)$ and $S_\pm = S_\pm(\ff,I)$. The proof now splits into two cases, first if $f_1 < f_2$ on $I$, and second if $f_1 = f_2$ on $I$.

{\bf Case 1}: Suppose first that $f_1 < f_2$ on $I$. Let $A_\pm = L_\pm(\ff,I,1)$ and $B_\pm = L_\pm - A_\pm$. By \eqref{Lqdef}, on $I$ we have
\begin{align*}
F_j' &= \frac j\pdim - \frac{\dimsum}{\dimprod} L_-\\
f_1' &= \frac 1\pdim - \frac{\dimsum}{\dimprod} A_-\\
\psi' &= -\frac{\dimprod}{\dimsum}\left((m+1) f_1' + \frac{d-1}{j-1} (F_j' - f_1')\right)\\
&= -\frac{(m+d)n}{d} + (m+1) A_- + \frac{d-1}{j-1} B_-
\end{align*}
and on the other hand, by \eqref{codimfI} we have
\begin{equation}
\label{S-0jS+jdv2}
\begin{split}
\dimprod - \delta(\ff,t) &\geq \#\big(S_-\cap\{1\}\big)\cdot \#\big(S_+\cap \OC 1d\big) + \#\big(S_-\cap \OC 1j\big) \cdot\#\big(S_+\cap \OC jd\big)\\
&= m A_- + B_- (\pdim - L_+)
\end{split}
\end{equation}
and thus
\[
\dimsing - \delta(\ff,t) \geq m A_- + B_- (\pdim - L_+) - \frac{mn}{d}\cdot
\]
So to demonstrate \eqref{psiprimebound}, it suffices to show that
\[
-n + (m+1) A_- + \frac{d-1}{j-1} B_- \leq m A_- + B_- (\pdim - L_+).
\]
Rearranging gives the equivalent formulation
\[
\frac{d-1}{j-1} B_- \leq (n - A_-) + B_- (m - L_+).
\]
If $B_- = 0$ this is obviously true (and since $n\geq 2$ by assumption, the inequality is strict in this case), and therefore if we backtrack we get that \eqref{psiprimebound} is true as well in this case. 
Otherwise, assume that $B_- > 0$. Then we can rearrange again to get
\[
\frac{d-1}{B_+ + B_-} \leq \frac{n - A_-}{B_-} + m - L_+,
\]
and subtracting 1 from both sides gives
\begin{equation}
\label{ETS2}
\frac{(n-L_-) + (m-L_+)}{B_+ + B_-} \leq \frac{n - L_-}{B_-} + m - L_+.
\end{equation}
This formula is true since $\frac1{B_+  + B_-} \leq \min(\frac1{B_-},1)$, and so backtracking shows that \eqref{psiprimebound} is true as well. If $f_2 = f_d$ on $I$, then $j=d$ and thus $L_+ = m$, $L_- = n$ and so equality holds (in \eqref{ETS2} and equivalently) in \eqref{psiprimebound}. Similarly, if $f_1 < f_2 < f_3 = f_d$ and $f_2' = -1/n$ on $I$, then $j=2$ and $B_+ = 0$, so $B_- = B_+ + B_- = j - 1 = 1$ and thus equality holds. This completes the proof of Case 1.

{\bf Case 2}: Next suppose that $f_1 = f_2$ on $I$. Then on $I$ we have
\[
\psi' = -\frac{\dimprod}{\dimsum} \frac{m+d}{j} F_j' = -\frac{(m+d)n}{d} + \frac{m+d}{j} L_-
\]
and on the other hand, as in \eqref{S-0jS+jd} we have
\begin{equation}
\label{S-0jS+jdv3}
\dimsing - \delta(\ff,I) \geq L_- (m - L_+) - \frac{mn}{d}
\end{equation}
so to demonstrate \eqref{psiprimebound}, it suffices to show that
\[
-n + \frac{m+d}{j} L_- \leq L_- (m - L_+).
\]
If $L_- = 0$ this is obvious (and the inequality is strict), so assume that $L_- > 0$. Then rearranging gives the equivalent formulation
\[
\frac{2m + n}{L_+ + L_-} \leq \frac{n}{L_-} + m - L_+.
\]
Write $M_+ = m - L_+$ and $M_- = n - L_-$. Then subtracting $1$ from both sides gives
\[
\frac{L_+ + 2M_+ + M_-}{L_+ + L_-} \leq \frac{M_-}{L_-} + M_+
\]
and multiplying by $L_+ + L_-$ and rearranging gives
\begin{equation}
\label{ETS1}
L_+ \leq \frac{M_- L_+}{L_-} + M_+ (L_+ + L_- - 2).
\end{equation}
We now demonstrate \eqref{ETS1}. First, note that since $L_+ + L_- = j \geq 2$, both terms on the right-hand side are nonnegative. So if either term is individually at least $L_+$, then \eqref{ETS1} holds. In particular, if $L_- \leq M_-$, then the first term is $\geq L_+$, and if $L_- \geq 2$ and $M_+ \geq 1$, then the second term is $\geq L_+$. Also, if $L_+ = 0$ then \eqref{ETS1} obviously holds. So assume that $L_+ > 0$, that $L_- > M_-$, and that either $L_- \leq 1$ or $M_+ = 0$.

If $L_- \leq 1$, then since $L_- > M_-$, we have $M_- = 0$. But since $n = L_- + M_-$, this contradicts our assumption that $n\geq 2$.

If $M_+ = 0$, then
\[
j = L_+ + L_- > L_+ + \frac{L_- + M_-}{2} = \frac{2L_+ + 2M_+ + L_- + M_-}{2} = \frac{2m + n}{2}
\]
and thus $\frac{n}{d-j} > 2 > \frac{m+d}{j}$. Since $\ff$ is balanced, this implies
\begin{align*}
&n f_d + (m+1)f_1 + (d-1) f_2 = n f_d + (m+d) f_j\\
\geq &\frac{n}{d-j} (f_{j+1} + \ldots + f_d) + \frac{m+d}{j} (f_1 + \ldots + f_j) > 0,
\end{align*}
contradicting \eqref{psiassumption}. Thus neither $L_- \leq 1$ nor $M_+ = 0$ can hold, and so \eqref{ETS1} holds, and backtracking yields \eqref{psiprimebound}. This completes our proof of Case 2, and thus completes the proof of Lemma \ref{lemmapsiprimebound}.
\end{proof}

We are now ready to prove the upper bound in Theorem \ref{theorempacking}. Let $\ff\in \Sing_\dims(\tau)$, i.e. $|f_1(t)| \geq \tau t$ for all sufficiently large $t$, be a balanced template, and let $T$ be a time such that $\delta(T) > mn-m$. Note that this implies that $1\in S_+(\ff,T)$. Let $T'$ be the largest time such that $f_1' = 1/m$ on $(T,T')$. If $T' > T$, then the convexity condition implies that $f_1(T') = f_2(T')$. On the other hand, if $T = T'$, then $f_1'(T) < 1/m$, and since $1\in S_+(\ff,T)$, this implies that $f_1(T) = f_2(T)$. So either way $f_1(T') = f_2(T')$.

Let $\gg:[0,T']\to\R^d$ be the standard template defined by the points $(0,0)$ and $(T',f_1(T'))$ (cf. Definition \ref{definitionstandardtemplate}). Then $f_1(T) = g_1(T)$ while $f_2(T) \leq g_2(T)$. Since $\ff$ is balanced, using the definition of $\gg$ this implies that $f_d(T) \geq g_d(T)$. Consequently $\psi_\ff(T) \geq \psi_\gg(T)$ and hence
\[
\Delta(\ff,T) \leq \dimsing - \psi_\ff(T) \leq \dimsing - \psi_\gg(T) = \Delta(\gg,T) = \overline\delta_\dims \left(\frac{-f_1(T')}{T'}\right).
\]
The first equality holds because for $\gg$ defined as above, on each interval of linearity one of the conditions 1,2 is satisfied (cf. Table \ref{tablestandardzerotemplate}), and the second equality is a restatement of \eqref{deltataupackingdef}.

Thus for all $T$ such that $\delta(T) > mn-m$, we have
\[
\Delta(\ff,T) \leq \max\left(mn-m, \; \max_{T' \geq T} \overline\delta_\dims\left(\frac{-f_1(T')}{T'}\right)\right),
\]
and it follows that the same is true for all $T$. Taking the limsup gives
\[
\overline\delta(\ff)  \leq \max\big(mn-m, \overline\delta_\dims(\tau)\big),
\]
where $\ff\in \Sing_\dims(\tau)$. Taking the supremum over all $\ff$ and applying Theorem \ref{theoremvariational4} completes the proof.

\section{Proof of Theorem \ref{theorempacking2}}
\label{subsectionN1PDgeq}

The proof is similar to that in Section \ref{subsectionhsmallHDgeq}.
Assume $n=1$. There are two cases to consider, when $0 < \tau < \frac{m-1}{2m}$ and when $\tau < \frac1{m^2}$.

{\bf Case 1}. Fix $0 < \tau < \frac{m-1}{2m}$. Fix $\lambda > 1$ and let $t_k = \lambda^k$ and $\epsilon_k = \tau t_k$. However, rather than letting $\ff = \ff[\tau,\lambda]$, we will introduce a new parameter $\gamma > 0$ (which we think of as being independent of $\lambda$), small enough so that $\mbf s[(\gamma,-\epsilon),(1,-\tau)]$ is well-defined for all $0\leq \epsilon \leq \tfrac{m-1}{2m} \gamma$ (it suffices to take $\gamma \leq \frac{4m}{m^2-1}(\frac{m-1}{2m}-\tau)$). Let
\[
\epsilon = \left(\frac{\tau+(\lambda-1)\tfrac{m-1}{2m}}{\lambda}\right)\gamma.
\]
We define $\ff$ as follows:
\begin{itemize}
\item On $[\gamma,1]$, we have $\ff = \mbf s[(\gamma,-\epsilon),(1,-\tau)]$ (cf. Figure \ref{figurem=2} for an example with $m=2$).
\item Extend $\ff$ to $[1,\gamma\lambda]$ via the requirements that $f_1' = f_2' = -\frac{m-1}{2m}$ and $f_3 = \ldots = f_d$ on $[1,\gamma\lambda]$.
\item Extend $\ff$ to $\Rplus$ via exponential equivariance. This is possible by the definition of $\epsilon$.
\end{itemize}

\begin{figure}
\begin{tikzpicture}[line cap=round,line join=round,>=triangle 45,x=4.9cm,y=3cm]
\clip(-0.9952263135477621,-1.8147247879868178) rectangle (2.5728329530894376,2.1852752120131784);
\draw [line width=1pt] (0,1) -- (0.5,1.25);
\draw [line width=1pt] (0.5,1.25) -- (2,-0.25);
\draw [line width=1pt] (0,-0.5) -- (0.5,-1);
\draw [line width=1pt]  (0.5,-1) -- (2,-0.25);
\draw [line width=1pt]  (0,-0.5) -- (2,0.5);
\draw [dotted, line width=1pt]  (-0.3,0) -- (2.1,0);
\begin{scriptsize}
\draw [fill=black] (0,-0.5) circle (2.5pt);
\draw[color=black] (-0.1,-0.3604764219737473) node {$(\gamma,-\epsilon)$};
\draw [fill=black] (2,-0.25) circle (2.5pt);
\draw[color=black] (2,-0.41211040890185206) node {$(1,-\tau)$};
\end{scriptsize}
\end{tikzpicture}
\caption{The joint graph of $\ff = \mbf s[(\gamma,-\epsilon),(1,-\tau)]$ with $m=2$ and $n=1$.}
\label{figurem=2}
\end{figure}
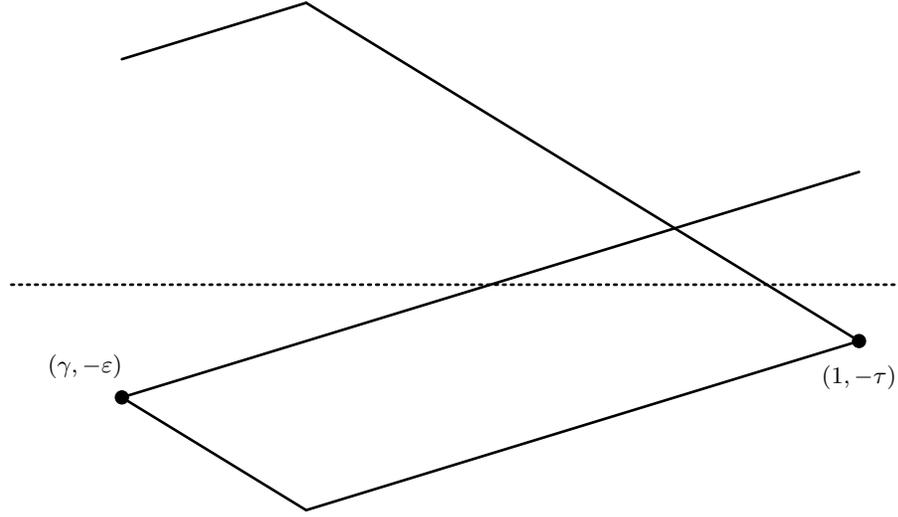

Now since $\delta(\ff,\cdot) = 1$ on $[1,\gamma\lambda]$, we have
\[
\Delta(\ff,\gamma\lambda) \geq \tfrac{\gamma\lambda-1}{\gamma\lambda}\cdot
\]
Taking the supremum over $\ff$ and applying Theorem \ref{theoremvariational4} yields
\[
\PD(\Sing_\dims(\tau)) \geq \tfrac{\gamma\lambda-1}{\gamma\lambda}\cdot
\]
Taking $\lambda\to\infty$ completes the proof.

{\bf Case 2}. Now suppose that $\tau < \frac1{m^2}$, and let $\tau' = \frac{(m-1)\tau}{1-m\tau} < \frac1m$. For each $\lambda > 1$ let $\ff_\lambda = \ff[\tau',\lambda]$ be the standard $1\times m$ template defined by $\tau'$ and $\lambda$ (as in Definition \ref{definitionstandardtemplate2parameter}). Claim \ref{claimuppercontractionlimit} shows that
\[
\lim_{\lambda\to\infty} \overline\delta(\ff_\lambda) = \delta_{1,m}(\tau') \geq mn - n = m - 1.
\]
Now the $m\times 1$ template $-\ff_\lambda$ has the same upper average contractivity as $\ff_\lambda$. Thus to complete the proof, it suffices to show that
\[
\tau(-\ff_\lambda) = \tau
\]
for all sufficiently large $\lambda$. Indeed,
\[
\tau(-\ff_\lambda) = \liminf_{t\to\infty} \tfrac1t f_d(t) = \tfrac1{t_0} f_d(t_0),
\]
where $t_0 > 1$ is the smallest time such that $f_2(t_0) = f_3(t_0)$. Since  $\ff(1) = (-\tau,-\tau,\tfrac2{m-1}\tau,\ldots,\tfrac2{m-1}\tau)$ and $\ff' = (-\tfrac1m,1,-\frac1m,\ldots,-\frac1m)$ on $(1,t_0)$ (cf. Figure \ref{figurediamond2}), we have  that
\[
f_d(t_0) = -\tau' + (t_0 - 1) = \tfrac2{m-1}\tau' - \tfrac1m(t_0 - 1).
\]
Thus
\begin{align*}
t_0 &= 1 + \tfrac m{m-1}\tau'\\
f_d(t_0) &= \tfrac1{m-1}\tau'\\
\tau(-\ff(\lambda)) &= \frac{f_d(t_0)}{t_0} = \frac{\tfrac1{m-1}\tau'}{1 + \tfrac m{m-1}\tau'} = \tau.
\end{align*}
This completes the proof in the case $\tau < \frac1{m^2}$. 

\begin{figure}[h!]
\begin{tikzpicture}
\clip(-1,-3) rectangle (8,2);
\draw[dashed] (-1,0) -- (8,0);
\draw (0,-0.5) -- (5,-3) -- (7,-1);
\draw (0,-0.5) --  (0.8,0.3);
\draw (3.5,0.75) -- (7,-1);
\draw[line width = 1.5] (0,0.5) --  (0.8,0.3);
\draw[line width = 1.5] (3.5,0.75) -- (5,1.5) -- (7,1);
\draw[line width = 2.5] (0.8,0.3) -- (3.5,0.75);
\fill (0,-0.5) circle (0.09);
\node at (0,-1.2) {$(1,-\tau')$};
\fill (7,-1) circle (0.09);
\node at (7,-0.33) {$(\lambda,-\tau' \lambda)$};
\fill (0.8,0.3) circle (0.09);
\node at (1,1) {$(t_0,f_d(t_0))$};
\end{tikzpicture}
\caption{The joint graph of $\ff_\lambda = \ff[\tau',\lambda]$ on the interval $[1,\lambda]$, as in Definition \ref{definitionstandardtemplate2parameter}.}
\label{figurediamond2}
\end{figure}
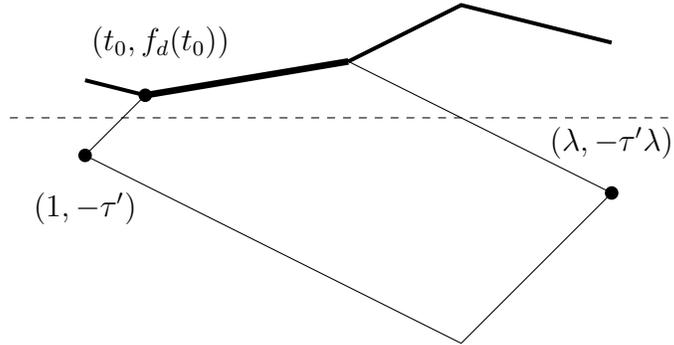

Finally, we leave the equality cases $\tau = \frac{m-1}{2m}$ and $\tau = \frac1{m^2}$ as exercises for the reader. Specifically, one glues together partial templates corresponding to a sequence of values $\tau_k \to \tau$ to get a template which is $\tau$-singular but has the desired packing dimension property.

\section{Proof of Theorem \ref{theoremN2}, lower bound for Hausdorff dimension}
\label{subsectionN2HDgeq}

Assume $n\geq 2$, and fix $0 < \tau < \frac1n$. As in Section \6\ref{subsectionPDgeq} we consider a two-parameter standard template $\ff[\tau,\lambda]$ (as in Definition \ref{definitionstandardtemplate2parameter}). Now if $\lambda$ is sufficiently large, then \eqref{standardtemplateB}-\eqref{standardtemplate3B} are satisfied, and thus there is a standard template $\ff = \ff_\lambda = \ff[\tau,\lambda]$ defined by the sequence of points $(t_k,-\epsilon_k)_0^\infty$, where $t_k = \lambda^k$ and $\epsilon_k = \tau t_k$.

Modifying the proof of Claim \ref{claimuppercontractionlimit} yields
\[
\lim_{\lambda\to\infty} \underline\delta(\ff_\lambda) = \inf_{0<T\leq1} \Delta(\gg,T),
\]
where $\gg \df \mbf s[(0,0),(1,-\tau)]$ (as in Definition \ref{definitionstandardtemplate}). Applying the variational principle (Theorem \ref{theoremvariational2}) to $\ff_\lambda$ gives us that 
\[
\HD(\Sing_\dims(\tau)) \geq \inf_{0 < T \leq 1} \Delta(\gg,T).
\]
Now $\delta(\gg,t) \geq mn-2m$ for all $t$, and $\delta(\gg,t) \geq mn-m$ for all $t\leq \frac{n(d-1)}{d}[\frac1n-\tau]$. It follows that
\[
\Delta(\gg,T) \geq mn-2m + \frac{mn(d-1)}{d}\left[\frac1n-\tau\right]
\]
for all $0 < T \leq 1$.

\section{Proof of Theorem \ref{theoremN1}, lower bound for Hausdorff dimension}
\label{subsectionN1HDgeq}

Assume $n=1$, and fix $0 < \tau < \frac{m-1}{2m}$, and let $\lambda$ be minimal such that \eqref{standardtemplate3B} holds, i.e.
\begin{equation}
\label{lambdadef}
\lambda \df 1 + \tfrac{d\tau}{2}\left(\tfrac{m-1}{2m}-\tau\right)^{-1}.
\end{equation}
As usual we let $t_k = \lambda^k$, $\epsilon_k = \tau t_k$, and $\ff = \ff[\tau,\lambda]$. 

Fix $t \geq 0$. If $\delta(\ff,t) = 0$, then $S_-(\ff,t) = \{1\}$ and thus $f'_1(t) = -1$, while if $\delta(\ff,t) \geq 1$ then we have the trivial bound $f'_1(t) \leq \frac1m$. Combining these two results into one formula yields
\[
f_1'(t) \leq -1 + \tfrac{m+1}{m}\delta(\ff,t) \text{ for all $t$}.
\]
Thus
\[
-\tfrac{m-1}{2m} < -\tau = f_1(1) \leq -1 + \tfrac{m+1}{m}\Delta(\ff,1),
\]
and rearranging gives
\[
\Delta(\ff,1) > \tfrac12.
\]
It follows that $\Delta(\ff,T) \geq \tfrac1{2T} \geq \tfrac1{2\lambda}$ for all $T\in[1,\lambda]$. The exponential equivariance of $\ff$ then implies that $\Delta(\ff,T) \geq \tfrac1{2\lambda}$ for all $T > 0$. So
\[
\underline\delta(\ff) \geq \tfrac1{2\lambda} \underset{\eqref{lambdadef}}{=} \tfrac12 - \Theta\left(\tfrac{m-1}{2m}-\tau\right)
\]
and applying Theorem \ref{theoremvariational4} completes the proof.

\section{Proof of Theorem \ref{theoremN2}, upper bound for Hausdorff dimension}
\label{subsectionN2HDleq}
Let $\ff$ be a $\tau$-singular template which is not trivially singular, i.e. $|f_1(t)| \geq \tau t$ for all sufficiently large $t$, and $f_{j+1}(t) - f_j(t) \nrightarrow \infty$ as $t\to\infty$ for all $j = 1,\ldots,d-1$. Then there exists a constant $C$ such that $f_2(T) \leq f_1(T) + C$ infinitely often. Fix $T$ such that $f_2(T) \leq f_1(T) + C$. Since $\ff$ is $\tau$-singular, we have $f_2(T) \leq f_1(T) + C \leq -\dynexp T + C$.

Since $1, 2 \in S_-(\ff,t)$ 

For all $t$ such that $f_1'(t) = f_2'(t) = -\frac1\qdim$, we have
\[
\dimprod - \delta(\ff,t) \geq 2\pdim
\]
and for all $t$ such that $f_i'(t) > -\frac1\qdim$ for some $i = 1,2$, we have
\[
f_i'(t) \geq \frac{1}{\qdim + 1}\left[\frac{1}{\pdim} - \frac{\qdim}{\qdim}\right] = -\frac{1}{\qdim} + \frac{\dimsum}{\dimprod(\qdim+1)}
\]
and thus
\[
f_1'(t) + f_2'(t) \geq -\frac2{\qdim} + \frac{\dimsum}{\dimprod(\qdim+1)}
\]
and at the same time $\dimprod - \delta(\ff,t) \geq 0$. Combining these two cases we have
\[
\dimprod - \delta(\ff,t) \geq 2\pdim - \frac{2\pdim^2\qdim(\qdim+1)}{\dimsum}\left[\frac2\qdim + f_1'(t) + f_2'(t)\right]
\]
and averaging over the interval $[0,T]$ gives
\begin{align*}
\dimprod - \Delta(\ff,T)
&\geq 2\pdim - \frac{2\pdim^2\qdim(\qdim+1)}{\dimsum}\left[\frac2\qdim + \frac{f_1(T)}{T} + \frac{f_2(T)}{T} \right]\\
&\geq 2\pdim - \frac{4\pdim^2\qdim(\qdim+1)}{\dimsum}\left[\frac1\qdim - \dynexp +\frac{C}{2T}\right].
\end{align*}
Taking the liminf as $T\to\infty$ and applying Theorem \ref{theoremvariational4} completes the proof.

\section{Proof of Theorem \ref{theoremN1}, upper bound for Hausdorff dimension}
\label{subsectionN1HDleq}
Let $\ff$ be a $\tau$-singular $m \times 1$ template, i.e. $|f_1(t)| \geq \tau t$ for all sufficiently large $t$. The proof spilts in two cases.

{\bf Case 1}. First suppose that both $f_1 = f_2$ and $f_2 = f_3$ infinitely often. 

Fix $T_1 > 0$ such that $f_2(T_1) = f_3(T_1)$, and let $T \geq T_1$ be minimal such that $f_1(T) = f_2(T)$. Let $x = \frac{\pdim - 1}{2\pdim} - \dynexp > 0$. For each $t$, let $j(t)$ denote the unique element of $S_-(\ff,t)$. Then
\[
f_1'(t) + f_2'(t)
\begin{cases}
= \frac1\pdim - 1 & j(t) = 1,2\\
\geq \frac1\pdim - 1 + \alpha & j(t) > 2
\end{cases}
\]
where $\alpha > 0$ is a constant. On the other hand,
\[
\tfrac1T\big(f_1(T) + f_2(T)\big) \leq -2\tau = \frac1\pdim - 1 + 2x.
\]
It follows that
\[
\lambda(\{t\leq T : j(t) > 2\}) = O(xT)
\]
where $\lambda$ is Lebesgue measure. Consequently $f_i(t) = \tfrac{t}{\pdim} + O(xT)$ for all $i > 2$ and $t \in [0,T]$. On the other hand, since $f_2' \geq -1$ it follows that for $t \in [0,T]$ we have
\[
f_2(t) \leq f_2(T) + T-t \leq -\left(\frac{\pdim - 1}{2\pdim} - x\right)T + T-t = \frac{\pdim+1}{2\pdim}T - t + O(xT),
\]
and thus we have $f_2 < f_3$ for all $t\in I \df (T/2 + cxT,T)$, where $c > 0$ is a constant. In particular we have $T_1 \leq T/2 + cxT$. By the minimality of $T$, it follows that $f_1 < f_2$ on $I$. Using the convexity condition it is possible to prove that $j(t)=2$ for all $t\in I$. Thus $f_1' = \frac1\pdim$ on $I$ and thus
\[
f_1(T/2) = f_1(T) - \tfrac1\pdim(T/2) + O(xT) \leq -\dynexp T - \tfrac1\pdim(T/2) + O(xT) = -(T/2) + O(xT).
\]
Consequently,
\begin{equation}
\label{statisticalT2}
\lambda(\{t\leq T/2 : j(t) > 1\}) = O(xT)
\end{equation}
and thus $\Delta(\ff,T/2) = O(xT)$.

{\bf Case 2a}. Now if $f_1 < f_2$ for all sufficiently large times, then it follows from the convexity condition that $j(t) = 1$ for all sufficiently large times, and thus $\underline\delta(\ff) = 0$.

{\bf Case 2b}. If $f_2 < f_3$ for all sufficiently large times, then it follows from the convexity condition that $j(t) \leq 2$ for all sufficiently large times, and thus
\[
2f_1(t) \leq f_1(t) + f_2(t) = -\frac{\pdim-1}{\pdim} t + C
\]
for some constant $C$. This demonstrates that $\dynexp \geq \frac{\pdim-1}{2\pdim}$. Since equality holds infinitely often, we have $\dynexp = \frac{\pdim-1}{2\pdim}$. Thus for $\dynexp < \frac{\pdim-1}{2\pdim}$, we have $f_2 = f_3$ infinitely often.

\section{Proof of Theorem \ref{theoremN1}, upper bound for packing dimension}
\label{subsectionN1PDleq}

Let $T_1 > 0$ be a local maximum of $\Delta(\ff,\cdot)$, and by contradiction suppose that $\Delta(\ff,T_1) > 1$. Then $\delta(\ff,I) > 1$, where $I$ is the interval of linearity for $\ff$ whose right endpoint is $T_1$. Equivalently, $j > 2$ on $I$, where $j$ is as in \6\ref{subsectionN1HDleq}. Let $T$ be as in \6\ref{subsectionN1HDleq}. Since $f_1 < f_2$ on $(T_1,T)$, by the convexity condition we have $j > 1$ on $(T_1,T)$ and thus by \eqref{statisticalT2} we have $T_1 = T/2 + O(xT)$. But then by the argument of \6\ref{subsectionN1HDleq}, we have
\[
\Delta(\ff,T_1) = \Delta(\ff,T/2) + O(x) = O(x)
\]
and thus if $x$ is sufficiently small, then $\Delta(\ff,T_1) < 1$, a contradiction.

\draftnewpage

\section{Proof of Theorem \ref{theoremspecialcase}}
\label{subsectionspecialcase}
Note that the packing dimension formula in Theorem \ref{theoremspecialcase} follows immediately from Theorem \ref{theorempacking}. Thus, we prove only the Hausdorff dimension formula. However, note that the first part of the proof could apply to the computation of packing dimension as well.

Fix $\tau > 0$, and let $\ff$ be a $1\times 2$ template which satisfies $\dynexp(\ff) = \tau$ but is not trivially singular. We claim that
\begin{align}
\label{specialcasebounds}
\underline\delta(\ff) &\leq \underline\delta(\tau),
\end{align}
where $\underline\delta(\tau)$ is the right-hand side of the first formula of Theorem \ref{theoremspecialcase}. This will prove the upper bound of that formula. Indeed, since $\ff$ is not trivially singular, the sets $F_- = \{t\geq 0 : f_1(t) = f_2(t)\}$ and $F_+ = \{t\geq 0 : f_2(t) = f_3(t)\}$ are both unbounded. Since $\ff$ is piecewise linear, we can write $F_-\cup F_+$ as the union of a sequence of intervals $[s_1,t_1] < [s_2,t_2] < \ldots$
\begin{claim}
We can assume without loss of generality that 
\[
F_+ = [s_1,t_1]\cup [s_3,t_3]\cup \ldots ~\text{and}~ F_- = [s_2,t_2]\cup [s_4,t_4]\cup \ldots
\]
\end{claim}
\begin{proof}
First, since $F_-$ and $F_+$ are disjoint, for each $k$ we have either $[s_k,t_k] \subset F_-$ or $[s_k,t_k] \subset F_+$. Now let $\gg:\Rplus\to\R^3$ be defined by the formulas
\[
\gg(t) = \begin{cases}
\left(-\tfrac{1}{2}f_3(t),-\tfrac{1}{2}f_3(t),f_3(t),\ldots,f_3(t)\right) &\text{ if } t\in (t_k,s_{k+1}),\; [s_k,t_k],[s_{k+1},t_{k+1}] \subset F_-\\
\left(f_1(t),-\tfrac{1}{2}f_1(t),\ldots,-\tfrac{1}{2}f_1(t)\right) &\text{ if } t\in (t_k,s_{k+1}),\; [s_k,t_k],[s_{k+1},t_{k+1}] \subset F_+\\
\ff(t) & \text{ otherwise}.
\end{cases}
\]
Then $\delta(\gg,t) \geq \delta(\ff,t)$ for all $t\geq 0$, so $\underline\delta(\gg) \geq \underline\delta(\ff)$ and $\overline\delta(\gg) \geq \overline\delta(\ff)$. Moreover, since the minima of the functions 
\[
t\mapsto \frac{-f_1(t)}{t} ~\text{and}~ t\mapsto \frac{-g_1(t)}{t}
\]
on an interval of the form $[t_k,s_{k+1}]$ are always attained at one of the endpoints of the interval, we have $\dynexp(\gg) = \dynexp(\ff)$. So it suffices to prove \eqref{specialcasebounds} with $\ff$ replaced by $\gg$. Now the corresponding sets $F_-$ and $F_+$ defined in terms of $\gg$ are clearly of the desired form, with the exception that the roles of $F_-$ and $F_+$ may be switched; this exception can be dealt with by truncating the template from the left so as to cut out the interval $[s_1,t_1]$.
\end{proof}

We observe that $f_1$ and $f_2$ ``split'' at times $t_{2k}$ and ``merge'' at times $s_{2k}$, while $f_2$ and $f_3$ ``split'' at times $t_{2k+1}$ and ``merge'' at times $s_{2k+1}$. Consequently
\begin{align*}
f_1'(t_{2k}^+) &< f_2'(t_{2k}^+), &
f_2'(t_{2k+1}^+) &< f_3'(t_{2k+1}^+),\\
f_1'(s_{2k}^-) &> f_2'(s_{2k}^-), &
f_2'(s_{2k+1}^-) &> f_3'(s_{2k+1}^-).
\end{align*}
It follows that if $j(t)$ denotes the unique element of $S_+(\ff,t)$, then
\begin{align*}
j(s_{2k}^+) &= j(t_{2k}^-) = 1,&
j(t_{2k}^+) &= j(s_{2k+1}^-) = 2,\\
j(s_{2k+1}^+) &= j(t_{2k+1}^-) = 2,&
j(t_{2k+1}^+) &= 3 > j(s_{2k+2}^-) = 1.
\end{align*}
Thus by the convexity condition, there exists sequences of numbers $t_{2k+1} < a_k \leq r_k < s_{2k+2}$ such that
\[
\ff'(t) =
\begin{cases}
\big(-\tfrac12,1,-\tfrac12\big) & t_{2k} < t < s_{2k+1}\\
\big(-\tfrac12,\tfrac14,\tfrac14\big) & s_{2k+1} < t < t_{2k+1}\\
\big(-\tfrac12,-\tfrac12,1\big) & t_{2k+1} < t < a_k\\
\big(-\tfrac12,1,-\tfrac12\big) & a_k < t < r_k\\
\big(1,-\tfrac12,-\tfrac12\big) & r_k < t < s_{2k+2}\\
\big(\tfrac14,\tfrac14,-\tfrac12\big) & s_{2k+2} < t < t_{2k+2}
\end{cases}
\]
(cf. Figure \ref{figuremainexample}). Evidently, we have
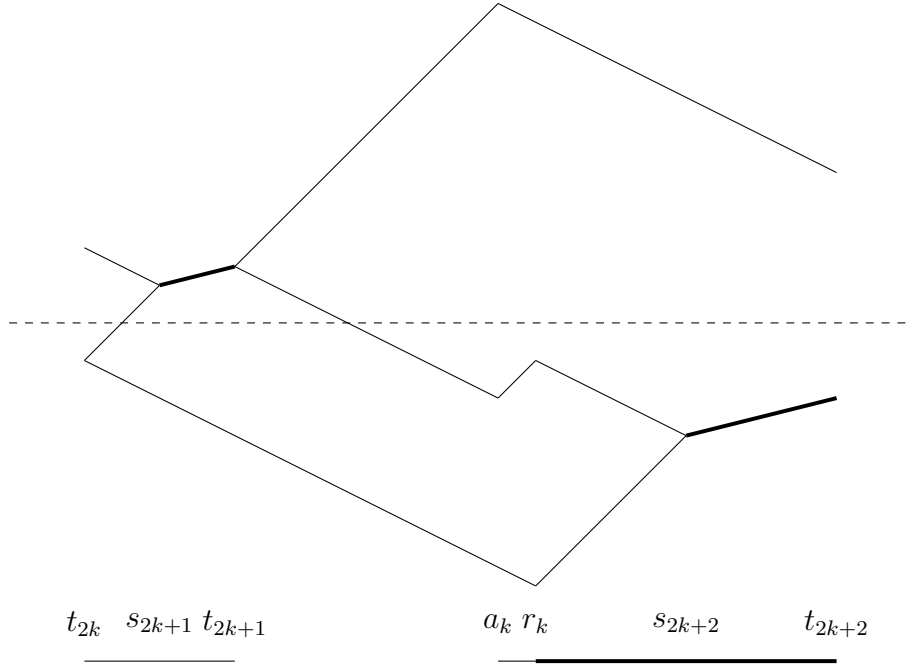
\begin{figure}
\begin{tikzpicture}
\clip(-1,-5) rectangle (11,5);
\draw[dashed] (-1,0) -- (11,0);
\node at (0,-4) {$t_{2k}$};
\node at (1,-4) {$s_{2k+1}$};
\node at (2,-4) {$t_{2k+1}$};
\node at (5.5,-4) {$a_k$};
\node at (6,-4) {$r_k$};
\node at (8,-4) {$s_{2k+2}$};
\node at (10,-4) {$t_{2k+2}$};
\draw (0,-4.5) -- (2,-4.5);
\draw (5.5,-4.5) -- (6,-4.5);
\draw[line width=1.5] (6,-4.5) -- (10,-4.5);
\draw (0,1)--(1,0.5);
\draw (0,-0.5)--(1,0.5);
\draw (0,-0.5)--(6,-3.5);
\draw[line width=1.5] (1,0.5) -- (2,0.75);
\draw (2,0.75) -- (5.5,4.25);
\draw (5.5,4.25) -- (10,2);
\draw (2,0.75) -- (5.5,-1);
\draw (5.5,-1) -- (6,-0.5);
\draw (6,-0.5) -- (8,-1.5);
\draw (6,-3.5) -- (8,-1.5);
\draw[line width=1.5] (8,-1.5) -- (10,-1);
\end{tikzpicture}
\caption{A piece of an arbitrary $1\times 2$ template.}
\label{figuremainexample}
\end{figure}
\[
\delta(\ff,t) = 3 - j(\ff,t) =
\begin{cases}
1 & t_{2k} < t < t_{2k+1}\\
0 & t_{2k+1} < t < a_k\\
1 & a_k < t < r_k\\
2 & r_k < t < t_{2k+2}.
\end{cases}
\]
Now let $A_k,B_k,C_k,D_k\in\R$ be chosen so that
\begin{align*}
f_1(t) &= A_k - \tfrac12 t \text{ for all } t\in [t_{2k},r_k],\\
f_1(t) &= B_k + t \text{ for all } t\in [r_k,s_{2k+2}],\\
f_3(t) &= C_k + t \text{ for all } t \in [t_{2k+1},a_k],\\
f_3(t) &= D_k - \tfrac12t \text{ for all } t\in [a_k,s_{2k+3}].
\end{align*}
Then the set of parameters
\[
\Big(A_k,B_k,C_k,D_k\Big)_{k\in\N}
\]
is a necessary and sufficient set of parameters for $\ff$ in the following sense: the map sending $\ff$ to this set of parameters is injective, and its image is the set of all sequences of parameters that satisfy the following inequalities:
\begin{align} \label{parameters1}
s_k &\leq t_k < s_{k+1},\\ \label{parameters2}
t_{2k+1} &< a_k \leq r_k < s_{2k+2}
\end{align}
where $s_k,t_k,a_k,r_k$ are defined by the equations
\begin{align} \label{rkdef}
0 &= (A_k - \tfrac12 r_k) - (B_k + r_k)\\ \label{Ckdef}
0 &= (C_k + a_k) - (D_k - \tfrac12 a_k)\\ \label{t0def}
0 &= 2\big(A_k - \tfrac12 t_{2k}\big) + \big(D_{k-1} - \tfrac12 t_{2k}\big)\\ \label{s0def}
0 &= 2\big(B_k + s_{2k+2}\big) + \big(D_k - \tfrac12 s_{2k+2}\big)\\ \label{Ckin}
0 &= \big(A_k - \tfrac12 t_{2k+1}\big) + 2 \big(C_k + t_{2k+1}\big)\\ \label{s1def}
0 &= \big(A_k - \tfrac12 s_{2k+1}\big) + 2 \big(D_{k-1} - \tfrac12 s_{2k+1}\big)
\end{align}
The idea now is to take a function $\ff$ defined by a sequence of parameters satisfying \eqref{parameters1}-\eqref{parameters2}, and to replace it by a function $\w\ff$ defined by a sequence of parameters
\[
\Big(\w A_k,\w B_k,\w C_k,\w D_k\Big)_{k\in\N}.
\]
If we can show that $\Delta(\w\ff,T) \geq \Delta(\ff,T)$ for all $T$, while $\what\tau(\w\ff) = \what\tau(\ff)$, then it suffices to prove \eqref{specialcasebounds} for $\w\ff$. A change that satisfies this inequality will be called an \emph{allowable change}. Note that if a change only affects the value of $\delta(\ff,\cdot)$ on two intervals $I_1,I_2$ such that $\max(I_1) < \min(I_2)$, increasing it on $I_1$ and decreasing it on $I_2$, with greater total area for the effect on $I_1$, then the change is allowable. We now show that we can make some allowable changes to simplify the structure of the template $\ff$.

\begin{claim}
We can without loss of generality assume that $a_k = r_k$ for all $k$.
\end{claim}
\begin{proof}
We claim that decreasing $C_k$ by $\epsilon$ while leaving all other parameters fixed is an allowable change. Indeed, this change will have the effect of increasing $t_{2k+1}$ by $\frac43\epsilon$ while increasing $a_k$ by $\frac23\epsilon$. This means that $\delta(\ff,\cdot)$ is increased by $1$ on an interval of length $\frac43\epsilon$ around $t_{2k+1}$, but decreased by $1$ on an interval of length $\frac23\epsilon$ around $a_k$. Thus, the change is allowable, and applying the maximum value of $\epsilon = \frac32(r_k - a_k)$ completes the proof.
\end{proof}

From now on we will not treat $C_k$ as an independent parameter, but rather assume that it is given by \eqref{Ckdef} together with the formula $a_k = r_k$. Note that in this case, \eqref{rkdef}, \eqref{Ckdef}, and \eqref{Ckin} combine to form the equation
\begin{equation}
\label{parameters4.5}
0 = \big(A_k - \tfrac12 t_{2k+1}\big) + 2 \big(D_k - A_k + B_k + t_{2k+1}\big).
\end{equation}

\begin{claim}
The following set of parameter changes is allowable:
\begin{align*}
\w A_k &= A_k + \epsilon\\
\w B_{k-1} &= B_{k-1} + \epsilon\\
\w D_{k-1} &= D_{k-1} - \epsilon
\end{align*}
\end{claim}
\begin{proof}
These changes lead to the following changes to $t_k,r_k$:
\begin{itemize}
\item no change to $t_{2k-1}$
\item decrease $r_{k-1}$ by $\tfrac23\epsilon$ (thus increasing $\delta(\ff,\cdot)$ by $2$ on an interval of this length)
\item increase $t_{2k}$ by $\tfrac23\epsilon$ (thus increasing $\delta(\ff,\cdot)$ by $1$ on an interval of this length)
\item increase $t_{2k+1}$ by $\tfrac23\epsilon$ (thus increasing $\delta(\ff,\cdot)$ by $1$ on an interval of this length)
\item increase $r_k$ by $\tfrac23\epsilon$ (thus decreasing $\delta(\ff,\cdot)$ by $2$ on an interval of this length)
\end{itemize}
The changes to $s_k$ can be ignored as they do not affect $\delta(\ff,\cdot)$, except to note that $\Delta s_k = \w s_k - s_k$ is always negative and so $\w t_k - \w s_k \geq t_k - s_k \geq 0$. The only decreasing effect, due to the change on $r_k$, is dominated by the increasing effect due to the change on $r_{k-1}$. Thus the changes are allowable.
\end{proof}

Now for each $k$, choose the maximum value of $\epsilon$ such that the changes lead to parameters satisfying \eqref{parameters1}-\eqref{parameters2} as well as the inequality
\[
f_1(t) \leq -\tau t \text{ for all } t,
\]
where $\tau < \what\tau(\ff)$ is arbitrary. Note that by piecewise linearity, this inequality is equivalent to saying that for all $k$ we have
\begin{align}
\label{parameters5}
f_1(t_{2k}) &\leq -\tau t_{2k}.
\end{align}
Then after the changes, \eqref{parameters5} will be satisfied with equality for every $k$. Equivalently,
\begin{equation}
\label{parameters5.5}
A_k - \tfrac12 t_{2k} = -\tau t_{2k}.
\end{equation}
Let $u_k = t_{2k}$, and note that
\[
\ff(u_k) = \left(-\tau u_k,-\tau u_k,2\tau u_k\right).
\]
This equality implies that for each $k$, we can define a template $\gg^{(k)}$ by letting $\gg^{(k)} = \ff$ on $[u_k,u_{k+1}]$ and then extending by exponential equivariance:
\[
\gg^{(k)}(\lambda t) = \lambda \gg^{(k)}(t) \text{ where } \lambda = u_{k+1}/u_k.
\]
Note that clearly, $\tau(\gg^{(k)}) = \tau$ for all $k$. From now on we will specialize to the Hausdorff dimension case of Theorem \ref{theoremspecialcase}.
\begin{claim}
We have
\begin{equation}
\label{fgkbound}
\underline\delta(\ff) \leq \sup_k \underline\delta(\gg^{(k)}).
\end{equation}
\end{claim}
\begin{proof}
Fix $\epsilon > 0$. Then there exist infinitely many $k$ such that $\Delta(\ff,u_{k+1}) \geq \Delta(\ff,u_k) - \epsilon$. For such a $k$, we have
\[
\Delta(\gg^{(k)},u_k) = \Delta(\ff,[u_k,u_{k+1}]) \geq \Delta(\ff,u_k) - O(\epsilon)
\]
since $u_{k+1}/u_k$ is bounded away from $1$. Thus
\[
\inf_{T\in [u_k,u_{k+1}]} \Delta(\ff,T) \leq \inf_{T\in [u_k,u_{k+1}]} \Delta(\gg^{(k)},T) + O(\epsilon) = \underline\delta(\gg^{(k)}) + O(\epsilon).
\]
Taking the liminf over $k$ and then letting $\epsilon\to 0$ gives \eqref{fgkbound}.
\end{proof}

Thus, we can without loss of generality assume that $\ff$ is exponentially equivariant, i.e. that
\begin{align}
\label{exponential}
A_k &= \lambda^k A,&
B_k &= \lambda^k B,&
D_k &= \lambda^k D
\end{align}
for some $A,B,D > 0$ and $\lambda > 1$. Now by rescaling, we can without loss of generality assume that $u_0 = 1$. Plugging $k=0$ into the formulas \eqref{rkdef}-\eqref{s1def}, \eqref{parameters4.5}, and \eqref{parameters5.5}, and solving for the appropriate variables yields
\begin{align*}
A &= \tfrac12 - \tau\\
D &= \lambda(\tfrac32 - 2A) = \lambda(\tfrac12 + 2\tau)\\
t_0 &= 1\\
s_1 &= \tfrac23(A+2\lambda^{-1} D) = 2 - 2A = 1 + 2\tau\\
t_1 &= \tfrac23(A - 2D - 2B)\\
r_0 &= \tfrac23(A-B)\\
s_2 &= -\tfrac23(2B+D)\\
t_2 &= \lambda.
\end{align*}
On the interval $[u_0,u_1] = [1,\lambda]$, the behavior of $\delta(\ff,\cdot)$ is as follows:
\begin{equation}
\label{deltaftspecialcase}
\delta(\ff,t) = \begin{cases}
1 & 1 < t < t_1\\
0 & t_1 < t < r_0\\
2 & r_0 < t < \lambda.
\end{cases}
\end{equation}
Now consider the change $\w B = B - \epsilon$. This change increases $t_1$ by $\frac43\epsilon$ and increases $r_0$ by $\frac23\epsilon$, this increasing $\delta(\ff,\cdot)$ by $1$ on an interval of length $\frac43\epsilon$ around $t_1$ and decreasing $\delta(\ff,\cdot)$ by $2$ on an interval of length $\frac23\epsilon$ around $r_0$. Thus the change is allowable, and by taking the maximum possible value of $\epsilon = \tfrac34(t_2-s_2)$, we can without loss of generality assume that $s_2 = t_2$, or equivalently that
\[
B = -\tfrac34\lambda - \tfrac12 D = \lambda(A - \tfrac32) = -\lambda(1+\tau)
\]
(cf. Figure \ref{figurespecialcase4}). Note that this implies
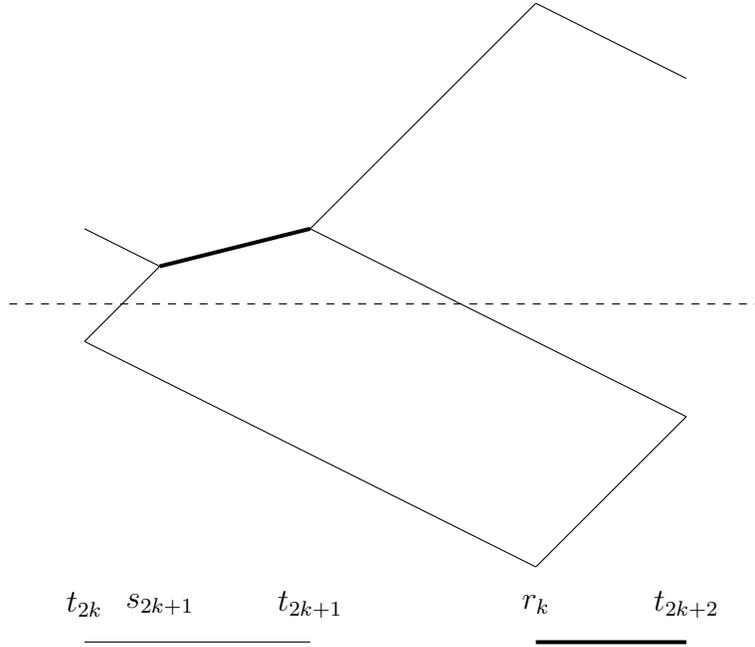
\begin{figure}
\begin{tikzpicture}
\clip(-1,-5) rectangle (9,5);
\draw[dashed] (-1,0) -- (9,0);
\node at (0,-4) {$t_{2k}$};
\node at (1,-4) {$s_{2k+1}$};
\node at (3,-4) {$t_{2k+1}$};
\node at (6,-4) {$r_k$};
\node at (8,-4) {$t_{2k+2}$};
\draw (0,-4.5) -- (3,-4.5);
\draw[line width=1.5] (6,-4.5) -- (8,-4.5);
\draw (0,1)--(1,0.5);
\draw (0,-0.5)--(1,0.5);
\draw (0,-0.5)--(6,-3.5);
\draw[line width=1.5] (1,0.5) -- (3,1);
\draw (3,1) -- (6,4);
\draw (6,4) -- (8,3);
\draw (3,1) -- (8,-1.5);
\draw (6,-3.5) -- (8,-1.5);
\end{tikzpicture}
\caption{A period of an exponentially periodic $1\times 2$ template, simplified using the arguments of this section.}
\label{figurespecialcase4}
\end{figure}
\[
t_1 = \tfrac23 A + \tfrac43 \lambda A.
\]
Now it is a problem of one-variable calculus: $\lambda$ is the only free parameter, and we must optimize $\underline\delta(\ff)$. Note that $\lambda$ is subject to the restriction
\[
\lambda \geq \frac{3/2 - 2A}{A} = \frac{1/2 + 2\tau}{1/2 - \tau}
\]
which comes from the inequality $s_1 \leq t_1$. Now from \eqref{deltaftspecialcase}, we have
\begin{align*}
\underline\delta(\ff) &= \Delta(\ff,r_0) = \Delta(\ff,[\lambda^{-1} r_0,r_0])
= \frac{1(t_1-1) + 2(1-\lambda^{-1} r_0)}{r_0 - \lambda^{-1} r_0}\cdot
\end{align*}
On the other hand,
\begin{align*}
t_1 &= \tfrac23 A + \tfrac43 \lambda A,&
r_0 &= \tfrac23 A - \tfrac23 \lambda A + \lambda.
\end{align*}
Let $x = \tau$ and $y = \tfrac13(\lambda - 1)$. Then
\begin{align*}
t_1 &= (\tfrac13-\tfrac23x)(3+6y) = (1-2x)(1+2y),\\
r_0 &= 1+3y-(\tfrac13-\tfrac23x)(3y) = 1+(2+2x)y,\\
\underline\delta(\ff) &= \frac{\lambda(t_1-1) + 2(\lambda - r_0)}{r_0(\lambda - 1)}\\
&= \frac{(1+3y)(-2x+(2-4x)y) + (2-4x)y}{(1+(2+2x)y)(3y)}\\
&= f_x(y) \df \frac23 \cdot \frac{-x + (2-7x)y + (3-6x)y^2}{y + (2+2x)y^2}\cdot
\end{align*}
We now need to find the maximum of the function $f_x$ on the interval $\CO{\frac{x}{1/2-x}}{\infty}$, assuming that $0 < x < 1/2$. The function $f_x$ has two critical points, given by the formulas\Footnote{Note that we found it easier to do these calculations first for the general case
\[
f(y) = \frac{-A+By+Cy^2}{Dy+Ey^2},
\]
then plug in the values $A=x$, $B= 2-7x$, $C=3-6x$, $D=1$, and $E=2+2x$, and finally multiply by $\frac23$. In the general case the formulas are
\begin{align*}
0 &= AD+2AEy-(BE-CD)y^2\\
y &= \frac{\epsilon\sqrt{Q} + AE}{BE-CD} = \frac{AD}{\epsilon\sqrt{Q} - AE} &\text{where $Q = (AE)^2+(AD)(BE-CD)$}\\
f(y) &= \frac1{D^2}\big(2AE + BD - 2 \epsilon\sqrt{A^2 E^2 + ABDE - ACD^2}\big).\noreason
\end{align*}}
\begin{align*}
0 &= x + (4x + 4x^2)y + (-1+4x+14x^2)y^2\\
y &= \frac{\epsilon\sqrt{x-6x^3+4x^4} + 2x + 2x^2}{1-4x-14x^2}
= \frac{x}{\epsilon\sqrt{x-6x^3+4x^4} - 2x - 2x^2 }\\
f_x(y) &= \frac43 - \frac43\epsilon\sqrt{x-6x^3+4x^4} - 2x + \frac83 x^2
\end{align*}
where $\epsilon = \pm 1$. Note that since the critical point corresponding to $\epsilon = -1$ is negative, it is not in the domain and so can be ignored. The critical point corresponding to $\epsilon = +1$ is positive if and only if $1 - 4x - 14x^2 > 0$, which in turn is true if and only if $x < \frac{3\sqrt 2 - 2}{14}$. In this case, it is easy to check that this critical point is in the domain of $f_x$, and that the critical point is a maximum. Thus in this case
\[
\sup_y f_x(y) = f_x(y_{\text{crit}}) = \frac43 - \frac43\sqrt{x-6x^3+4x^4} - 2x + \frac83 x^2.
\]
On the other hand, if $x \geq \frac{3\sqrt 2 - 2}{14}$, then this critical point is negative or undefined, and thus $f_x$ has no critical points on its domain. It can be verified that $f_x$ is increasing in this case, so its supremum is equal to its limiting value:
\[
\sup_y f_x(y) = \lim_{y\to\infty} f_x(y) = \frac{2}{3}\cdot \frac{3-6x}{2+2x} = \frac{1-2x}{1+x}\cdot
\]
Since $\underline\delta(\ff) \leq \sup_y f_x(y)$, this completes the proof of the upper bound. To prove the lower bound, note that if $y\in \CO{\frac{x}{1/2-x}}{\infty}$, then there is a unique exponentially periodic template $\ff$ satisfying the formulas appearing in the above proof, and this template satisfies $\underline\delta(\ff) = f_x(y)$. Thus $\HD(\Sing_{1,2}(\omega)) \geq f_x(y)$, and taking the supremum over $y$ proves the lower bound. Note that the exponentially periodic template $\ff$ is the same as the standard template defined by the sequence of points $(t_k,-\epsilon_k) = (\lambda^k,-\tau \lambda^k)$, where $\tau = x$ and $\lambda = 1+3y$.

\section{Proof of Theorem \ref{theoremsingonaverage}}
\label{subsectionsingonaverage}
Let $\ff$ be a template, and let $\phi$ be as in \6\ref{subsectionmainupperbound}. We claim that
\[
\phi'(t) \leq \dimsing - \delta(\ff,t) + \frac{mn}{m+n} g(t),
\]
where $g(t) = 1$ if $\ff(t) = \0$ and $g(t) = 0$ otherwise. Indeed, when $\ff(t) \neq \0$, this follows from Lemma \ref{lemmaphiprimebound}, and when $\ff(t) = \0$ it follows from direct calculation using the fact that $\phi'(t) = 0$ and $\delta(\ff,t) = mn$. Now fix $T > 0$. Integrating over $[0,T]$ gives
\[
0 \lesssim_\plus \phi(T) - \phi(0) \leq \int_0^T \left[\dimsing - \delta(\ff,t) + \frac{mn}{m+n} g(t)\right] \;\dee t = T\left[\dimsing - \Delta(\ff,T) + \frac{mn}{m+n} G(T)\right],
\]
where $G(T)$ is the average of $g$ on $[0,T]$. It follows that
\[
\overline\delta(\ff) \leq \limsup_{T\to\infty} \left[\dimsing + \frac{mn}{m+n} G(T)\right] = \dimsing + \frac{mn}{m+n}\limsup_{T\to\infty} G(T) = \PP(\ff)\dimsing + (1-\PP(\ff)) mn,
\]
where
\[
\PP(\ff) \df \liminf_{T\to\infty} (1-G(T))
\]
is the proportion of time spent near infinity. Applying Theorem \ref{theoremvariational2} gives
\[
\dim_H(\{A:\PP(A)=p\}) \leq \dim_P(\{A:\PP(A)=p\}) \leq p\dimsing + (1-p)mn.
\]
For the reverse direction, fix $p$ and $\epsilon > 0$ small. Define $\ff$ on $[1,1+\epsilon]$ as follows:
\begin{itemize}
\item Let $\ff = \gg[(0,0),(1+p\epsilon,0)]$ on $[1,1+p\epsilon]$
\item Let $\ff(t) \equiv \0$ on $[1+p\epsilon,1+\epsilon]$
\end{itemize}
and extend by exponential equivariance. It is easy to see that $\PP(\ff) = p$ and
\[
\dim_P(\DD(\ff)) \geq \dim_H(\DD(\ff)) \geq p\dimsing + (1-p)mn - O(\epsilon).
\]
This completes the proof.

\section{Proof of Theorem \ref{theoremstarkov}}
\label{subsectionstarkov}

Let $\phi$ be a function such that $\phi(t) \to \infty$ as $t \to \infty$, and without loss of generality suppose that $\phi$ is increasing. Let $(t_k,-\epsilon_k)$ be a sequence of points such that:
\begin{itemize}
\item[(i)] $\Delta t_k \leq \frac12 \phi(t_k)$ for all $k$;
\item[(ii)] $\epsilon_k \leq \frac12 \phi(t_k)$ for all $k$;
\item[(iii)] $\epsilon_k \to \infty$ as $k\to\infty$;
\item[(iv)] $\epsilon_k/\Delta t_k \to 0$ and $\epsilon_{k+1}/\Delta t_k \to 0$ as $k\to\infty$.
\end{itemize}
Then let $\ff$ be the standard template defined by the sequence of points $(t_k,-\epsilon_k)$. Conditions (i) and (ii) imply that $f_1(t) \geq -\phi(t_k) \geq -\phi(t)$ for all $k\in\N$ and $t\in [t_k,t_{k+1}]$. Condition (iii) implies that $\ff$ is singular. Finally, condition (iv) implies that $\underline\delta(\ff) = \dimsing$, since
\begin{align*}
\Delta(\ff,[t_k,t_{k+1}])
&= \Delta\left(\mbf s\left[\left(0,-\frac{\epsilon_k}{\Delta t_k}\right),\left(0,-\frac{\epsilon_{k+1}}{\Delta t_k}\right)\right],1\right)\\
&\to \Delta(\mbf s[(0,0),(1,0)],1) = \dimsing \text{ as } k\to\infty.
\end{align*}

\section{Proof of Theorem \ref{theoremkmessenger}}
\label{subsectionkmessenger}

Fix $2 \leq k \leq d - 1$ and $j\in \{k-1,k\}$, and let $\ff$ be a template with the following properties:
\begin{align} \label{ksing1}
f_{k-1}(t) &\to -\infty \text{ as } t\to \infty,\\ \label{ksing2}
f_{k+1}(t) &\to +\infty \text{ as } t \to\infty,\\ \label{ft0}
\frac1t\ff(t)&\to 0 \text{ as } t\to\infty,\\ \label{Sj+-to1}
\frac1T \lambda\big([0,T]\cap (S_j^+ \cup S_j^-)\big) &\to 1 \text{ as } T\to\infty,
\end{align}
where $S_j^+$ (resp. $S_j^-$) is the set of all times $t\geq 0$ such that the following hold:
\begin{itemize}
\item $f_1(t) = \ldots = f_j(t) < f_{j+1}(t) = \ldots = f_d(t)$,
\item $(L_+,L_-) = (\lceil \frac{j m}{d}\rceil,\lfloor \frac{j n}{d}\rfloor)$ (resp. $(L_+,L_-) = (\lfloor \frac{j m}{d}\rfloor,\lceil \frac{j n}{d}\rceil)$), where $L_\pm = L_\pm(\ff,t,j)$.
\end{itemize}
Such a template can be constructed by alternating long $S_j^\pm$ intervals with short intervals along which $f_k$ crosses 0 and returns, in a manner consistent with the rule on changes of slopes (cf. Figure \ref{hourglass}). The key point is that if $t\in S_j^+$ then $f_1'(t) \geq 0$, but if $t\in S_j^-$ then $f_1'(t) \leq 0$ (with equality if and only if $\frac{j m}{d}$ is an integer). Note that the template $\ff$ is not trivially singular.

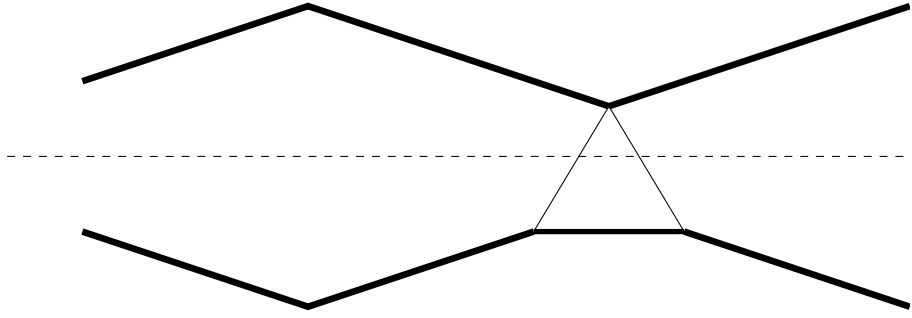
\begin{figure}[h!]
\begin{tikzpicture}
\clip(-1,-3) rectangle (11,3);
\draw[dashed] (-1,0) -- (11,0);
\draw[line width=2.5] (0,-1) -- (3,-2) -- (6,-1);
\draw[line width=2.5] (0,1) -- (3,2) -- (7,2/3);
\draw (7,2/3) -- (8,-1);
\draw (6,-1) -- (7,2/3);
\draw[line width=2] (6,-1) -- (8,-1);
\draw[line width=2.5] (8,-1) -- (11,-2);
\draw[line width=2.5] (7,2/3) -- (11,2);
\end{tikzpicture}
\caption{A piece of a template $\ff$ with the desired properties, as described in \6\ref{subsectionkmessenger} (Proof of Theorem \ref{theoremkmessenger}). The triangular portion of the figure can be made arbitrarily small in proportion to the rest.}
\label{hourglass}
\end{figure}

To compute the lower contractivity of $\ff$, we observe that for $t\in S_j^+$, we have
\begin{align*}
f_1'(S_j^+) \df f_1'(t) &= \frac1j\left[\frac{\lceil\frac{jm}{d}\rceil}{m} - \frac{\lfloor\frac{jn}{d}\rfloor}{n}\right]\\
&= \frac1j\left[\frac{\frac{jm}{d} + \{-\frac{jm}{d}\}}{m} - \frac{\frac{jn}{d} - \{\frac{jn}{d}\}}{n}\right]\\
&= \frac1j \frac{m+n}{mn}\left\{\frac{jn}{d}\right\}\\
mn - \delta(\ff,S_j^+) \df mn - \delta(\ff,t) &= L_- (m-L_+) = \left\lfloor \frac{j n}{d}\right\rfloor \left(m - \left\lceil \frac{j m}{d}\right\rceil\right)\\
&=  \left(\frac{jn}{d} - \left\{\frac{jn}{d}\right\}\right)\left(m - \frac{jm}{d} - \left\{-\frac{jm}{d}\right\}\right)\\
&= \frac{j(d-j)mn}{d^2} - \frac{(d-j)m + jn}{d} \left\{\frac{jn}{d}\right\} + \left\{\frac{jn}{d}\right\}^2.
\end{align*}
Similarly, for $t\in S_j^-$ we have
\begin{align*}
f_1'(S_j^-) \df f_1'(t) &= -\frac1j \frac{m+n}{mn} \left\{\frac{jm}{d}\right\}\\
mn - \delta(\ff,S_j^-) \df mn - \delta(\ff,t) &= \frac{j(d-j)mn}{d^2} + \frac{(d-j)m + jn}{d} \left\{\frac{jm}{d}\right\} + \left\{\frac{jm}{d}\right\}^2.
\end{align*}
On the other hand, for $t\notin S_j^+ \cup S_j^-$ we have $-\frac1n \leq f_1'(t) \leq \frac1m$ and $0\leq \delta(\ff,t) \leq mn$. If $\frac{jm}{d}$ is an integer, then by \eqref{Sj+-to1} we have
\[
\delta(\ff,S_j^+) = \delta(\ff,S_j^-) = f_\dims(j)
\]
and we are done. Otherwise, by \eqref{ft0} and \eqref{Sj+-to1} we have
\[
\frac1T \lambda\big([0,T]\cap S_j^\pm\big) \to \alpha^\pm \text{ as } T\to \infty,
\]
where $\alpha^+ + \alpha^- = 1$ and
\[
\alpha^+ f_1'(S_j^+) + \alpha^- f_1'(S_j^-) = 0.
\]
It follows that
\begin{align*}
\alpha^+ &= \left\{\frac{jm}{d}\right\},&
\alpha^- &= \left\{\frac{jn}{d}\right\},
\end{align*}
and thus
\begin{align*}
\underline\delta(\ff) &= \alpha^+ \delta(\ff,S_j^+) + \alpha^- \delta(\ff,S_j^-) = f_\dims(j).
\end{align*}
This completes the proof.

\section{Proof of Theorem \ref{theoremSSR}}
\label{subsectionSSR}

Part (i) follows directly from Lemma \ref{lemmaapproximatetemplate}, since we can take $\Lambda = u_A \Z^d$ where $A$ is the matrix in question. To prove part (ii), consider the template $\ff$ that we need to approxiomate by a successive minima function $\hh_A$. If $\overline\delta(\ff) > 0$, then by Theorem \ref{theoremvariational2}, the packing dimension of $\DD(\ff)$ is positive and thus $\DD(\ff)$ is nonempty. If we take $A\in\DD(\ff)$, then $\hh_A \asymp_\plus \ff$.
On the other hand, suppose that $\overline\delta(\ff) = 0$, and consider the set
\[
Z =  \{t\geq 0 : \Delta(\ff,t) > 0\}
\]
Then the density of $Z$ is zero, i.e. $\lim_{T\to\infty} \frac1T |Z\cap [0,T]| = 0$, where $|\cdot  |$ denotes $1$-dimensional Lebesgue measure. On the other hand, for all $t\notin Z$ we must have $\ff'(t) = (-\frac1n,\ldots,-\frac1n,\frac1m,\ldots,\frac1m)$. It follows that $f_n(t) < f_{n+1}(t)$ for all sufficiently large $t$. Then the convexity and quantized slope conditions (see Definition \ref{definitiontemplate}) imply that $F_n$ must be piecewise linear with only finitely many intervals of linearity. 
Now it follows, using the fact that $Z$ has zero density, that $F_n'(t) = -1$ for all sufficiently large $t$, which in turn implies that $\ff(t) \asymp_\plus (-\frac1n,\ldots,-\frac1n,\frac1m,\ldots,\frac1m) t$. Now there exist matrices $A$ such that $\hh_A(t) \asymp_\plus (-\frac1n,\ldots,-\frac1n,\frac1m,\ldots,\frac1m) t$ (for example, matrices with rational entries) and so this completes the proof.

\section{Proof of Theorem \ref{theoremJS}}
\label{sectionJS}
A matrix $A$ is badly approximable if and ony if its successive minima function $\hh_A$ is bounded. Thus, by Theorem \ref{theoremvariational2}, the Hausdorff dimension of the set of badly approximable matrices is equal to the supremum of $\underline\delta$ over bounded templates. Since $\underline\delta(\0) = mn$ and $\underline\delta(\ff) \leq mn$ for all templates $\ff$, this supremum is equal to $mn$.

\section{Proof of Theorem \ref{theoremBD}}
\label{sectionBD}

Analogously to the uniform dynamical exponent, we define the regular (non-uniform) dynamical exponent of a map $\ff:\Rplus\to\R^d$ to be the number
\begin{equation*}\label{RDE}
\dynexp(\ff) \df \limsup_{t\to\infty} \frac{-1}t f_1(t).
\end{equation*}
Now let $\ff$ be a template with $\tau(\ff)=\tau \in [0,\frac1n]$ and consider the potential function
\[
\phi(t) = \phi_\ff(t) = mn|f_1(t)|.
\]
\begin{lemma}
Let $I$ be an interval of linearity for $\ff$. Then
\[
\phi'(I) \leq mn - \delta(I),
\]
with equality in the following cases:
\begin{itemize}
\item $\ff = \0$ on $I$
\item $f_1' = -\tfrac1n$ and $f_2 = f_d$ on $I$.
\end{itemize}
\end{lemma}
\begin{subproof}
Let $j$ be the largest value such that $f_1 = f_j$ on $I$, and let $L_\pm = L_\pm(\ff,I,j)$. Then
\[
\phi'(I) = mn\frac1j\left[\frac{L_-}{n} - \frac{L_+}{m}\right]
\]
while
\[
mn - \delta(I) \geq L_- (m - L_+).
\]
So we need to show that
\[
\frac1j[m L_- - n L_+ ] \leq L_-  (m - L_+).
\]
Indeed, since $L_- \leq n$ and $j \geq 1$, we have
\[
\frac1j[m L_- - n L_+ ] \leq \frac1j[m L_- - L_- L_+ ] = \frac1j L_- (m - L_+) \leq L_- (m - L_+).
\]
Equality holds when $L_+ = m$ and $L_- = n$, and when $L_+ = 0$ and $L_- = 1$.
\end{subproof}

Integrating gives
\[
\phi(T) = mn|f_1(T)| \leq T(mn - \Delta(T)).
\]
Dividing by $T$ and then taking the limsup gives
\[
mn\tau \leq mn - \underline\delta(\ff).
\]
Rearranging gives
\[
\underline\delta(\ff) \leq mn(1 - \tau).
\]
Thus, by Theorem \ref{theoremvariational2}, we have
\[
\HD(\{\omega\text{-approximable matrices}\})
= \sup_{\ff:\tau(\ff)=\tau} \underline\delta(\ff) \leq mn(1 - \tau).
\]

Let us now show the reverse inequality. For each $\lambda > (1 + \tau/m) / (1 - \tau/n)$, let $\ff_\lambda$ be as in Figure \ref{JSBD}, i.e. $\ff_\lambda$ is exponentially $\lambda$-periodic and $f_{\lambda,1}$ is maximal with respect to the restriction $f_{\lambda,1}(1) = -\tau$. Then $\tau(\ff_\lambda) = \tau$, while $\underline\delta(\ff_\lambda) = mn(1 - \tau) / (1 - \lambda^{-1})$. So as $\lambda\to\infty$, we have $\underline\delta(\ff_\lambda) \to mn(1 - \tau)$ and the proof is complete.

\begin{figure}[h!]
\begin{tikzpicture}
\clip (0,-4) rectangle (39,2.5);
\scalebox{0.5}{
\draw[dashed] (1,0) -- (32,0);
\draw (3,0) -- (8,-3) -- (12,0);
\draw[line width=3] (12,0) -- (15,0);
\draw[line width=3] (3,0) -- (8,2) -- (12,0);
\draw (15,0) -- (25,-6) -- (31,0);
\draw[line width=3] (15,0) -- (25,4) -- (31,0);
\node at (20,-2) {\scalebox{1.8}{$-\frac1n$}};
\node at (28,-2) {\scalebox{1.8}{$\frac1m$}};
\fill (25,-6) circle (0.1);
\node at (25,-6.7) {\scalebox{1.3}{$(1,-\tau)$}};
}
\end{tikzpicture}
\caption{The joint graph of $\ff_\lambda$ as described above.}
\label{JSBD}
\end{figure}


\section{Proof of Theorem \ref{theoremconjectureBGMRV}}

To prove this theorem, we use Theorem \ref{theoremvariationaluniform} directly rather than its corollary Theorem \ref{theoremvariational2}. The proof is similar to that in Section \ref{subsectionmainlowerbound}. Fix $\epsilon > 0$, let $C_\epsilon > 0$ be as in the statement of Theorem \ref{theoremvariationaluniform}, and fix $t > 0$ large and $\gamma > 1$ close to $1$. Consider the template $\ff$ which is given on each interval $[\gamma^k,\gamma^{k+1}]$ by the standard template $\ss[(\gamma^k,-2C_\epsilon),(\gamma^{k+1},-2C_\epsilon)]$. Then as in Section \ref{subsectionmainlowerbound} we get $\underline\delta(\ff) \geq \dimsing - o(\gamma - 1)$, so by Theorem \ref{theoremvariationaluniform} we have $\HD(\DD(\ff,C_\epsilon)) \geq \dimsing - o(\gamma - 1)$. To complete the proof, we need to show that for all $\bfA \in \DD(\ff,C_\epsilon)$, we have $\bfA \in \FS(\dims)$. Indeed, we have $\|\Mink_\bfA(t)\| \geq \|\ff(t)\| - \|\Mink_\bfA(t) - \ff(t)\| \geq 2C_\epsilon - C_\epsilon = C_\epsilon$ for all $t$, which implies $\bfA \in \DI(\dims)$. Moreover, we have $\|\Mink_\bfA(\gamma^k)\| \lesssim C_\epsilon$ for all $k$ which implies $\bfA \notin \Sing(\dims)$. Finally, since $\|\Mink_\bfA(t\gamma^k)\| \to \infty$ for some constant $t$, we have $\bfA \notin \BA(\dims)$.

\draftnewpage
\part{Dimension games}
\label{partgames}
\section{Preliminaries on measures and dimensions}
\label{sectiondimensionprelim}
We first recall the basics of Hausdorff and packing measures and dimensions, \cite{BishopPeres, Falconer_book2013}. Hausdorff measure and dimension were introduced in 1918 by Hausdorff \cite{Hausdorff}, while packing measure and dimension were introduced by Tricot in 1982 \cite{Tricot}. Sullivan independently (re)discovered packing measures and dimensions when studying the limit sets of geometrically finite Kleinian groups in 1984 \cite{Sullivan_entropy}.

The \emph{$s$-dimensional Hausdorff measure} of a set $A\subset\R^D$ is
\label{HM}
\[
\scrH^s(A) \df \sup_{\epsilon > 0} \inf\left\{\sum_{i = 1}^\infty (\diam(U_i))^s: 
\begin{split}
&\text{$(U_i)_1^\infty$ is a countable cover of $A$}\\
&\text{~~~with $\diam(U_i)\leq\epsilon \all i$}
\end{split}
\right\} \cdot
\]
Dual to the Hausdorff measure, which is defined via economical coverings by small balls, it is natural to define a measure in terms of dense packings by small disjoint balls. This leads to the notion of the \emph{$s$-dimensional packing measure} of a set $A\subset\R^D$, which is defined as 
\label{PM}
\[
\scrP^s(A) \df \inf\left\{ \sum_{i = 1}^\infty \w\scrP^s(A_i) : A \subset \bigcup_{i = 1}^\infty A_i\right\},
\]   
where
\[
\w\scrP^s(A) \df \inf_{\epsilon > 0} \sup\left\{\sum_{j = 1}^\infty (\diam(B_j))^s:
\begin{split}
&\text{$(B_j)_1^\infty$ is a countable disjoint collection of balls}\\
&\text{~~~with centers in $A$ and with $\diam(B_j)\leq\epsilon \all j$}
\end{split}\right\}\cdot
\]
Given the measures defined above, we define the \emph{Hausdorff dimension} and \emph{packing dimension} of a set $A\subset\R^D$ as follows:
\label{HDPD}
\begin{align*}
\HD(A) &\df \inf \{ s : \scrH^s(A) = 0 \} =  \sup \{ s : \scrH^s(A) = \infty \}\\
\PD(A) &\df \inf \{ s : \scrP^s(A) = 0 \} = \sup \{ s : \scrP^s(A) = \infty \}.
\end{align*} 
We recall two basic facts (see \cite[\6 3.2 and \6 3.5]{Falconer_book2013}) about these dimensions. First, they are both {\it monotonic}, i.e. if $E \subset F \subset \R^D$, then $\HD(E) \leq \HD(F)$ and $\PD(E) \leq \PD(F)$. Second, the packing dimension is bounded below by the Hausdorff dimension, i.e. for $F \subset \R^D$, we have $\HD(F) \leq \PD(F)$.

In the sequel we will apply the following consequence of the Rogers--Taylor--Tricot density theorem for Hausdorff and packing measures \cite[Theorem 2.1]{Simmons7}, which provides a method of computing the Hausdorff and packing dimensions of a Borel set in terms of local geometric-measure-theoretic information.
For each point $\xx\in\R^D$ define the lower and upper pointwise dimensions of a measure $\mu$ at $\xx$ by 
\label{localdim}
\[
\underline\dim_\xx(\mu) \df \liminf_{\rho\to 0} \frac{\log\mu(B(\xx,\rho))}{\log\rho} \;\;\;\text{and}\;\;\;
\overline\dim_\xx(\mu) \df \limsup_{\rho\to 0} \frac{\log\mu(B(\xx,\rho))}{\log\rho} \cdot
\]
Note that the limits may be replaced by limits over any sequence $\rho_n\to 0$ such that $\rho_n / \rho_{n + 1}$ is bounded, without affecting the values. Also note that it is possible for the dimensions to take any value in $[0,\infty]$.

\begin{theorem}
\label{theoremRTT}
Fix $D\in\N$ and let $\mu$ be a locally finite Borel measure on $\R^D$. Then for every Borel set $A\subset\R^D$,
\begin{itemize}
\item If $\underline\dim_\xx(\mu) \geq s$ for all $x \in A$ and $\mu(A)>0$, then $\HD(A)\geq s$.
\item If $\underline\dim_\xx(\mu) \leq s$ for all $x \in A$, then $\HD(A)\leq s$.
\item If $\overline\dim_\xx(\mu) \geq s$ for all $x \in A$ and $\mu(A)>0$, then $\PD(A)\geq s$.
\item If $\overline\dim_\xx(\mu) \leq s$ for all $x \in A$, then $\PD(A)\leq s$.
\end{itemize}
\end{theorem}

The statement above is closest to \cite[Proposition 2.3]{Falconer_book3}. Readers interested in studying further refinements are referred to Cutler's weak and strong duality principles in \cite[Theorems 1.4 and 1.5]{Cutler}. See \cite[\68]{MSU} for a self-contained proof of the density theorem for measures in the setting of metric spaces.

\section{A characterization of Hausdorff and packing dimensions using games}
\label{sectiongame}

Schmidt's game is a two-player topological game introduced in a seminal paper of Wolfgang M. Schmidt in 1966 \cite{Schmidt1} as a technique to analyze Diophantine sets that are exceptional with respect to both measure and category. 
Schmidt's paper led to a plethora of applications at the interface of dynamical systems, Diophantine approximation and fractal geometry, which often involve various modifications of his eponymous game. For a small sample of such research, see \cite{Dani2, McMullen_absolute_winning, KleinbockWeiss2, BFKRW, CCM, An1, BeresnevichVelani5, FSU4, BHNS}.

The proof of our variational principle is based on a new variant of Schmidt's game which is in principle capable of computing the Hausdorff and packing dimensions of any Borel set. In Schmidt's original game, players take turns choosing a descending sequence of balls and compete to determine whether or not the intersection point of these balls is in a certain target set. The key feature of our new variant is that instead of requiring the rate at which the players' moves contribute information to the game to be constant, the new variant allows the rate of information transfer to be variable, with the first player, Alice, getting to choose the rate of information transfer. However, Alice is penalized if she exerts too much control over the game over long periods of time without giving her opponent Bob a chance to exert control over the game. 

\begin{definition}
\label{definitiongames1}
Given $0 < \beta < 1$ and $\delta >0$, Alice and Bob play the \emph{$\delta$-dimensional Hausdorff (resp. packing) $\beta$-game} as follows:
\begin{itemize}
\item The turn order is alternating, with the first turn being the $0$th turn and Alice playing first. Thus, Bob's $k$th turn occurs after Alice's $k$th turn and before Alice's $(k + 1)$st turn.
\item Alice begins by choosing a starting radius $\rho_0 > 0$.
\item On the $k$th turn, Alice chooses a nonempty finite $3\rho_k$-separated set\Footnote{\label{rhosep} A set $A$ is called \emph{$\rho$-separated} if $\dist(x,y) > \rho$ for all distinct $x,y \in A$.} $A_k \subset \R^D$, and Bob responds by choosing a ball $B_k \df B(\xx_k,\rho_k)$, where $\xx_k \in A_k$ and $\rho_k \df \beta^k \rho_0$. (We can think of Alice's choice $A_k$ as representing the collection of balls $\{B(\xx,\rho_k) : \xx\in A_k\}$ from which Bob chooses his ball.)
\item On the $0$th turn, there is no further restriction on Alice's choice $A_0$, but on each subsequent turn $(k+1)$, she must choose $A_{k+1}$ so as to satisfy
\begin{equation}
\label{Alicerules}
A_{k + 1} \subset B(\xx_k,(1 - \beta)\rho_k).
\end{equation}
Note that this condition guarantees (see Figure \ref{figurepotentialgame}) that
\[
B_0 \supset B_1 \supset B_2 \supset \cdots
\]
\end{itemize}
\begin{figure}
\centering
\begin{tikzpicture}[scale=0.45]
\clip (-1.1,-1.1) rectangle (5.1,5.1);
\redball{ (2,2) circle (3)}
\end{tikzpicture}
\begin{tikzpicture}[scale=0.45]
\clip (-1.1,-1.1) rectangle (5.1,5.1);
\draw (2,2) circle (3);
\draw (0.1,2.2) circle (1);
\draw (2.8,1) circle (1);
\redball{ (2.5,3.5) circle (1)}
\end{tikzpicture}
\begin{tikzpicture}[scale=0.45]
\clip (-1.1,-1.1) rectangle (5.1,5.1);
\draw (2,2) circle (3);
\draw (0.1,2.2) circle (1);
\draw (2.8,1) circle (1);
\draw (2.5,3.5) circle (1);
\draw (2.5,2.9) circle (0.333);
\draw (2.5,4.1) circle (0.333);
\draw (1.9,3.5) circle (0.333);
\redball{ (3.1,3.5) circle (0.333)}
\end{tikzpicture}
\caption{Three consecutive rounds of the Hausdorff/packing game. On each round Alice presents Bob with a set of balls to choose between (represented by the set of centers of those balls), and Bob chooses one of the balls, which are colored/shaded above.}
\label{figurepotentialgame}
\end{figure}
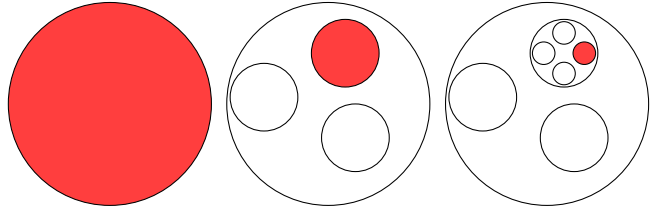
After infinitely many turns have passed, the point
\begin{equation}
\label{outcome}
\xx_\infty = \lim_{k\to\infty} \xx_k \in \bigcap_{k = 0}^\infty B_k
\end{equation}
is computed (note that the right-hand side is always a singleton). It is called the \emph{outcome} of the game. Also, we let $\AA = (A_k)_{k\in\N}$, and we compute the number
\begin{equation}
\label{Hausdorff}
\underline\delta(\AA) \df \liminf_{k\to\infty} \frac{1}{k} \sum_{i = 0}^k \frac{\log\#(A_i)}{-\log(\beta)}
\end{equation}
resp.
\begin{equation}
\label{packing}
\overline\delta(\AA) \df \limsup_{k\to\infty} \frac{1}{k} \sum_{i = 0}^k \frac{\log\#(A_i)}{-\log(\beta)},
\end{equation}
which represents Alice's \emph{score}. Alice's goal will be to ensure that the outcome is in a certain set $S$, called the \emph{target set}, and simultaneously to guarantee that her score is at least $\delta$. To be precise, a set $S \subset \R^D$ is said to be \emph{$\delta$-dimensionally Hausdorff (resp. packing) $\beta$-winning} if Alice has a strategy to simultaneously ensure that the outcome $\xx_\infty$ is in $S$, and that her score $\underline\delta(\AA)$ (resp. $\overline\delta(\AA)$) is at least $\delta$. The set $S$ is said to be \emph{$\delta$-dimensionally Hausdorff (resp. packing) winning} if it is $\delta$-dimensionally Hausdorff (resp. packing) $\beta$-winning for all sufficiently small $\beta > 0$. (Equivalently, we could say that Alice's score is automatically set equal to zero whenever $\xx_\infty \notin S$, in which case we would say that $S$ is $\delta$-dimensionally Hausdorff $\beta$-winning if Alice has a strategy to ensure that her score is at least $\delta$.)
\end{definition}

The following result is one of the key ingredients in the proof of the variational principle:

\begin{theorem}
\label{theoremHPgame}
The Hausdorff (resp. packing) dimension of a Borel set $S \subset \R^D$ is the supremum of $\delta$ such that $S$ is $\delta$-dimensionally Hausdorff (resp. packing) winning.
\end{theorem}
\begin{remark}
The theorem remains true (with the same proof) if $\R^D$ is replaced by any doubling\Footnote{A metric space is \label{doubling}\emph{doubling} if there exists constants $C,r_0$ such that every ball of radius $0<r\leq r_0$ can be covered by at most $C$ balls of radius $r/2$.} metric space.
\end{remark}
A key fact used in the proof is that since $S$ is Borel, the Borel determinacy theorem \cite{Martin_determinacy} implies that for all $\delta,\beta$, the $\delta$-dimensional Hausdorff and packing $\beta$-games are determined, meaning that either Alice or Bob has a winning strategy. This follows from \cite[Theorem 3.1]{FLS}, since the games can be viewed as ``games played on complete metric spaces'' in the language of \cite{FLS}, specifically with $X = \R^D \times \N^\N$ (the latter factor representing the number of balls that Alice chooses in each step).
\begin{proof}
We prove the theorem for the case of Hausdorff dimension; the argument in the case of packing dimension is nearly identical.\\

\noindent We begin by proving the lower bound. Suppose that $S$ is $\delta$-dimensionally Hausdorff winning, and we must show that $\HD(S) \geq \delta$. Fix $\beta > 0$ such that $S$ is $\delta$-dimensionally Hausdorff $\beta$-winning, and consider a strategy for Alice to win the  $\delta$-dimensional Hausdorff $\beta$-game with target set $S$. Now for each $k\geq 0$, let $E_k$ denote the union of all sets $A_k$ that Alice might choose according to her strategy in response to some possible sequence of moves that Bob could play, and let $\rho_k = \beta^k \rho_0$. Then the set
\[
C \df \bigcap_{k = 0}^\infty \bigcup_{\xx_k\in E_k} B(\xx_k,\rho_k)
\]
is the set of all possible outcomes of the game when Alice plays her winning strategy. It is a closed and totally disconnected set, contained entirely in $S$. Note that by induction and the restrictions on Alice's possible moves, for all $k$, $E_k$ is $3\rho_k$-separated.

To bound the Hausdorff dimension of $C$, we introduce a probability measure on $C$ by considering the scenario where Alice plays according to her winning strategy and Bob plays randomly: on the $k$th turn, Bob chooses the point $\xx_k \in A_k$ uniformly at random, independently of all previous choices. This yields a random game whose outcome is distributed according to some probability measure $\mu$ on $C$. Now fix $\xx\in C$, and for each $k\geq 0$ let $\xx_k\in E_k$ be chosen so that $\xx\in B(\xx_k,\rho_k)$. Then since $E_k$ is $3\rho_k$-separated, if Bob plays in a way such that the final outcome is in $B(\xx,\rho_k)$, then on the $k$th turn he must choose the ball $B(\xx_k,\rho_k)$. It follows that
\[
B(\xx,\rho_k)\cap C \subset B(\xx_k,\rho_k)
\]
and thus
\[
\mu(B(\xx,\rho_k)) \leq \mu(B(\xx_k,\rho_k)) = \left(\prod_{i = 0}^k \#(A_i)\right)^{-1}.
\]
So the lower pointwise dimension of $\mu$ at $\xx$ is
\begin{align*}
\underline\dim_\xx(\mu)
&\df \liminf_{\rho\to 0} \frac{\log\mu(B(\xx,\rho))}{\log\rho}\\
&= \liminf_{k\to\infty} \frac{\log\mu(B(\xx,\rho_k))}{\log\rho_k} \since{$\rho_k = \beta^k \rho_0$}\\
&\geq \liminf_{k\to\infty} \frac{\log\mu(B(\xx_k,\rho_k))}{\log\rho_k}\\
&= \liminf_{k\to\infty} \frac{-\sum_{i = 0}^k \log\#(A_i)}{k\log\beta + \log\rho_0}\\
&= \underline\delta(\AA) \geq \delta
\end{align*}
since Alice is using a winning strategy. Since $\xx\in C$ was arbitrary and $\mu(C) = 1$, applying the Rogers--Taylor--Tricot Theorem \ref{theoremRTT} proves the lower bound $\HD(S) \geq \delta$.\\

\noindent To prove the upper bound, suppose that $S$ is not $\delta$-dimensionally Hausdorff winning, and we will show that $\HD(S) \leq \delta$. Fix $0 < \beta \leq 1/2$ small enough so that $S$ is not $\delta$-dimensionally Hausdorff $\beta$-winning. Then Alice does not have a winning strategy for the $\delta$-dimensional Hausdorff $\beta$-game with target set $S$. Since this game is determined as we mentioned earlier, we know that Bob must have a winning strategy for it, which we now fix.

Fix a radius $\rho_0 > 0$, and for each $k\in\N$ 
\begin{itemize}
\item let $E_k$ be a maximal $\frac 13\beta^k\rho_0$-separated subset of $\R^D$, and
\item  let $E_k^{(1)},\ldots,E_k^{(p)}$ be disjoint $3\beta^k\rho_0$-separated subsets of $E_k$ such that $E_k = \bigcup_{i = 1}^p E_k^{(i)}$.
\end{itemize}
Since $\R^D$ is a doubling metric space (see Footnote \ref{doubling}), it is possible to choose $p$ to be independent of $k$ and $\beta$. We define a family of counter-strategies for Alice as follows. Consider the $k$th turn for some $k\in\N$, and if $k \geq 1$ then let $B_{k-1} = B(\xx_{k-1},\rho_{k-1})$ be the move that Bob just played. Let
\[
\w B_{k-1} = \begin{cases}
B(\xx_{k-1},(1-\beta)\rho_{k-1}) & k\geq 1\\
B(\0,\kappa + \rho_0) & k = 0
\end{cases},
\]
where $\kappa > 0$ is a large constant. Next let
\[
X_k^{(i)} = E_k^{(i)}\cap \w B_{k-1},\;\;\;\;
N_k^{(i)} = \#(X_k^{(i)}),\;\;\;\;
A_k^{(i,N_k^{(i)})} = X_k^{(i)} ~.
\]
From then on, we define the moves $A_k^{(i,j)}$ and $B_k^{(i,j)}$ by backwards recursion as follows:
\begin{itemize}
\item if $A_k^{(i,j)}$ is defined for some $j\geq 1$, then $$B_k^{(i,j)} = B(\xx_k^{(i,j)},\rho_k)$$ is Bob's response if Alice plays $A_k^{(i,j)}$.
\item if $A_k^{(i,j)}$ and $B_k^{(i,j)} = B(\xx_k^{(i,j)},\rho_k)$ are both defined for some $j \geq 1$, then $$A_k^{(i,j-1)} \df A_k^{(i,j)} \butnot \{\xx_k^{(i,j)}\}.$$
\end{itemize}
Note that $\#(A_k^{(i,j)}) = j$ for all $j = 0,\ldots,N_k^{(i)}$.

In what follows we will consider two counter-strategies for Alice: a random one and a deterministic one. The purpose of the random strategy is to construct a measure on $S$, whereas the purpose of the deterministic strategy is to prove a bound on this measure in order to apply the Rogers--Taylor theorem. Note that since Bob's strategy wins against deterministic counter-strategies, it wins against random ones as well.

Now consider the scenario where Bob plays according to his winning strategy and Alice plays randomly: on the $k$th turn, Alice chooses a move $A_k^{(i_k,j_k)}$ where the integers $i_k$ and $j_k$ are chosen independently of previous choices $i_1,i_2,\dots,i_{k-1}$ and $j_1,j_2,\dots,j_{k-1}$ of with respect to a probability distribution satisfying
\begin{equation}
\label{ikjkdist}
\prob(i_k = i, \; j_k = j) \geq c j^{-(1+\epsilon)},
\end{equation}
where $\epsilon > 0$ is fixed and $c > 0$ is a constant depending on $\epsilon$ and $p$. By the Kolmogorov extension theorem, this yields a random sequence of plays whose outcome is distributed according to some probability measure $\mu$ on $\R^D$.

Now fix $\xx\in S\cap B(\0,\kappa)$. For each $k\in\N$, there exists $\xx_k \in E_k$ such that $\dist(\xx,\xx_k) \leq \frac{1 - \beta}{1 + \beta} \rho_k$. Note that $\xx_0 \in B(\0,\kappa + \rho_0)$, and $\xx_{k+1} \in B(\xx_k,(1-\beta)\rho_k)$ for all $k$. It follows that Alice can guarantee that the outcome is equal to $\xx$ by playing the move $A_k = A_k^{(I_k,J_k)}$ on the $k$th turn for some sequences of integers $(I_k)_{k\in\N}$, $(J_k)_{k\in\N}$. Since Bob's strategy is winning and $\xx\in S$, it follows Alice's score is less than $\delta$, i.e.
\[
\underline\delta(\AA) < \delta.
\]
Let $G_1$ denote the sequence of plays described above, and let $G_2$ be a sequence of plays where on the $k$th turn, Alice chooses a set $A_k^{(i_k,j_k)}$, and Bob responds according to his winning strategy, such that $i_k = I_k$ and $j_k = J_k$ for all $k \in \{ 0,\ldots,\ell \}$. Then the $\ell$th ball of $G_2$ is equal to the $\ell$th ball of $G_1$, and thus the outcome of $G_2$ is within $2\rho_\ell$ of the outcome of $G_1$, i.e. $\xx$. Thus if we think of $G_2$ as being chosen randomly, then
\begin{align*}
\mu\big(B(\xx,2\rho_\ell)\big) 
&\geq \prob\big(i_k = I_k,\; j_k = J_k \all k\leq \ell\big)\\ 
&\geq \prod_{k = 0}^\ell c J_k^{-(1 + \epsilon)}\\
&= c^\ell \exp\left(-(1+\epsilon)\sum_{k = 0}^\ell \log\#(A_k)\right)
\end{align*}
and so
\begin{align*}
\underline\dim_\xx(\mu)
&= \liminf_{\ell\to\infty} \frac{\log\mu(B(\xx,2\rho_\ell))}{\log(2\rho_\ell)}\\
&\leq \liminf_{\ell\to\infty} \frac{\ell\log(c) + (1+\epsilon)\sum_{k = 0}^\ell -\log\#(A_k)}{\ell\log\beta + \log(2\rho_0)}\\
&= \frac{\log(c)}{\log(\beta)} + (1 + \epsilon)\underline\delta(\AA)\\
&< \frac{\log(c)}{\log(\beta)} + (1 + \epsilon)\delta.
\end{align*}
Since $\xx\in S$ was arbitrary, applying the Rogers--Taylor Theorem \ref{theoremRTT} again yields
\[
\HD(S) \leq \frac{\log(c)}{\log(\beta)} + (1 + \epsilon)\delta.
\]
Letting $\beta,\epsilon \to 0$ completes the proof.
\end{proof}

\section{Playing games with Diophantine targets}
\label{sectiongamesgeometry}

In practice, when we play the Hausdorff or packing game with a target set defined in terms of the parametric geometry of numbers, it is helpful to use a different formalism to encode Alice and Bob's moves. First of all, note that for each $k$, the ball $B_k = B(\xx_k,\rho_k)$ is homeomorphic to the unit ball $B(\0,1)$ via the similarity transformation
\[
T_k(\zz) = \xx_k + \rho_k\zz.
\]
By replacing $A_{k+1}$ and $\xx_{k+1}$ by their preimages under $T_k$, and leaving $A_0$ and $\xx_0$ the same, we can see that we can make the following changes to the rules of the $\delta$-dimensional Hausdorff (resp. packing) $\beta$-game without affecting the existence of winning strategies for either player:\footnote{The moves in the new modified game $\w A_k$ and $\w \xx_k$ are related to the analogous moves $A_k$ and $\xx_k$ in the original game via the formulas $A_k = \xx_{k-1} + \rho_{k-1} \w A_{k}$ and $\xx_k = \xx_{0} + \sum_{i = 1}^k \beta^i \rho_{-1} \w \xx_i$.}

\begin{itemize}
\item For $k\geq 1$, instead of requiring that Alice's choice $A_k$ is $3\rho_k$-separated, we require that it is $3\beta$-separated.
\item Instead of \eqref{Alicerules}, Alice must choose $A_{k+1}$ to satisfy
\begin{equation}
\label{Alicerules2}
A_{k+1} \subset B(\0,1-\beta)
\end{equation}
whereby Bob then chooses a point $\xx_{k+1} \in A_{k+1}$.
\item The outcome of the game, instead of being computed by \eqref{outcome}, is computed by the formula
\begin{equation}
\label{outcome2}
\xx_\infty \df \xx_0 + \sum_{k = 1}^\infty \beta^k \rho_{-1} \xx_k,
\end{equation}
where $\rho_{-1} \df \beta^{-1}\rho_0$ (using the definition of $\rho_k$ in Definition \ref{definitiongames1}).
\end{itemize}
We will call the version of the Hausdorff (resp. packing) game resulting from these rule changes the \label{definitionmodifiedgame}\emph{modified Hausdorff (resp. packing) game}. It will be the version we use in the proof of Theorem \ref{theoremvariational2} (Variational principle, version 2) in Part \ref{partproofvariational}.

\comtushar{We need to make the change $d \to D$ in the section before.\ddr}
Let $D=mn$ in Section \ref{sectiongame}, and let us identify $\R^D$ with $\MM$, the space of $m\times n$ matrices with real entries. Thus we replace $\xx_k$ by $X_k$ etc. Further, we will assume that the target set is of the form (recalling notation from below Theorem \ref{theoremSSR})
\[
S = \DD(\SS) = \bigcup_{\ff\in\SS} \DD(\ff) = \bigcup_{\ff\in\SS} \{\bfA \in \MM : \Mink_\bfA \asymp_\plus \ff\}
\]
for some collection $\SS$ of functions from $\Rplus$ to $\R^d$ closed under finite perturbations (i.e. if whenever $\ff\in\SS$ and $\gg \asymp_\plus \ff$, we have $\gg\in \SS$). In this case, we can track the ``progression'' of the game by associating a unimodular lattice to each turn of the game. Specifically, for each $k\geq 0$ let
\begin{equation}
\label{Lambdakdef}
\Lambda_{k+1} \df g_{-\alpha\log(\beta^{k+1} \rho_{-1})} u_{Y_k} \Z^{\dimsum}, ~\text{where}~
\alpha \df \frac{\dimprod}{\dimsum} ~\text{and}~ Y_k \df X_0 + \sum_{i=1}^k \beta^i \rho_{-1} X_i .
\end{equation}
Here, we use uppercase letters $(X,Y)$ instead of bold letters $(\xx,\yy)$ because we are working with matrices rather than with vectors. Then for $k\geq 1$, $\Lambda_k$ and $\Lambda_{k+1}$ are related by the formula
\begin{equation}
\label{Lambdarecursive}
\Lambda_{k+1} = g_{-\alpha\log(\beta)} u_{X_k} \Lambda_k.
\end{equation}
This is because
\begin{equation}
\label{danisemi}
u_X g_{-\alpha\log(\lambda)} = g_{-\alpha\log(\lambda)} u_{\lambda X} \text{ for all } \lambda, X.
\end{equation}
\begin{notation}
\label{gamma}
To simplify the notation in \eqref{Lambdarecursive}, we let
\begin{align*}
\tbeta &\df -\alpha\log(\beta) > 0,& 
g &\df g_\tbeta,
\end{align*}
so that
\begin{equation*}
\Lambda_{k+1} = g u_{X_k} \Lambda_k.
\end{equation*}
Intuitively, this means that $\Lambda_k$ is well-defined at the start of turn $k$, and that Alice and Bob's choices on turn $k$ can be thought of as a process of choosing $\Lambda_{k+1}$ indirectly by choosing $X_k$.
\end{notation}

The significance of the sequence of lattices $(\Lambda_k)_1^\infty$ is given by the following lemma, where we use the notation
\begin{equation*}
\label{hLambda}
\hh(\Lambda) \df (\log\lambda_1(\Lambda),\ldots,\log\lambda_d(\Lambda)).
\end{equation*}
\begin{lemma}
\label{lemmaMS}
Let $\jj:\Rplus\to\R^d$ be the function defined on $\gamma\Z$ by the formula
\[
\jj(k\gamma) \df \hh(\Lambda_k)
\]
and extended to $\Rplus$ via linear interpolation. Then 
\[
\jj\asymp_\plus \hh_{X_\infty}
\] 
where $X_\infty$ is as in \eqref{outcome2}. In particular, $X_\infty\in\DD(\SS)$ if and only if $\jj\in\SS$.
\end{lemma}
\begin{proof}
Fix $k$, and write
\[
Z_k \df \sum_{i=0}^\infty \beta^i X_{k+i}.
\]
Then
\[
u_{Z_k} \Lambda_k = g_{-\alpha \log(\rho_{-1}) + k\gamma} u_{X_\infty} \Z^{\dimsum}.
\]
Since $Z_k \in B(\0,1)$, this implies that
\[
\ff(k\gamma) = \hh(\Lambda_k) \asymp_\plus \hh(u_{Z_k} \Lambda_k) = \hh(g_{-\alpha\log(\rho_{-1}) + k\gamma} u_{X_\infty} \Z^{\dimsum}) = \hh_{X_\infty}(-\alpha\log(\rho_{-1}) + k\gamma)
\]
and thus $\ff \asymp_\plus \hh_{X_\infty}$. Since $\SS$ is closed under finite perturbations, it follows that $\ff$ is in $\SS$ if and only if $\hh_{X_\infty}$ is, i.e. if and only if $X_\infty \in \DD(\SS)$.
\end{proof}

\draftnewpage
\part{Proof of the variational principle}
\label{partproofvariational}

Throughout Part \ref{partproofvariational}, $\| \cdot \|$ generally denotes the Euclidean norm, i.e. $\| \xx \|^2 \df \sum_i x_i^2$.  This is allowed since the variational principle (Theorem \ref{theoremvariationaluniform}) is independent of the choice of norm. In certain places we will use the max norm, i.e. $\| \xx \|_{\infty} \df \max_i |x_i|$. Also, in certain places $\|\cdot\|$ denotes other kinds of norms such as the covolume of a lattice.

\section{Preliminaries}
This section collects various notation and lemmata employed in our proof of the variational principle, viz. Theorem \ref{theoremvariationaluniform}. Though some of these results may be considered elementary by experts familiar with the geometry of numbers, we include such for the benefits of self-containment. Thus, for instance, we begin by recalling Minkowski's second theorem on successive minima for the reader's convenience.

\begin{theorem}[Minkowski, {\cite[Theorem V in \6VIII.4.3]{Cassels3}}]
\label{mink2}
Let $\Lambda$ be a lattice in a vector space $V\subset \R^d$. Then
\[
\prod_{j=1}^{\dim(V)} \lambda_j(\Lambda) \asymp \|\Lambda\|,
\]
where $\|\Lambda\|$ denotes the covolume of $\Lambda$, and \hyperref[subsectionDani]{$\lambda_j(\Lambda)$} is the $j$th minimum of $\Lambda$ with respect to the unit ball of $V$.
\end{theorem}

\begin{definition}
\label{definitionLambdarational}
Let $\Lambda \subset \R^d$ be a lattice. A subspace $V\subset \R^d$ is called \emph{$\Lambda$-rational} if $V\cap\Lambda$ is a lattice in $V$. Denote the set of all $q$-dimensional $\Lambda$-rational subspaces of $\R^d$ by $\VV_q(\Lambda)$.
\end{definition}

\begin{notation}
\label{Lambdarational}
If $V$ is a $\Lambda$-rational subspace of $\R^d$, we denote the covolume of $V\cap\Lambda$ in $V$, with respect to the Euclidean metric on $V$ inherited from $\R^d$, by $\|V\|$. Although this notation is misleading since $\|V\|$ depends on $\Lambda$ and not just on $V$, in practice this should not be a problem as it should generally be clear what $\Lambda$ is (for instance, if $V$ is $\Lambda$-rational, then $g V$ is $g \Lambda$-rational, so we can take $\|g V\|$ to be the covolume of $g(V\cap \Lambda)$).
\end{notation}

The following result is well-known.
\begin{proposition}[Exterior product formula]
\label{propositionexteriorproduct}
If $\vv_1,\ldots,\vv_k$ is a basis of $V\cap \Lambda$, then
\[
\|V\| = \|\vv_1\wedge\cdots\wedge\vv_k\|.
\]
\end{proposition}

\begin{notation}
\label{contracting}
We denote the subspace of $\R^d$ contracted by the \hyperref[subsectionDani]{$(g_t)$ flow} (defined in \6 \ref{subsectionDani}) by $\vert$, i.e.
\[
\vert \df \{\0\}\times \R^\qdim.
\]
The conical $\epsilon$-neighborhood of a subspace $V \subset \R^d$ will be denoted
\[
\CC(V,\epsilon) = \{\rr : \dist(\rr,V) \leq \epsilon\|\rr\|\}
\]
where $\dist$ denotes infimal distance. Given $0 < \beta < 1$ in the definition of the $\delta$-dimensional Hausdorff (resp. packing) $\beta$-game (see Definition \ref{definitiongames1}), and following Notation \ref{gamma} and \eqref{Lambdakdef} from Section \ref{sectiongamesgeometry}, we write
\begin{align*}
\tbeta &= -\frac{mn}{m+n}\log(\beta),&
g &= g_\tbeta.
\end{align*}
Following Definition \ref{definitiontemplate}, we write
\[
F_q = \sum_{i=1}^q f_i.
\]
\end{notation}

The following lemmas will be used in the proof of Theorem \ref{theoremvariationaluniform}.

\begin{lemma}
\label{lemmaminkowskibasis}
Let $\Lambda \leq \R^d$ be a lattice. Then there exists a basis $\{\rr_1,\ldots,\rr_d\}$ of $\Lambda$ such that if $V_q = \sum_{i=1}^q \R\rr_i$, then
\[
\log\|V_q\| \asymp_\plus \sum_{i=1}^q \log\lambda_i(\Lambda).
\]
Moreover,
\begin{equation}
\label{rilambdai}
\log\|\rr_i\| \asymp_\plus \log\lambda_i(\Lambda) \text{ for all } i.
\end{equation}
\end{lemma}
\begin{proof}
 
Let $\{\rr_1,\ldots,\rr_d\}$ be a Minkowski reduced basis of $\Lambda$ (see \cite[Proposition 5.3]{Helfrich}). 
Now let $h$ be the change of basis matrix changing $\{\ee_1,\ldots,\ee_d\}$ into $\{\rr_1,\ldots,\rr_d\}$ and write $h = kan$ where $k\in \SO(d)$, $a$ is a diagonal matrix, and $n$ is an upper triangular matrix. Since $\{\rr_1,\ldots,\rr_d\}$ is a Minkowski reduced basis, $a_i \gtrsim a_{i+1}$ for all $i$ and $n$ is bounded, and thus $g = k(ana^{-1})$ is also bounded. Note that $\rr_i = ga\ee_i$ for all $i$.
Then for all $i$, 
\[
\|\rr_i\| \asymp a_i \asymp \lambda_i(a \Z^d) \asymp \lambda_i(g a \Z^d) = \lambda_i(\Lambda).
\]
On the other hand, for each $q = 1,\ldots,d$, we have
\[
\log\|V_q\| \asymp_\plus \log\|a E_q\| = \log(a_1\ldots a_q) =
\sum_{i=1}^q \log(a_i) \asymp_\plus \sum_{i=1}^q
\log\lambda_i(\Lambda),
\]
where $E_q = \sum_{i=1}^q \R\ee_i$.
\end{proof}

\begin{lemma}
\label{lemmaminkowskibasis2}
Let $\Lambda \leq \R^d$ be a lattice, and let $V_q$ be as in Lemma \ref{lemmaminkowskibasis}. Then
\begin{equation}
\label{minkowskibasis2}
\log\|V_q'\| \gtrsim_\plus \sum_{i=1}^{q-1} \log\lambda_i(\Lambda) + \log\lambda_{q+1}(\Lambda) \text{ for all } V_q' \in \VV_q(\Lambda) \butnot \{V_q\}.
\end{equation}
\end{lemma}
\begin{proof}
Fix $V_q'\in \VV_q(\Lambda) \butnot \{V_q\}$. By Minkowski's second theorem (Theorem \ref{mink2}), we have
\begin{equation}
\label{secondtheoremprime}
\log\|V_q'\| \asymp_\plus \sum_{i=1}^q \log\lambda_i(\Lambda\cap V_q').
\end{equation}
For all $i = 1,\ldots,q-1$, we have
\begin{equation}
\label{lambdaiVqprime}
\lambda_i(\Lambda\cap V_q') \geq \lambda_i(\Lambda).
\end{equation}
For the $i=q$ term, we use a different argument to get a better bound. Let $E$ (resp. $E'$) be a spanning set for $\Lambda\cap V_q$ (resp. $\Lambda\cap V_q'$). Then $E\cup E'$ is a spanning set for $\Lambda\cap (V_q+V_q')$. Since $\dim(V_q+V_q') \geq q+1$, it follows that
\[
\max_{\rr\in E\cup E'} \|\rr\| \geq \lambda_{q+1}(\Lambda).
\]
Taking the infimum over all $E,E'$ gives
\[
\max\big(\lambda_q(\Lambda\cap V_q),\lambda_q(\Lambda\cap V_q')\big) \geq \lambda_{q+1}(\Lambda).
\]
On the other hand, it follows from \eqref{rilambdai} that $\lambda_q(\Lambda\cap V_q) \asymp \lambda_q(\Lambda) \leq \lambda_q(\Lambda\cap V_q')$. Thus,
\[
\lambda_q(\Lambda\cap V_q') \gtrsim \lambda_{q+1}(\Lambda).
\]
Combining with \eqref{secondtheoremprime} and \eqref{lambdaiVqprime} yields \eqref{minkowskibasis2}.
\end{proof}

\begin{lemma}
\label{lemmapushoffvert}
Recall from \eqref{slopeset} that $d_+ = \pdim$ and $d_- = \qdim$. Let $q=\dir_+ + \dir_-$ be a decomposition of $q$ with $\dir_\pm \in [0,d_\pm] \cap \Z$. Let $\Lambda$ be a lattice and let $V$ be a $q$-dimensional $\Lambda$-rational subspace such that
\[
\dir_- \geq \sup_{\|\bfC\|\leq\beta} \dim(u_\bfC V\cap \vert).
\]
Then for all $t\geq 0$,
\begin{equation}
\label{gtVV}
\log\|g_t V\| - \log\|V\| \gtrsim_{\plus,\beta} \left(\frac{\dir_+}{m} - \frac{\dir_-}{n}\right)t.
\end{equation}
The reverse inequality holds if $\dim(V\cap \vert) = \dir_-$.
\end{lemma}
We recall that $0 < \beta < 1$ is the parameter used in the definition of the $\delta$-dimensional Hausdorff/packing $\beta$-game.
\begin{proof}
Let $(\rr_i)_1^k = (\pp_i,\qq_i)_1^k \in V$ be a maximal set with respect to the following properties: $\|\pp_i\| > \beta \|\qq_i\|$ for all $i$, and $\pp_i$ is perpendicular to $\pp_j$ whenever $i\neq j$. For each $j = 1,\ldots,k$ let
\[
W_j\df V\cap \bigcap_{i=1}^j (\pp_i,\0)^\perp
\] 
and note that for all $\rr = (\pp,\qq)\in W_k$, we have $\|\pp\| \leq \beta\|\qq\|$. Now let $X = \{\qq\in\R^\qdim : \exists \pp\in\R^\pdim \; (\pp,\qq)\in W_k\}$, and let $\bfC:X\to\R^\pdim$ be defined so that $(-\bfC\qq,\qq)\in W_k$. Now, $\bfC$ can be extended to a map from $\R^\qdim$ to $\R^\pdim$ without increasing its operator norm. It follows that $\|\bfC\|\leq \beta$ and $u_\bfC W_k \subset \vert$, which implies that
\[
q - k = \dim(W_k) = \dim(u_\bfC W_k) \leq \dim(u_\bfC V\cap \vert) \leq \dir_-.
\]
Rearranging gives $k \geq \dir_+$. Finally, let $(\rr_i)_{k+1}^q$ be an arbitrary basis of $W_k$. Then there exists a constant $\alpha > 0$ such that
\[
\|V\| = \alpha\|\rr_1\wedge\cdots\wedge\rr_q\|
\]
and
\[
\|g_t V\| = \alpha\|g_t \rr_1 \wedge\cdots\wedge g_t\rr_q\|.
\]
In particular
\[
\|V\| \leq \alpha\|\rr_1\|\cdots\|\rr_k\|\cdot \|\rr_{k+1}\wedge\cdots\wedge\rr_q\| \lesssim_\beta \alpha \|\pp_1\|\cdots\|\pp_k\|\cdot\|\rr_{k+1}\wedge\cdots\wedge\rr_q\|
\]
while
\[
\|g_t V\| = \alpha\|e^{t/m}[(\pp_1,\0) + o_\beta(\| \pp_1 \|)]\wedge\cdots\wedge e^{t/m}[(\pp_k,\0) + o_\beta(\| \pp_k \|)]\wedge g_t \rr_{k+1} \wedge\cdots\wedge g_t\rr_q\|.
\]
Since $((\pp_i,\0))_1^k$ are orthogonal to each other and also to $\rr_{k+1},\ldots,\rr_q$, it follows that if $t$ is sufficiently large in comparison to $\beta$ we have
\begin{align*}
\|g_t V\| &\gtrsim_{\phantom\beta} \alpha e^{kt/m} \|\pp_1\|\cdots\|\pp_k\|\cdot\|g_t \rr_{k+1} \wedge\cdots\wedge g_t\rr_q\|\\
&\gtrsim_{\phantom\beta} \alpha e^{kt/m} \|\pp_1\|\cdots\|\pp_k\|\cdot e^{-(q-k)t/n}\|\rr_{k+1} \wedge\cdots\wedge\rr_q\|\\
&\gtrsim_\beta e^{kt/m-(q-k)t/n}\|V\|.
\end{align*}
Since $k\geq \dir_+$, this completes the proof of \eqref{gtVV}.

Now suppose $\dim(V\cap\vert) = L_-$, and we will show that the reverse inequality of \eqref{gtVV} holds. Let $(\rr_i)_1^{L_-}$ be an orthonormal basis of $V\cap\vert$, and extend to an orthonormal basis $(\rr_i)_1^q$ of $V$. Then by Proposition \ref{propositionexteriorproduct},
\begin{align*}
\frac{\|g_t V\|}{\|V\|} 
&\asymp \|g_t \rr_1 \wedge \cdots \wedge g_t \rr_q\|\\
&\lesssim \|g_t \rr_1\| \cdots \|g_t \rr_q\|\\
&\leq (e^{-t/n} \|\rr_1\|) \cdots (e^{-t/n} \|\rr_{L_-}\|) (e^{t/m} \|\rr_{L_- + 1}\|) \cdots (e^{t/m} \|\rr_q\|)\\
&= \exp\left(\left(\frac{L_+}{m} - \frac{L_-}{n}\right) t \right).
\end{align*}
This completes the proof.
\end{proof}

We finish with an elementary observation about the slopes of line segments appearing in templates (see Definition \ref{definitiontemplate}) that is used in proving Lemma \ref{lemmahperturbation}, which in turn is needed in the proof of the lower bound of the variational principle.

\begin{observation}
\label{observationrationalslopes}
If $\ff$ is a template then for all $t\geq 0$ we have
\[
f_j'(t) - f_i'(t) \in \tfrac1q\Z \;\text{ for some\; $q\leq \dimprod d^2$}.
\]
\end{observation}
\begin{proof}
For all $i,t$ we have
\[
f_i'(t) = \frac{p}{\dimprod q}
\]
for some $p\in\Z$ and $q = 1,\ldots,d$.  So we have
\[
f_j'(t) - f_i'(t) = \frac{p_2}{\dimprod q_2} - \frac{p_1}{\dimprod q_1} = \frac{p}{\dimprod q_1 q_2}
\]
and we have $\dimprod q_1 q_2 \leq \dimprod d^2$.
\end{proof}

\draftnewpage
\section{Proof of Theorem \ref{theoremvariationaluniform}, lower bound}
\label{sectionlower}
Let $\ff$ be a template. We must show that for all $\epsilon > 0$, there exists $C_\epsilon > 0$ such that
\begin{align}
\label{lowerbound}
\HD(\DD(\ff,C_\epsilon)) &\geq \underline\delta(\ff) - \epsilon,&
\PD(\DD(\ff,C_\epsilon)) &\geq \overline\delta(\ff) - \epsilon
\end{align}
where $\DD(\ff,C_\epsilon) \df \DD(\NN(\{\ff\},C_\epsilon))$. To this end, we will play the modified Hausdorff and packing games with target set $S = \DD(\ff,C_\epsilon)$. It turns out that the same strategy will work for Alice in both games.

The proof can be divided into four basic stages:
\begin{enumerate}[1.]
\item {\it Reduction}: We can without loss of generality assume that the template $\ff$ appearing in the statement of the theorem is in a special form which is convenient to the later argument.
\item {\it Mini-strategy}: For any template $\gg$ (not necessarily the same as the $\ff$ appearing in the theorem), Alice can guarantee that if $\bfA$ is the outcome of the game, then the successive minima function $\hh_\bfA$ remains close to $\gg$ for a certain interval of time before diverging from it. This interval can be an interval of linearity of $\gg$, or the union of any fixed number of intervals of linearity. However, the upper bound on $|\hh_\bfA - \gg|$ rapidly grows as the allowed number of intervals of linearity increases.
\item {\it Error correction}: If the value of the successive minima function $\hh_\bfA$ at a certain time $t$ is slightly off from the value of $\ff$ at $t$, then we can perturb $\ff$ into a partial template $\gg$ such that $\gg(t) = \hh_\bfA(t)$. Alice can then follow the perturbed template $\gg$ rather than the original template $\ff$.
\item {\it Uniform error bounds}: The error correction techniques from stage (3) are sufficient to guarantee that the final successive minima function $\hh_\bfA$ remains at a bounded distance from the desired template $\ff$, and that the inequalities $\underline\delta(\AA) \geq \underline\delta(\ff) - \epsilon$ and $\overline\delta(\AA) \geq \overline\delta(\ff) - \epsilon$ are satisfied.
\end{enumerate}

Stage 2 is in some sense the most important one because it makes the connection between the parametric geometry of numbers and the theory of templates. In the other stages, for the most part we do not deal with parametric geometry of numbers directly.

\subsection{Reduction}

There are two key features we would like to assume of our template $\ff$: its corner points\Footnote{I.e. points where the derivative of $\ff$ is undefined.} should be appropriately spaced, and each corner point should have only one ``purpose''.

\begin{definition}
\label{definitionintegral}
Given an $\tspace > 0$, a template $\ff$ is \emph{$\tspace$-integral} if
\begin{itemize}
\item[(I)] its corner points are multiples of $\tspace$, and
\item[(II)] for all $t\in \tspace\N$ we have $f_i(t) \in \frac{\tspace}{mnd!}\Z$ for all $1\leq i \leq d$.
\end{itemize}
By the quantized slope condition (see Definition \ref{definitiontemplate}) it suffices to check (II) for any $t\in \tspace\N$ (e.g. $t=0$) to obtain it for all $t\in \tspace\N$.
\end{definition}

\begin{definition}[Cf. Figure \ref{figuresplitsmergerstransfers}]
\label{definitionsimple}
Let $\ff$ be a template, let $t > 0$ be a corner point of $\ff$, and let $I_-,I_+$ be the two maximal intervals of linearity for $\ff$ such that $I_- = (t_-,t)$ and $I_+ = (t,t_+)$ for some $t_- < t < t_+$.
\begin{itemize}
\item We call $t$ a \emph{split} (resp. \emph{merge}) if there exists $q = 1,\ldots,d-1$ such that $f_q(t) = f_{q+1}(t)$, but $f_q < f_{q+1}$ on $I_+$ (resp. on $I_-$).
\item We call $t$ a \emph{transfer} if there exists $q = 1,\ldots,d-1$ such that $f_q(t) < f_{q+1}(t)$ and $\dir_+(\ff,I_+,q) > \dir_+(\ff,I_-,q)$ (equiv. $F_q'(I_+) > F_q'(I_-)$).
\end{itemize}
Finally, we call the template $\ff$ \emph{simple} if the sets of splits, merges, and transfers are pairwise disjoint.

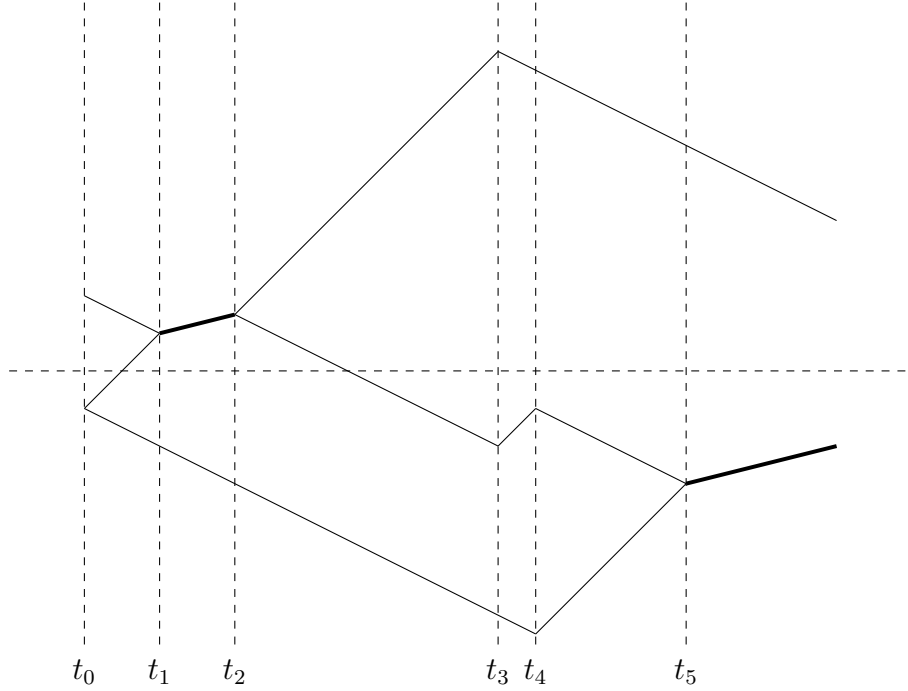
\begin{figure}
\begin{tikzpicture}
\clip(-1,-5) rectangle (11,5);
\draw[dashed] (-1,0) -- (11,0);
\node at (0,-4) {$t_{0}$};
\node at (1,-4) {$t_{1}$};
\node at (2,-4) {$t_{2}$};
\node at (5.5,-4) {$t_{3}$};
\node at (6,-4) {$t_{4}$};
\node at (8,-4) {$t_{5}$};
\draw (0,1)--(1,0.5);
\draw (0,-0.5)--(1,0.5);
\draw (0,-0.5)--(6,-3.5);
\draw[line width=1.5] (1,0.5) -- (2,0.75);
\draw (2,0.75) -- (5.5,4.25);
\draw (5.5,4.25) -- (10,2);
\draw (2,0.75) -- (5.5,-1);
\draw (5.5,-1) -- (6,-0.5);
\draw (6,-0.5) -- (8,-1.5);
\draw (6,-3.5) -- (8,-1.5);
\draw[line width=1.5] (8,-1.5) -- (10,-1);
\draw[dashed] (0,-3.64) -- (0,5);
\draw[dashed] (1,-3.64) -- (1,5);
\draw[dashed] (2,-3.64) -- (2,5);
\draw[dashed] (5.5,-3.64) -- (5.5,5);
\draw[dashed] (6,-3.64) -- (6,5);
\draw[dashed] (8,-3.64) -- (8,5);
\end{tikzpicture}
\caption{In this figure of a portion of an arbitrary $1\times 2$ template, the corner points $t_0$ and $t_2$ are splits, $t_1$ and $t_5$ are merges, and $t_3$ and $t_4$ are transfers.}
\label{figuresplitsmergerstransfers}
\end{figure}
\end{definition}

\begin{remark}
In any template, every corner point is either a split, a merge, or a transfer.
\end{remark}

We now show that we can assume without loss of generality that the template $\ff$ appearing in the statement of Theorem \ref{theoremvariationaluniform} is both simple and integral. Since a similar argument will be needed for the proof of the upper bound of Theorem \ref{theoremvariationaluniform} (specifically, showing that a successive minima function can always be approximated by a template (cf. Lemma \ref{lemmaapproximatetemplate} below)), we prove this lemma in slightly greater generality than may appear to be necessary.

\begin{lemma}
\label{lemmafindtemplate}
Fix $\tspace > 0$, and let $\ff:\Rplus\to \R^d$ be a map (not necessarily a template) satisfying the following conditions:
\begin{itemize}
\item[\text{(I)}] $f_1 \leq \cdots \leq f_d$.
\item[\text{(II)}] For all $t_1 < t_2$ and $i = 1,\ldots,d$ we have
\[
-\frac{1}{\qdim} \leq \frac{f_i(t_2) - f_i(t_1)}{t_2 - t_1} \leq \frac{1}{\pdim}\cdot
\]
\item[\text{(III)}] For all $q = 1,\ldots,d$ and for every interval $I$ such that
\begin{equation}
\label{qsplit}
f_{q+1}> f_q \text{ on } I,
\end{equation}
there exists a convex, piecewise linear function $F_{q,I}:I\to\R$ with slopes in $Z(q)$ (cf. \eqref{slopeset}) which satisfies
\begin{equation}
\label{FqI}
F_q \df \sum_{i=1}^q f_i \asymp_\plus F_{q,I} \text{ on } I
\end{equation}
and if $\ff$ is a template, then
\begin{equation}
\label{FqIgeq}
F_{q,I}' \geq F_q' \text{ on } I.
\end{equation}
\end{itemize}
Then there exists a simple $\tspace$-integral template $\gg$ which approximates $\ff$ to within an additive constant, i.e. satisfies $\gg \asymp_\plus \ff$. The implied constant depends on $\tspace$ and on the implied constant of \eqref{FqI} but not directly on $\ff$. Moreover, if $\ff$ is a template, then $\gg$ can be chosen so that for all $q,t,t'$ such that $g_q(t) < g_{q+1}(t)$ and $|t'-t|\leq \tspace$, we have $f_{q+1}(t') - f_q(t') \geq \tspace$ and $G_q'(t) \geq F_q'(t')$, and consequently
\begin{align}
\label{reductionimprovement}
\underline\delta(\gg) &\geq \underline\delta(\ff),&
\overline\delta(\gg) &\geq \overline\delta(\ff).
\end{align}
\end{lemma}
\begin{remark*}
Any template $\ff$ satisfies conditions (I)-(III) (and in fact, one can take $F_{q,I} = F_q$ in (III)).
\end{remark*}
\begin{remark*}
In the proof below, all implied constants are assumed to depend on $\tspace$. 
\end{remark*}
\begin{proof}
Let \label{definitiontstar}$\tstar \df d\cdot(d^2)!\tspace \in \tspace\N$, and let \label{definitiontbigspace}$\tbigspace \df \dimprod d^{4d}\tstar$. Fix $q = 1,\ldots,d$, and let $\II_q$ be the collection of all intervals satisfying \eqref{qsplit} whose endpoints are in $\tbigspace\N\cup\{\infty\}$, and which are maximal with respect to these two properties. Note that this implies that  $\II_q$ is a disjoint collection if intervals. For each $I\in \II_q$, let $F_{q,I}:I\to\R$ be a convex, piecewise linear function as in (III).

By first moving the corner points of $F_{q,I}$ to the left and then increasing $F_{q,I}$ by an additive constant, we may without loss of generality suppose that the following hold:
\begin{itemize}
\item[(IV)] the corner points of $F_{q,I}$ are all integer multiples of $\tbigspace$; and
\item[(V)] the values of $F_{q,I}$ at integer multiples of $\tbigspace$ are all in the set $(d^{4(d-q)} + d^{4d}\Z)\tstar$. (The displacement term $d^{4(d-q)}$ will help us guarantee that the resulting template $\gg$ is simple.)
\end{itemize}
Here, we have used the fact that $Z(q) \subset \frac1{mn} \Z$ for all $q$, to ensure that the conditions are not inconsistent. Moving the corner points to the left rather than to the right guarantees that \eqref{FqIgeq} is still satisfied, via the convexity condition. We can also assume that $F_{d,I_*} \equiv \tstar$, where $I_* = \Rplus$ is the unique element of $\II_d$.

\begin{claim}
\label{claimexplicitlipschitz}
There exist collections of disjoint intervals $\w\II_q$ ($q=1,\ldots,d$) satisfying 
\begin{equation}\label{fqasymp}
f_q \asymp_\plus f_{q+1} ~\text{on}~ \Rplus \butnot \bigcup(\w\II_q),
\end{equation}
and functions $\w F_{q,I} : I \to \R$ ($q=1,\ldots,d$, $I \in \w \II_q$) satisfying (the analogues of) \text{(III)-(V)} as well as the following:
\begin{itemize}
\item[(VI)] for all $1 \leq q_1 < q_2 \leq d$, $I_1\in \w\II_{q_1}$, and $I_2\in \w\II_{q_2}$, we have
\begin{equation}
\label{explicitlipschitz}
-\frac1\qdim \leq \frac{\w F_{q_2,I_2}' - \w F_{q_1,I_1}'}{q_2 - q_1} \leq \frac1\pdim \text{ on } I_1 \cap I_2.
\end{equation}
\end{itemize}
\end{claim}
\begin{subproof}
Fix a constant $C_2 > 0$ to be determined, and let $C_1 = 2d^2 C_2$. Let $\JJ_q = \{I \in \II_q : |I| > C_1\}$ and $S = \bigcup_q \bigcup_{I\in \JJ_q} S_{q,I}$, where $S_{q,I}$ is the union of the set of corner points of $F_{q,I}$ and the set of endpoints of $I$. Consider the equivalence relation $\sim$ on $S$ where $a\sim b$ means that one can reach $b$ starting from $a$ via a series of ``jumps'' of size $\leq C_2$, while remaining in $S$.

We claim that each equivalence class for $\sim$ has cardinality at most $2d^2$. Indeed, otherwise there exist points $t_0 < \ldots < t_{2d^2}$ in $S$ such that $t_{i+1} - t_i \leq C_2$ for all $i$. By the pigeonhole principle there exists $q = 1,\ldots,d$ such that $\#(\bigcup_{I\in \JJ_q} S_{q,I}\cap [t_0,t_{2d^2}]) \geq 2d+1$. Since $\#(S_{q,I}) \leq \#(Z(q)) + 1 = \min(q,d-q) + 1 \leq d$ for all $I\in \II_q$, applying the pigeonhole principle again shows that there exist at least 3 intervals $I\in \JJ_q$ such that $I \cap [t_0,t_{2d^2}] \neq \emptyset$, and since $\JJ_q \subset \II_q$ these intervals must be disjoint. But then the middle interval is a subset of $[t_0,t_{2d^2}]$, which contradicts $I\in \JJ_q$, since $t_{2d^2} - t_0 \leq 2d^2 C_2 = C_1$.

Now let $\sigma:S\to S$ be the map which sends each equivalence class under $\sim$ to its smallest element, and note that $|\sigma(a) - a| \leq C_1$ for all $a\in S$. Write $\sigma((a,b)) = (\sigma(a),\sigma(b))$ for all $a,b\in S$, and
for each $q = 1,\ldots,d$ let $\w\II_q = \{\sigma(I) : I \in \JJ_q\}$.
For each $q = 1,\ldots,d$ and $I\in \JJ_q$, let $\w F_{q,I}$ be a piecewise linear function which is equal to $F_{q,I}$ at the left endpoint of $I$, such that if $(a,b) \subset I$ is a maximal interval of linearity for $F_{q,I}$ of slope $z$, then $(\sigma(a),\sigma(b))$ is (if nonempty) a maximal interval of linearity for $\w F_{q,I}$ of slope $z$.

Now if $t\in \Rplus \butnot \bigcup(\w\II_q)$, then either $t\in \sigma(I)$ for some interval $I = (a,b)$ (with $a,b\in S$) disjoint from $\bigcup(\JJ_q)$, or $t = \sigma(a)$ for some $a\in S\butnot\bigcup(\JJ_q)$. Thus $|t - a| \leq C_1$ for some $a\notin \bigcup(\JJ_q)$, and thus $|t - b| \leq 2C_1$ for some $b\notin \bigcup(\II_q)$, and thus by the definition of $\II_q$ we have $f_q(t) \asymp_{\plus,C_1} f_q(a) \asymp_\plus f_{q+1}(a) \asymp_{\plus,C_1} f_{q+1}(t)$, i.e. \eqref{fqasymp} holds. Moreover, by letting $(a_i,a_{i+1})$ be the maximal intervals of linearity of $F_{q,I}$ and inducting on $i$, we get that $\w F_{q,\sigma(I)} \asymp_{\plus,C_1} F_{q,I}$ for all $q,I$. Thus, \eqref{FqI} holds with $F$ replaced by $\w F$. Moreover, since $\sigma(S) \subset S \subset \tbigspace \Z$, condition (IV) holds for $\w F$. Also, by translating each $\w F_{q,I}$ by a constant if necessary, we can without loss of generality assume that condition (V) holds for $\w F$.

If $\ff$ is a template, then to demonstrate \eqref{FqIgeq}, suppose $\sigma((a,b))\in \w\II_q$; by the convexity condition, $F_q'$ is increasing on $\bigcup(\II_q)$, and so since $\sigma(b) \leq b$, on $\sigma((a,b))$ we have $F_q' \leq F_q'\given (a,b) = \w F_{q,\sigma((a,b))}'$.

Finally, we need to show that \eqref{explicitlipschitz} holds for $\w F$.  Indeed, let $(\sigma(a),\sigma(b)) \subset I_1\cap I_2$ be a maximal interval of linearity for $\w F_{q,I_2} - \w F_{q,I_1}$. Since $\sigma(a) < \sigma(b)$, we have $b - a \geq C_2$. Let $k = q_2 - q_1$. Then by condition (II) we have
\[
-\frac k\qdim \leq \sum_{i=q_1+1}^{q_2} \frac{f_i(b) - f_i(a)}{b - a} \leq \frac k\pdim
\]
and thus if $z$ is the constant value of $F_{q_2,I_2}' - F_{q_1,I_1}'$ on $(a,b)$, then
\[
-\frac k\qdim - \frac{4C_3}{C_2} \leq z = \frac{F_{q_2,I_2}(b) - F_{q_2,I_2}(a) - F_{q_1,I_1}(b) + F_{q_1,I_1}(a)}{b - a} \leq \frac k\pdim + \frac{4C_3}{C_2}
\]
where $C_3$ is the implied constant of \eqref{FqI}. On the other hand, we have $z \in \frac1\dimprod\Z$ by condition (III). So if we choose $C_2 > 4\dimprod C_3$, then we get $-k/\qdim \leq z \leq k/\pdim$, completing the proof of the claim.
\end{subproof}

Next, for each $q = 1,\ldots,d$ let $F_{q,\ast}:\Rplus\to \R\cup \{\ast\}$ be defined by the formula
\[
F_{q,\ast}(t) = \begin{cases}
\w F_{q,I}(t) & \text{ if } t\in I \text{ for some } I\in \w\II_q\\
\ast & \text{otherwise},
\end{cases}
\]
and let 
\begin{equation}
\label{Gqstar}
G_{q,\ast}(t) = F_{q,\ast}(t) + q(d-q)C_4, 
\end{equation}
where $C_4\in d^{4d}\tstar\N$ is large to be determined.
Let $G_{0,\ast}(t) = 0  \in (d^{4(d-0)} + d^{4d}\Z)\tstar$ for all $t$.
Then since we assumed that $F_{d,I_*} \equiv \tstar$, where $I_* = \Rplus$, it follows that $G_{d,\ast}(t) = \tstar$ for all $t \in I_*$. Here and henceforth we let $\ast + x = \ast + \ast = \ast$ for any $x \in \R$.

At this point, the intuitive idea is to try to define the template $\gg$ by solving the equations
\begin{equation}
\label{intuitivegdef}
G_{q,\ast}(t) = \begin{cases}
\sum_{i = 1}^q g_i(t) & \text{ if } g_q(t) < g_{q+1}(t)\\
\ast & \text{ if } g_q(t) = g_{q+1}(t).
\end{cases}
\end{equation}
However, the formula \eqref{intuitivegdef} is not necessarily solvable with respect to $\gg$, due to the fact that the natural candidate for a solution does not necessarily satisfy $g_1(t) \leq \cdots \leq g_d(t)$. To address this issue, we introduce the concept of the convex hull function of a set:

\begin{definition}
\label{definitionconvexhull}
The \emph{convex hull function} of a set $\Gamma\subset \R^2$ is the largest convex function $h:I\to\R$ such that $h(x) \leq y$ for all $(x,y)\in \Gamma$, where $I$ is the smallest interval containing the projection of $\Gamma$ onto the first coordinate.
\end{definition}

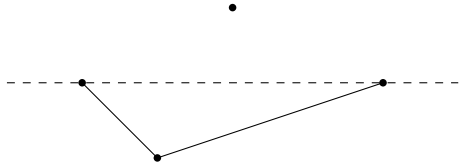
\begin{figure}[h!]
\begin{tikzpicture}
\clip(-1,-1.3) rectangle (5,1.3);
\draw[dashed] (-1,0) -- (5,0);
\draw (0,0) -- (1,-1);
\draw (1,-1) -- (4,0);
\fill (0,0) circle (0.05);
\fill (1,-1) circle (0.05);
\fill (2,1) circle (0.05);
\fill (4,0) circle (0.05);
\end{tikzpicture}
\caption{The convex hull function $h$ of the set $\{(0,0),(1,-1),(2,1),(4,0)\}$. Since $1 > h(2) = -2/3$, the convex hull function does not change when the point $(2,1)$ is removed.}
\end{figure}

We can now define $\gg$ via the formula
\[
g_q(t) = h_t(q) - h_t(q-1),
\]
where $h_t:[0,d]\to\R$ is the convex hull function of the set
\[
\Gamma(t) = \{(q,G_{q,\ast}(t)) : q = 0,\ldots,d,\; G_{q,\ast}(t) \neq \ast\}.
\]
To complete the proof of Lemma \ref{lemmafindtemplate}, we must show
\begin{enumerate}[\bf (A)]
\item that $\gg = (g_1,\ldots,g_d)$ is a simple $\tspace$-integral template,
\item that $\gg \asymp_\plus \ff$,
\item that if $\ff$ is a template, then for all $q,t,t'$ such that $g_q(t) < g_{q+1}(t)$ and $|t'-t|\leq \tspace$, we have $f_{q+1}(t') - f_q(t') \geq \tspace$ and $G_q'(t) \geq F_q'(t')$, and consequently \eqref{reductionimprovement} holds.
\end{enumerate}


{\bf Proof of (A)}: We start with showing that $\gg$ is continuous. From this it is easy to see that it is piecewise linear, the first step to proving that it is a template. Fix $t > 0$, and write $h(t^\pm) = \lim_{s\to t^\pm} h(s)$. (The limit exists since $h$ is linear on intervals of the form $(a,s)$ and $(s,b)$.sa) Let
\[
\Gamma(t^\pm) = \lim_{s\to t^\pm} \Gamma(s) = \{(q,G_{q,\ast}(t^\pm)) : q = 0,\ldots,d, \; G_{q,\ast}(t^\pm) \neq \ast\}.
\]
We need to show that $\Gamma(t^-)$ and $\Gamma(t^+)$ have the same convex hull function. For this purpose, it suffices to show that any point in one of these sets but not the other is not an element of the graph of the corresponding convex hull function (which implies that the convex hull function does not change when the point is removed).

Indeed, fix $q = 1,\ldots,d-1$ and suppose that $G_{q,\ast}(t^+) \neq \ast$ but $G_{q,\ast}(t^-) = \ast$. Let $0\leq p < q$ and $d\geq r > q$ be maximal and minimal, respectively, such that $G_{p,\ast}(t^+),G_{r,\ast}(t^+) \neq \ast$. Then by \eqref{fqasymp} we have
\[
f_{p+1}(t) \asymp_\plus \ldots \asymp_\plus f_q(t) \asymp_\plus f_{q+1}(t) \asymp_\plus \ldots \asymp_\plus f_r(t)
\]
and thus by \eqref{FqI},
\[
\frac{F_{q,\ast}(t^+) - F_{p,\ast}(t^+)}{q - p} \asymp_\plus f_q(t) \asymp_\plus f_{q+1}(t) \asymp_\plus \frac{F_{r,\ast}(t^+) - F_{q,\ast}(t^+)}{r - q}\cdot
\]
Then by \eqref{Gqstar} it follows that
\begin{align*}
\frac{G_{r,\ast}(t^+) - G_{q,\ast}(t^+)}{r - q} - \frac{G_{q,\ast}(t^+) - G_{p,\ast}(t^+)}{q - p}
&\asymp_\plus \left[\frac{r(d-r) - q(d-q)}{r-q} - \frac{q(d-q) - p(d-p)}{q-p}\right] C_4\\
&=_\pt -(r-p)C_4 \leq -2C_4.
\end{align*}
So if $C_4$ is sufficiently large, then
\[
\frac{G_{r,\ast}(t^+) - G_{q,\ast}(t^+)}{r - q} < \frac{G_{q,\ast}(t^+) - G_{p,\ast}(t^+)}{q - p}
\]
i.e. the slope of the line from $(q,G_{q,\ast}(t^+))$ to $(r,G_{r,\ast}(t^+))$ is less than the slope of the line from $(p,G_{p,\ast}(t^+))$ to $(q,G_{q,\ast}(t^+))$. It follows that $(q,G_{q,\ast}(t^+))$ lies above the graph of the convex hull function of $\Gamma(t^+)$. Since $q$ was arbitrary, this shows that $\Gamma(t^-)$ and $\Gamma(t^+)$ have the same convex hull function. Thus $\gg(t^-) = \gg(t^+)$, and $\gg$ is continuous at $t$.

We next demonstrate that $\gg$ satisfies conditions (I)-(III) of Definition \ref{definitiontemplate}. (I) follows from the fact that convex hull functions are convex, while (II) follows from \eqref{explicitlipschitz} in Claim \ref{claimexplicitlipschitz}. To demonstrate (III), fix $q = 1,\ldots,d$ and let $I$ be an interval of linearity for $\gg$ such that $g_q < g_{q+1}$ on $I$. Fix $t\in I$. Since $h_t(q) - h_t(q-1) < h_t(q+1) - h_t(q)$ (with the convention $h_t(d+1) = +\infty$), the point $(q,h_t(q))$ is an extreme point of the convex hull of $\Gamma(t)$ and thus $(q,h_t(q)) \in \Gamma(t)$, i.e. $G_{q,\ast}(t) = h_t(q)$. It follows that
\begin{equation}
\label{GqastGq}
\sum_{i = 1}^q g_i = G_{q,\ast} \text{ on } I.
\end{equation}
Since $G_{q,\ast}\given I$ is convex and piecewise linear with slopes in $Z(q)$, it follows that the same is true for $\sum_1^q g_i \given I$. Thus, $\gg$ is a template.

To show that $\gg$ is simple and $\tspace$-integral, we first observe that by condition (IV) all transfers occur at integer multiples of $\tbigspace$. Let $t$ be a split or a merge with corresponding index $q$. Then $(q,G_{q,\ast}(s))$ is an extreme point of the convex hull of $\Gamma(s)$ when $s$ approaches $t$ from one side, but not from the other side. So there exist $0\leq p < q < r\leq d$ such that the point $(q,G_{q,\ast}(t))$ lies on the line segment connecting $(p,G_{p,\ast}(t))$ and $(r,G_{r,\ast}(t))$. Thus, we have $\Phi(t) = 0$ where
\[
\Phi(s) = (r-q)G_{p,\ast}(s) + (q-p)G_{r,\ast}(s) - (r-p)G_{q,\ast}(s).
\]
Write $t = t' + t''$ where $t'$ is a multiple of $\tbigspace$ and $0 \leq t'' < \tbigspace$. Then by assumption $G_j(t') \in (d^{4(d-j)} + d^{4d}\Z)\tstar$ for all $j$. Thus $\frac1\tstar \Phi(t') \in \Z$, and furthermore
\begin{align*}
\tfrac{1}{\tstar} \Phi(t')
&\equiv (r-q)d^{4(d-p)} + (q-p)d^{4(d-r)} - (r-p)d^{4(d-q)}\\ 
&\equiv (q-p) d^{4(d-r)}\\ 
& \not\equiv 0 \note{modulo $d^{4(d-q)}$}.
\end{align*}
In particular $\Phi(t') \neq 0 = \Phi(t)$, so $t'' > 0$ and thus $t$ is not a transfer. Thus, the set of splits and the set of merges are both disjoint from the set of transfers.

Since $G_{p,\ast},G_{q,\ast},G_{r,\ast}$ are linear on $[t',t'+\tbigspace]$, so is $\Phi$. Let $z$ denote the constant value of $\Phi'$ on $[t',t'+\tbigspace]$, and note that
\begin{align*}
0\neq z &= \left(\tfrac1\pdim+\tfrac1\qdim\right) \Big[ (r-q)\dir_+(p) + (q-p)\dir_+(r) - (r-p)\dir_+(q)\Big]\\
&\in \left(\tfrac1\pdim+\tfrac1\qdim\right)\{-(r-p)q,\ldots,(r-q)p+(q-p)r\} \subset \left(\tfrac1\pdim+\tfrac1\qdim\right) \{-d^2,\ldots,d^2\}.
\end{align*}
Thus
\begin{align*}
t'' &= -\frac{\Phi(t')}{z} \in \frac{\tstar \Z}{\big(\tfrac1\pdim + \tfrac1\qdim\big)(d^2)!} \subset \frac{\tstar \Z}{d\cdot(d^2)!},
\end{align*}
so since $\tstar = d\cdot(d^2)!\tspace$ we have $t'' \in \Z\tspace$. Since transfers also occur at integer multiples of $\tspace$, this implies that condition (I) of Definition \ref{definitionintegral} is satisfied. To check condition (II), note that we have $G_{q,\ast}(t) \in \tstar\Z$ whenever $t\in \tbigspace\N$, and thus since $G_{q,\ast}$ has slopes in $Z(q) \subset \frac1{mn}\Z$, we have $G_{q,\ast}(t) \in \frac{\tspace}{mn} \Z$ whenever $t\in \tspace\N$. Now for each $q = 1,\ldots,d$ and $t\in \tspace\N$, there exist $p < q \leq r$ such that
\[
g_q(t) = \frac{G_{r,\ast}(t) - G_{p,\ast}(t)}{r - p} \in \frac{\tspace}{mnd!} \Z.
\]
Thus $\gg$ is $\tspace$-integral.

Next, since $t'' > 0$, it follows that $\Phi$ is linear in a neighborhood of $t$, and thus there exist points near $t$ for which $\Phi$ is strictly negative. At these points, we have $g_q = g_{q+1}$. It follows that $t$ is not both a split and a merge with respect to the same index $q$.

By contradiction, suppose that $t$ is both a split and a merge, with corresponding indices $q_1 \neq q_2$. We can apply the above argument twice: for each $i = 1,2$ we get indices $0 \leq p_i < q_i < r_i \leq d$, a function $\Phi_i$, and a slope $z_i$. We have
\[
-\frac{\Phi_1(t')}{z_1} = t'' = -\frac{\Phi_2(t')}{z_2}
\]
and thus
\[
\frac{\Phi_1(t')}{\tstar} \cdot \frac{z_2}{\tfrac1\pdim + \tfrac1\qdim} = \frac{\Phi_2(t')}{\tstar} \cdot \frac{z_1}{\tfrac1\pdim + \tfrac1\qdim}\cdot
\]
So there exist $a_1,a_2\in \{-d^2,\ldots,d^2\}\butnot\{0\}$ such that
\begin{align*}
&a_1[(r_1-q_1)d^{4(d-p_1)} + (q_1-p_1)d^{4(d-r_1)} - (r_1-p_1)d^{4(d-q_1)}]\\\equiv\; &a_2[(r_2-q_2)d^{4(d-p_2)} + (q_2-p_2)d^{4(d-r_2)} - (r_2-p_2)d^{4(d-q_2)}] \note{modulo $d^{4d}$}.
\end{align*}
Comparing the base $d^4$ expansions of both sides shows that $(p_1,q_1,r_1) = (p_2,q_2,r_2)$, contradicting that $q_1 \neq q_2$. Thus, the set of splits and the set of merges are disjoint. This completes the proof of (A), viz. that $\gg = (g_1,\ldots,g_d)$ is a simple $\tspace$-integral template.

{\bf Proof of (B)}: We next show that $\gg \asymp_\plus \ff$. Indeed, fix $t\geq 0$. Let $h_1$ and $h_2$ be the convex hull functions of $\Gamma(t)$ and $\{(q,F_q(t)) : G_{q,\ast}(t) \neq \ast\}$, respectively. Since $G_{q,\ast}(t) \asymp_\plus F_q(t)$ for all $q$ such that $G_{q,\ast}(t) \neq \ast$, we have $h_1 \asymp_\plus h_2$. Since $f_1(t) \leq \cdots \leq f_d(t)$, the map $q \mapsto F_q(t)$ is convex and thus $h_2(q) \geq F_q(t)$. On the other hand, by \eqref{fqasymp}, the map $q\mapsto F_q(t)$ is approximately linear on segments $[p,r]$ where $q\in (p,r)$ implies $G_{q,\ast}(t) = \ast$, and thus $h_2(q) \lesssim_\plus F_q(t)$. Combining, we get $h_1(q) \asymp_\plus F_q(t)$. But then $g_q(t) = h_1(q) - h_1(q-1) \asymp_\plus F_q(t) - F_{q-1}(t) = f_q(t)$ for all $q$, i.e. $\gg(t) \asymp_\plus \ff(t)$, and we are done with the proof of (B).

{\bf Proof of (C)}: Next, suppose that $\ff$ is a template, and fix $q,t,t'$ such that $g_q(t) < g_{q+1}(t)$ and $|t' - t| \leq \tspace$. We will show that $f_{q+1}(t') - f_q(t') \geq \tspace$ and $G_q'(t) \geq F_q'(t')$. Indeed, let $p < q < r$ be maximal and minimal, respectively, such that $G_{p,\ast}(t),G_{r,\ast}(t) \neq \ast$. Then by \eqref{GqastGq} and the definitions of $G_{q,\ast},F_{q,\ast}$ we have
\[
G_q(t) = G_{q,\ast}(t) = F_{q,\ast}(t) + q(d-q) C_4 = F_q(t) + q(d-q) C_4,
\]
and similarly for $G_p(t)$, $G_r(t)$. Then
\begin{align*}
f_{q+1}(t) - f_q(t) &\asymp_\plus \frac1{r-q}(F_r(t) - F_q(t)) - \frac1{q-p}(F_q(t) - F_p(t)) \by{\eqref{fqasymp}}\\
&=_\pt \frac1{r-q}(G_r(t) - r(d-r) C_4 - G_q(t) + q(d-q) C_4) \noreason\\
&\;\;\;\; - \frac1{q-p}(G_q(t) - q(d-q) C_4 - G_p(t) + p(d-p) C_4)\noreason\\
&\geq_\pt (r+q-d)C_4 - (q+p-d)C_4 \by{convexity of $q\mapsto G_q(t)$}\\
&=_\pt (r-p)C_4 \geq 2C_4.
\end{align*}
It follows that $f_{q+1}(t') - f_q(t') \geq 2 C_4 - 2\tspace$. Choosing $C_4 \geq 2\tspace$, we get $f_{q+1}(t') - f_q(t') \geq  \tspace$. On the other hand, since $F_{q,\ast}' = G_{q,\ast}' = G_q'$ near $t$,  by \eqref{FqIgeq} we have $G_q'(t) \geq F_q'(t)$.

Finally, to demonstrate \eqref{reductionimprovement}, let $I$ be an interval on which both $\ff$ and $\gg$ are linear. For all $q$ such that $g_q < g_{q+1}$ on $I$, the previous argument gives $G_q' \geq F_q'$ on $I$, and thus $\dir_+(\gg,I,q) \geq \dir_+(\ff,I,q)$ (the right-hand side being well-defined since $f_q < f_{q+1}$ on $I$). It follows that
\begin{equation}
\label{S+gf}
\#\big(S_+(\gg,I)\cap \OC 0q_\Z\big) \geq \#\big(S_+(\ff,I)\cap \OC 0q_\Z\big)
\end{equation}
for all $q$ such that $g_q < g_{q+1}$ on $I$ (cf. Definition \ref{definitiondimtemplate}). Combining with \eqref{Splusdef1} shows that \eqref{S+gf} holds for all $q = 1,\ldots,d$, and thus since
\[
\delta(\ff,I) = \sum_{q = 1}^{d-1} \#\big(S_+(\ff,I)\cap \OC 0q_\Z\big) - \binom m2,
\]
we have $\delta(\gg,I) \geq \delta(\ff,I)$. Since $I$ was arbitrary, we get \eqref{reductionimprovement}. This concludes the proof of (C).

Having proved (A), (B), and (C), we have completed the proof of Lemma \ref{lemmafindtemplate}.
\end{proof}

\begin{lemma}
\label{lemmaapproximatetemplate}
If $\Lambda$ is a unimodular lattice in $\R^d$, then the successive minima function $\hh = (h_1,\ldots,h_d)$, where
\[
h_i(t) = \log\lambda_i(g_t \Lambda),
\]
satisfies conditions $\text{(I)-(III)}_{\ff=\hh}$ of Lemma \ref{lemmafindtemplate}, meaning that it can be approximated by a template.
\end{lemma}
\begin{proof}
Condition (I) is immediate from the definition, while condition (II) follows from some simple calculations which we leave to the reader. To demonstrate property (III), fix $j = 1,\ldots,d-1$ and an interval $[T_1,T_2]$ such that $h_{j+1}(t) > h_j(t)$ for all $t\in [T_1,T_2]$. For each $t\in [T_1,T_2]$ let\footnote{I.e. $V_j(t)$ is the smallest subspace containing $\{ \rr\in \Lambda : \|g_t \rr\| \leq e^{h_j(t)}\}$. See Convention \ref{spanconvention}.}
\[
V_j(t) = \lb \rr\in \Lambda : \|g_t \rr\| \leq e^{h_j(t)} \rb = \lb \rr\in \Lambda : \|g_t \rr\| < e^{h_{j + 1}(t)} \rb.
\]
The assumption on $[T_1,T_2]$ guarantees that the map $t\mapsto V_j(t)$ is continuous on this interval, and since this map takes only rational values, it is therefore constant. So $V_j(t)$ is independent of $t$. By Minkowski's second theorem (Theorem \ref{mink2}), for all $t\in [T_1,T_2]$ we have
\[
\prod_{i = 1}^j \lambda_i(g_t \Lambda) \asymp \|g_t V_j(t)\| = \| g_t V_j \|,
\]
where we use $\| \cdot \|$ to denote covolume, see Notation \ref{Lambdarational}.
To continue further, we use the exterior product formula for covolume (Proposition \ref{propositionexteriorproduct}):
\[
\|g_t V_j\| = \|g_t v_1\wedge\cdots\wedge g_t v_j\|
\]
where $v_1,\ldots,v_j$ is a basis of $V_j\cap \Lambda$. The expression on the right-hand side is a member of the space $\bigwedge^j \R^d$, which has a basis of the form $\{e_S : S\subset\{1\ldots,d\},\#(S) = j\}$. Thus,
\[
\|g_t V_j\| \asymp \max_{\#(S) = j} \lb g_t v_1\wedge\cdots\wedge g_t v_j,e_S\rb = \max_{\#(S) = j} \| \pi_S g_t V_j \|,
\]
where $\pi_S$ denotes the coordinate projection from $\R^d$ to $\R^S$. The logarithm of the right-hand side is the maximum of linear maps whose slopes are in the set $Z(j)$. Thus, the function
\[
F_{j,[T_1,T_2]}(t) \df \max_{\#(S) = j} \log \| \pi_S g_t V_j\|
\]
satisfies the appropriate conditions, cf. \eqref{FqI}.
\end{proof}

\subsection{Mini-strategy}
\label{subsectionministrategy}

Suppose that Alice and Bob have played the first $k$ turns of the modified Hausdorff/packing game (from Section \ref{sectiongamesgeometry}), and that Alice wants to play so as to guarantee that the successive minima function of the outcome will be close to a given template $\gg$ for some short period of time starting at $k\tbeta$. Recall from Notation \ref{gamma} that $\tbeta = -\alpha\log(\beta)$, where $\alpha = \frac{\dimprod}{\dimsum}$ and $0<\beta<1$ is the parameter in the modified Hausdorff/packing game. 
Whether or not she can do this depends both on the template $\gg$ and on the lattice $\Lambda_k$ given by \eqref{Lambdakdef}. Intuitively, we expect that she can do it if $\hh(\Lambda_k)$ is close to $\gg(k\tbeta)$, and $\Lambda_k$ is ``positioned in a way so as to allow Alice to continue this correspondence for larger values of $k$''. If the lattice $\Lambda_k$ is positioned appropriately, we will call it a \emph{$C$-match} for $\gg$ at time $k\tbeta$. We give the formal definition as follows:

\begin{definition}
\label{definitionCmatch}
Let $\gg$ be a $\tbeta$-integral partial template, and fix $C > 0$. A lattice $\Lambda\subset\R^d$ is a \emph{$C$-match} for $\gg$ at time $t\in \tbeta \N$ if
\begin{itemize}
\item[(I)] We have
\begin{equation}
\label{Cmatch}
\|\Mink(\Lambda) - \gg(t)\| < C.
\end{equation}
\item[(II)] There is a family of nested $\Lambda$-rational subspaces $(V_q)_{q\in Q(t)}$, where 
\[
Q(t) \df \{q:g_q(t) < g_{q+1}(t)\},
\] 
such that for all $q\in Q(t)$, we have $\dim(V_q) = q$,
\begin{equation}
\label{lambdaVq}
\Big| \log\lambda_i(\Lambda \cap V_q) - h_i(\Lambda) \Big| \leq C \text{ for all } 1 \leq i \leq q,
\end{equation}
and
\begin{equation}
\label{VqL}
\dim(V_q\cap \vert) \geq \dir_-(\gg,I,q),
\end{equation}
where $I$ is an interval of linearity for $\gg$ whose left endpoint is $t$.
\end{itemize}
\end{definition}

Fix $C_1>0$.\Footnote{Note that the constants $C_1, C_2$ appearing in this section and the next two are independent of those with the same names in the previous subsection (in the proof of Lemma \ref{lemmafindtemplate}).} We now show that if $\Lambda_{k_1}$ is a $C_1$-match for $\gg$ at time $t_1 = k_1 \tbeta$, then it is possible for Alice to follow $\gg$ for any fixed number of intervals of linearity to within an additive constant depending on $C_1$:
\begin{lemma}
\label{lemmaministrategy}
Fix $k_1,k_2\in \N$ with $k_2 > k_1$ and let $t_i = k_i \tbeta$. Let $\gg:\CO{t_1}\infty\to\R^d$ be a $\tbeta$-integral partial template, and let $\niol$ be the number of maximal intervals of linearity of the function $\gg\given (t_1,t_2)$. Suppose that on the $k_1$th turn of the modified Hausdorff/packing game, $\Lambda_{k_1}$ is a $C_1$-match for $\gg$ at time $t_1$. Then Alice has a strategy for turns $k_1,\ldots,k_2-1$ of the modified Hausdorff/packing game guaranteeing the following:
\begin{itemize}
\item[(i)] For all $k \in [k_1,k_2]_\Z$,
\begin{equation}
\label{C2def}
\Mink(\Lambda_k) \asymp_{\plus,C_1,N,\beta} \gg(k\tbeta).
\end{equation}
\item[(ii)] The final lattice $\Lambda_{k_2}$ is a $C_2$-match for $\gg$ at time $t_2$, where $C_2$ is a constant depending only on $C_1$, $\niol$ and $\beta$.
\item[(iii)] We have
\[
\Delta(\AA,[k_1,k_2]) \df \frac{1}{k_2 - k_1} \sum_{k = k_1}^{k_2 -1} \frac{\log\#(A_k)}{-\log(\beta)} = \delta(\gg,[t_1,t_2]) + O\left(\tfrac1{\tbeta} + \tfrac1{k_2 - k_1}\right),
\]
where the implied constant may depend on $C_1$ and $\niol$ but does not depend on $\beta$.
\end{itemize}
\end{lemma}
\begin{proof}
By induction, it suffices to prove the lemma in the case where $\niol = 1$, i.e. where $\gg$ is linear on $I = (t_1,t_2)$.

Let $\Lambda \df \Lambda_{k_1}$, and let
\begin{align*}
Q' &\df \{q : g_q(t_1) < g_{q+1}(t_1)\},&
Q &\df \{q : g_q < g_{q+1} \text{ on } I\}.
\end{align*}
Note that using the notation from Definition \ref{definitionCmatch}, we have
\[
Q' = Q(t_1) \subset Q(t_1) \cup Q(t_2) = Q.
\]
In the sequel, for each $q \in Q'$, let $V_q$ be as in Definition \ref{definitionCmatch}, as guaranteed by the fact that $\Lambda$ is a $C_1$-match for $\gg$ at time $t_1$.

\begin{claim}
\label{claimextendflag}
If $\beta$ is sufficiently small, then there exists a family of $\Lambda$-rational subspaces $(V_q)_{q\in Q}$ extending $(V_q)_{q\in Q'}$ with the following properties:
\begin{itemize}
\item[(i)] $\dim(V_q) = q$ for all $q\in Q$.
\item[(ii)] $V_p \subset V_q$ for all $p,q\in Q$ such that $p < q$.
\item[(iii)] $\log\|V_q\| \asymp_\plus \sum_1^q g_i(t_1)$ for all $q\in Q$, where the implied constant may depend on $C_1$.
\item[(iv)] There exists $\bfB\in B_\MM(\0,1 - \beta)$ such that for all $q\in Q$,
\[
\dim(u_\bfB V_q\cap \vert) = \dir_-(q) \df \dir_-(\gg,I,q)
\]
and
\[
\dim(u_{\bfB + \bfC} V_q\cap \vert) \leq \dir_-(q) \text{ for all } \|\bfC\| \leq 2\beta^{1/2}.
\]
\end{itemize}
\end{claim}
\begin{subproof}
Fix $\epsilon > 0$ small and independent of $\beta$, and let $S_\pm = S_\pm(\gg,I)$ as defined in \eqref{Splusdef1} and \eqref{Sminusdef1}. We will define the family $(V_q)_{q\in Q}$ and a sequence of linearly independent lattice vectors $(\rr_i)_{i\in S_-}$ by simultaneous recursion: Fix $j\in S_-$ and suppose that $\rr_i$ has been defined for all $i\in S_-(j) \df \{i\in S_- : i < j\}$. Let $q\in Q$ and $r\in Q'$ be maximal and minimal, respectively, such that $q < j \leq r$. If $V_q$ has not been defined yet, then let $V_q \subset V_r$ be a $\Lambda$-rational subspace of dimension $q$ such that
\begin{align}
\label{Vpdef}
V_p &\subset V_q \;\;\all Q\ni p < q,&
\rr_i &\in V_q \;\;\all S_-\ni i \leq q,
\end{align}
chosen so as to minimize $\|V_q\|$ subject to these restrictions. Then
\begin{align*}
\dim(V_r\cap \vert) &\underset{\eqref{VqL}}{\geq} \dir_-(r) = \#(S_-(r+1)) > \#(S_-(j))
\end{align*}
since $j\in S_-$. Further, we observe that 
\[
\dim(V_r) > \dim\left(V_q + \textstyle\sum_{i\in S_-(j)} \R\rr_i\right)
\]
since 
\begin{align*}
\dim\left(V_q + \textstyle\sum_{i\in S_-(j)} \R\rr_i\right) &\leq q + \#(S_-(j)) - \#(S_-(q+1))\\
&\leq q + j - (q+1)\\
&< j\leq r.
\end{align*}
We claim that it is possible to choose $\rr_j \in \Lambda\cap V_r$ such that \Footnote{In the equations below, $\ang$ denotes the angle between two vectors, or between a vector and a vector subspace.}
\begin{align}
\label{rjdef1}
\ang(\rr_j,\vert) &\leq \epsilon,&
\ang(\rr_j,\R\rr_i) &\geq_\pt \pi/2 - \epsilon \;\;\all i \in S_-(j),\\ 
\label{rjdef2}
\ang\left(\rr_j,V_q+\textstyle\sum_{i\in S_-(j)}\R\rr_i\right) &\geq \epsilon^2,&
\log\|\rr_j\| &\lesssim_{\plus,C_1} g_j(t_1).
\end{align}
Indeed, one produces $\rr_j$ by first choosing a unit vector 
\[
\uu_1 \in V_r \cap \vert \cap \bigcap_{i \in S_-(j)} \rr_i^\perp,
\] 
choosing a second unit vector $\uu_2\in V_r$ so that\Footnote{We use \label{nbhd}$\NN(A,\epsilon)$ to denote the $\epsilon$-neighborhood of a set $A\subset\R^d$.} 
\[
B(\uu_2,\epsilon/3) \subset \R B_\ang(\uu_1,\epsilon)\butnot \NN \left(V_q + \sum_{i\in S_-(j)}\R\rr_i,2\epsilon^2 \right),
\] 
and finally choosing 
\[
\rr_j \in \Lambda\cap V_r \cap B(\tau \uu_2, \tau \epsilon/3),
\] 
where $\tau = C \lambda_r(\Lambda \cap V_r)$ for a constant $C$ (depending on $\epsilon$) large enough to guarantee that $\Lambda\cap V_r$ is a $(\tau \epsilon/3)$-net in $V_r$. 
Now both sides of \eqref{rjdef1} follow since $\rr_j \in  \R B_\ang(\uu_1,\epsilon)$. The left-hand side of \eqref{rjdef2} follows since $\tau^{-1}\rr_j \notin \NN(V_q + \sum_{i\in S_-(j)}\R\rr_i,2\epsilon^2)$ and $\|\tau^{-1}\rr_j\| \geq 1 - \epsilon/3$. The right-hand side of \eqref{rjdef2} follows from the fact that $\log\|\rr_j\| \asymp_\plus \log\tau \asymp_\plus \log\lambda_r(\Lambda \cap V_r) \asymp_{\plus,C_1} g_j(t_1)$ (by \eqref{Cmatch} and \eqref{lambdaVq} and since $g_j(t_1) = g_r(t_1)$ since $\CO jr$ is disjoint from $Q'$). This completes the proof of \eqref{rjdef1}-\eqref{rjdef2}, and thus the construction of $(\rr_i)_1^d$ and $(V_q)_{q\in Q}$.

Note that by construction, the family $(V_q)_{q\in Q}$ satisfies (i) and (ii). To demonstrate (iii), first we observe that it holds for $q\in Q'$ by Minkowski's second theorem (Theorem \ref{mink2}). By induction, suppose that (iii) holds for all $p<q$, where $q\in Q\butnot Q'$, and let $p\in Q$ and $r\in Q'$ be maximal and minimal, respectively, such that $p < q\leq r$. Then by \eqref{rjdef2}, \eqref{Cmatch}, and \eqref{lambdaVq}, we have
\begin{align*}
\log\left\|V_p + \sum_{i\in S_-(p,q)} \R\rr_i\right\| &\leq \log\|V_p\| + \sum_{i\in S_-(p,q)} \log\|\rr_i\|\\
&\lesssim_{\plus,C_1} \sum_{i\leq p} g_i(t_1) + \sum_{i\in S_-(p,q)} g_i(t_1)\\
&=_\pt \sum_{i\leq p + \#(S_-(p,q))} g_i(t_1), \since{$(p,r)\cap Q = \emptyset$}
\end{align*}
where $S_-(p,q) = \{i\in S_- : p < i \leq q\}$. Thus by \eqref{Cmatch}, it is possible to choose a $\Lambda$-rational subspace $V_q\subset V_r$ satisfying \eqref{Vpdef} such that $\log\|V_q\| \lesssim_{\plus,C_1} \sum_{i\leq q} g_i(t_1)$. The reverse inequality follows directly from Minkowski's second theorem (Theorem \ref{mink2}).

To demonstrate (iv), let $\vert' = \sum_{j\in S_-} \R\rr_j$. Then \eqref{rjdef1} implies that $\dist_\GG(\vert,\vert') = O(\epsilon)$, where $\dist_\GG$ denotes distance in the Grassmannian variety of $n$-dimensional subspaces of $\R^d$, which we denote by $\GG = \GG(d,n)$. It follows that if $\epsilon$ is sufficiently small, then there exists $\bfB\in B_\MM(\0,1-\beta)$ such that $u_{-\bfB} \vert = \vert'$. Then for all $q\in Q$, by \eqref{Vpdef} we have
\[
\dim(u_{\bfB} V_q \cap \vert) = \dim(V_q\cap \vert') \geq \#\{i\in S_- : i \leq q\} = \dir_-(q).
\]
Conversely, fix $\|\bfC\| \leq 2\beta^{1/2}$ and $q\in Q$. Let us define
\[
W \df u_\bfC V_q \cap \sum_{S_-\ni i > q} \R\rr_i .
\]
Then  
\[
\dim(u_{\bfB+\bfC}V_q\cap\vert) = \dim(u_\bfC V_q\cap \vert') \leq \dir_-(q) + \dim(W).
\]
Thus if $\dim(W) = 0$, then we are done with proving (iv). So by contradiction, suppose that $\dim(W) > 0$, i.e. that there exists
\[
\0\neq \rr\in W.
\]
Write $\rr = \sum_{S_-\ni i > q} c_i \rr_i$ for some constants $c_i\in\R$. Let $S_-\ni j > q$ be chosen so as to maximize $\theta^{-j} |c_j|\cdot\|\rr_j\|$, where $\theta > 0$ is small. 
Since $\rr \neq \0$, we have $c_i \neq 0$ for some $i$, and thus $\theta^{-j} |c_j| \cdot \|\rr_j\| \geq \theta^{-i} |c_i| \cdot \|\rr_i\| > 0$.
Then
\begin{align*}
\rr_j &= \frac1{c_j}\left(\rr - \sum_{i\neq j} c_i \rr_i\right)
\end{align*}
and thus
\begin{align*}
\frac1{\|\rr_j\|}\dist\left(\rr_j,V_q+\textstyle\sum_{i\in S_-(j)} \R\rr_i\right)
&\leq \frac1{|c_j|\cdot\|\rr_j\|} \left[\|\rr - u_{-\bfC}\rr\| + \sum_{i>j} |c_i|\cdot\|\rr_i\|\right]\\
&\lesssim \frac{\|\bfC\|\cdot\|\rr\|}{|c_j|\cdot\|\rr_j\|} + \sum_{i>j} \theta^{i-j}\\
&\lesssim 2\beta^{1/2} \max_i \frac{|c_i|\cdot\|\rr_i\|}{|c_j|\cdot\|\rr_j\|} + \theta\\
&\lesssim \theta^{1-d} \beta^{1/2} + \theta.
\end{align*}
Letting $\theta \df \beta^{1/(2d)} \leq \epsilon^3$ gives
\[
\ang\left(\rr_j,V_q+\textstyle\sum_{i\in S_-(j)} \R\rr_i\right) \lesssim \epsilon^3,
\]
which contradicts the first half of \eqref{rjdef2} if $\epsilon$ (or equivalently $\beta$) is sufficiently small. This completes the proof of (iv), and thus of Claim \ref{claimextendflag}.
\end{subproof}

Now for the purposes of defining Alice's strategy, fix $k = k_1,\ldots,k_2-1$, and suppose that the game has progressed to turn $k$, so that Bob's matrices $\bfB_{k_1},\ldots,\bfB_{k-1} \in B_\MM(\0,1-\beta)$ have all been defined. For each $q\in Q$ let
\[
V_q^{(k)} \df (g u_{\bfB_{k-1}})\cdots (g u_{\bfB_{k_1}}) V_q,
\]
where $g = g_\tbeta$ and $\tbeta$ are as in Notation \ref{gamma}. 
Recall we defined $\Lambda \df \Lambda_{k_1}$ and so Claim \ref{claimextendflag} implies $V_q$ is $\Lambda_{k_1}$-rational, and thus $V_q^{(k)}$ is $\Lambda_k$-rational.
 


Now let $\OC pq_\Z$ be an interval of equality for $\gg$ on $I$, and consider the quotient lattice 
\[
\Gamma_k \df \Lambda_k\cap V_q^{(k)}/V_p^{(k)}
\] 
(more precisely, $\Gamma_k$ is the image of $\Lambda_k$ under the quotient map $V_q^{(k)} \to V_q^{(k)} / V_p^{(k)}$). Let $(\rr_i^{(k)})_1^{q-p}$ be a basis for $\Gamma_k$ such that
\[
\|\rr_i^{(k)}\| \asymp \lambda_i(\Gamma_k) \text{ for all } 1\leq i \leq q-p.
\]
For each $j = 1,\ldots,q-p$ let
\[
V_{p+j}^{(k)} = V_p^{(k)} + \sum_{i=1}^j \R \rr_i^{(k)}.
\]
Next, a matrix $\bfB$ will be called \emph{good on turn $k$} if for all $j = 1,\ldots,d$ we have
\begin{equation}
\label{good1}
\dim\big(u_\bfB V_j^{(k)}\cap \vert\big) = \dir_-(j) \df \#(S_-\cap [1,j])
\end{equation}
and
\begin{align}
\label{good2}
\dim\big(u_{\bfB+\bfC} V_j^{(k)}\cap \vert\big) \leq \dir_-(j) \text{ for all $\|\bfC\| \leq 2\beta^{1/2}$.}
\end{align}
Note that the notion of being good on turn $k$ is dependent on the choice of subspaces $V_j^{(k)}$ and is therefore not canonical.

Alice's strategy on turn $k$ can now be given as follows:\\
\begin{center}
\fbox{
\begin{minipage}{27em}
Let $A_k$ be a $3\beta$-separated subset (to be specified later) of the set of matrices in $B_\MM(\0,1-\beta)$ that are good on turn $k$.
\end{minipage}
}
\end{center}
Note that by Claim \ref{claimextendflag}(iv), we can always take $A_k \neq \emptyset$.

Now, to prove that Alice's strategy guarantees (i)-(iii) in Lemma \ref{lemmaministrategy}, consider a possible sequence of responses from Bob, i.e. a sequence $(\bfB_k)_{k=k_1}^{k_2-1}$ such that for each $k$, we have $\bfB_k\in A_k$. For each $k = k_1,\ldots,k_2$ let
\[
Z_k \df \sum_{\ell = k_1}^{k - 1} \beta^{\ell - k_1} \bfB_\ell \in B(\0,1),
\]
so that for all $q\in Q$,
\[
V_q^{(k)} = g^{k-k_1} u_{Z_k} V_q
\]
(cf. \eqref{danisemi}). Now fix $q\in Q$, and let $\dir_\pm \df \dir_\pm(\gg,I,q)$. Fix $k = k_1+1,\ldots,k_2$. Since $\bfB_{k-1}$ is good on turn $k-1$, we have
\[
\dim(u_{Z_k} V_q \cap \vert) = \dim(V_q^{(k)}\cap \vert) = \dir_-
\]
and since $\bfB_{k_1}$ is good on turn $k_1$ and $\|Z_k - \bfB_{k_1}\| \leq \frac{\beta}{1-\beta}$ and $\frac{\beta}{1-\beta} + \beta < 2\beta^{1/2}$ once $\beta$ is sufficiently small, we have
\[
\dim\big(u_{Z_k + \bfC} V_q\cap \vert\big) \leq \dir_- \text{ for all }\|\bfC\| \leq \beta.
\]
Now by Lemma \ref{lemmapushoffvert}, these two formulas imply that
\begin{align*}
\log\|V_q^{(k)}\| - \log\|V_q\|
&\asymp_{\plus\phantom{,\beta}} \log\|g^{k-k_1} u_{Z_k} V_q\| - \log\|u_{Z_k} V_q\|\\
&\asymp_{\plus,\beta} \left(\frac{\dir_+}{\pdim} - \frac{\dir_-}{\qdim}\right)(k-k_1)\tbeta\\
&=_{\phantom{\plus,\beta}} \sum_{i=1}^q g_i(k\tbeta) - \sum_{i=1}^q g_i(t_1). \by{\eqref{Lqdef}}
\end{align*}
Combining with condition (iii) of Claim \ref{claimextendflag} shows that
\begin{equation}
\label{Vqk2}
\log\|V_q^{(k)}\| \asymp_{\plus,\beta} \sum_{i=1}^q g_i(k\tbeta).
\end{equation}
Now let $\OC pq_\Z$ be an interval of equality for $\gg$ on $I$, and let 
\[
\Gamma_k = \Gamma_k(p,q) \df \Lambda_k \cap V_q^{(k)}/V_p^{(k)}
\] 
as above.
\begin{claim}
\label{claimlocalmatch}
We have
\[
\log\lambda_j(\Gamma_k) \asymp_{\plus,\beta} g_{p+j}(k\gamma)
\]
for all $j=1,\ldots,q-p$ and $k=k_1,\ldots,k_2$.
\end{claim}
\begin{subproof}
Write
\[
\eta_j(k) \df \log\lambda_j(\Gamma_k) - g_{p+j}(k\tbeta).
\]
By \eqref{Vqk2} and Minkowski's second theorem (Theorem \ref{mink2}), we have
\begin{equation}
\label{Vqk3}
\sum_{i=1}^{q-p} \eta_i(k) \asymp_{\plus,\beta} \log\|\Gamma_k\| - \sum_{i=p+1}^q g_i(k\tbeta) \asymp_\plus 0.
\end{equation}
First suppose that $\diff_- = 0$, where $\diff_\pm = \diff_\pm(\gg,I,p,q)$. Then for all $j,k$ we have
\[
g_{p+j}(k\gamma) - g_{p+j}(k_1\gamma) = \frac{(k-k_1)\gamma}{m} \geq \log\lambda_j(\Gamma_k) - \log\lambda_j(\Gamma_{k_1}),
\]
and \eqref{Vqk3} implies that approximate equality holds. Similar logic works if $\diff_+ = 0$.

So suppose that $\diff_+,\diff_- > 0$. Let $K$ be a large constant (depending on $\beta$). To complete the proof of Claim \ref{claimlocalmatch} we will show that
\begin{equation}
\label{IHschmidt}
-\frac{K}{\diff_+} \leq \eta_j(k) \leq \frac{K}{\diff_-}
\end{equation}
for all $j=1,\ldots,q-p$ and $k=k_1,\ldots,k_2$, by induction on $k$. Indeed, suppose that \eqref{IHschmidt} holds for $k$, and we will prove that it holds for $k' = k + \ell_0$, where $\ell_0$ is a large integer. (We then take $k' = k,k+1,\ldots,k+\ell_0-1$ as the base cases of the induction.) By \eqref{Vqk3}, we have
\[
j \eta_j(k) \geq \sum_{i = 1}^j \eta_i(k) \asymp_{\plus,\beta} -\sum_{i = j+1}^{q-p} \eta_i(k) \geq -\frac{(q-p-j)K}{\diff_-}\cdot
\]
Letting $j = \diff_+ + 1$ shows that
\[
\eta_{\diff_+ + 1}(k) \gtrsim_{\plus,\beta} -\frac{\diff_- - 1}{\diff_-}\frac{K}{\diff_+ + 1} = -\frac{K}{\diff_+} + \alpha K,
\]
where $\alpha > 0$ is a positive constant.

Note that since $g_{p+j}(t) = g_q(t)$ for all $t\in I$ and $j = 1,\ldots,q-p$, we have $\eta_1 \leq \cdots \leq \eta_{q-p}$, so if \eqref{IHschmidt} fails for $k' = k + \ell_0$, then either $\eta_1(k') < -K/\diff_+$ or $\eta_{q-p}(k') > K/\diff_-$. By contradiction suppose that $\eta_1(k') < -K/\diff_+$ (the other case is similar). Then $\eta_1(k) \asymp_{\plus,\ell_0,\beta} -K/\diff_+$ and thus
\[
\eta_{\diff_+ + 1}(k) - \eta_1(k) \gtrsim_{\plus,\ell_0,\beta} \alpha K.
\]
If $K$ is sufficiently large in comparison to $\ell_0$, then it follows that there exists $j' = 1,\ldots,\diff_+$ such that
\begin{equation}
\label{jgap}
\eta_{j'+1}(k) - \eta_{j'}(k) \geq \frac{\alpha K}{\diff_+ + 1}\cdot
\end{equation}
It follows from \eqref{jgap} that if $K$ is sufficiently large (in comparison to $\ell_0$), then
\[
V_{p+j'}^{(\ell)} = b_{k,\ell} V_{p+j'}^{(k)} \text{ for all }\ell = k,\ldots,k'
\]
where $b_{k,\ell} \df (g u_{\bfB_{\ell-1}})\cdots (g u_{\bfB_k})$. Thus since $\bfB_k,\ldots,\bfB_{\ell-1}$ are good on their respective turns, Lemma \ref{lemmapushoffvert} shows that
\[
\log\|V_{p+j'}^{(k')}\| - \log\|V_{p+j'}^{(k)}\| \asymp_{\plus,\beta} (k'-k)\tbeta \left(\frac{\dir_+(p+j')}{\pdim} - \frac{\dir_-(p+j')}{\qdim}\right).
\]
Subtracting \eqref{Vqk2} (with $q=p$ and separately $k=k$, $k=k'$) and using the asymptotic
\[
\log\|V_{p+j'}^{(\ell)}\| - \log\|V_p^{(\ell)}\| \asymp_{\plus,\beta} \sum_{i=1}^{j'} \log \lambda_i(\Gamma_\ell)
\]
and the relations
\begin{align*}
\dir_+(p+j') &= \dir_+(p) + j',&
\dir_-(p+j') &= \dir_-(p)
\end{align*}
(valid since $j' \leq \diff_+$) show that
\begin{equation}
\label{sumtojprime}
\sum_{i=1}^{j'} \log\lambda_i(\Gamma_{k'}) - \sum_{i=1}^{j'} \log\lambda_i(\Gamma_k)
\asymp_{\plus,\beta} (k'-k)\tbeta \frac{j'}{\pdim}\cdot
\end{equation}
On the other hand, since $\log\|b_{k,k'}\| \lesssim_\plus (k'-k)\gamma/m$, we have
\[
\log\lambda_i(\Gamma_{k'}) - \log\lambda_i(\Gamma_k) \lesssim_\plus (k'-k)\tbeta \frac{1}{\pdim} \text{ for all $i = 1,\ldots,j'$}
\]
and by \eqref{sumtojprime}, approximate equality holds. In particular
\[
\eta_1(k') - \eta_1(k) \asymp_{\plus,\beta} (k'-k)\tbeta \left[\frac{1}{\pdim} - \frac{1}{\diff_+ + \diff_-}\left(\frac{\diff_+}{\pdim} - \frac{\diff_-}{\qdim}\right)\right].
\]
The right-hand side is strictly positive, so if $\ell_0$ is sufficiently large, then the left-hand side is also positive. But this contradicts our assumption that $\eta_1(k') < -K/\diff_+ \leq \eta_1(k)$, thus demonstrating \eqref{IHschmidt}. This concludes the proof of Claim \ref{claimlocalmatch}.
\end{subproof}

Next, note that for any $1 \leq i \leq q-p$ we have that 
\[
\lambda_i(\Gamma_k(p,q))
\leq
\lambda_{p+i}(\Lambda_k \cap V_q^{(k)})
\lesssim
\max_{\substack{\OC{p'}{q'}\\q'\leq q}}\lambda_{q'-p'}(\Gamma_k(p',q')),
\]
where the maximum is taken over all intervals of equality $\OC{p'}{q'}$ for $\gg$ that satisfy $q'\leq q$. Indeed, the first inequality can be demonstrated by observing that the projection of a set of $p+i$ linearly independent vectors in $\Lambda_k\cap V_q$ contains a linearly independent set of $i$ vectors in $\Gamma_k(p,q)$. For the second inequality, denote the right-hand side by $\lambda$ and note that by pulling back vectors appropriately, we can recursively construct bases of $\Lambda_k\cap V_{q'}$ for all $q' \leq q$, such that the largest vector in each basis has norm $\lesssim \lambda$.

Now using Claim \ref{claimlocalmatch}, we have that 
\[
g_{p+i}(k\gamma) 
\lesssim_{\plus,\beta}
\log \left( \lambda_{p+i}(\Lambda_k \cap V_q^{(k)}) \right)
\lesssim_{\plus,\beta}
\max_{\substack{\OC{p'}{q'}\\q'\leq q}} g_{q'}(k \gamma)
=
g_q(k\gamma)
=
g_{p+i}(k\gamma), 
\]
where the maximum is taken as before. Thus we have that for $p < j \leq q$
\begin{equation}
\label{LambdakVq}
\log\lambda_j(\Lambda_k\cap V_q^{(k)})  \asymp_{\plus,\beta} g_j(k\gamma).
\end{equation}

To demonstrate \eqref{Cmatch} and \eqref{lambdaVq}, we pick $\rr \in \Lambda_k \butnot V_p^{(k)}$. Now consider the projection map
\[
\pi : V_{q} \to V_{q} / V_{p}
\]
and note that 
\[
\|\rr\| \geq \|\pi(\rr)\| 
\geq \lambda_1(\Gamma_k(p,q))
\asymp_{\times,\beta} \exp(g_{p+1}(k\gamma)).
\]
Therefore we have that $\log\lambda_{p+1}(\Lambda_k) \gtrsim_{\plus,\beta} g_{p+1}(k\gamma)$. Thus for $p < j \leq q$, we also have 
\[
\log\lambda_j(\Lambda_k) \gtrsim_{\plus,\beta} g_j(k\gamma).
\]
On the other hand, by the monotonicity of the successive mimina functional we have 
\[
\log\lambda_j(\Lambda_k\cap V_q) \geq \log\lambda_j(\Lambda_k).
\] 
Therefore, using \eqref{LambdakVq} and the previous two display equations, we get 
\[
g_j(k\gamma) \asymp_{\plus,\beta} \log\lambda_j(\Lambda_k\cap V_q) \asymp_{\plus,\beta} \log\lambda_j(\Lambda_k),
\] 
and thus \eqref{C2def} holds. This completes the proof of condition (i) of Lemma \ref{lemmaministrategy}. 

We proceed to prove conditions (ii) and (iii).
By \eqref{C2def}, \eqref{Cmatch} and \eqref{lambdaVq} hold with $\Lambda = \Lambda_{k_2}$, $t = t_2$, and $C=C_2$, where $C_2$ is the implied constant of \eqref{C2def}. 
Observe that by \eqref{VqL} we have
\[
\dim(V_q(\Lambda_{k_2})\cap \vert) = \dim(V_q^{(k_2)}\cap \vert) \geq \dir_-(\gg,I,q) \geq \dir_-(\gg,I_+,q),
\]
where $I_+$ is the interval of linearity for $\gg$ whose left endpoint is $t_2$. Note that the last inequality is due to the assumption of convexity in (III) of Definition \ref{definitiontemplate}. It follows that condition (II) of Definition \ref{definitionCmatch} holds with $\Lambda = \Lambda_{k_2}$ and $t = t_2$, which completes the proof of (ii).

To demonstrate (iii), it suffices to show that
\begin{itemize}
\item[(a)] $\#(A_{k_1}) \geq 1$, and
\item[(b)] $\#(A_k) \gtrsim \beta^{-\delta}$ for all $k > k_1$, where $\delta = \delta(\gg,I)$.
\end{itemize}
Note that (a) is true by part (iv) of Claim \ref{claimextendflag}. To demonstrate (b), fix $k > k_1$, and observe that since $\bfB_{k-1}$ is good on turn $k-1$, for all $q\in Q$ we have
\begin{equation}
\label{indhypgood}
\dim(V_q^{(k)}\cap \vert) = \dir_-(q)
\end{equation}
and
\begin{equation}
\label{goodonkminusone}
\dim(u_\bfC V_q^{(k)} \cap \vert) \leq \dir_-(q) \;\;\all\; \bfC \in B_\MM(\0,\beta^{-1/2}).
\end{equation}
We now construct a basis of $\R^d$ as follows. 
\begin{claim}
\label{claimorthobasis}
There exists an almost orthonormal basis $(\rr_i)_1^d$ of $\R^d$ (meaning that $\rr_i \cdot \rr_j = \delta_{ij} + o(1)$ as $\beta\to 0$ for all $i,j$), which contains a subset that is an orthonormal basis of $\vert$ and for each $q \in Q$ contains an almost orthonormal basis of $V_q^{(k)}$. 
\end{claim}
\begin{subproof}
Let $\OC pq_\Z$ be an interval of equality for $\gg$ on $I$, and let $\diff_\pm \df \diff_\pm(p,q)$ (as defined in \eqref{Mpqdef}). Let $(\rr_i)_{p+1}^{p+\diff_+}$ be an orthonormal basis of
\[
W_+(p,q) \df V_q^{(k)} \cap (V_p^{(k)})^\perp \cap (V_q^{(k)}\cap\vert)^\perp
\]
and let $(\rr_i)_{p+\diff_++1}^q$ be an orthonormal basis of
\[
W_-(p,q) \df V_q^{(k)} \cap \vert \cap (V_p^{(k)}\cap \vert)^\perp.
\]
Such bases exist because \eqref{indhypgood} allows us to compute the dimensions of these spaces. Then we claim that $(\rr_i)_1^d$ is an almost orthonormal basis of $\R^d$ (meaning that $\rr_i \cdot \rr_j = \delta_{ij} + o(1)$ as $\beta\to 0$ for all $i,j$), and that $(\rr_i)_{i\in S_-}$ is an orthonormal basis of $\vert$ (where $S_-$ is defined in \eqref{Sminusdef1}). 

Indeed, to see why $(\rr_i)_1^d$ is almost orthonormal, we fix $i < j$ and consider four cases:
\begin{align}
\label{++}
&\rr_i \in W_+(p_1,q_1), \;\; \rr_j \in W_+(p_2,q_2)\\
\label{+-}
&\rr_i \in W_+(p_1,q_1), \;\; \rr_j \in W_-(p_2,q_2)\\
\label{-+}
&\rr_i \in W_-(p_1,q_1), \;\; \rr_j \in W_+(p_2,q_2)\\
\label{--}
&\rr_i \in W_-(p_1,q_1), \;\; \rr_j \in W_-(p_2,q_2)
\end{align}
for $p_1 \leq q_1$ and $p_2 \leq q_2$. In the three cases \eqref{++}, \eqref{-+} and \eqref{--}, we have that $\rr_i \cdot \rr_j =\delta_{ij}$ by part (ii) of Claim \ref{claimextendflag}.
Note that since $\vert = \sum_{\OC pq_\Z} W_-(p,q)$, it follows from \eqref{--} that $(\rr_i)_{i\in S_-}$ is an orthonormal basis of $\vert$. 

So we are left to consider the case \eqref{+-}. Note that in this case we may assume that $p_1 < q_1 \leq p_2 < q_2$ (since if $p_1 = q_1$ and $p_2 = q_2$, then \eqref{+-} reduces to \eqref{-+}).

Let $V \df V_{q_1}^{(k)}$. Then $\rr_i \in W_+(p_1,q_1) \subset V \cap (V \cap \vert)^\perp$ and $\rr_j \in W_-(p_2,q_2) \subset \vert \cap (V\cap \vert)^\perp$. Write $\rr_i = (\pp,\qq')$ and $\rr_j = (\0,\qq)$.
Now let 
\[
Y \vv \df -\frac{(\qq'\cdot\vv)}{(\qq'\cdot\qq')}\pp.
\] 
Then note that $u_Y (\pp,\qq') \in \vert$ and $u_Y(V\cap \vert) = V\cap\vert$ (since $\rr_i \in (V\cap \vert)^\perp$). Therefore we have that 
\[
\dim(u_Y(V)\cap\vert) > \dim(V\cap\vert).
\]
Thus by \eqref{goodonkminusone}, we have that 
\[
\|Y\| > \beta^{-1/2}.
\]
Since $\|Y\| \leq \|\pp\| / \|\qq'\| \leq 1 / \|\qq'\|$ it follows that $\|\qq'\| < \beta^{1/2}$.
Therefore 
\[
|\rr_i \cdot \rr_j| = |\qq\cdot\qq'| \leq \|\qq'\| < \beta^{1/4}.
\]
Thus $\rr_i \cdot \rr_j = \delta_{ij} + o(1)$ as $\beta\to 0$ for all $i,j$ as claimed. This concludes the proof of Claim \ref{claimorthobasis}.
\end{subproof}

Let $\ZZ$ be the space of all $d\times d$ matrices $X$ such that for all $i,j$ such that $X_{i,j} \neq 0$, we have $i < j$, $i\in S_+$, and $j\in S_-$. Evidently, $\dim(\ZZ) = \delta(\gg,I)$. Now let $\bfR$ be the matrix whose column vectors are $\rr_1,\ldots,\rr_d$. Then for all $\bfD\in \ZZ$, the matrix $\bfR\cdot ( I + \bfD )\cdot\bfR^{-1}$ preserves the subspaces $(V_q^{(k)})_{q\in Q}$. Now define a map $\Phi:\ZZ\to \MM$ as follows: for each $\bfD\in \ZZ$, $\bfB = \Phi(\bfD)$ is the unique matrix such that
\[
u_\bfB \vert = \bfR \cdot ( I + \bfD ) \cdot\bfR^{-1} \vert.
\]
It is easy to check that in a neighborhood of the origin, $\Phi$ is a bi-Lipschitz embedding with bi-Lipschitz constant depending only on $\max(\|\bfR\|,\|\bfR^{-1}\|)$. But since the basis $(\rr_i)_1^d$ is almost orthonormal as proved in Claim \ref{claimorthobasis}, there is a uniform bound on this constant as long as $\beta$ is sufficiently small.

\begin{claim}
Let $C$ be the bi-Lipschitz constant of $\Phi$. There exist $W \in B_\ZZ(\0,1/(2C))$ and a constant $0 < \epsilon \leq 1/(2C)$ such that for all $W'\in B_\ZZ(W,\epsilon)$, \eqref{good2} holds for $X = \Phi(W')$ for all $j = 1,\ldots,d$, as long as $\beta$ is sufficiently small.
\end{claim}
\begin{subproof}
By an induction argument, it suffices to consider only one value of $j$, specifically $j = p + M_+(p,q) \in \OC pq_\Z$ where $\OC pq_\Z$ is an interval of equality. Note that since $M_+ + M_- = q-p$, it follows that $M_- = q-j$.
Now the map 
\[
\Psi:W \mapsto (I + W) \cdot \bfR^{-1} ((V_q^{(k)}\cap\vert) / V_p^{(k)})
\] 
is easily seen to be an open mapping from $\ZZ$ to the Grassmannian\Footnote{We use the notation $\GG \df \GG_k(V)$ to denote the Grassmannian variety of $k$-dimensional subspaces of a Euclidean vector space $V$. Each $(\GG, d_{\GG})$ is a compact metric space, see \cite[Lemma 3.2]{Morris2} for two equivalent (symmetric) means to define the metric on the Grassmanian.} $\GG \df \GG_{M_-}(\R^q\times\{\0\} / \R^p \times \{\0\})$ in a neighborhood of $W = \0$, since it is a composition of a projection map ($W\mapsto (W_{i,j})_{i,j\in \OC pq_\Z}$) and a homeomorphism. Thus in particular there exists $W$ such that $\dist_\GG(\Psi(W),\VV) > 0$ where $\VV = \{V : \dim(\bfR^{-1} (V_j^{(k)} / V_p^{(k)}) \cap V) > 0\}$. A compactness argument (since $V_j^{(k)}$ and $\bfR$ both range over compact sets) shows we can take $\dist_\GG(\Psi(W),\VV) \asymp 1$.\Footnote{Otherwise, we could find sequences $V_{j,\ell}$ and $\bfR_\ell$ such that the corresponding $\dist_\GG(\Psi_\ell(W),\VV_\ell) \to 0$ as $\ell\to \infty$, and passing to a convergent subsequence would give $\dist_\GG(\Psi_\infty(W),V_\infty) = 0$ in the limit.} Now let $\epsilon = (1/2) \dist_\GG(\Psi(W),\VV)$, and let $W'\in B_\ZZ(W,\epsilon)$ and $X = \Phi(W')$. Then if $\|Y\| \leq 2\beta^{1/2}$, then we have $\dist_\GG(\Psi(W),u_{-Y} \Psi(W)) \lesssim \beta^{1/2}$ and so if $\beta$ is sufficiently small we have $u_{-Y} \Psi(W) \notin \VV$, which implies \eqref{good2}.
\end{subproof}

Now let $A_k'$ be a maximal $3C\beta$-separated subset of $B_\ZZ(W,\epsilon)$. Then $A_k = \Phi(A_k')$ is a $3\beta$-separated subset of $B(\0,1-\beta)$ consisting entirely of matrices good on turn $k$. It follows that
\[
\#(A_k) = \#(A_k') \asymp \beta^{-\delta}.
\]
This concludes the proof of condition (iii) of Lemma \ref{lemmaministrategy}, and therefore of the entire lemma.
\end{proof}

\subsection{Error correction}
\label{subsectionerrorcorrection}
Fix $\tspace \in \tbeta\N$, let $\ff$ be a simple $\tspace$-integral template, fix $t_0\in \tspace\N$, and let $\Pert = (\pert_1,\ldots,\pert_d)\in\R^d$ be a vector such that $\pert_i \leq \pert_{i+1}$ for all $i$ such that $f_i(t_0) = f_{i+1}(t_0)$. Such a vector will be called a \emph{perturbation vector} of $\ff$ at $t_0$. For convenience, for each $k\in\N$ let $t_k = t_0 + k \tspace$. We define the function $\aa:\N\cup\{-1\}\to\R^d$ recursively as follows:
\begin{itemize}
\item $\aa(-1) = \Pert$.
\item Fix $k\geq 0$ such that $\aa(k-1)$ has been defined, and let $I_k = (t_k,t_{k+1})$. If $\OC pq_\Z$ is an interval of equality for $\ff$ on $I_k$ (cf. Definition \ref{definitiondimtemplate}), then for all $i = p+1,\ldots,q$, we let
\begin{equation}
\label{akrecursive}
a_i(k) = 
\begin{cases}
a_i(k-1) & \text{ if } f_{p+1}' = \ldots = f_q' \in \exceptionalset \text{ on } (t_0,t_{k+1})\\
\displaystyle\frac{1}{q-p} \sum_{j=p+1}^q a_j(\index-1)
& \text{ otherwise.}
\end{cases}
\end{equation}
\end{itemize}
The idea is that we will construct a new template by displacing $\ff$ by $\aa(k)$ on each interval $I_k$, and then changing the resulting function into a template by modifying it slightly to deal with the issues that arise near multiples of $\tspace$. The motivation for the equation \eqref{akrecursive} will become apparent when we analyze when it is possible to perform such a modification. Note that by induction, for all $k$ we have
\begin{equation}
\label{akh}
\|\aa(k)\|_{\infty} \leq \|\Pert\|_{\infty}
\end{equation}
and
\begin{equation}
\label{aikai1k}
a_i(k) \leq a_{i+1}(k) \text{ whenever } f_i = f_{i+1} \text{  on } I_k.
\end{equation}

\begin{lemma}
\label{lemmahperturbation}
Let the notation be as above. If $\displaystyle \|\Pert\|_{\infty} < \cspace = \frac{\tspace}{2mnd!}$, then there exists a partial template $\gg:\CO{t_0}\infty\to\R^d$ such that
\begin{equation}
\label{gfh}
\gg(t_0) = \ff(t_0) + \Pert
\end{equation}
and such that for all $k$, we have
\begin{equation}
\label{hpert}
\gg = \ff + \aa(k) \text{ on } \w I_k \df \big(t_k + \timeerror,t_{k+1} - \timeerror\big),
\end{equation}
where $\aa$ is as above, and
\[
\timeerror \df 2 \dimprod d^2\|\Pert\|_{\infty}.
\]
Moreover, we have
\begin{equation}
\label{Spmgtft}
S_\pm(\gg,t) = S_\pm(\ff,t) \text{ for } t\geq t_0
\end{equation}
and in particular
\begin{equation}
\label{deltagdeltaf}
\delta(\gg,t) = \delta(\ff,t) \text{ for } t\geq t_0.
\end{equation}
\end{lemma}
The partial template $\gg$ constructed in the proof below will be called the \emph{$\Pert$-perturbation} of $\ff$ at $t_0$.

\begin{proof}
We will first show that for all $k\geq 0$, if $\gg$ is any function satisfying \eqref{hpert}, then $\gg \given \w I_k$ is a partial template. Indeed, since $\gg$ is linear on $\w I_k$, it suffices to check conditions (I) and (II) of Definition \ref{definitiontemplate}, along with the following weakening of condition (III):
\begin{itemize}
\item[(III$'$)] For all $j = 1,\ldots,d-1$ such that $g_j < g_{j+1}$ on $\w I_k$, we have $G_j'(\w I_k) \in Z(j)$.
\end{itemize}
Condition (II) is obvious, so we check (I) and (III$'$).
\begin{subproof}[Proof of \text{(I)}]
Fix $i = 1,\ldots,d-1$, and we will show that $g_i \leq g_{i+1}$ on $\w I_k$. There are three cases:
\begin{itemize}
\item If $f_i = f_{i+1}$ on $I_k$, then by \eqref{aikai1k} we have $a_i(k) \leq a_{i+1}(k)$ and thus $g_i \leq g_{i+1}$ on $\w I_k$.
\item If $f_i(t_k) = f_{i+1}(t_k)$ but $f_i < f_{i+1}$ on $I_k$, then we have $f_i'(I_k) < f_{i+1}'(I_k)$, and thus
\begin{align*}
f_{i + 1} - f_i &> (f_{i+1}'(I_k) - f_i'(I_k)) \timeerror \note{on $\w I_k$}\\
&\geq \tfrac{1}{mnd^2}\timeerror \by{Observation \ref{observationrationalslopes}}\\
&= 2\|\Pert\|_{\infty}\\
&\geq |a_{i+1}(k) - a_i(k)|, \by{\eqref{akh}}
\end{align*}
so $g_i < g_{i+1}$ on $\w I_k$. Similar logic applies if $f_i(t_{k+1}) = f_{i+1}(t_{k+1})$ but $f_i < f_{i+1}$ on $I_k$.
\item If $f_i(t_k) < f_{i+1}(t_k)$ and $f_i(t_{k+1}) < f_{i + 1}(t_{k+1})$, then since $\ff$ is $\tspace$-integral we have
\begin{align*}
f_{i+1} - f_i &\geq \mathrm{min}\big(f_{i+1}(t_k) - f_i(t_k),f_{i+1}(t_{k+1}) - f_i(t_{k+1})\big) \hspace{-1.3 in} \note{on $I_k$}\\
&\geq \tfrac{\tspace}{mnd!} \by{Definition \ref{definitionintegral}}\\
&= 2\cspace\\
&> 2\|\Pert\|_{\infty} \by{hypothesis}\\
&\geq |a_{i+1}(k) - a_i(k)| \by{\eqref{akh}}
\end{align*}
and thus $g_i < g_{i+1}$ on $\w I_k$.
\QEDmod\qedhere\end{itemize}
\end{subproof}
\begin{subproof}[Proof of \text{(III$'$)}]
Fix $j = 1,\ldots,d-1$ such that $g_j < g_{j+1}$ on $\w I_k$. There are two cases:
\begin{itemize}
\item If $f_j < f_{j+1}$ on $I_k$, then $G_j'(\w I_k) = F_j'(I_k) \in Z(j)$.
\item If $f_j = f_{j+1}$ on $I_k$, then $a_j(k) < a_{j+1}(k)$. Moreover, $j$ and $j+1$ are in the same interval of equality $\OC pq_\Z \ni j,j+1$ for $\ff$ on $I_k$. By \eqref{akrecursive} we have $f_{p+1}'(I_k) = \ldots = f_q'(I_k) \in \exceptionalset$. Without loss of generality suppose that $f_{p+1}'(I_k) = \ldots = f_q'(I_k) = \tfrac1\pdim$. Then we have
\[
G_j'(\w I_k) = F_j'(I_k) = F_p'(I_k) + \frac{j-p}{\pdim} = \frac{\dir_+(\ff,I_k,p) + (j-p)}{\pdim} - \frac{\dir_-(\ff,I_k,p)}{\qdim} \in Z(j).
\]
(The intuition behind this calculation is that $\frac1\pdim$ and $-\frac1\qdim$ are ``free slopes'' that can be used by an individual $f_j$ without the need for averaging; cf. the model of ``particle physics'' described in the paragraph below Definition \ref{definitiondimtemplate}.)
\QEDmod\qedhere\end{itemize}
\end{subproof}
Next, we demonstrate \eqref{Spmgtft} for $t\in \w I_k$ (note that \eqref{deltagdeltaf} follows from \eqref{Spmgtft}). Let $\OC pq_\Z$ be an interval of equality for $\ff$ on $I_k$. By the proof of (I) above, we have $g_p < g_{p+1}$ and $g_q < g_{q+1}$ on $\w I_k$. Let
\[
\diff_\pm = \diff_\pm(\ff,I_k,p,q) = \dir_\pm(\gg,\w I_k,q) - \dir_\pm(\gg,\w I_k,p).
\]
If $\diff_+ > 0$ and $\diff_- > 0$, then $a_{p+1}(k) = \ldots = a_q(k)$ and thus $\OC pq_\Z$ is an interval of equality for $\gg$ on $\w I_k$, which implies that $S_+(\ff,I_k)\cap \OC pq_\Z = S_+(\gg,\w I_k)\cap \OC pq_\Z$. On the other hand, if $\diff_+ = 0$, then $S_+(\ff,I_k)\cap \OC pq_\Z = \emptyset = S_+(\gg,\w I_k)\cap \OC pq_\Z$, and if $\diff_- = 0$, then $S_+(\ff,I_k)\cap \OC pq_\Z = \OC pq_\Z = S_+(\gg,\w I_k)\cap \OC pq_\Z$. Since $\OC pq_\Z$ was arbitrary we have $S_+(\ff,I_k) = S_+(\gg,\w I_k)$ and thus $\delta(\ff,I_k) = \delta(\gg,\w I_k)$.

Finally, we describe how to define $\gg$ on an interval of the form
\begin{equation}
\label{Jk}
J_k \df
\begin{cases}
\big[t_k - \timeerror,t_k + \timeerror\big] & \text{if } k > 0\\
\big[t_0,t_0 + \timeerror\big] & \text{if } k = 0
\end{cases}
\end{equation}
We now consider two cases:\\

{\bf Case 1.} If $\aa(k-1) = \aa(k)$, then we can continue to use the formula $\gg = \ff + \aa(k)$ on $J_k$. Minor modifications to the previous argument show that $\gg\given \w I_{k-1}\cup J_k \cup \w I_k$ is a partial template (where we use the convention that $\w I_{-1} = \emptyset$), and $S_\pm(\ff,J_k) = S_\pm(\gg,J_k)$.\\

{\bf Case 2.} Suppose that $\aa(k-1) \neq \aa(k)$. By \eqref{akrecursive}, this means that $t_k$ is either a merge, a transfer, or $t_0$. We restrict our attention to the case where $t_k$ is a merge; the other cases are similar. Define $\gg$ on $J_k$ as follows: Let $\OC pq_\Z$ be an interval of equality for $\ff$ on $I_k$ which is not an interval of equality for $\ff$ on $I_{k-1}$, and let $\diff_\pm = \diff_\pm(\ff,I_k,p,q)$, so that $\diff_+ + \diff_- = q - p$, and $S_+(\ff,J_k)\cap\OC pq_\Z = \{p+1,\ldots,p+\diff_+\}$. Note that $\diff_+,\diff_- > 0$, as otherwise we would have $f_{p+1}' = f_q' \in \exceptionalset$ on $I_{k-1}$, and thus $\OC pq_\Z$ would be an interval of equality for $\ff$ on $I_{k-1}$. We define the piecewise linear functions $g_{p+1},\ldots,g_q$ on $J_k$ by imposing the following conditions:
\begin{itemize}
\item We have
\begin{equation}
\label{atminJk}
\gg(\min(J_k)) = \ff(\min(J_k)) + \aa(k-1).
\end{equation}
\item We have
\begin{equation}
\label{onJk}
\sum_{i = p+1}^q g_i' = \sum_{i = p+1}^q f_i' = \frac{\diff_+}{\pdim} - \frac{\diff_-}{\qdim} \text{ on $J_k$}
\end{equation}
(the second equality holds because $t_k$ cannot be a transfer, since $f$ is simple).
\item For all $p < i \leq p+\diff_+$ and $t\in J_k$, we have $g_i'(t) = \frac1\pdim$ unless $g_i(t) = g_{p+\diff_+ +1}(t)$, in which case $g_i'(t) = z(t)$, where $z(t)$ is defined below.
\item For all $p+\diff_+ < i \leq q$ and $t\in J_k$, we have $g_i'(t) = -\frac1\qdim$ unless $g_i(t) = g_{p+\diff_+}(t)$, in which case $g_i'(t) = z(t)$, where $z(t)$ is defined below.
\end{itemize}
The number $z(t)$ appearing in the last two conditions can be computed by plugging the values of $g_i'$ appearing in those conditions into \eqref{onJk} and then solving for $z(t)$. If $g_{p+\diff_+}(t) < g_{p+\diff_+ + 1}(t)$, then $z(t)$ is taken to be undefined. In all of the above formulas, derivatives should be assumed to be taken from the right.

It is easy to check that these conditions uniquely determine the functions $g_{p+1},\ldots,g_q$ on the interval $J_k$, and that $S_+(\gg,t)\cap\OC pq_\Z = \{p+1,\ldots,p+\diff_+\} = S_+(\ff,t)\cap\OC pq_\Z$ for all $t\in J_k$. Since $\CO{t_0}\infty = \bigcup J_k \cup \bigcup \w I_k$, this implies that $S_\pm(\gg,t) = S_\pm(\ff,t)$ for all $t\geq t_0$.

To ensure that this does not lead to an inconsistency with \eqref{hpert}, we need to check that
\begin{equation}
\label{maxJk}
\gg(\max(J_k)) = \ff(\max(J_k)) + \aa(k).
\end{equation}
Since $\OC pq_\Z$ is an interval of equality for $\ff$ on $I_k$ and since $\diff_+,\diff_- > 0$, by \eqref{akrecursive} the map $i \mapsto f_i(\max(J_k)) + a_i(k)$ is constant on $\OC pq_\Z$. 

Suppose first that the map $i\mapsto g_i(\max(J_k))$ is also constant on $\OC pq_\Z$. Let $h(t) \df (G_q - G_p)(t) - (F_q - F_p)(t)$. Then \eqref{atminJk} implies that $h(\min(J_k)) = \sum_{p+1}^q a_i(k-1)$. Now by \eqref{akrecursive} this gives $h(\min(J_k)) = \sum_{p+1}^q a_i(k)$, and by \eqref{onJk} we have $h' = 0$ and thus $h(\max(J_k)) = \sum_{p+1}^q a_i(k)$. Rearranging gives the sum of the $i$th coordinate of \eqref{maxJk} over $i\in \OC pq_\Z$. Since both sides' $i$th coordinates are independent of $i$ for $i\in\OC pq_\Z$, this demonstrates \eqref{maxJk}.

On the other hand, suppose that $i\mapsto g_i(\max(J_k))$ is not constant on $\OC pq_\Z$. Then either there exists $p < i \leq p+\diff_+$ such that $g_i(t) < g_{p+\diff_+ +1}(t)$ for all $t\in J_k$, or there exists $p+\diff_+ < i \leq q$ such that $g_i(t) > g_{p+\diff_+}(t)$ for all $t\in J_k$. Without loss of generality suppose the first case holds. Then $g_{p+1}(t) < g_{p+\diff_+ + 1}(t)$ for all $t\in J_k$, and thus $g_{p+1}'(t) = \frac1\pdim$ for all $t\in J_k$. Now let
\begin{align*}
F(t) &\df \sum_{i = p+1}^q \big[f_i(t) - f_{p+1}(t)\big],\\
G(t) &\df \sum_{i = p+1}^q \big[g_i(t) - g_{p+1}(t)\big].
\end{align*}
Then $F(t),G(t) \geq 0$, $F(t_k) = 0$, and
\[
F'(t) \geq G'(t) = -Z \df \left[\frac{\diff_+}{\pdim} - \frac{\diff_-}{\qdim}\right] - \frac{q - p}{\pdim} = -\diff_-\left[\frac1{\pdim}+\frac1{\qdim}\right].
\]
It follows that
\begin{align*}
0 &\leq G(\max(J_k))
= G(\min(J_k)) - Z|J_k|\\
&\leq F(\min(J_k)) + 2d\|\Pert\|_{\infty} - Z|J_k|\\
&\leq F(t_k) + Z\big[k>0\big]\timeerror + 2d\|\Pert\|_{\infty} - Z|J_k|
= 2d\|\Pert\|_{\infty} - Z\timeerror < 0,
\end{align*}
where the last inequality follows from the definition of $\timeerror$ and the inequality $\diff_- \geq 1$. This is a contradiction, and therefore \eqref{maxJk} holds and so $\gg$ is continuous in a neighborhood of $\max(J_k)$. 

Thus $\gg$ is continuous on $\CO{t_0}\infty$. Indeed,  we have that $\gg$ is piecewise linear on $J_k$ by definition and on $\w I_k$ by \eqref{hpert}. Further, $\gg$ is continuous at the transition points $\min(J_k) \in \w I_{k-1}\cap J_k$ and $\max(J_k) \in J_k \cap \w I_k$. The former follows from \eqref{atminJk} and the latter from \eqref{maxJk}.

Recall that we have previously shown that $\gg \given \w I_k$ is a partial template. We leave the verification of the other conditions of Definition \ref{definitiontemplate} as an exercise to the reader. This concludes the proof of Lemma \ref{lemmahperturbation}.
\end{proof}

Now we combine the concept of perturbation vectors with the concept of $C$-matches introduced in \6\ref{subsectionministrategy}. The following lemma shows that by perturbing a template, it is possible to improve the constant $C$ appearing in Definition \ref{definitionCmatch}:

\begin{lemma}
\label{lemmaC0match}
Let $\Lambda$ be a $\cspace$-match (where $\cspace = \frac{\tspace}{2mnd!}$) for an $\tspace$-integral template $\ff$ at $t_0\in \tspace\N$, and let $\gg$ be the $\Pert$-perturbation of $\ff$ at $t_0$, where $\Pert \in (d!)^3(d^3)!\tbeta\Z^d$ is a perturbation vector of size $\|\bb\|_\infty < \cspace - C_1$ such that
\begin{equation}
\label{C0match}
\big\|\Mink(\Lambda) - [\ff(t_0) + \Pert]\big\|_{\infty} < C_1
\end{equation}
for some constant $C_1 \leq \cspace$. Suppose that $t_0$ is not a split with respect to $\ff$. Then $\gg$ is $\tbeta$-integral on $[t_0,t_2]$ and $\Lambda$ is a $C_1$-match for $\gg$ at $t_0$.
\end{lemma}

\begin{proof}
To show that $\gg$ is $\tbeta$-integral, we need to check both conditions (I) and (II) of Definition \ref{definitionintegral}. To show (II), we note that since $\ff$ is $\tspace$-integral, $\tspace \in \tbeta\N$, and $\Pert\in (d^3)!\tbeta\Z^d$, it follows by \eqref{gfh} that for all $1\leq i \leq d$ we have
\[
g_i(t_0) \in \frac{\tbeta}{mnd!}\Z.
\]
This is sufficient by the remark at the end of Definition \ref{definitionintegral}.
To show (I), we need to prove that all the corner points of $\gg$ are multiples of $\tbeta$. Suppose that $t$ is a corner point for $\gg$. Then $t \in J_k$ for some $k \in \{0,1,2\} $ (cf. \eqref{Jk}), and we let $t_* \df \min(J_k)$.
Recall from Lemma \ref{lemmahperturbation} that $s \df 2mnd^2 \|\bb\|_\infty \in (d^3)!\gamma\Z$, and so $t_\ast \in (d^3)!\gamma\Z$ as well.
Since $t$ is a corner point of $g$, there exists an interval of equality $\OC pq_\Z$ such that (i) the function $z$ appearing in the proof of Lemma \ref{lemmahperturbation} is well-defined at $t$ and (ii) if $p \leq i < p + M_\plus < j \leq q$ are minimal and maximal, respectively, so that $g_{i+1}(t) = \ldots = g_j(t) = z(t)$, then either $g_{i+1}'=1/m$ on $(t_*,t)$ or $g_j'=-1/n$ on $(t_*,t)$. Without loss of generality suppose the former holds, and also note that the definition of $z(t)$ implies that $g_{p+1}'=\ldots=g_i'=1/m$ and $g_{j+1}'=\ldots=g_q'=-1/n$ on $(t_*,t)$. Then, letting $\Delta t \df t - t_*$, on the one hand we have that 
\begin{align*}
\sum_{\ell = p+1}^q g_{\ell}(t)
&= \sum_{\ell = p+1}^i g_{\ell}(t) + \sum_{j+1}^q g_{\ell}(t) + (j - i) g_{i+1}(t)\\
&= \sum_{\ell = p+1}^i g_{\ell}(t_*) + (i - p) \frac{\Delta t}{m} +
\sum_{\ell = j+1}^q g_{\ell}(t_*) - (q - j) \frac{\Delta t}{n} +
(j - i) \left(g_{i+1}(t_*) + \frac{\Delta t}{m} \right),
\end{align*}
while, on the other, using \eqref{onJk}, we have that 
\[
\sum_{\ell = p+1}^q g_{\ell}(t)  = \sum_{\ell = p+1}^q g_{\ell}(t_*) + \left( \frac{M_+}{m} - \frac{M_-}{n}\right) \Delta t. 
\]
Solving for $\Delta t$ gives
\[
(A/m + A/n) \Delta t =  \sum_{\ell = i+1}^j g_{\ell}(t_*) - (j - i) g_{i+1}(t_*)
\]
where $A \df j - (p + M_+) \in \OC 0{M_-}_\Z$.

Now, since $\bb\in (d!)^3(d^3)!\gamma\Z^d$, we have $t_\ast \in (d!)^3(d^3)!\gamma\Z$ and further it follows from the definition in \eqref{akrecursive} that $a_{\ell}(k) \in (d!)^{3-k-1}(d^3)!\gamma\Z$. Thus for all $p < \ell \leq q$,
\[
g_{\ell}(t_*) = f_{\ell}(t_*) + a_{\ell}(k) \in (d!)^{3-k-1}(d^3)!\gamma\Z
\]
and thus
\[
(A/m + A/n) \Delta t \in (d^3)! \gamma \Z .
\]
and therefore that 
\[
(m+n) A (t - t_*) \in (d^3)! \gamma \Z.
\]
It thus follows that $t \in \gamma \Z$. This completes the proof that $\gg$ is $\tbeta$-integral on $[t_0,t_2]$.

Since $\gg(t_0) = \ff(t_0) + \Pert$, \eqref{C0match} implies that condition (I) of Definition \ref{definitionCmatch} holds with $C=C_1$ and $t = t_0$. Let $I_+$ be an interval of linearity for both $\ff$ and $\gg$ whose left endpoint is $t_0$, and let $\OC pq_\Z$ be an interval of equality for $\ff$ on $I_+$. 
Then $f_q(t) < f_{q+1}(t)$ for $t\in I_+$, so since $\ff$ is $\tspace$-integral and $\|\bb\|_\infty < \cspace - C_1$, by \eqref{C0match} we have $\mink_q(\Lambda) < \mink_{q+1}(\Lambda)$. For each $j\in \OC pq_\Z\cap Q(t_0)$, let $V_j$ be a $\Lambda$-rational subspace of the linear span of $\{\rr\in\Lambda:\|\rr\|\leq\lambda_j(\Lambda)\}$ of dimension $j$. 
Then
\begin{align*}
\dim(V_j\cap \vert) &\geq \max\big(\dim(V_p\cap\vert),\dim(V_q\cap\vert)-(q-j)\big)\\
&\geq \max\big(\dir_-(\ff,I_+,p),\dir_-(\ff,I_+,q)-(q-j)\big) = \dir_-(\ff,I_+,j) = \dir_-(\gg,I_+,j),
\end{align*}
where the second-to-last equality follows from the assumption that $t_0$ is not a split for $\ff$. This demonstrates \eqref{VqL}.

 This concludes the proof of Lemma \ref{lemmaC0match}.
\end{proof}

\subsection{Uniform error bounds}
\label{subsectionsummary}

We are now ready to complete the proof of \eqref{lowerbound}. First, by Lemma \ref{lemmafindtemplate} we can without loss of generality assume that $\ff$ is simple and that its corner points are all multiples of $2\tspace$, where $\tspace = \kspace \tbeta$, $\kspace \in (d!)^3 (d^3)! \N$ is large to be determined, and $\tbeta$ is as above. After translating by $\tspace$, we can assume that the corner points are at odd multiples of $\tspace$ instead of even multiples. We can now define Alice's strategy as follows: Fix $\ell\in\N$ and let $k_\ell = 2\ell \kspace$, and suppose that the game has progressed to turn $k_\ell$. This means that the lattice $\Lambda^{(\ell)} \df \Lambda_{k_\ell}$ has already been defined.
\begin{itemize}
\item If $\Lambda^{(\ell)}$ is not a $\cspace$-match for $\ff$ at $t_\ell \df k_\ell \tbeta = 2\ell \kspace \tbeta$, then Alice resigns (plays arbitrarily) on turn $k_\ell$.
\item Suppose that $\Lambda^{(\ell)}$ is a $\cspace$-match for $\ff$ at $t_\ell$. Let $\Pert = \Pert^{(\ell)}$ be the element of $(d!)^3(d^3)! \tbeta\Z^d$ closest to $\Mink(\Lambda^{(\ell)}) - \ff(t_\ell)$ (using any tiebreaking mechanism). Then $\Pert$ is a perturbation vector satisfying \eqref{C0match} with $\Lambda = \Lambda^{(\ell)}$, $t_0 = t_\ell$, and $C_1 = (1/2)(d!)^3(d^3)! \tbeta$. Let $\gg = \gg^{(\ell)}$ be the $\Pert$-perturbation of $\ff$ at $t_\ell$. Then by Lemma \ref{lemmaC0match}, $\gg$ is $\tbeta$-integral and $\Lambda^{(\ell)}$ is a $C_1$-match for $\gg$ at $t_\ell$. This allows us to apply Lemma \ref{lemmaministrategy} (setting $k_1$ and $k_2$ Lemma \ref{lemmaministrategy} to be $k_\ell$ and $k_{\ell+1}$, respectively), and on turns $k_\ell,\ldots,k_{\ell+1}-1$ Alice plays the strategy given by this lemma.
\end{itemize}

We assume that Alice does not resign at turn $k_{\ell}$.
Let $t_\ell' \df (2\ell + 1)\tspace$. Since $\ff$ is linear on $I_0^{(\ell)} \df [t_\ell,t_\ell']$ and $I_1^{(\ell)} \df [t_\ell',t_{\ell+1}]$, 
it follows that $\gg$ is linear on $[t_\ell + s,t_\ell' - s]$ and $[t_\ell' + s,t_{\ell+1}]$. On the other hand, note that in the proof of Lemma \ref{lemmahperturbation}, \eqref{onJk} and the following bullet points imply that on each interval $J_k = [t_\ell,t_\ell + s]$ or $[t_\ell' - s,t_\ell' + s]$, $\gg$ only changes slopes at points $t$ such that $g_i < g_j$ on $(\min(J_k),t)$ and $g_i = g_j$ on $(t,\max(J_k))$ for some $i < j$. It follows that $\gg$ has at most $d-1$ maximal intervals of linearity on $J_k$, and thus at most $2d$ maximal intervals of linearity on $I_\ell$. In particular we have $\niol\leq 2d$ in Lemma \ref{lemmaministrategy}.

To compute the relation between $\Pert^{(\ell)}$ and $\Pert^{(\ell+1)}$, we let $\aa^{(\ell)}:\N\cup\{-1\}\to\R^d$ be the function defined in \6\ref{subsectionerrorcorrection}, so that $\aa^{(\ell)}(-1) = \Pert^{(\ell)}$. Then we have
\begin{align*}
\gg^{(\ell)} &= \ff + \aa^{(\ell)}(0) \text{ on } \w I_0^{(\ell)} = [t_\ell+\timeerror,t_\ell'-\timeerror],\\
\gg^{(\ell)} &= \ff + \aa^{(\ell)}(1) \text{ on } \w I_1^{(\ell)}\cup J_2^{(\ell)}\cup \w I_2^{(\ell)} = [t_\ell'+\timeerror,t_{\ell+1}'-\timeerror].
\end{align*}
The second equality follows from the fact that $t_2^{(\ell)} = t_{\ell+1}$ is not a corner point of $\ff$, so $\aa^{(\ell)}(1) = \aa^{(\ell)}(2)$ and thus Case 1 of the proof of Lemma \ref{lemmahperturbation} applies. In particular, we have
\[
\gg^{(\ell)}(t_{\ell+1}) = \ff(t_{\ell+1}) + \aa^{(\ell)}(1).
\]
On the other hand, according to part (ii) of Lemma \ref{lemmaministrategy}, $\Lambda^{(\ell+1)} \df \Lambda_{k_{\ell+1}}$ is a $C_2$-match for $\gg^{(\ell)}$ at $t_{\ell+1}$, where $C_2$ is a constant depending only on $C_1$. Thus, using (I) of Definition \ref{definitionCmatch}, we have
\[
\big\|\Mink(\Lambda^{(\ell+1)}) - [\ff(t_{\ell+1}) + \aa^{(\ell)}(1)]\big\|_{\infty} \leq C_2,
\]
and so by the definition of $\Pert^{(\ell+1)}$, we have
\begin{equation}
\label{hrecursive}
\|\Pert^{(\ell+1)} - \aa^{(\ell)}(1)\|_{\infty} \leq C_2 + C_1,
\end{equation}
assuming that Alice does not resign on turn $k_{\ell+1}$.

Assume now that there exists a constant $B > 0$ (which is independent of Bob's strategy) such that
\begin{equation}
\label{ETSsummary}
\|\Pert^{(\ell)}\|_{\infty} \leq B \text{ for all }\ell \text{ such that Alice does not resign on or before turn $k_\ell$}.
\end{equation}
Fix $\ell$ such that Alice does not resign on or before turn $k_\ell$. Then $\Lambda^{(\ell+1)}$ is a $C_2$-match for $\gg^{(\ell)}$ at $t_{\ell+1}$, and is therefore a $(C_2+B)$-match for $\ff$ at $t_{\ell+1}$, since $\|\aa^{(\ell)}(1)\|_{\infty} \leq \|\Pert^{(\ell)}\|_{\infty} \leq B$. Letting $k_\tspace$ be large enough so that $\tspace \geq 4\dimprod d!(C_2+B)$, we see that $\Lambda^{(\ell+1)}$ is a $\cspace$-match for $\ff$ at $t_{\ell+1}$, and thus Alice does not resign on turn $k_{\ell+1}$. So by induction Alice never resigns.

So for all $\ell \in \N$, $\Lambda^{(\ell)}$ is a $C$-match for $\ff$ at $t_\ell$, where $C \df C_2 + B$. It follows from Definition \ref{definitionCmatch} that for all $\ell \in \N$
\[
\| \hh(\Lambda_\ell) - \ff(t_\ell) \|_\infty \leq C
\]
and so using Lemma \ref{lemmaMS} gives us that the final outcome $X_\infty$ (as defined by \eqref{outcome2}) is in the target set $\DD(\ff,C_\epsilon)$, where $C_\epsilon$ is a constant depending on $C$ (and thus on $\epsilon$ as per the next paragraph). 

To compute Alice's score, we use part (iii) of Lemma \ref{lemmaministrategy} to get that
\begin{align*}
\Delta(\AA,[0,k_\ell])
= \frac1\ell \sum_{j = 0}^{\ell - 1} \Delta(\AA,[k_j,k_{j+1}])
&= \frac1\ell \sum_{j = 0}^{\ell - 1} \delta(\gg^{(j)},[t_j,t_{j+1}]) + O\left(\tfrac1{\tbeta} + \tfrac1{2\kspace}\right)\\
&= \frac1\ell \sum_{j = 0}^{\ell - 1} \delta(\ff,[t_j,t_{j+1}]) + O\left(\tfrac1{\tbeta} + \tfrac1{2\kspace} + \tfrac{s}{\eta}\right)\\
&= \delta(\ff,[0,t_\ell]) + O\left(\tfrac1{\tbeta} + \tfrac1{\kspace} + \tfrac{B}{\eta}\right)
\end{align*}
and thus after taking liminfs on both sides we have
\begin{align*}
\underline\delta(\AA) &= \underline\delta(\ff) + O\left(\tfrac1{\tbeta} + \tfrac1{\kspace} + \tfrac{B}{\eta}\right).
\end{align*}
Given $\epsilon > 0$, we can choose $\beta$ small enough (and so $\tbeta$ large enough) and $\eta$ (and thus $\kspace$) large enough so that the last term is less than $\epsilon$, which shows that $\underline\delta(\AA) \geq \underline\delta(\ff) - \epsilon$, and thus $\DD(\ff,C_\epsilon)$ is $(\underline\delta(\ff)-\epsilon)$-dimensionally Hausdorff $\beta$-winning. Applying Theorem \ref{theoremHPgame} shows that $\HD(\DD(\ff,C_\epsilon)) \geq \underline\delta(\ff)-\epsilon$. 
Now in the above argument we can replace all $\underline\delta$s by $\overline\delta$s, and all liminfs by limsups, to prove that  $\PD(\DD(\ff,C_\epsilon)) \geq \overline\delta(\ff)-\epsilon$.
This completes the proof of \eqref{lowerbound} assuming \eqref{ETSsummary}. In what follows we will prove \eqref{ETSsummary}.

The basic idea is as follows: Since the perturbation vectors $\Pert^{(\ell)}$ satisfy the approximate functional equation \eqref{hrecursive} (where $\aa^{(\ell)}(-1) = \Pert^{(\ell)}$), we can view each vector $\Pert^{(\ell+1)}$ approximately as the vector $\Pert^{(\ell)}$ with some ``mixing'' done to it, in accordance with \6\ref{subsectionerrorcorrection}. Since the perturbation vectors satisfy the approximate relation $\sum_1^d \pert_q^{(\ell)} \asymp_\plus 0$, this mixing process will tend to cause $\pert_q^{(\ell)} \asymp_\plus 0$ for all $q = 1,\ldots,d$, but only if all $q$ are mixed together. Since it may be a long time between times when some $q$ is mixed with $q+1$, we need to keep track of what is happening on long intervals where $q$ and $q+1$ do not mix. This leads to our next definition:

Given $q = 0,\ldots,d$, an interval $[\ell_1,\ell_2]$ will be called a \emph{$q$-interval} if either
\[
\label{qinterval}
f_q < f_{q+1} \text{ on }(t_{\ell_1-1}',t_{\ell_2}')
\]
or
\[
f_q' = f_{q+1}' = c \text{ on } (t_{\ell_1-1}',t_{\ell_2}'), \text{ where } c \in \exceptionalset \cdot
\]
Note that every interval is both a $0$-interval and a $d$-interval (according to our convention that $f_0 = -\infty$ and $f_{d+1} = +\infty$).
\begin{claim}
\label{claiml1l2}
Fix $q=1,\ldots,d-1$ and let $[\ell_1,\ell_2]$ be a $q$-interval. Then there exists a constant $\alpha = \alpha(q,\ell_1,\ell_2)$ such that for all $\ell = \ell_1,\ldots,\ell_2$, we have
\begin{equation}
\label{l1l2}
\sum_{i=1}^q \pert_i^{(\ell)} \asymp_{\plus,\beta} \alpha(q,\ell_1,\ell_2).
\end{equation}
\end{claim}
\begin{proof}
First suppose that $f_q < f_{q+1} \text{ on }(t_{\ell_1-1}',t_{\ell_2}')$. By the definition of $\Pert^{(\ell)}$, we have
\[
\sum_{i=1}^q \pert_i^{(\ell)} \asymp_{\plus,\beta} \sum_{i=1}^q h_i(\Lambda^{(\ell)}) - F_q(t_\ell).
\]
Let $V^{(\ell)}_q$ be as in Lemma \ref{lemmaminkowskibasis} (applied to the lattice $\Lambda^{(\ell)}$), and recall we have 
\[
\sum_{i=1}^q h_i(\Lambda^{(\ell)}) \asymp_{\plus} \log\|V^{(\ell)}_q\|.
\]
Since $f_q < f_{q+1}$ on $(t_{\ell_1-1}',t_{\ell_2}')$, we have $g u_{\bfB_\ell} V^{(\ell)}_q = V^{(\ell+1)}_q$ for all $\ell$. 

Since we know that Alice is following her strategy, as defined in the proof of Lemma \ref{lemmaministrategy}, then for each turn $k = k_{\ell_1} - \kspace,\ldots,k_{\ell_2} + \kspace$ the matrix $X_k$ is good on turn $k$ (as defined in \eqref{good1}-\eqref{good2}\Footnote{Technically the notion of being good on turn $k$ depends on the choice of $\ff$ in Lemma \ref{lemmaministrategy}, which in the current situation is $\ff = \gg$. However, by \eqref{Spmgtft} the notion of being good on turn $k$ is the same for $\ff = \ff$ and $\ff = \gg$.}). Thus the hypotheses of Lemma \ref{lemmapushoffvert} are satisfied (with $\dir_\pm = \dir_\pm(q) = \#(S_\pm \cap [0,q])$), and it then follows that 
\[
\log\|V^{(\ell')}_q\| - \log\|V^{(\ell)}_q\| \asymp_{\plus,\beta} (t_{\ell'} - t_\ell) z
\]
for any $\ell < \ell'$ such that $F_q' = z$ on $(t_\ell,t_{\ell'})$. Note that a similar argument appeared earlier in the paragraph containing \eqref{Vqk2}.
Now since $F_q$ is piecewise linear on $(t_{\ell_1-1}',t_{\ell_2}')$ with a bounded number of intervals of linearity, it follows that
\[
\log\|V^{(\ell)}_q\| - \log\|V^{(\ell_1)}_q\| \asymp_{\plus,\beta} \int_{t_{\ell_1}}^{t_\ell} F_q' = F_q(t_\ell) - F_q(t_{\ell_1}).
\]
Thus, letting
\[
\alpha(q,\ell_1,\ell_2) \df \log\|V^{(\ell_1)}_q\| - F_q(t_{\ell_1})
\]
completes the proof of the claim in the case $f_q < f_{q+1} \text{ on }(t_{\ell_1-1}',t_{\ell_2}')$.

Now suppose that $f_q = f_{q+1}$ and $f_q' = f_{q+1}' = \tfrac1m$ on $I_0 = (t_{\ell_1-1}', t_{\ell_2}')$ (the case where $f_q' = f_{q+1}' = -\tfrac1n$ on $I_0$ proceeds similarly). For each interval of linearity $I\subset I_0$ and for each $t\in I$, let $\OC{p(t)}{r(t)}_\Z$ be the interval of equality for $\ff$ on $I$ that contains $q$. Since $\tfrac1m$ is the maximum possible derivative for any $f_j$ and since $f_q' = f_{q+1}' = \tfrac1m$ on $I_0$, it follows that $f_q = f_{q+1}$ can only merge with $f_j$ if $j > q+1$, and can only split from $f_j$ if $j < q$; equivalently, $p$ and $r$ are increasing functions. In fact, since $\ff$ is simple, such merges and splits are not possible at all, as they would require a simultaneous transfer to account for the fact that $f_j'(t^-) < \frac1m = f_j'(t^+)$ (in the case of a merge) or $f_j'(t^+) > -\frac1n = f_j'(t^-)$ (in the case of a split). Thus, $p$ and $r$ are constant. But then by the same logic as before we have
\begin{align*}
\log\|V^{(\ell')}_p\| - \log\|V^{(\ell)}_p\| 
&\asymp_{\plus,\beta} (t_{\ell'} - t_\ell) z_p,\\
\log\|V^{(\ell')}_r\| - \log\|V^{(\ell)}_r\| 
&\asymp_{\plus,\beta} (t_{\ell'} - t_\ell) z_r,
\end{align*}
where $z_p,z_r \in [-\tfrac1n,\tfrac1m]$ satisfy 
\begin{equation}
\label{zrzp}
z_r - z_p = \frac{r-p}{m} \cdot
\end{equation}
Now let
\[
h \df (g u_{\bfB_{\ell'-1}})\cdots (g u_{\bfB_{\ell}})
\]
and note that $h V_j^{(\ell)} = V_j^{(\ell')}$. Furthermore, we have
\[
\|h\| \lesssim \exp \left(2(\ell'-\ell)\kspace \cdot\tfrac1m \right) = \exp\left((t_{\ell'} - t_\ell)\cdot\tfrac1m\right)
\]
where $\|\cdot\|$ denotes the operator norm. Letting $\wbar h:V_q^{(\ell)} / V_p^{(\ell)} \to V_q^{(\ell')} / V_p^{(\ell')}$ be the induced map, we have $\|\wbar h\| \leq \|h\|$ and thus
\[
\|V_q^{(\ell')} / V_p^{(\ell')}\| \lesssim \|\wbar h\|^{q-p} \|V_q^{(\ell)} / V_p^{(\ell)}\|
\lesssim \exp\left((t_{\ell'} - t_\ell)\tfrac{q-p}{m}\right) \| V_q^{(\ell)} / V_p^{(\ell)} \|.
\]
Since the covolume of a quotient space is the quotient of the covolumes, taking logarithms gives
\begin{align*}
\log\|V_q^{(\ell')}\| - \log \|V_q^{(\ell)}\| 
&\lesssim_\plus \log\|V_p^{(\ell')}\| - \log \|V_p^{(\ell)}\| + (t_{\ell'} - t_\ell)\tfrac{q-p}m\\
&\asymp_\plus (t_{\ell'} - t_\ell)\left(z_p + \tfrac{q-p}m\right)
\end{align*}
A similar argument shows that
\[
\log\|V_q^{(\ell')}\| - \log \|V_q^{(\ell)}\| \gtrsim_\plus  (t_{\ell'} - t_\ell)\left(z_r - \tfrac{r-q}m\right).
\]
Next, we observe that \eqref{zrzp} can be rearranged to yield
\[
z_q \df z_p + \frac{q-p}m =  z_r - \frac{r-q}m
\]
and thus we have
\[
\log\|V_q^{(\ell')}\| - \log \|V_q^{(\ell)}\| \asymp_\plus (t_{\ell'} - t_\ell) z_q.
\]
The proof can be continued in the same way as in the earlier case. 
A similar argument applies if $f_q' = f_{q+1}' = -\tfrac1n$ on $(t_{\ell_1-1}', t_{\ell_2}')$. This concludes the proof of Claim \ref{claiml1l2}.
\end{proof}

Now fix $\ell\in\N$ and $q = 1,\ldots,d-1$. If $\ell$ is contained in a $q$-interval then we let
\[
\alpha(q,\ell) = \alpha(q,\ell_1,\ell_2),
\]
where $[\ell_1,\ell_2]$ is the longest $q$-interval containing $\ell$. Otherwise, we let $\alpha(q,\ell) = *$. Next, we let $\cc^{(\ell)} \in \R^d$ be the unique vector such that
\begin{align} \label{equality}
\sum_{i=1}^q c_i^{(\ell)} &= \alpha(q,\ell) &\text{ when } \alpha(q,\ell)\in\R,\\
c_q^{(\ell)} &= c_{q+1}^{(\ell)} &\text{ when } \alpha(q,\ell) = *.
\end{align}
Then by \eqref{akrecursive} and \eqref{l1l2}, we have 
\[
\cc^{(\ell)} \asymp_{\plus,\beta} \Pert^{(\ell)}.
\]

For convenience, we introduce a slightly modified version of intervals of equality (see Definition \ref{definitiondimtemplate}). We call an interval $\OC pq_\Z$ an \emph{interval of mixing} for $\ff$ on $I$ if either
\begin{itemize}
\label{intervalmixing}
\item $\OC pq_\Z$ is an interval of equality for $\ff$ on $I$, and $f_q' \notin \exceptionalset$ on $I$, or
\item $q = p+1$ and $f_q' \in \exceptionalset$ on $I$.
\end{itemize}
Note that if $\OC pq_\Z$ is an interval of mixing for $\ff$ on $I_\ell \df (t_{\ell-1}',t_\ell')$, then $[\ell,\ell]$ is both a $p$-interval and a $q$-interval.

Let $\OC pq_\Z$ be an interval of mixing for $\ff$ on $I_\ell$. Then by \eqref{akrecursive}, we have $a_i^{(\ell-1)}(1) = a$ for all $i \in \OC pq_\Z$, where $a$ is a constant. By \eqref{hrecursive}, we have $|\pert_i^{(\ell)} - a| \leq C_2+C_1$ for all $i\in \OC pq_\Z$, and thus by \eqref{akrecursive}, we have $|a_i^{(\ell)}(0) - a| \leq C_2 + C_1$ for all $i\in \OC pq_\Z$. On the other hand, for $i = 1,\ldots,d$ such that $f_i'\in \exceptionalset$ on $I_\ell$, \eqref{akrecursive} implies that $a_i^{(\ell)}(0) = a_i^{(\ell)}(-1) = \pert_i^{(\ell)}$. Thus
\[
\|\aa^{(\ell)}(0) - \Pert^{(\ell)}\| \leq 2(C_2+C_1)
\]
and consequently $\aa^{(\ell)}(0) \asymp_{\plus,\beta} \cc^{(\ell)}$. On the other hand, by \eqref{hrecursive} we have $\aa^{(\ell)}(1) \asymp_{\plus,\beta} \cc^{(\ell+1)}$, so by \eqref{akrecursive}, for every interval of mixing $\OC pq_\Z$ for $\ff$ on $I_{\ell+1}$, we have
\begin{equation}
\label{hiPhi}
c_i^{(\ell+1)} \asymp_{\plus,\beta}
\frac{1}{q-p} \sum_{j=p+1}^q c_j^{(\ell)}.
\end{equation}
(This is the reason we use the term ``interval of mixing''; the quantities $(c_i^{(\ell)})$ get ``mixed'' within the interval of mixing.) Let $B$ be a large number, fix $\epsilon\in \{\pm1\}$, and write $\wbar c_i^{(\ell)} = B + \epsilon c_i^{(\ell)}$. We claim that there exist constants $C(1),\ldots,C(d) \geq 0$, independent of $B$, such that if $B \geq \max_i C(i)/i$, then for all $j < k$ and $\ell\in\N$ we have
\begin{equation}
\label{kjBb}
\sum_{i = j + 1}^k \wbar c_i^{(\ell)} \geq C(k-j).
\end{equation}
Indeed, when $\ell = 0$, we have $\cc^{(0)} = \0$ and thus $\sum_{i=j+1}^k \wbar c_i^{(\ell)} = (k-j)B \geq C(k-j)$. For the inductive step, fix $\ell\in\N$ and suppose that \eqref{kjBb} holds for all $j < k$. Fix $j < k$, and we will show that
\begin{equation}
\label{kjBb2}
\sum_{i = j + 1}^k \wbar c_i^{(\ell+1)} \geq C(k-j).
\end{equation}
{\bf Case 1.} Suppose that $[\ell,\ell+1]$ is both a $j$-interval and a $k$-interval. Then $\alpha(j,\ell) = \alpha(j,\ell+1)\in\R$ and similarly for $k$. So
\[
\sum_{i = j + 1}^k c_i^{(\ell)} = \alpha(k,\ell) - \alpha(j,\ell) = \alpha(k,\ell+1) - \alpha(j,\ell+1) = \sum_{i = j + 1}^k c_i^{(\ell+1)}
\]
and thus
\[
\sum_{i = j + 1}^k \wbar c_i^{(\ell+1)} = \sum_{i = j + 1}^k \wbar c_i^{(\ell)} \geq C(k-j).
\]
{\bf Case 2.} Suppose that $[\ell,\ell+1]$ is a $j$-interval but not a $k$-interval. Let $\OC pq_\Z \ni k$ be an interval of mixing for $\ff$ on either $I_\ell$ or $I_{\ell+1}$. Then by \eqref{hiPhi}, we have
\begin{equation}
\label{hiPhicorollary}
c_i^{(\ell+1)} \asymp_{\plus,\beta} \frac1{q-p} \sum_{i=p+1}^q c_i^{(\ell)} \text{ for all } i\in \OC pq_\Z.
\end{equation}
In the latter case this follows directly from \eqref{hiPhi}, while if $\OC pq_\Z$ is an interval of mixing for $\ff$ on $I_\ell$, then by \eqref{hiPhi} we have $c_i^{(\ell)} \asymp_{\plus,\beta} c$ for all $i\in \OC pq_\Z$ for some constant $c$, and applying \eqref{hiPhi} again gives \eqref{hiPhicorollary}. On the other hand, since $[\ell,\ell+1]$ is a $j$-interval we have $j\leq p$, and thus the previous case gives
\[
\sum_{i=j+1}^p \wbar c_i^{(\ell)} = \sum_{i=j+1}^p \wbar c_i^{(\ell+1)}.
\]
So
\begin{align*}
\sum_{i=j+1}^k \wbar c_i^{(\ell+1)}
&\asymp_{\plus,\beta} \sum_{i=j+1}^p \wbar c_i^{(\ell)}
+ \frac{k-p}{q-p} \sum_{i=p+1}^q \wbar c_i^{(\ell)}\\
&\geq_\pt \frac{q-k}{q-p}C(p-j) + \frac{k-p}{q-p} C(q-j) \by{\eqref{kjBb}}\\
&\geq_\pt 0 + \frac1d C(q-j) \geq \frac1d C(k+1-j).
\end{align*}
Let $C$ denote the implied constant of the asymptotic, and let $C(1),\ldots,C(d)$ be defined by the recursive formula
\begin{align*}
C(1) &\df 0,&
C(k+1) &\df d(C(k)+C).
\end{align*}
Then we have demonstrated \eqref{kjBb2}, completing the inductive step.\\

\noindent {\bf Case 3.} If $[\ell,\ell+1]$ is a $k$-interval but not a $j$-interval, or is neither a $j$-interval nor a $k$-interval, then the proof is similar to Case 2. We leave the details to the reader.\\

This completes the proof of \eqref{kjBb}, which in turn implies \eqref{ETSsummary}, thereby completing the proof of \eqref{lowerbound}. Thus, we have completed proving the lower bounds in Theorem \ref{theoremvariationaluniform}.

\draftnewpage
\section{Proof of Theorem \ref{theoremvariationaluniform}, upper bound}
\label{sectionupper}
Let $\SS$ be a class of functions from $\Rplus$ to $\R^d$. We claim that for all $\epsilon > 0$ there exists $C_\epsilon > 0$ such that
\begin{align*}
\HD(S) &\leq \sup_{\ff\in\NN(\SS,C_\epsilon)\cap \TT_\dims} \underline\delta(\ff) + \epsilon,&
\PD(S) &\leq \sup_{\ff\in\NN(\SS,C_\epsilon)\cap \TT_\dims} \overline\delta(\ff) + \epsilon,
\end{align*}
where $S = \DD(\SS)$ is as in \eqref{MSdef}. As in the proof of the lower bounds, we will play the modified Hausdorff and packing games with target set $S$ and parameter $0<\beta<1$. This time, we will define a strategy for Bob for sufficiently small $\beta$.\\

{\bf Definition of the strategy.} Suppose that the game has progressed to turn $\turn$, with corresponding lattice $\Lambda_\turn$ as in \6\ref{sectiongamesgeometry}. Let $\{\rr_1,\ldots,\rr_d\}$ be a Minkowski basis of $\Lambda_\turn$ (cf. Lemma \ref{lemmaminkowskibasis}), and for each $q = 0,\ldots,d$ let $V_q = \sum_{i=1}^q \R\rr_i$. Essentially, Bob's strategy will be to ``push the subspaces $V_q$ away from $\vert$ as much as possible given Alice's move''. To make this precise, fix $\bfB\in B_\MM(\0,1-\beta)$, and for each $q = 0,\ldots,d$ let
\begin{align}
\label{Dqdef}
\dir_q^- = \dir_q^-(k,\bfB) &\df \sup_{\|\bfC\| \leq 2\beta} \dim(u_{\bfB + \bfC} V_q\cap \vert),&
\dir_q^+ &\df q - \dir_q^-.
\end{align}
Let
\begin{align*}
S_+ = S_+(k,\bfB) &\df \{q = 1,\ldots,d : L_q^+ = L_{q-1}^+ + 1 \text{ and } L_q^- = L_{q-1}^-\},\\
S_- = S_-(k,\bfB) &\df \{q = 1,\ldots,d : L_q^- = L_{q-1}^- + 1 \text{ and } L_q^+ = L_{q-1}^+\},
\end{align*}
and note that $S_+\cup S_- = \dset$ and $\#(S_\pm) = d_\pm$, where $d_+ = \pdim$, $d_- = \qdim$ (see \eqref{slopeset}). Also note that $L_q^\pm = \#(S_\pm \cap \OC 0q_\Z)$ for all $q = 1,\ldots,d$. Finally, let $\delta(\turn,\bfB) = \delta(S_+,S_-)$, where as in \eqref{deltaT+T-},
\[
\delta(T_+,T_-) \df \#\{(i_+,i_-) \in T_+\times T_-: i_+ < i_-\}.
\]

Bob's strategy on turn $\turn$ can now be given as follows: If Alice makes the move $A_\turn \subset B_\MM(\0,1-\beta)$, then Bob responds by choosing $\bfB_\turn \in A_\turn$ so as to maximize $\delta(\turn,\bfB_\turn)$. Note that larger values of $\delta(\turn,\bfB_\turn)$ correspond to larger values of $L_q^+$ and correspondingly smaller values of $L_q^-$, which in turn correspond to the intuitive idea of ``pushing $V_q$ away from $\vert$ (by a distance of at least $2\beta$)''.

The following claim will be used to relate scores in the Hausdorff and packing games with the dimensions of templates.

\begin{claim}
\label{claimAkupper}
For all $\turn$ we have
\[
\#(A_\turn) \lesssim \beta^{-\delta(\turn,\bfB_\turn)}.
\]
\end{claim}
\begin{proof}
Let $\delta = \delta(\turn,\bfB_\turn)$. Clearly,
\[
A_\turn \subset \{\bfB : \delta(\turn,\bfB) \leq \delta\} \subset \bigcup_{T_+,T_-} \big\{\bfB : \dir_q^-(\bfB) \geq \#(T_-\cap \OC 0q_\Z) \text{ for all }q = 1,\ldots,d\big\},
\]
where the union is taken over all sets $T_+,T_-\subset \dset$ such that $T_+\cap T_- = \emptyset$, $T_+ \cup T_- = \dset$, $\#(T_\pm) = d_\pm$, and
\[
\delta(T_+,T_-) \leq \delta.
\]
Fix $T_+,T_-$ as above, and for each $q = 1,\ldots,d$ let $\what\dir_q^- = \#(T_-\cap \OC 0q_\Z)$. We need to estimate the size of the set
\[
A_\turn(T_+,T_-) \df \{\bfB \in A_\turn : \dir_q^-(\bfB) \geq \what\dir_q^- \text{ for all }q\}.
\]
Since $A_\turn = \bigcup_{T_+,T_-} A_\turn(T_+,T_-)$, to complete the proof it suffices to show that
\[
\#(A_\turn(T_+,T_-)) \lesssim \beta^{-\delta}.
\]
Note that for each $\bfB\in B_\MM(\0,1-\beta)$ and $q = 1,\ldots,d$, we have $\dir_q^-(\bfB) \geq \what\dir_q^-$ if and only if $\bfB$ is in the $2\beta$-neighborhood of the algebraic set
\[
\ZZ_q = \{\bfB : \dim(u_\bfB V_q \cap \vert) \geq \what\dir_q^-\} \subset \MM.
\]
Thus,
\[
A_\turn(T_+,T_-) \subset \bigcap_{q=1}^d \NN(\ZZ_q,2\beta).
\]
Let $\ZZ = \bigcap_{q=1}^d \ZZ_q$. We claim that
\begin{equation}
\label{2betaCbeta}
\bigcap_{q=1}^d \NN(\ZZ_q,2\beta) \subset \NN(\ZZ,2d \beta)
\end{equation}
Indeed, fix $\bfB \in \bigcap_q \NN(\ZZ_q,2\beta)$. For each $q = 1,\ldots,d$, choose $\bfB_q \in \ZZ_q \cap B(\bfB,2\beta)$, and let
\[
\what V_q = u_{\bfB_q} V_q\cap \vert.
\]
Next, for $q = 1,\ldots,d$ we recursively define
\[
W_q = \what V_q \cap \what W_{q-1}^\perp,\;\;\;\;
\what W_q = W_1 + \ldots + W_q,
\]
with the understanding that $\what W_0 = \{\0\}$. Note that since $\bfB_q \in \ZZ_q$,
\[
\dim(\what W_q) = \dim\big(\what W_{q-1} + (\what V_q\cap \what W_{q-1}^\perp)\big) \geq \dim(\what V_q) \geq \what\dir_q^-.
\]
Let $\bfD$ be the unique matrix such that $u_{-\bfD} \vv = u_{-\bfB_q} \vv$ for all $q = 1,\ldots,d$ and $\vv\in W_q$. Since $\vert = \what W_d = W_1 + \ldots + W_d$ is an orthogonal decomposition, such a $\bfD$ exists, and we have $\|\bfD - \bfB\| \leq \sum_{q=1}^d \|\bfB_q - \bfB\| \leq 2d \beta$. Now fix $q = 1,\ldots,d$. For all $p = 1,\ldots,q$ and $\vv\in W_p$, we have $u_{-\bfD} \vv = u_{-\bfB_p} \vv \in V_p \subset V_q$. This implies that $u_{-\bfD} \what W_q \subset V_q$ and thus
\[
\dim(u_\bfD V_q\cap \vert) \geq \dim(\what W_q) \geq \what\dir_q^-,
\]
so $\bfD \in \ZZ_q$. Since $q$ was arbitrary, we have $\bfD\in \ZZ$, and thus $\bfB \in \NN(\ZZ,2d\beta)$. This completes the proof of \eqref{2betaCbeta}.

So $A_\turn(T_+,T_-) \subset \NN(\ZZ,2d \beta)$, where $A_\turn(T_+,T_-)$ is a $3\beta$-separated set and $\ZZ$ is an algebraic set whose diagram in the sense of \cite[Definition 4.2]{YomdinComte} is constant (i.e. independent of $\turn$, $\beta$, and $\ff$). By \cite[Corollary 5.7]{YomdinComte}, it follows that
\[
\#(A_\turn(T_+,T_-)) \lesssim \beta^{-\dim(\ZZ)},
\]
whereas we wish to show that $\#(A_\turn(T_+,T_-)) \lesssim \beta^{-\delta}$. So to complete the proof we must show that $\dim(\ZZ) \leq \delta$.

Consider first the case where the subspaces $V_q$ ($q = 1,\ldots,d$) are all coordinate subspaces, i.e. $V_q = \sum_{i\in I_q} \R \ee_i$ for some $I_q \subset \dset$, and where $\dim(V_q\cap \vert) = \what\dir_q^-$ for all $q$. In this case, we write $I_- = \{m+1,\ldots,d\}$, so that $\vert = \sum_{i\in I_-} \R\ee_i$. Let $\sigma$ be the unique permutation of $\dset$ such that for each $q = 1,\ldots,d$, we have $I_q = \{\sigma(1),\ldots,\sigma(q)\}$. Then since
\[
\#(I_q\cap I_-) = \what L_q^- = \#(T_-\cap \OC 0q_\Z) \text{ for all }q,
\]
we have $I_- = \sigma(T_-)$.

It is readily verified that $\bfB \in \ZZ$ if and only if $X_{i,j} = 0$ for all $i = 1,\ldots,m$ and $j = 1,\ldots,n$ such that
\[
\sigma^{-1}(i) > \sigma^{-1}(m+j).
\]
Thus, $\dim(\ZZ)$ is equal to the number of pairs $(i,j)\in \{1,\ldots,m\}\times\{1,\ldots,n\}$ such that $\sigma^{-1}(i) < \sigma^{-1}(m+j)$, or equivalently the number of pairs $(i_+,i_-) \in T_+\times T_-$ such that $i_+ < i_-$. In other words, $\dim(\ZZ) = \delta(T_+,T_-) \leq \delta$.

For the general case, note that the map $\bfB \mapsto u_{-\bfB} \vert$ is a coordinate chart for the Grassmannian variety $\GG = \GG(d,n)$ of $n$-dimensional subspaces of $\R^d$. So it suffices to show that $\dim(\ZZ') \leq \delta$, where
\begin{equation*}
\ZZ' = \{W\in \GG : \dim(V_q\cap W) \geq \what\dir_q^- \text{ for all }q\}.
\end{equation*}
Let $W$ be a smooth point of $\ZZ'$ (i.e. a point where the tangent space to $\ZZ'$ at $W$ is defined) such that the local dimension of $\ZZ'$ at $W$ is equal to $\dim(\ZZ')$. Then $\dim(V_q\cap W) = \what\dir_q^-$ for all $q$. Moreover, there is a basis of $\R^d$ such that the subspaces $V_q$ ($q = 1,\ldots,d$) and $W$ are all coordinate subspaces with respect to this basis. So from the previous argument, it follows that $\dim(\ZZ'\cap U) = \delta$, where $U$ is a neighborhood of $W$ (depending on the basis). Since the local dimension of $\ZZ'$ at $W$ is equal to $\dim(\ZZ')$, this shows that $\dim(\ZZ') = \delta$.
\end{proof}

Now suppose that the game is played according to Bob's strategy, let $\bfA$ denote the outcome, and suppose that the corresponding successive minima function $\hh_\bfA$ is in $\SS$. By Lemma \ref{lemmafindtemplate}, there exists a template $\gg$ such that $\gg \asymp_\plus \hh_\bfA$. Fix a large constant $C_1 \geq \gamma$. Applying Lemma \ref{lemmafindtemplate} again, there exists a template $\ff$ such that $\ff \asymp_{\plus,C_1} \gg$ and such that for all $q,t,t'$ such that $f_q(t) < f_{q+1}(t)$ and $|t'-t|\leq C_1$, we have $g_{q+1}(t') - g_q(t') \geq C_1$ and $F_q'(t) \geq G_q'(t')$. Since $\hh_\bfA \in \SS$, we have $\ff\in \NN(\SS,C_\epsilon)$ for some constant $C_\epsilon$ depending on $\beta$ and $C_1$ (which will depend on $\epsilon$).
\begin{claim}
\label{claimfgeqgame}
We have
\begin{align*}
\underline\delta(\ff) &\geq \underline\delta - O(1/\log\beta),&
\overline\delta(\ff) &\geq \overline\delta - O(1/\log\beta),
\end{align*}
where $\underline\delta$ and $\overline\delta$ denote Alice's scores (see Definition \ref{definitiongames1}) in the Hausdorff and packing games, respectively.
\end{claim}
\begin{proof}
It suffices to show that for all $k\in\N$ and $t'\in [k\gamma,(k+1)\gamma]$,
\[
\delta(\ff,t') \geq \frac{\log\#(A_k) - O(1)}{-\log(\beta)}\cdot
\]
Indeed, fix such $k,t'$, and let $t = k\gamma$. By Claim \ref{claimAkupper}, we have
\[
\delta(\turn,\bfB_\turn) \geq \frac{\log\#(A_k) - O(1)}{-\log(\beta)},
\]
so to complete the proof it suffices to show that
\[
\delta(\ff,t') \geq \delta(\turn,\bfB_\turn).
\]
Indeed, fix $q = 1,\ldots,d-1$ such that $f_q(t') < f_{q+1}(t')$, and we will show that
\begin{equation}
\label{Lq+Lq+}
L_+(\ff,t',q) \geq \dir_q^+.
\end{equation}
Indeed, first note that by assumption, and since $C_1 \geq \gamma$, the inequality $f_q(t') < f_{q+1}(t')$ implies that $g_{q+1}(t) - g_q(t) \geq C_1$. Now by the definition of $\gg$ and Lemma \ref{lemmaMS}, we have
\[
\gg(t) \asymp_\plus \hh_\bfA(t) \asymp_\plus \hh(\Lambda_\turn)
\]
and thus we in fact get $\log\lambda_{q+1}(\Lambda_\turn) - \log\lambda_q(\Lambda_\turn) \gtrsim_\plus C_1$.

Now let
\[
\bfD_\turn = \sum_{\ell = \turn}^\infty \beta^{\ell-\turn} \bfB_\ell \in B_\MM(\bfB_\turn,\beta) \subset B_\MM(\0,1).
\]
By \eqref{Dqdef}, we have
\[
\sup_{\|\bfC\| \leq \beta} \dim(u_{\bfD_\turn + \bfC} V_q\cap \vert) \leq \dir_q^-.
\]
Thus by Lemma \ref{lemmapushoffvert}, for all $s \geq 0$, we have
\begin{equation}
\label{firsts}
\log\|g_s u_{\bfD_\turn} V_q\| \gtrsim_{\plus,\beta} \log\|V_q\| + \left(\frac{\dir_q^+}{m} - \frac{\dir_q^-}{n}\right)s.
\end{equation}
On the other hand, since $\log\lambda_{q+1}(\Lambda_\turn) - \log\lambda_q(\Lambda_\turn) \gtrsim_\plus C_1$, for all $V_q'\in \VV_q(\Lambda_\turn) \butnot \{V_q\}$, by Lemma \ref{lemmaminkowskibasis2} we have
\[
\log\|V_q'\| - \log\|V_q\| \gtrsim_\plus C_1
\]
and thus for all $0 \leq s \leq \frac{mn}{qd}C_1$, since $\log\|g_s^{-1}\| \leq s/n$, we have
\begin{align*}
\log\|g_s u_{\bfD_\turn} V_q'\|
&\gtrsim_\plus \log\|V_q'\| - \frac qn s
\gtrsim_\plus \log\|V_q\| + C_1 - \frac qn s\\
&\geq_\pt \log\|V_q\| + \frac qm s \geq_\pt \log\|V_q\| + \left(\frac{\dir_q^+}{m} - \frac{\dir_q^-}{n}\right)s.
\end{align*}
Combining with \eqref{firsts} gives
\[
\inf_{V_q'\in\VV_q(\Lambda_k)} \log\|g_s u_{\bfD_\turn} V_q'\| \gtrsim_{\plus,\beta} \log\|V_q\| + \left(\frac{\dir_q^+}{m} - \frac{\dir_q^-}{n}\right)s.
\]
On the other hand, since $\gg\asymp_\plus \hh_\bfA$, by Lemmas \ref{lemmaminkowskibasis} and \ref{lemmaMS}, we have
\begin{align*}
\log\|V_q\| &\asymp_\plus \sum_{i=1}^q \log\lambda_i(\Lambda_k) \asymp_\plus G_q(t),\\
\inf_{V_q'\in\VV_q(\Lambda_k)} \log\|g_s u_{\bfD_\turn} V_q'\| &\asymp_\plus \sum_{i=1}^q \log\lambda_i(g_s u_{\bfD_\turn} \Lambda_k) \asymp_\plus G_q(t+s),
\end{align*}
so
\[
G_q(t+s) \gtrsim_{\plus,\beta} G_q(t) + \left(\frac{\dir_q^+}{m} - \frac{\dir_q^-}{n}\right)s.
\]
Rearranging gives
\[
\int_t^{t+s} G_q' \gtrsim_{\plus,\beta} \left(\frac{\dir_q^+}{m} - \frac{\dir_q^-}{n}\right)s.
\]
Suppose that $G_q' < \frac{\dir_q^+}{m} - \frac{\dir_q^-}{n}$ on $[t,t+s]$. Then since $\gg$ is a template,
\[
G_q' \leq \frac{\dir_q^+ - 1}{m} - \frac{\dir_q^- + 1}{n} \text{ on } [t,t+s]
\]
and thus
\[
\left(\frac{\dir_q^+ - 1}{m} - \frac{\dir_q^- + 1}{n}\right)s \gtrsim_{\plus,\beta}\left(\frac{\dir_q^+}{m} - \frac{\dir_q^-}{n}\right)s
\]
which implies $s \asymp_{\plus,\beta} 0$, i.e. $|s| \leq C_2$ for some constant $C_2$. Let $C_1,s$ be chosen so that $C_2 < s \leq \min(C_1,\frac{mn}{qd}C_1)$ and $\tbeta \leq C_1$. Then the inequality $|s| \leq C_2$ contradicts the definition of $s$, so the hypothesis that $G_q' < \frac{\dir_q^+}{m} - \frac{\dir_q^-}{n}$ on $[t,t+s]$ must be incorrect, i.e. we must have $G_q'(t'') \geq \frac{\dir_q^+}{m} - \frac{\dir_q^-}{n}$ for some $t''\in [t,t+s]$. Now since $t',t''\in [t,t+C_1]$, we have $|t''-t'| \leq C_1$, and thus by our assumptions on $\ff$ we have
\[
\frac{\dir_+(\ff,t',q)}{m} - \frac{\dir_-(\ff,t',q)}{n} = F_q'(t') \geq G_q'(t'') \geq \frac{\dir_q^+}{m} - \frac{\dir_q^-}{n}
\]
demonstrating \eqref{Lq+Lq+}.

To summarize, we have
\[
\#\big(S_+(\ff,t')\cap \OC 0q_\Z\big) \geq \#\big(S_+(\turn,\bfB_k)\cap \OC 0q_\Z\big)
\]
for all $q$ such that $f_q(t') < f_{q+1}(t')$. It follows from \eqref{Splusdef1} that the same inequality holds for all $q=1,\ldots,d-1$. Since
\[
\delta(T_+,T_-) = \sum_{q=1}^{d-1} \#\big(T_+\cap \OC 0q\big) - \binom m2,
\]
(where $\delta$ is as in \eqref{deltaT+T-}), we have
\[
\delta(\ff,t') = \delta(S_\pm(\ff,t')) \geq \delta(S_\pm(\turn,\bfB_k)) = \delta(\turn,\bfB_\turn).
\qedhere\]
\end{proof}

Fix $\epsilon > 0$ and let $\delta = \sup_{\ff\in \NN(\SS,C_\epsilon)\cap \TT_{m,n}}\underline\delta(\ff) + \epsilon$. Then by the previous Claim \ref{claimfgeqgame}, we have $\delta > \underline\delta$ as long as $\beta$ is sufficiently small. So by Theorem \ref{theoremHPgame}, we have $\delta \geq \HD(\DD(\SS))$. Since $\delta$ was arbitrary, we have
\[
\HD(\DD(\SS)) \leq \sup_{\ff\in \NN(\SS,C_\epsilon)\cap \TT_{m,n}}\underline\delta(\ff) + \epsilon.
\]
A similar argument gives the bound for the packing dimension, thereby completing the proof of the upper bounds in Theorem \ref{theoremvariationaluniform}.

\draftnewpage
\part{Appendix and references}
\label{partappendix}
\appendix
\section{Translating between Schmidt--Summerer's notation and ours}
\label{appendix}

This appendix explains the relations between certain concepts and notation in our paper and in Schmidt--Summerer's \cite{SchmidtSummerer3} to provide a guide for readers of both. 

Schmidt--Summerer are working in the framework of simultaneous approximation, so $n=1$ for them, and further: their $n$ is our $d=m+1$, their $y$ is our $r$, their $\xi$ is our $A$. In particular, note that they have $r = (q,p)$ instead of $r = (p,q)$. Their $\Lambda(\xi)$ would translate to $u_A \Z^d$ in our paper, and what they call $\KK(Q)$ is what we would call $g_{-t} B$, where $Q = e^t$ and $B = [-1,1]^d$. Finally, their $T$ is our $g_{-1}$.

Schmidt--Summerer's set-up encodes the same geometric information as ours since 
\[
\lambda_i(g_t u_A \Z^d, B) = \lambda_i(u_A \Z^d, g_{-t} B) .
\] 
Therefore, in their notation the right-hand side is $\lambda_i(\Lambda(\xi),\KK(Q))$. Similarly, $L_i(q)$ in their notation
is the same as $h_i(t)$ in our notation, where $q = t/(n-1)$.
The connection between our notion of a {\it template} (see Definition \ref{definitiontemplate}) and Schmidt--Summerer's {\it $(n,\gamma)$-systems} (see \cite[\62]{SchmidtSummerer3}) is as follows: if $P$ is an $(n,0)$-system then 
\[
\hh(t) = \frac{n}{n-1} P(t) - \frac{t}{n-1}\] 
is an $(n-1) \times 1$ template.

We further remark that after Schmidt--Summerer consider the limiting case of an $(n,0)$-system in \cite[\63]{SchmidtSummerer3}, they go on, in \cite[\64, pg. 62]{SchmidtSummerer3}, to conjecture that the study of these systems should suffice to determine the spectra of the family of exponents of approximation that they are interested in. The rest of their paper develops a theory of covers of an $(n,\gamma)$-system, which is then applied to prove relations between several exponents of approximation.

Interested readers are also referred to Roy's paper \cite{Roy3} for translating between Schmidt--Summerer's notation and his. In contrast to Schmidt--Summerer who work in the simultaneous approximation framework, Roy works in the dual framework of approximation by linear forms. Roy defines the notion of a {\it rigid system} (a special case of $(n,0)$-systems) in the introduction of \cite{Roy3} and goes on to prove that every $(n,\gamma)$-system can be approximated by a rigid system up to bounded additive difference (see \cite[Theorem 1.3]{Roy3}). Roy's rigid systems translate to our \emph{$\tspace$-integral templates} (see Definition \ref{definitionintegral}).

\hrulefill

{\bf Data availability statement}: Data sharing not applicable to this article as no datasets were
generated or analysed during the current study.


\bibliographystyle{amsplain}

\bibliography{bibliography}

\end{document}